\documentclass[letterpaper,12pt]{report}

\usepackage{pifont}
\usepackage{tabularx} 
\usepackage{amsmath,environ}  
\usepackage{graphicx} 
\usepackage[margin=0.9in,letterpaper]{geometry} 
\usepackage{cite} 
\usepackage[final]{hyperref} 
\usepackage{cleveref}
\hypersetup{
	colorlinks=true,       
	linkcolor=blue,        
	citecolor=blue,        
	filecolor=magenta,     
	urlcolor=blue         
}
\usepackage[T1]{fontenc} 
\usepackage{textcomp}    
\usepackage{lmodern}     
\usepackage{fix-cm}      

\usepackage{enumitem}
\usepackage{graphicx} 
\usepackage{amsfonts}
\usepackage{amssymb}
\usepackage{comment}

\usepackage{amsthm}
\usepackage{subfigure}
\usepackage{grffile}
\usepackage{multicol}
\usepackage{algorithm,algorithmic}
\usepackage{empheq}
\usepackage{mathtools}
\usepackage{xcolor}
\usepackage{tikz}
\usetikzlibrary{shapes.geometric, arrows}
\tikzstyle{process} = [rectangle, minimum width=1cm, minimum height=1cm, text centered, draw=black, fill=orange!30]
\tikzstyle{arrow} = [thick,->,>=stealth]
\usetikzlibrary{decorations.pathreplacing}
\usepackage[framemethod=tikz]{mdframed}

\usepackage{blindtext}

\usepackage[scr = dutchcal]{mathalpha}
\DeclareMathOperator{\trace}{trace}

\newcommand{\mcl}[1]{\mathcal{ #1}}
\newcommand{\mbf}[1]{\mathbf{ #1}}

\newcommand{\norm}[1]{\Vert #1\Vert}

\newcommand{\hinf}{\ensuremath{H_{\infty}}}
\newcommand{\ip}[2]{\left\langle{#1},{#2}\right\rangle}
\renewcommand{\th}{\ensuremath{\theta}}

\newcommand{\bmat}[1]{\begin{bmatrix} #1\end{bmatrix}}

\newcommand{\mat}[1]{\begin{matrix}#1\end{matrix}}

\newcommand{\sbmat}[1]{\left[\begin{smallmatrix}
		#1\end{smallmatrix} \right]}

\newcommand{\R}{\mathbb{R}}
\newcommand{\C}{\mathbb{C}}

\newcommand{\N}{\mathbb{N}}

\newcommand{\myint}{\int_{a}^{b}}
\newcommand{\myinta}[1]{\int_{a}^{#1}}
\newcommand{\myintb}[1]{\int_{#1}^{b}}

\newcommand{\precceq}{\preccurlyeq}

\let\bl\bigl
\let\bbl\Bigl
\let\bbbl\biggl
\let\bbbbl\Biggl
\let\br\bigr
\let\bbr\Bigr
\let\bbbr\biggr
\let\bbbbr\Biggr

\newtheorem{thm}{Theorem}
\newtheorem{defn}[thm]{Definition}
\newtheorem{lem}[thm]{Lemma}

\newtheorem{cor}[thm]{Corollary}

\newtheorem{ex}{\textbf{Example}}

\newenvironment{boxEnv}[1]
  {\mdfsetup{
    frametitle={\colorbox{red!30}{#1\space}},
    innertopmargin=10pt,
    frametitleaboveskip=-\ht\strutbox,
    roundcorner = 5pt,
    backgroundcolor = yellow!30,
    outerlinecolor = blue!50!black,
    }
  \begin{mdframed}%
  }
  {\end{mdframed}}
  \newenvironment{Statebox}[1]
  {\mdfsetup{
    frametitle={\colorbox{white}{#1\space}},
    innertopmargin=5pt,
    frametitleaboveskip=-\ht\strutbox,
    roundcorner = 0pt,
    backgroundcolor = black!2,
    outerlinecolor = black,
    }
  \begin{mdframed}%
  }
  {\end{mdframed}}

\newcounter{codebox}

\newenvironment{codebox}[1][]{%
    \refstepcounter{codebox}
  \mdfsetup{
    frametitle={\colorbox{blue!30}{Code Block \thecodebox\space}},
    innertopmargin=10pt,
    frametitleaboveskip=-\ht\strutbox,
    roundcorner = 5pt,
    backgroundcolor = green!20,
    outerlinecolor = blue!50!black,
    }
  \begin{mdframed}%
  }
  {\end{mdframed}}

\NewEnviron{codeMLB}{%
\fontsize{11}{13}
  \begin{flalign*}\tt
    \BODY
  \end{flalign*}
}

\newenvironment{matlab}
  {\par\noindent\normalfont\par\nopagebreak%
  \begin{mdframed}[frametitle=,
     linewidth=1pt,
     linecolor=black,
     bottomline=false,topline=false,rightline=false,
     innerrightmargin=0pt,innertopmargin=0pt,innerbottommargin=0pt,
     innerleftmargin=1em,
     skipabove=.5\baselineskip
   ] \fontsize{11}{13} \tt}
  {\end{mdframed}}

\newcommand{\fourpi}[4]{\text{$\mcl{P}${\footnotesize $\bmat{#1,& \hspace{-3mm}#2 \\ #3,& \hspace{-3mm} \left\{#4\right\}}$}}}
\newcommand{\threepi}[1]{\mcl{P}_{\{#1\}}}

\allowdisplaybreaks

\begin{document}
	\vfill
	\title{ \includegraphics[width=0.35\linewidth]{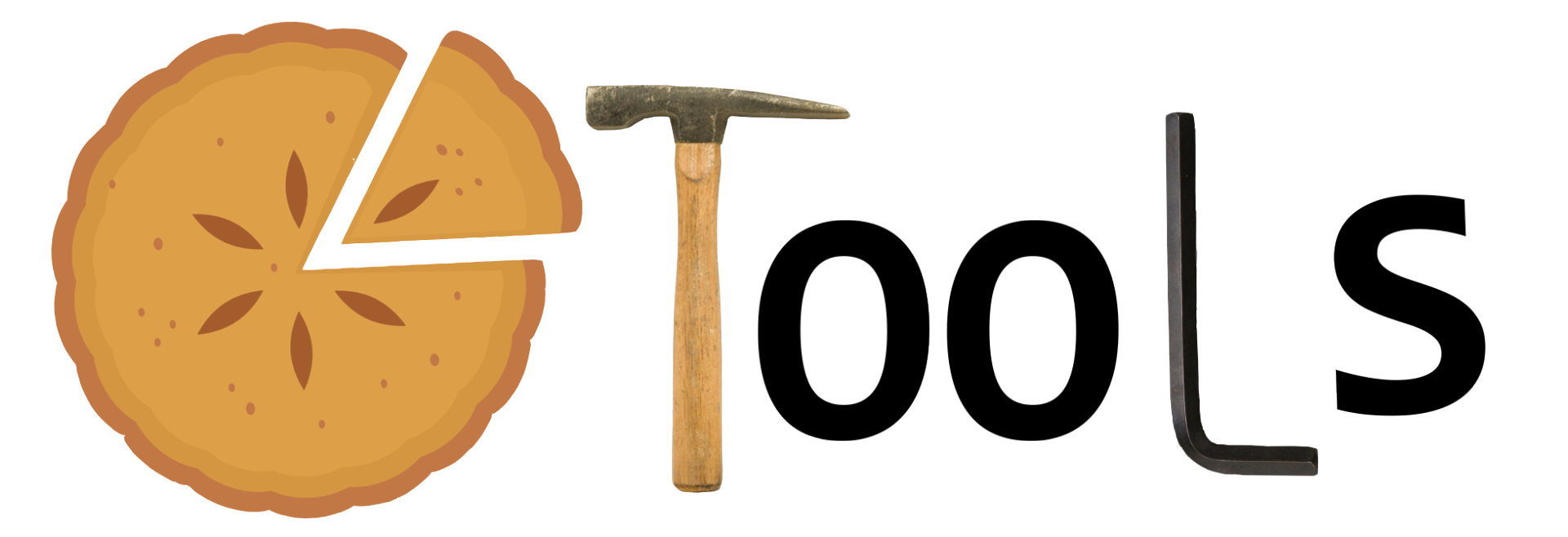}\\
		PIETOOLS 2025: User Manual}
	\author{S. Shivakumar \and A. Das \and D. Braghini \and D. Jagt \and Y. Peet \and M. Peet}
	\date{\today}
	\maketitle

	\chapter*{Copyrights and license information}
	PIETOOLS is free software: you can redistribute it and/or modify
	it under the terms of the GNU General Public License as published by
	the Free Software Foundation, either version 3 of the License, or
	(at your option) any later version.
	
	This program is distributed in the hope that it will be useful,
	but WITHOUT ANY WARRANTY; without even the implied warranty of
	MERCHANTABILITY or FITNESS FOR A PARTICULAR PURPOSE.  See the
	GNU General Public License for more details.
	
	You should have received a copy of the GNU General Public License along with this program; if not, write to the Free Software Foundation, Inc., 59 Temple Place, Suite 330, Boston, MA  02111-1307  USA. 
	\chapter*{Notation}
	\begin{tabular}{l l}
		$\R$ &Set of real numbers $(-\infty,\infty)$\\
		$\partial_s^i \mbf{x}$& $\frac{\partial^i\mbf{x}}{\partial s^i}$ where $s$ is in a compact subset of $\R$\\
		$\dot{\mbf{x}}$& $\frac{\partial\mbf{x}}{\partial t}$ where $t$ is in $[0,\infty)$\\
		$L_2^n[a,b]$& Set of Lebesgue-integrable functions from $[a,b]\to \R^n$\\
		$RL^{m,n}[a,b]$& $\R^m\times L_2^n[a,b]$\\
		$H_k^n[a,b]$& $\{f \in L_2^n[a,b]\mid \partial_s^if \in L_2^n[a,b] \forall i\le k\}$\\
		$0_{m\times n}$& Zero matrix of dimension $m\times n$\\
		$0_n$& Zero matrix of dimension $n\times n$\\
		$I_n$& Identity matrix of dimension $n\times n$\\
		$\Delta_a$ & Dirac operator on $f:C\to X$, $\Delta_a(f) = f(a)$, for $a\in C$\\
		$\mcl{B}(X,Y)$ & Space of bounded linear operators from $X$ to $Y$
	\end{tabular}
	
	\newpage

	\tableofcontents

	\chapter{About PIETOOLS}\label{ch:about}
	PIETOOLS is a free MATLAB toolbox for manipulating Partial Integral (PI) operators and solving Linear PI Inequalities (LPIs), which are convex optimization problems involving PI variables and PI constraints. 
	PIETOOLS can be used to:
	\begin{itemize}
		\item define PI operators in 1D and 2D
		\item declare PI operator decision variables (positive semidefinite or indefinite)
		\item add operator inequality constraints
		\item solve LPI optimization problems
	\end{itemize}  
	The interface is inspired by YALMIP and the program structure is based on that used by SOSTOOLS. By default the LPIs are solved using SeDuMi~\cite{sedumi}, however, the toolbox also supports use of other SDP solvers such as Mosek, sdpt3 and sdpnal.
	
	To install and run PIETOOLS, you need:
	\begin{itemize}
		\item MATLAB version 2014a or later (we recommend MATLAB 2020a or higher. Please note some features of PIETOOLS, for example PDE input GUI, might be unavailable if an older version of MATLAB is used)
		\item The current version of the MATLAB Symbolic Math Toolbox (This is installed in most default versions of Matlab.)
		\item An SDP solver (SeDuMi is included in the installation script.)
	\end{itemize}

\section{PIETOOLS 2025 Release Notes} 
PIETOOLS 2025 provides several improvements, including more accurate and sophisticated stability tests, the ability to define PIEs with infinite-dimensional inputs and outputs, features for robust analysis and control, along with enhanced 2D functionality.

\paragraph{Major Updates:}
\begin{enumerate}
\item Several new notions of PDE stability have been defined and implemented. The default stability test is now PIE to PDE stability, which is the weakest of these notions. 
\begin{itemize}
\item Scripts now exist to test Lyapunov and exponential in both the PIE to PDE, PDE to PDE sense. 
\end{itemize}
\item PIE and PDE objects can now have infinite-dimensional inputs and outputs.
\begin{itemize}
\item $L_2$-gain analysis routines now support infinite-dimensional inputs and outputs.
\item PIESIM supports infinite-dimensional inputs and outputs in 1D (made after PIETOOLS 2025 release)
\end{itemize}
\item A test for well-posedness of univariate PDEs has been included and conversion scripts updated.
\item PIESIM allows for simulation of 2D PDEs
\item Finite-dimensional disturbances and control inputs can be entered as user-defined arrays in PIESIM (made after PIETOOLS 2025 release)
\end{enumerate}

\paragraph{Minor Updates}

\begin{enumerate}
\item The ability to utilize second order time derivatives in the command line interface has been modified.
\item Updated \texttt{lpigetdegrees}
\end{enumerate}

\section{Installing PIETOOLS}

PIETOOLS 2025 is compatible with Windows, Mac or Linux systems and has been verified to work with MATLAB version 2020a or higher, however, we suggest to use the latest version of MATLAB.
	
Before you start, \textbf{make sure} that you have 
\begin{enumerate}
	\item MATLAB with version 2014a or newer. (MATLAB 2020a or newer for GUI input)
	\item MATLAB has permission to create and edit folders/files in your working directory.
\end{enumerate}

\subsection{Installation}
PIETOOLS 2025 can be installed in two ways.
\begin{enumerate}
	\item \textbf{Using install script:} The script installs the following files --- tbxmanager (skipped if already installed), SeDuMi 1.3 (skipped if already installed), SOSTOOLS 4.00 (always installed), PIETOOLS 2025 (always installed). Adds all the files to MATLAB path.
	\begin{itemize}
		\item Go to \url{https://github.com/CyberneticSCL/PIETOOLS} or \\\url{control.asu.edu/pietools/}. 
		\item Download the file \textbf{pietools\_install.m} and run it in MATLAB.
		\item \textbf{Run the script from the folder it is downloaded in to avoid path issues.}
	\end{itemize}
	\item \textbf{Setting up PIETOOLS 2025 manually:}
	\begin{itemize}
		\item Download and install C/C++ compiler for the OS. 
		\item Install an SDP solver. SeDuMi can be obtained from \href{http://sedumi.ie.lehigh.edu/?page_id=58}{\texttt{this link}}.
		\item Download SeDuMi and run \textbf{install\_sedumi.m} file. 
		\begin{itemize}
			\item Alternatively, install MOSEK, obtain license file and add to MATLAB path.
		\end{itemize}
		\item Download \textbf{PIETOOLS\_2025.zip} from \href{http://control.asu.edu/pietools/pietools.html}{\texttt{this link}}, unzip, and add to MATLAB path.
		\end{itemize}
\end{enumerate}

\chapter{Scope of PIETOOLS}\label{ch:scope}

In this chapter, we briefly motivate the need for a new computational tool for the analysis and control of ODE-PDE systems as well as DDEs. We lightly touch upon, without going into details, the class of problems PIETOOLS can solve. 

\section{Motivation}
	Semidefinite programming (SDP) is a class of optimization problems that involve the optimization of a linear objective over the cone of positive semidefinite matrices. The development of efficient interior-point methods for SDP problems made LMIs a powerful tool in modern control theory. As Doyle stated in \cite{doyle}, LMIs played a central role in postmodern control theory akin to the role played by graphical methods like Bode plots, Nyquist plots, etc., in classical control theory. However, most of the applications of LMI techniques were restricted to finite-dimensional systems, until the sum-of-squares method came into the limelight. The sum-of-squares (SOS) optimization methods found application in control theory, for example searching for Lyapunov functions or finding bounds on singular values. 
    This gave rise to many toolboxes such as SOSTOOLS~\cite{sostools}, SOSOPT~\cite{sosopt}, etc., that can handle SOS polynomials in MATLAB. However, unlike the LMI methods for linear ODEs, SOS methods for analysis and control of PDEs still required ad-hoc interventions. For example, searching for a Lyapunov function that certifies stability of a PDE, one usually hits a roadblock in the form of boundary conditions, requiring the use of e.g. integration by parts, Poincar\'e inequality, or H\"older's Inequality to resolve.
	
	In an ideal world, we would prefer to define a PDE, specify the boundary conditions and let a computational tool take care of the rest. To resolve this problem, either we teach a computer to perform these ``ad-hoc'' interventions 
 or come up with a method that does not require such interventions, to begin with. To achieve the latter, we developed the Partial Integral Equation (PIE) representation of PDEs, which is an alternative representation of dynamical systems, parameterized by Partial Integral (PI) operators. The PIE representation can be used to represent a broad class of distributed parameter systems and is algebraic -- eliminating the use of boundary conditions and continuity constraints~\cite{shivakumar_2019CDC}, \cite{das_2019CDC}. Consequently, many LMI-based methods for analysis of finite-dimensional systems can be extended to infinite-dimensional ones using PIEs. The PIETOOLS software offers the tools to do exactly that -- making e.g. stability analysis, controller synthesis, and simulation of linear infinite-dimensional systems as intuitive as it is for finite-dimensional ones.
 

\section{PIETOOLS for Analysis, Control, and Simulation of ODE-PDE Systems}
Using PIETOOLS 2025 for controlling and simulating ODE-PDE models has been made intuitive, 
requiring little knowledge of the mathematical details on PIE operators. To illustrate, let us see how we can simulate a simple ODE-PDE model, and synthesize a controller.
	
 \subsection{Defining Models in PIETOOLS}
Any control problem necessarily starts with declaring the model, and PIETOOLS makes it extremely simple to do so. Suppose that we are interested in modeling a coupled ODE-PDE system, such as a system with ODE dynamics given by
\begin{align}
 \label{part1:ode}
\dot{x}(t) = -x(t)+ u(t),
\end{align}
with controlled input $u$, and PDE dynamics given by a one-dimensional wave equation 
\begin{align}
    \label{part1:pde0}
    \ddot{\mbf{x}}(t,s) = c^2  \partial_s^2 \mbf x(t,s) -b \partial_s \mbf x(t,s) + s w(t),\qquad s\in (0,1),~t\geq 0,
\end{align}
with velocity $c$, added viscous damping coefficient $b$ and external disturbance $w$. Since the PDE has a second-order derivative in time, we should make a change of variables to appropriately define a state space. For this example, we do so by introducing $\phi = (\partial_s \mbf{x}, \dot{\mbf x} )$, so that the dynamics of $\phi$ are governed by
\begin{align}
    \label{part1:pde}
    \dot{\phi}(t,s) = \bmat{0 & 1 \\ c & 0}  \partial_s \phi (t,s) +\bmat{0 & 0\\0 & -b} \phi (t,s) + \bmat{0\\s} w(t), s\in (0,1), t\geq 0.
\end{align}
We can also add a regulated output to our system, for instance
\begin{align}
    \label{part1:output}
    z(t) = \bmat{r(t) \\ u(t)}=\bmat{\mbf{x}(t,1)-\mbf{x}(t,0)\\ u(t)}=\bmat{\int_0^1 \phi_{1}(t,s)ds\\ u(t)}.
\end{align}
To define the presented model in MATLAB, we first declare the spatial and temporal variable $s$ and $t$, and create the state, input, and output variables, using the \texttt{pvar} and \texttt{pde\_var} functions:
\begin{matlab}
\begin{verbatim}
>> pvar s t
>> x = pde_var();      phi = pde_var(2,s,[0,1]);
>> w = pde_var(`in');  u = pde_var(`control');
>> z = pde_var(`out',2); 
\end{verbatim}\end{matlab}
Then, we can define an equation in terms of these variables as a \texttt{pde\_struct} object, using standard operators such as \texttt{`+',`-',`*',`diff',`subs',`int'}, etc., declaring e.g. our wave equation for $c=1$ and $b=0.01$ as:
\begin{matlab}
\begin{verbatim}
>> b = 0.01;    c = 1;
>> eq_dyn = [diff(x,t,1)==-x+u;
             diff(phi,t,1)==[0 1; c 0]*diff(phi,s,1)+[0;s]*w+[0 0;0 -b]*phi];
>> eq_out= z==[int([1 0]*phi,s,[0,1]); u];
>> odepde = [eq_dyn;eq_out]; 
\end{verbatim}
\end{matlab}

    
 
\subsection{Declaring Boundary Conditions}
A general PDE model is incomplete without boundary conditions, but in PIETOOLS, boundary conditions can be declared in much the same way as the system dynamics. For example, to declare the following Dirichlet and Neumann boundary conditions,
\[
 \dot{\mbf x} (t, s= 0)= 0,\qquad \partial_s \mbf x (t, s= 1)= x(t),
\]
we can simply call
\begin{matlab}
>> bc1 = [0 1]*subs(phi,s,0) == 0;\\
>> bc2 = [1 0]*subs(phi,s,1) == x;\\
>> odepde = [odepde;[bc1;bc2]];
\end{matlab}
Once the full system is declared, we clean up the structure and fill in any gaps by calling the function \texttt{initialize} as
\begin{matlab}
 >> odepde = initialize(odepde);
 \end{matlab}
Whenever equations are successfully initialized, a summary of the encountered states, inputs, and outputs is displayed in the Command Window. To verify if PIETOOLS got the right equations, the user just needs to type the name of the system variable (''\texttt{odepde}'' in this example) in the command window without a semicolon for PIETOOLS to display the added equations. We encourage the user to always check the equations before proceeding.

\subsection{Simulating ODE-PDE Models}
Having declared an ODE-PDE system, one of the first things a practitioner may do is to simulate the system. If you look at traditional PDE literature, a big challenge in simulation of PDEs is that every different kind of PDE requires different techniques to discretize. In PIETOOLS, however, there is only one command to simulate any linear ODE-PDE coupled system of your choice:
\begin{matlab}
    >> solution = PIESIM(odepde, opts, uinput);
\end{matlab}
Here, aside from the desired system to simulate (our \texttt{odepde}), we have to declare two additional arguments. The first are the options for simulation, determining e.g. the number of spatial and temporal points in discretization of the solution. For our illustration, we pass the following options to the function, informing PIESIM not to automatically plot the solution, to use 8 Chebyshev polynomials in expanding the solution in space, and to simulate the solutions up to $t=12$ seconds with a time-step of $3\cdot 10^{-2}$ seconds:
\begin{matlab}
    >> opts.plot = 'no';\\ 
    >> opts.N = 8;\\        
    >> opts.tf = 12;\\
    >> opts.dt = 3*1e-2;
\end{matlab}
%
%
Next, in order to simulate solutions, we must of course also pass values for all of the inputs, as well as initial values of the states, which is all done with the input \texttt{uinput}. In our case, we wish to simulate the zero-state response of the system, perturbed by an exponentially decaying sinusoidal signal, which we specify as
\begin{matlab}
    >> syms st sx;\\
    >> uinput.ic = [0,0,0];\\
    >> uinput.u = 0;\\
    >> uinput.w = sin(5*st)*exp(-st); 
\end{matlab}
Note that we set the control input to zero to simulate an open-loop response. 

For more details on the PIESIM input arguments, we refer the reader to Chapter~\ref{ch:PIESIM}. With the output structure \texttt{solution}, PIESIM returns discretized time-dependent arrays corresponding to the time vector used in the simulations and the resulting state variables and output. The simulated evolution of the second PDE state variable $\phi_{2}(t)=\dot{\mbf{x}}(t)$ and regulated output $r(t)$ for our example is depicted in Figure \ref{fig:Ch2_OL}.

\begin{figure}[htbp]
    \centering
    \includegraphics[width=\textwidth]{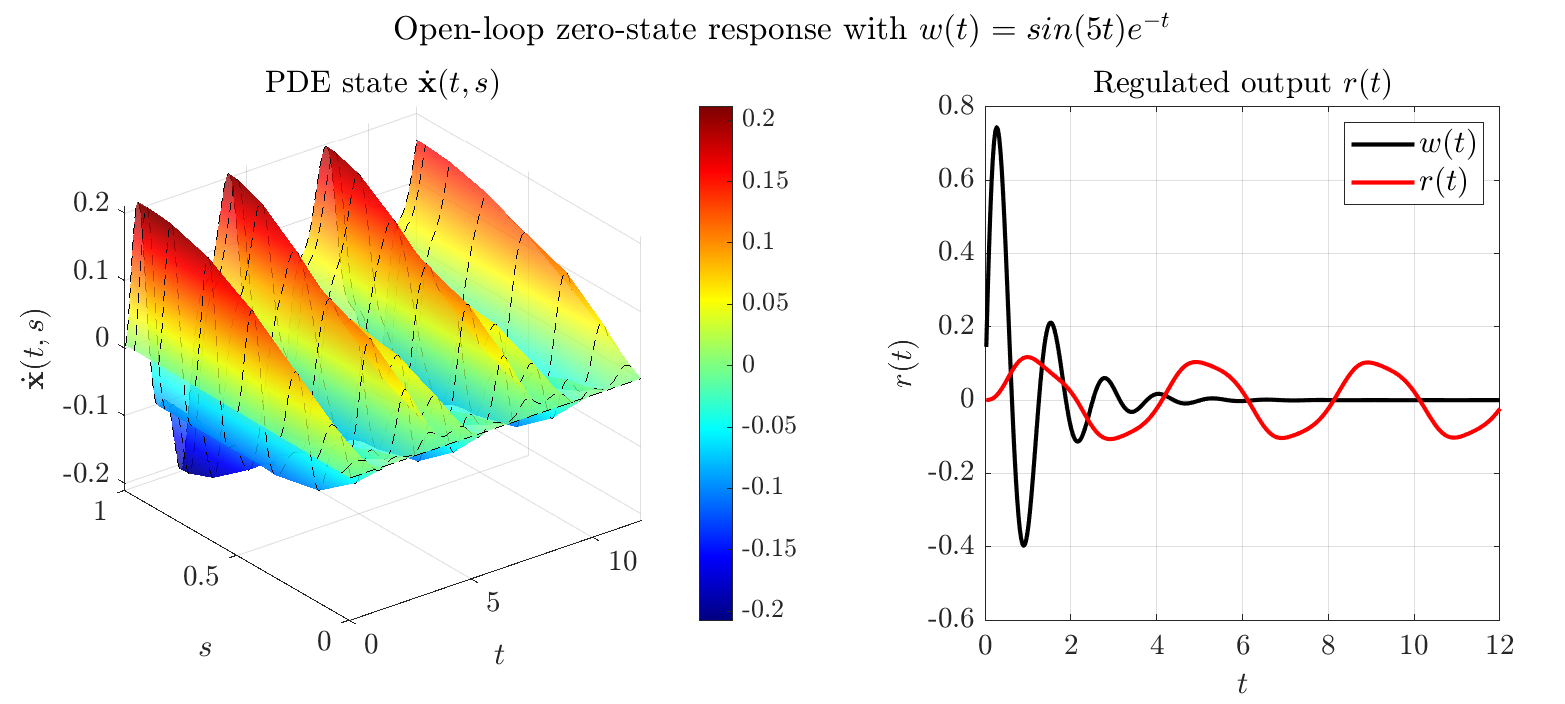}
    \caption{Transient response of the state variable $\dot{\mbf{x}}(t,s)$ and regulated output $r(t)$ by simulating the ODE-PDE model \eqref{part1:ode} and \eqref{part1:pde} with $u(t) = 0$ for external disturbance $w(t) = sin(5t) e^{-t}$.}
    \label{fig:Ch2_OL}
\end{figure}

\subsection{Analysis and Control of the ODE-PDE Model Using PIEs}
Apart from simulation, you may be interested in knowing whether the model is internally stable or not. Moreover, what would be a good control input such that the effect of external disturbances for a specific choice of output can be suppressed? In PIETOOLS, such analysis and synthesis tasks are typically performed by first converting the ODE-PDE model to a new representation called a Partial Integral Equation (PIE), which is parametrized by a special class of operators, and then solving convex optimization problems (see Chapter.~\ref{ch:LPIs} for more details). 

To illustrate, for our example (as for any model declared in PIETOOLS), conversion to a PIE is simply done by calling the function \texttt{convert} as
\begin{matlab}
    >> PIE = convert(odepde,`pie');  
    \begin{verbatim}
    --- Reordering the state components to allow for representation as PIE ---

    The order of the state components x has not changed.

    --- Converting the PDE to an equivalent PIE --- 

    --- Conversion to PIE was successful --- \end{verbatim}\end{matlab}
Note here that, although conversion to a PIE is done automatically, the user should be aware that the order of e.g. state variables may be changed in the process. PIETOOLS will always display a message informing the user of such changes, and in this case, we find that no re-ordering is performed.

Once the model is converted to a PIE, analysis, and control can be performed by calling one of the executive functions. 
There are numerous executive functions available, including for stability analysis, for computing norms of the system, and for performing optimal estimator and controller synthesis -- see Chapter~\ref{ch:LPI_examples}. 

For our ODE-PDE model defined by~\eqref{part1:ode} and~\eqref{part1:pde}, it can be proven that the system is asymptotically stable only when $b > 0$. We can verify stability for the value $b=0.01$ by calling the stability executive for the PIE representation of our ODE-PDE model as follows
\begin{matlab}
    >> settings = lpisettings(`heavy');\\
    >> lpiscript(PIE,`stability',settings);
\end{matlab}
Here, we must also declare settings used for running the stability test, which must be specified as one of the following: \texttt{extreme}, \texttt{stripped}, \texttt{light}, \texttt{heavy}, \texttt{veryheavy}, or \texttt{custom}. For details on the optimization settings, the reader is referred to section.~\ref{sec:executives-settings}. 

For our system, PIETOOLS is able to successfully solve the stability program, and it will inform the user of this fact by displaying the following output:
\begin{matlab}
    \begin{verbatim}
        The System of equations was successfully solved.\end{verbatim}
\end{matlab}

Although our system is indeed found to be stable, our simulation results (Fig.~\ref{fig:Ch2_OL}) show that solutions do not converge to zero very quickly, continuing to oscillate long after the disturbance has already vanished. One way to quantify stability of the system is with the $H_{\infty}$ norm or $L_{2}$-gain of the system, measuring the worst-case amplification of the ``energy`` of the regulated output over that of the disturbance. To compute an upper bound $\gamma$ on the value of this $H_{\infty}$ norm, we can run the corresponding executive for the PIE representation of our system as
\begin{matlab}
    >> [$\sim$,$\sim$, gam] = lpiscript(PIE,'l2gain',settings);
\end{matlab}
If successful, PIETOOLS will display the obtained bound on the $H_{\infty}$ norm as follows
\begin{matlab}
    \begin{verbatim}
    The H-infty norm of the given system is upper bounded by:
    51.5744 \end{verbatim}
\end{matlab}
The obtained bound suggest that the $H_{\infty}$ norm of the system is quite large. To improve the system's rejection of disturbances, we therefore design a state-feedback controller that provides a control input $u(t)$ which minimizes the $H_{\infty}$ norm of the closed-loop system, provided that such a controller exists. To synthesize this state feedback controller we can call yet another executive:
 \begin{matlab}
    >> [$\sim$, Kval, gam\_val] = lpiscript(PIE,'hinf-controller',settings);
\end{matlab}
which will make PIETOOLS search for an operator $\mcl K$ (stored in variable \texttt{Kval}) corresponding to the optimal controller gain, and display the closed-loop $\hinf$ norm if successful. For this example, the resulting controller substantially improves performance, achieving an $H_{\infty}$ norm of the closed-loop system of just $0.8183$
 \begin{matlab}
    \begin{verbatim}
    The closed-loop H-infty norm of the given system is upper bounded by:
    0.8183
    \end{verbatim}
\end{matlab}
 The controller is generally a 4-PI linear operator, as detailed in Chapter.~\ref{ch:PIs}, which has an image parameterized by matrix-valued polynomials. The resulting controller can be displayed by entering its variable name in the command window. Keep in mind that PIETOOLS disregards the monomials with coefficients lower than an accuracy defaulted to $10^-4$. 

Using PIESIM, we can also simulate the response of the resulting closed-loop system. The simulated evolution of the PDE state $\dot{\mbf{x}}(t)$ and regulated output $r(t)$ are plotted in Figure~\ref{fig:Ch2_CL}, along with the feedback control effort $u(t)$ used to achieve this response. The regulated output response of the open- and closed-loop system are displayed together in Figure~\ref{fig:Ch2_output}, showing that the imposed feedback indeed makes the system ``more stable'', driving the output to 0 substantially faster.

\begin{figure}[htbp]
    \centering
    \includegraphics[width=\textwidth]{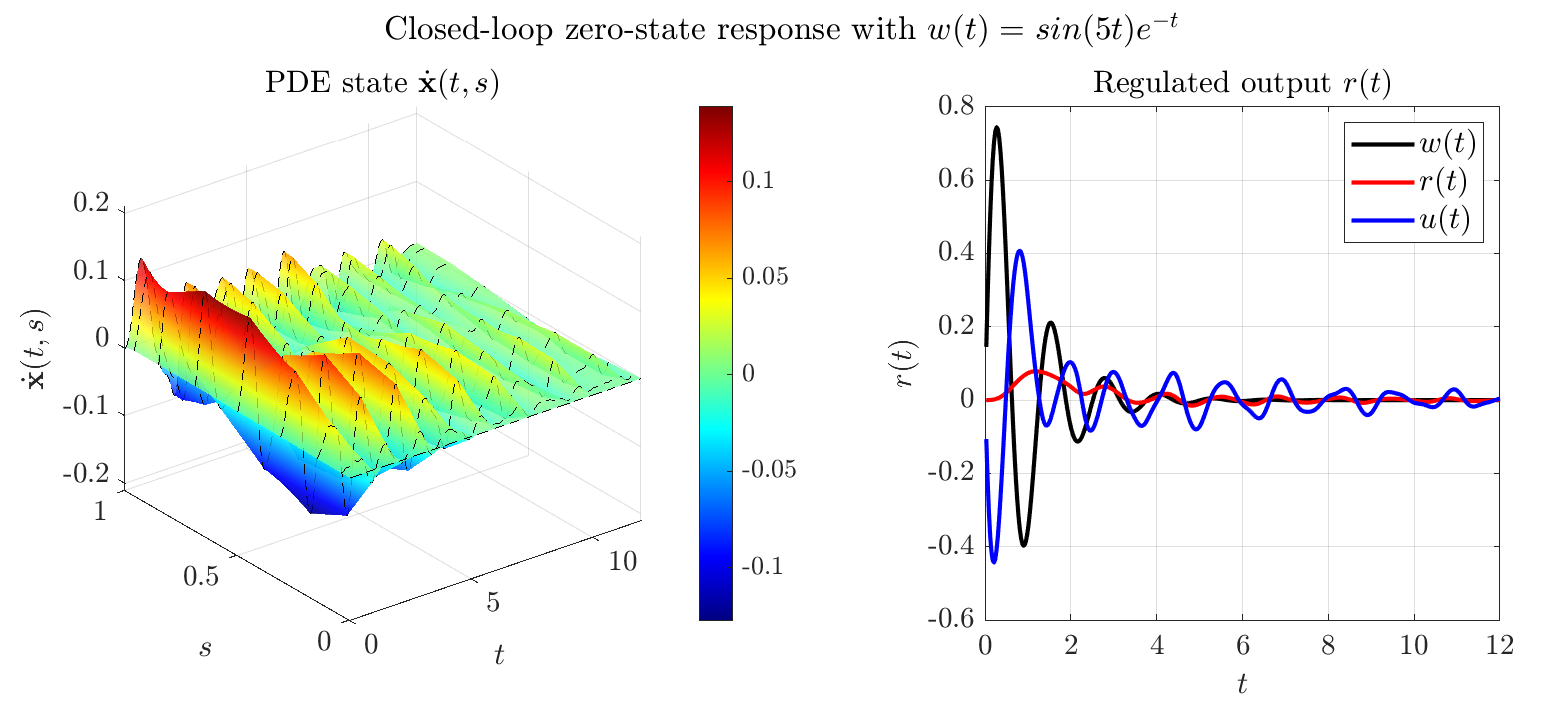}
    \caption{Transient response of the state variable $\dot{\mbf{x}}(t,s)$ and regulated output $r(t)$ on the closed-loop system for external disturbance $w(t) = sin(5t) e^{-t}$.}
    \label{fig:Ch2_CL}
\end{figure}

\begin{figure}[htbp]
    \centering
    \includegraphics[width=\textwidth]{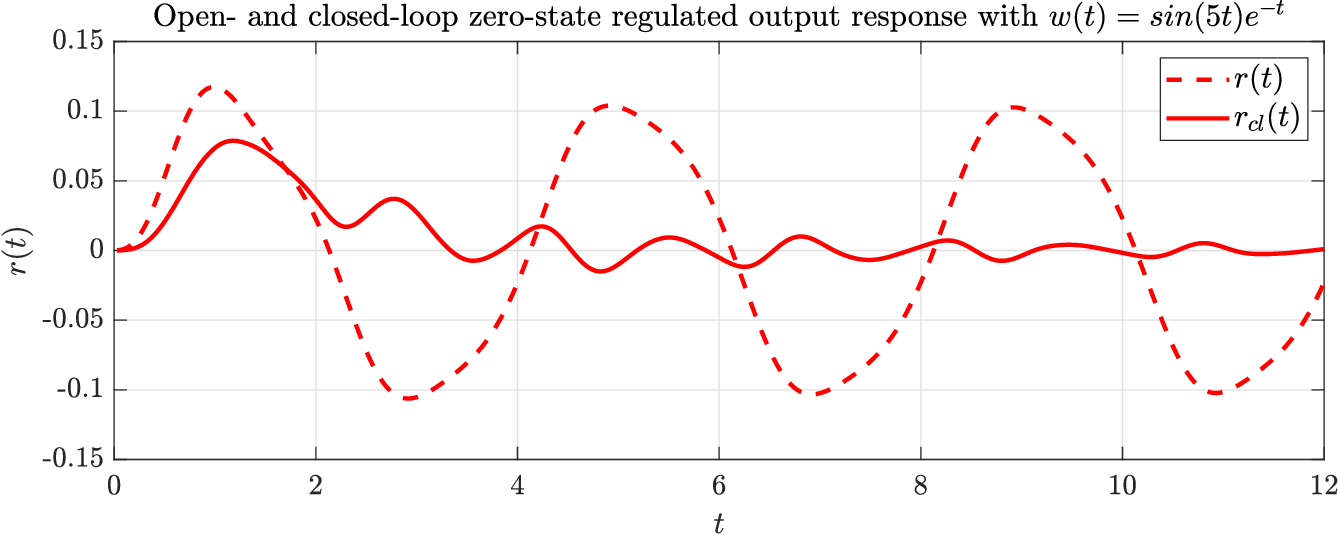}
    \caption{Open- $(r(t))$ and closed-loop ($r_{cl}(t)$) transient response of the regulated output $r(t)$ for external disturbance $w(t) = sin(5t) e^{-t}$.}
    \label{fig:Ch2_output}
\end{figure}

The full code used to produce the presented plots is provided in the file\\ ``PIETOOLS\_Code\_Illustrations\_Ch2\_Introduction.m''. We also encourage the user to look at the various PIETOOLS demo files for an overview on how to perform controller synthesis of ODE-PDE synthesis and simulate the resulting closed-loop system.



\section{Summary}

In this chapter, we gave an introduction on how PIETOOLS can be used to solve various control-relevant problems involving linear ODE-PDE models. The example depicted here was highly sensitive to disturbances. Figure.~\ref{fig:Ch2_OL} shows that, even after the applied disturbance has ceased, the output signal $r(t)$ remains affected, failing to converge to zero within the simulation time. This behaviour is measured by the computed $H_{\infty}$ norm of the open-loop system. 

On the other hand, with the synthesized feedback controller given by PIETOOLS, the closed-loop system quickly rejects the disturbance, as is clear from Figures.~\ref{fig:Ch2_CL}. The increase in performance can be certified by the considerable reduction in the value of the $H_{\infty}$ norm and by comparing the behaviour of the outputs without and with the controller in Figure.~\ref{fig:Ch2_output}.

\chapter{PI Operators in PIETOOLS}\label{ch:PIs}

PIETOOLS primarily functions by manipulation of Partial Integral (PI) operators, which is made simple by introduction of MATLAB classes that represent PI operators. 
In PIETOOLS 2025, there are two types of PI operators: PI operators with known parameters, \texttt{opvar/opvar2d} class objects, and PI operators with unknown parameters, \texttt{dopvar/dopvar2d} class objects. In this chapter, we outline the classes used to represent PI operators with known parameters. The information in this chapter is divided as follows: Section~\ref{sec:PIs:1D} and Section~\ref{sec:PIs:2D} provide brief mathematical background, and corresponding MATLAB implementation, about PI operators in 1D and 2D, respectively. Section~\ref{sec:PIs:overview} provides an overview of the structure of \texttt{opvar/opvar2d} classes in PIETOOLS. For more theoretical background on PI operators, we refer to Appendix~\ref{appx:PI_theory}. For more information on operations that can be performed on \texttt{opvar/opvar2d} class objects, we refer to Chapter~\ref{ch:opvar}.

\section{Declaring PI Operators in 1D}\label{sec:PIs:1D}

In this Section, we illustrate how 1D PI operators can be represented in PIETOOLS using \texttt{opvar} class objects. Here, we say that an operator $\mcl{P}$ is a 1D PI operator if it acts on functions $\mbf{v}(s)$ depending on just one spatial variable $s$, and the operation it performs can be described using partial integrals. We further distinguish 3-PI operators, acting on functions $\mbf{v}\in L_2^{n}[a,b]$, and 4-PI operators, acting on functions $\sbmat{v_0\\\mbf{v}_1}\in\sbmat{\R^{n_0}\\L_2^{n_1}[a,b]}$. Both types of operators can be represented using \texttt{opvar} class objects, as we show in the remainder of this section.

\subsection{Declaring 3-PI Operators}

We first consider declaring a 3-PI operator in PIETOOLS. Here, for given parameters $R=\{R_0,R_1,R_2\}$, the associated 3-PI operator $\mcl{P}[R]:L_2^n[a,b]\rightarrow L_2^m[a,b]$ is given by
\begin{align}\label{eq:3PI_standard_form}
 \bl(\mcl{P}[R]\mbf{v}\br)(s)&=R_0(s)\mbf{v}(s)+\int_{a}^{s}R_1(s,\theta)\mbf{v}(\theta)d\theta + \int_{s}^{b}R_2(s,\theta)\mbf{v}(\theta)d\theta,   &   s&\in[a,b],
\end{align}
for any $\mbf{v}\in L_2^n[a,b]$. In PIETOOLS, we represent such 3-PI operators using \texttt{opvar} class objects. For example, suppose we wish to declare a very simple PI operator $\mcl{A}:L_2^2[-1,1]\rightarrow L_2^2[-1,1]$, defined by
\begin{align}\label{eq:ex_3PI_1}
    \bl(\mcl{A}\mbf{v}\br)(s)&=\int_{-1}^{s}\underbrace{\bmat{1&2\\3&4}}_{R_1}\mbf{v}(\theta)d\theta,   &   s&\in[-1,1].
\end{align}
To declare this operator, we first initialize an empty \texttt{opvar} object \texttt{A}, by simply calling \texttt{opvar} as:
\begin{matlab}
\begin{verbatim}
 >> opvar A
 A = 
    [] | []
    --------
    [] | []
 A.R = 
      [] | [] | []
 >> A.I = [-1,1];
\end{verbatim}    
\end{matlab}
Here, the first line initialize a $0\times 0$ \texttt{opvar} object with all empty parameters \texttt{[]}. The second line, \texttt{A.I=[-1,1]}, then sets the spatial interval associated to the operator equal to $[-1,1]$, indicating that it maps the function space $L_2[-1,1]$. 

Next, we set the parameters of the operator. For a 3-PI operator such as $\mcl{A}$, only the paramaters in the field \texttt{A.R} will be nonzero, where \texttt{A.R} itself has fields \texttt{R0}, \texttt{R1} and \texttt{R2}. For our simple operator, only the parameter $R_1$ in the 3-PI Expression~\eqref{eq:3PI_standard_form} is nonzero, so we only have to assign a value to the field \texttt{R1}:
\begin{matlab}
\begin{verbatim}
 >> A.R.R1 = [1,2; 3,4];
 A = 
    [] | []
    --------
    [] | []
 A.R = 
      [0,0] | [1,2] | [0,0]
      [0,0] | [3,4] | [0,0]
\end{verbatim}    
\end{matlab}
where the fields \texttt{A.R.R0} and \texttt{A.R.R2} automatically default to zero-arrays of the appropriate dimensions. With that, the \texttt{opvar} object \texttt{A} represents the PI operator $\mcl{A}$ as defined in~\eqref{eq:ex_3PI_1}.

Next, suppose we wish to implement a slightly more complicated operator $\mcl{B}:L_2[0,1]\rightarrow L_2^2[0,1]$, defined as
\begin{align*}
    \bl(\mcl{B}\mbf{x}\br)(s)&=\underbrace{\bmat{1\\s^2}}_{R_{0}}\mbf{v}(s)+\int_{0}^{s}\underbrace{\bmat{2s\\s(s-\theta)}}_{R_1}\mbf{x}(\theta)d\theta+\int_{s}^{1}\underbrace{\bmat{3\theta\\\frac{3}{4}(s^2-s)}}_{R_2}\mbf{x}(\theta)d\theta, &
    s&\in[0,1].
\end{align*}
For this operator, the parameters $R_i(s,\theta)$ are all polynomial functions. Such polynomial functions can be represented in PIETOOLS using the \texttt{polynomial} class (from the `multipoly' toolbox), for which operations such as addition, multiplication and concatenation have already been implemented. This means that polynomials such as the functions $R_i$ can be implemented by simply initializing polynomial variables $s$ and $\theta$, and then using these variables to define the desired functions:
\begin{matlab}
\begin{verbatim}
 >> pvar s s_dum
 >> R0 = [1; s^2]
 R0 = 
   [   1]
   [ s^2]
   
 >> R1 = [2*s; s*(s-s_dum)]
 R1 = 
   [           2*s]
   [ s^2 - s*s_dum]
   
 >> R2 = [3*s_dum; (3/4)*(s^2-s)]
 R2 = 
   [           3*s_dum]
   [ 0.75*s^2 - 0.75*s]
\end{verbatim}    
\end{matlab}
Here, the first line calls the function \texttt{pvar} to initialize the two polynomial variables \texttt{s} and \texttt{s\_dum}, which we use to represent the spatial variable $s$ and dummy variable $\theta$, respectively. Then, we can add and multiply these variables to represent any desired polynomial in $(s,\theta)$, allowing us to implement the parameters $R_0(s)$, $R_1(s,\theta)$ and $R_2(s,\theta)$. Having defined these parameters, we can then represent the operator $\mcl{B}$ as an \texttt{opvar} object \texttt{B} as before:
\begin{matlab}
\begin{verbatim}
 >> opvar B;
 >> B.I = [0,1];
 >> B.var1 = s;      B.var2 = s_dum;
 >> B.R.R0 = R0;     B.R.R1 = R1;     B.R.R2 = R2
 B =
       [] | [] 
       ---------
       [] | B.R 

 B.R =
       [1] |         [2*s] |         [3*s_dum] 
     [s^2] | [s^2-s*s_dum] | [0.75*s^2-0.75*s] 
\end{verbatim}    
\end{matlab}
Note here that, in addition to specifying the spatial domain $[0,1]$ of the variables using the field \texttt{B.I}, we also have to specify the actual variables $s$ (\texttt{s}) and $\theta$ (\texttt{s\_dum}) that appear in the parameters, using the fields \texttt{B.var1} and \texttt{B.var2}. Here \texttt{var1} should correspond to the primary spatial variable, i.e. the variable $s$ on which the function $\mbf{u}(s):=\bl(\mcl{B}\mbf{v}\br)(s)$ will actually depend, and \texttt{B.var2} should correspond to the dummy variable, i.e. the variable $\theta$ which is used solely for integration.

\begin{boxEnv}{\textbf{Warning}}
By default, dummy variables in PIETOOLS are always assigned the same name as the primary variable, but with \texttt{\_dum} added (e.g. \texttt{s} and \texttt{s\_dum}). If users declare their own PI operators, they are highly recommended to use the same convention when setting their primary spatial variables and dummy variables, to avoid unintended errors when performing e.g. analysis and simulation.
\end{boxEnv}

\subsection{Declaring 4-PI Operators}

In addition to 3-PI operators, 4-PI operators can also be represented using the \texttt{opvar} structure. Here, for a given matrix $P$, given functions $Q_1,Q_2$, and 3-PI parameters $R=\{R_0,R_1,R_2\}$, we define the associated 4-PI operator $\mcl{P}\sbmat{P&Q_1\\Q_2&R}:\sbmat{\R^{n_0}\\L_2^{n_1}[a,b]}\rightarrow \sbmat{\R^{m_0}\\L_2^{m_1}[a,b]}$
\begin{align*}
    \bbl(\mcl{P}\sbmat{P&Q_1\\Q_2&R}\mbf{v}\bbr)(s)&=
    \left[\begin{array}{ll}
        Pv_0        \hspace*{-0.1cm}~& +\ \int_{a}^{b}Q_1(s)\mbf{v}_1(s)ds  \\
        Q_2(s)v_0   \hspace*{-0.1cm}& +\ \bl(\mcl{P}[R]\mbf{v}_1\br)(s)
    \end{array}\right],  &
    s&\in[a,b],
\end{align*}
for $\mbf{v}=\sbmat{v_0\\\mbf{v}_1}\in \sbmat{\R^{n_0}\\L_2^{n_1}[a,b]}$. To represent operators of this form, we use the same \texttt{opvar} structure as before, only now also specifying values of the fields \texttt{P}, \texttt{Q1} and \texttt{Q2}. For example, suppose we wish to declare a 4-PI operator $\mcl{C}:\sbmat{\R^2\\L_2[0,3]}\rightarrow \sbmat{\R\\L_2^2[0,3]}$ defined as
\begin{align*}
    \bl(\mcl{C}\mbf{x}\br)(s)&=
    \bbbbl[\begin{array}{ll}
        \overbrace{\sbmat{-1&2}}^{P} x_0        \hspace*{-0.1cm}~& +\ \int_{0}^{3}\overbrace{(3-s^2)}^{Q_1}\mbf{x}_1(s)ds  \\
        \underbrace{\sbmat{0&-s\\s&0}}_{Q_2}v_0   \hspace*{-0.1cm}& +\ \underbrace{\sbmat{1\\s^3}}_{R_0}\mbf{v}_1(s) + \int_{0}^{s}\underbrace{\sbmat{s-\theta\\\theta}}_{R_1}\mbf{v}_1(\theta)d\theta + \int_{s}^{3}\underbrace{\sbmat{s\\ \theta-s}}_{R_2}\mbf{v}_1(\theta)d\theta,
    \end{array}\bbbbr], &
    s&\in[0,3].
\end{align*}
for $\mbf{v}=\sbmat{v_0\\\mbf{v}_1}\in \sbmat{\R^{2}\\L_2^{1}[0,3]}$. To declare this operator, we first construct the polynomial functions defining the parameters $P$ through $R_2$, using \texttt{pvar} objects \texttt{s} and \texttt{tt} to represent $s$ and $\theta$:
\begin{matlab}
\begin{verbatim}
 >> pvar s tt
 >> P = [-1,2];
 >> Q1 = (3-s^2);        
 >> Q2 = [0,-s; s,0];
 >> R0 = [1; s^3];        R1 = [s-tt; tt];     R2 = [s; tt-s];
\end{verbatim}    
\end{matlab}
Having defined the desired parameters, we can then define the operator $\mcl{C}$ as
\begin{matlab}
\begin{verbatim}
 >> opvar C;
 >> C.I = [0,3];
 >> C.var1 = s;      C.var2 = tt;
 >> C.P = P;
 >> C.Q1 = Q1;
 >> C.Q2 = Q2;
 >> C.R.R0 = R0;     C.R.R1 = R1;     C.R.R2 = R2
 C =
     [-1,2] | [-s^2+3] 
     ------------------
     [0,-s] | C.R 
      [s,0] |   

 C.R =
      [1] | [s-tt] |     [s] 
    [s^3] |   [tt] | [-s+tt] 
\end{verbatim}    
\end{matlab}
using the field \texttt{R} to specify the 3-PI sub-component, and using the fields \texttt{P}, \texttt{Q1} and \texttt{Q2} to set the remaining parameters.

\section{Declaring PI Operators in 2D}\label{sec:PIs:2D}

In addition to PI operators in 1D, PI operators in 2D can also be represented in PIETOOLS, using the \texttt{opvar2d} data structure. Here, similarly to how we distinguish 3-PI operators and 4-PI operators for 1D function spaces, we will distinguish 2 classes of 2D operators. In particular, we distinguish the standard 9-PI operators, which act on just functions $\mbf{v}\in L_2\bl[[a,b]\times[c,d]\br]$, and the more general 2D PI operator, acting on coupled functions $\sbmat{v_0\\\mbf{v}_x\\\mbf{v}_y\\\mbf{v}_2}\in\sbmat{\R^{n_0}\\L_2^{n_x}[a,b]\\L_2^{n_y}[c,d]\\L_2^{n_2}[[a,b]\times[c,d]}$.

\subsection{Declaring 9-PI Operators}

For given parameters $N=\sbmat{N_{00}&N_{01}&N_{02}\\N_{10}&N_{11}&N_{12}\\N_{20}&N_{21}&N_{22}}$, the associated 9-PI operator $\mcl{P}[N]:L_2^n\bl[[a,b]\times[c,d]\br]\rightarrow L_2^m\bl[[a,b]\times[c,d]\br]$ is given by
{\small
\begin{align*}
    \left(\mcl{P}[N]\mbf{v}\right)(x,y)= N_{00}(x,y)\mbf{v}(x,y) &+\hspace*{0.0cm} \int_{c}^{y}\! N_{01}(x,y,\nu)\mbf{v}(x,\nu)d\nu + \int_{y}^{d}\! N_{02}(x,y,\nu)\mbf{v}(x,\nu)d\nu \nonumber\\
    +\int_{a}^{x}\! N_{10}(x,y,\theta)\mbf{v}(\theta,y)d\theta &+ \int_{a}^{x}\!\int_{c}^{y}\! N_{11}(x,y,\theta,\nu)\mbf{v}(\theta,\nu)d\nu d\theta + \int_{a}^{x}\!\int_{y}^{d}\! N_{12}(x,y,\theta,\nu)\mbf{v}(\theta,\nu)d\nu d\theta  \nonumber\\
    +\int_{x}^{b} N_{20}(x,y,\theta)\mbf{v}(\theta,y)d\theta &+ \int_{x}^{b}\!\int_{c}^{y}\! N_{21}(x,y,\theta,\nu)\mbf{v}(\theta,\nu)d\nu d\theta + \int_{x}^{b}\!\int_{y}^{d}\! N_{22}(x,y,\theta,\nu)\mbf{v}(\theta,\nu)d\nu d\theta,
 \end{align*}
}
for any $\mbf{v}\in L_2\bl[[a,b]\times[c,d]\br]$. In PIETOOLS 2025, we represent such operators using \texttt{opvar2d} class objects, which are declared in a similar manner to \texttt{opvar} objects. For example, to delcare a simple operator $\mcl{D}:L_2^2\bl[[0,1]\times[1,2]\br]\rightarrow L_2^2\bl[[0,1]\times[1,2]\br]$ defined as
\begin{align*}
    \bl[\mcl{D}\mbf{v}\br](s_1,s_2)&=\int_{0}^{s_1}\int_{s_2}^{2}\underbrace{\bmat{s_1^2&s_1s_2\\s_1s_2 &s_2^2}}_{N_{12}}\mbf{v}(\theta_1,\theta_2)d\theta_2 d\theta_1,   &   (s_1,s_2)&\in[0,1]\times[1,2],
\end{align*}
we first declare the parameter $N_{12}$ defining this operator by representing $s_1$ and $s_2$ by \texttt{pvar} objects \texttt{s1} and \texttt{s2}
\begin{matlab}
\begin{verbatim}
 >> pvar s1 s2
 >> N12 = [s1^2, s1*s2; s1*s2, s2^2];
\end{verbatim}    
\end{matlab}
Then, we initialize an empty \texttt{opvar2d} object \texttt{D} to represent $\mcl{D}$, and assign the variables $(s_1,s_2)$ and their domain $[0,1]\times[1,2]$ to this operator as
\begin{matlab}
\begin{verbatim}
 >> opvar2d D;
 >> D.var1 = [s1;s2];
 >> D.I = [0,1; 1,2];
\end{verbatim}    
\end{matlab}
Note here that, in \texttt{opvar2d} objects, \texttt{var1} is a column vector listing each of the spatial variables $(s_1,s_2)$ on which the result $\mbf{u}(s_1,s_2)=\bl(\mcl{D}\mbf{v}\br)(s_1,s_2)$ depends. Accordingly, the field \texttt{I} in an \texttt{opvar2d} object also has two rows, with each row specifying the interval on which the variable in the associated row of \texttt{var1} exists. Having initialized the operator, we then assign the parameter \texttt{N12} to the appropriate field. Here, the parameters defining a 9-PI operator are stored in the $3\times 3$ cell array \texttt{D.R22}, with \texttt{R22} referring to the fact that these parameters map 2D functions to 2D functions. Within this array, element \texttt{\{i,j\}} for $i,j\in\{1,2,3\}$ corresponds to parameter $N_{i-1,j-1}$ in the operator, and so we can specify parameter $N_{12}$ using element \texttt{\{2,3\}}:
\begin{matlab}
\begin{verbatim}
 >> D.R22{2,3} = N12
 D =
     [] |    [] |    [] |    [] 
     --------------------------
     [] | D.Rxx |    [] | D.Rx2 
     --------------------------
     [] |    [] | D.Ryy | D.Ry2 
     --------------------------
     [] | D.R2x | D.R2y | D.R22 

 D.Rxx =
     [] | [] | [] 

 D.Rx2 =
     [] | [] | [] 

 D.Ryy =
     [] | [] | [] 

 D.Ry2 =
     [] | [] | [] 

 D.R2x =
     [] | [] | [] 

 D.R2y =
     [] | [] | [] 

 D.R22 =
     [0,0] | [0,0] |        [0,0] 
     [0,0] | [0,0] |        [0,0] 
     ---------------------------- 
     [0,0] | [0,0] | [s1^2,s1*s2] 
     [0,0] | [0,0] | [s1*s2,s2^2] 
     ---------------------------- 
     [0,0] | [0,0] |        [0,0] 
     [0,0] | [0,0] |        [0,0] 
\end{verbatim}    
\end{matlab}
We note that, in the resulting structure, there are a lot of empty parameters, such as \texttt{D.Rxx}. As we will discuss in the next subsection, these parameters correspond to maps to and from other functions spaces, just like the parameters \texttt{P} and \texttt{Qi} in the \texttt{opvar} structure. Since the operator $\mcl{D}$ maps only functions $L_2^2\bl[[0,1]\times[1,2]\br]\rightarrow L_2^2\bl[[0,1]\times[1,2]\br]$, all parameters mapping different function spaces are empty for the object \texttt{D}.

Suppose now we want to declare a 9-PI operator $\mcl{E}:L_2\bl[[0,1]\times[-1,1]\br]\rightarrow L_2\bl[[0,1]\times[-1,1]\br]$ defined by
\begin{align*}
    \bl(\mcl{E}\mbf{v}\br)(s_1,s_2)&=\overbrace{x^2y^2}^{N_{00}}\mbf{v}(s_1,s_2) + \int_{-1}^{s_2}\overbrace{s_1(s_2-\theta_2)}^{N_{01}}\mbf{v}(s_1,\theta_2)d\theta_2   \\
    &+ \int_{s_1}^{1}\underbrace{(s_1-\theta_1)s_2}_{N_{20}}\mbf{v}(\theta_1,s_2)d\theta_1 + \int_{s_1}^{1}\int_{-1}^{s_2}\underbrace{(s_1-\theta_1)(s_2-\theta_2)}_{N_{21}}\mbf{v}(\theta_1,\theta_2)d\theta_2 d\theta_1
\end{align*}
As before, we first set the values of the parameters $N_{ij}$, using \texttt{s1}, \texttt{s2}, \texttt{th1} and \texttt{th2} to represent $s_1$, $s_2$, $\theta_1$ and $\theta_2$ respectively:
\begin{matlab}
\begin{verbatim}
 >> pvar s1 s2 th1 th2
 >> N00 = s1^2 * s2^2;     N01 = s1*(s2-th2);
 >> N20 = (s1-th1)*s2;     N21 = (s1-th1)*(s2-th2);
\end{verbatim}    
\end{matlab}
Next, we initialize an \texttt{opvar2d} object \texttt{E} with the appropriate variables and domain as
\begin{matlab}
\begin{verbatim}
 >> opvar2d E;
 >> E.var1 = [s1;s2];     E.var2 = [th1; th2];
 >> E.I = [0,1; -1,1];
\end{verbatim}    
\end{matlab}
where in this case we set both the primary variables, using \texttt{var1}, and the dummy variables, using \texttt{var2}. Note here that the domains of the first and second dummy variables are the same as those of the first and second primary variables, and are defined in the first and second row of \texttt{I} respectively. Finally, we assign the parameters $N_{ij}$ to the appropriate elements of \texttt{R22}
\begin{matlab}
\begin{verbatim}
 >> E.R22{1,1} = N00;     E.R22{1,2} = N01;
 >> E.R22{3,1} = N20;     E.R22{3,2} = N21
 E =
     [] |    [] |    [] |    [] 
     --------------------------
     [] | E.Rxx |    [] | E.Rx2 
     --------------------------
     [] |    [] | E.Ryy | E.Ry2 
     --------------------------
     [] | E.R2x | E.R2y | E.R22 

 E.R22 =

        [s1^2*s2^2] |                [s1*s2-s1*th2] | [0] 
     ---------------------------------------------------- 
                [0] |                           [0] | [0] 
     ---------------------------------------------------- 
     [s1*s2-s2*th1] | [s1*s2-s1*th2-s2*th1+th1*th2] | [0] 
\end{verbatim}    
\end{matlab}
so that \texttt{E} represents the desired operator.

\subsection{Declaring General 2D PI Operators}

The most general PI operators that can be represented in PIETOOLS 2025 are those mapping $\sbmat{\R^{n_0}\\L_2^{n_x}[a,b]\\L_2^{n_y}[c,d]\\L_2^{n_2}\bl[[a,b]\times[c,d]\br]}\rightarrow \sbmat{\R^{m_0}\\L_2^{m_x}[a,b]\\L_2^{m_y}[c,d]\\L_2^{m_2}\bl[[a,b]\times[c,d]\br]}$, defined by parameters $R=\sbmat{R_{00}&R_{0x}&R_{0y}&R_{02}\\R_{x0}&R_{xx}&R_{xy}&R_{x2}\\R_{y0}&R_{yx}&R_{yy}&R_{y2}\\R_{20}&R_{2x}&R_{2y}&R_{22}}$ as
{\small
\begin{align*}
    \bl(\mcl{P}[R]\mbf{x}\br)(s)=
    \left[\!\!\begin{array}{llll}
        R_{00}v_0 & \hspace*{-0.2cm}+\ \int_{a}^{b}R_{0x}(x)\mbf{v}_{x}(x)dx & \hspace*{-0.2cm}+\ \int_{c}^{d}R_{0y}(y)\mbf{v}_{y}(y)dy & \hspace*{-0.2cm}+\ \int_{a}^{b}\int_{c}^{d}R_{02}(x,y)\mbf{v}_{2}(x,y)dydx \\
        R_{x0}(x)v_0 & \hspace*{-0.2cm}+\ \bl(\mcl{P}[R_{xx}]\mbf{v}_{x}\br)(x) & \hspace*{-0.2cm}+\ \int_{c}^{d}R_{xy}(x,y)\mbf{v}_{y}(y)dy & \hspace*{-0.2cm}+\ \int_{c}^{d}\bl(\mcl{P}[R_{x2}]\mbf{v}_2\br)(x,y) dy \\
        R_{y0}(y)v_0 & \hspace*{-0.2cm}+\ \int_{a}^{b}R_{yx}(x,y)\mbf{v}_{x}(x)dx & \hspace*{-0.2cm}+\ \bl(\mcl{P}[R_{yy}]\mbf{v}_{y}\br)(y) & \hspace*{-0.2cm}+\ \int_{a}^{b}\bl(\mcl{P}[R_{y2}]\mbf{v}_2\br)(x,y) dx \\
        R_{20}(x,y)v_0 & \hspace*{-0.2cm}+\ \bl(\mcl{P}[R_{2x}]\mbf{v}_{x}\br)(x,y) & \hspace*{-0.2cm}+\ \bl(\mcl{P}[R_{2y}]\mbf{v}_{y}\br)(x,y) & \hspace*{-0.2cm}+\ \bl(\mcl{P}[R_{22}]\mbf{v}_2\br)(x,y) \\
    \end{array}\!\right]
\end{align*}
}
for $\mbf{v}=\sbmat{v_0\\\mbf{v}_x\\\mbf{v}_y\\\mbf{v}_2}\in \sbmat{\R^{n_0}\\L_2^{n_x}[a,b]\\L_2^{n_y}[c,d]\\L_2^{n_2}\bl[[a,b]\times[c,d]\br]}$, where $\mcl{P}[R_{xx}]$, $\mcl{P}[R_{yy}]$, $\mcl{P}[R_{x2}]$, $\mcl{P}[R_{y2}]$, $\mcl{P}[R_{2x}]$ and $\mcl{P}[R_{2y}]$ are 3-PI operators, and where $\mcl{P}[R_{22}]$ is a 9-PI operator. These types of PI operators are also represented using the \texttt{opvar2d} class, specifying each of the parameters $R_{ij}$ using the associated fields \texttt{Rij}. For example, suppose we want to implement a PI operator $\mcl{F}:\sbmat{\R\\L_2[0,2]\\L_2\bl[[0,2]\times[2,3]\br]}\rightarrow \sbmat{L_2^{2}[0,2]\\L_2\bl[[0,2]\times[2,3]\br]}$, defined as
\begin{align*}
    \bl(\mcl{F}\mbf{v}\br)(x,y)&=
    \bbbl[\begin{array}{lll}
         \overbrace{\sbmat{1\\x}}^{R_{x0}}v_0 & \hspace*{-0.2cm}+ \overbrace{\sbmat{x\\x^2}}^{R_{xx}^{0}}\mbf{v}_1(x) + \int_{x}^{2}\overbrace{\sbmat{1\\(\theta-x)}}^{R_{xx}^{2}}\mbf{v}_1(\theta)d\theta  & \hspace*{-0.2cm}+\int_{2}^{3}\int_{0}^{x}\overbrace{\sbmat{y\\y^2(x-\theta)}}^{R_{x2}^{1}}\mbf{v}_2(\theta,y)dy \\
         & \hspace*{0.35cm} \underbrace{y^2}_{R_{2x}^{0}}\mbf{v}_1(x) + \int_{0}^{x}\underbrace{y}_{R_{2x}^{1}}\mbf{v}_1(\theta)d\theta & \hspace*{-0.2cm}+ \int_{0}^{x}\int_{2}^{y}\underbrace{\theta\nu}_{R_{22}^{11}} \mbf{v}_2(\theta,\nu)d\nu d\theta
    \end{array}\bbbr],
\end{align*}
for $\mbf{v}=\sbmat{v_0\\\mbf{v}_1\\\mbf{v}_2}\in \sbmat{\R\\L_2[0,2]\\L_2\bl[[0,2]\times[2,3]\br]}$.  To declare this operator, we define the parameters as before as
\begin{matlab}
\begin{verbatim}
 >> pvar x y theta nu
 >> Rx0 = [1; x];
 >> Rxx_0 = [x; x^2];     Rxx_2 = [1; theta-x];
 >> Rx2_1 = [y; y^2 * (x-theta)];
 >> R2x_0 = y^2;         R2x_1 = y;
 >> R22_11 = theta*nu;
\end{verbatim}    
\end{matlab}
and then declare the \texttt{opvar2d} object as
\begin{matlab}
\begin{verbatim}
 >> opvar2d F;
 >> F.var1 = [x; y];      F.var2 = [theta; nu];
 >> F.I = [0,2; 2,3];
 >> F.Rx0 = Rx0;
 >> F.Rxx{1} = Rxx_0;     F.Rxx{3} = Rxx_2;
 >> F.Rx2{2} = Rx2_1;
 >> F.R2x{1} = R2x_0;     F.R2x{2} = R2x_1;
 >> F.R22{2,2} = R22_11;
\end{verbatim}    
\end{matlab}
yielding a structure 
\begin{matlab}
\begin{verbatim}
 >> F
 F =
 
      [] |    [] |    [] |    [] 
     ---------------------------
     [1] | F.Rxx |    [] | F.Rx2 
     [x] |       |       |       
     ---------------------------
      [] |    [] | F.Ryy | F.Ry2 
     ---------------------------
     [0] | B.R2x | F.R2y | F.R22 

 F.Rxx =
 
       [x] | [0] |       [1] 
     [x^2] | [0] | [theta-x] 

 F.Rx2 =
 
     [0] |                [y] | [0] 
     [0] | [-theta*y^2+x*y^2] | [0] 

 F.Ryy =
 
     [] | [] | [] 

 F.Ry2 =
 
     [] | [] | [] 

 F.R2x =
 
     [y^2] | [y] | [0] 

 F.R2y =
 
     [] | [] | [] 

 F.R22 =
 
     [0] |        [0] | [0] 
     ---------------------- 
     [0] | [nu*theta] | [0] 
     ---------------------- 
     [0] |        [0] | [0] 
\end{verbatim}
\end{matlab}
Representing the operator $\mcl{F}$.

In the following subsection, we provide an overview of how the \texttt{opvar} and \texttt{opvar2d} data structures are defined.

\newpage

\section{Overview of \texttt{opvar} and \texttt{opvar2d} Structure}\label{sec:PIs:overview}

\subsection{\texttt{opvar} class}

Let $\mcl{B}:\bmat{\R^{n_0}\\L_2^{n_1}[a,b]}\rightarrow \bmat{\R^{m_0}\\L_2^{m_1}[a,b]}$ be a 4-PI operator of the form
\begin{align}\label{eq:4PI_standard_form}
    \bl(\mcl{B}\mbf{x}\br)(s)=
    \left[\begin{array}{ll}
        Px_0        \hspace*{-0.1cm}~& +\ \int_{a}^{b}Q_1(s)\mbf{x}_1(s)ds  \\
        Q_2(s)x_0   \hspace*{-0.1cm}& +\ R_{0}(s)\mbf{x}_1(s) + \int_{a}^{s}R_{1}(s,\theta)\mbf{x}_1(\theta)d\theta + \int_{s}^{b}R_{2}(s,\theta)\mbf{x}_1(\theta)d\theta
    \end{array}\right]
\end{align}
for $\mbf{x}=\bmat{x_0\\\mbf{x}_1}\in \bmat{\R^{n_0}\\L_2^{n_1}[a,b]}$. Then, we can represent this operator as an \texttt{opvar} object \texttt{B} with fields as defined in Table~\ref{tab:opvar_fields}.

\begin{table}[!th]
\renewcommand{\arraystretch}{1.0}
\fontsize{11}{13}
 \begin{tabular}{p{1.0cm}p{2.00cm}p{12.75cm}}
 \hline
    \texttt{B.dim}    & \texttt{= [m0,n0; \hspace*{0.4cm} m1,n1]} 
    &  $2\times 2$ array of type \texttt{double} specifying the dimensions of the function spaces $\sbmat{\R^{m_0}\\L_2^{m_1}[a,b]}$ and $\sbmat{\R^{n_0}\\L_2^{n_1}[a,b]}$ the operator maps to and from;\\
    \texttt{B.var1} & \texttt{= s1}    &  $1\times 1$ \texttt{pvar} (\texttt{polynomial} class) object specifying the spatial variable $s$; \\
    \texttt{B.var2} & \texttt{= s1\_dum}   &  $1\times 1$ \texttt{pvar} (\texttt{polynomial} class) object specifying the dummy variable $\theta$;   \\
    \texttt{B.I} & \texttt{= [a,b]}       &  $1\times 2$ array of type \texttt{double}, specifying the interval $[a,b]$ on which the spatial variables $s$ and $\theta$ exist; \\
    \texttt{B.P} & \texttt{= P} & $m_0\times n_0$ array of type \texttt{double} or \texttt{polynomial} defining the matrix $P$; \\
    \texttt{B.Q1} & \texttt{= Q1} & $m_0\times n_1$ array of type \texttt{double} or \texttt{polynomial} defining the function $Q_1(s)$; \\
    \texttt{B.Q2} & \texttt{= Q2} & $m_1\times n_0$ array of type \texttt{double} or \texttt{polynomial} defining the function $Q_2(s)$; \\
    \texttt{B.R.R0} & \texttt{= R0} & $m_1\times n_1$ array of type \texttt{double} or \texttt{polynomial} defining the function $R_0(s)$; \\
    \texttt{B.R.R1} & \texttt{= R1} & $m_1\times n_1$ array of type \texttt{double} or \texttt{polynomial} defining the function $R_1(s,\theta)$; \\
    \texttt{B.R.R2} & \texttt{= R2} & $m_1\times n_1$ array of type \texttt{double} or \texttt{polynomial} defining the function $R_2(s,\theta)$; \\
    \hline
 \end{tabular}
 \caption{Fields in an \texttt{opvar} object \texttt{B}, defining a general 4-PI operator as in Equation~\eqref{eq:4PI_standard_form}}
\label{tab:opvar_fields}
\end{table}

\subsection{\texttt{opvar2d} class}

Let $\mcl{D}:\sbmat{\R^{n_0}\\L_2^{n_x}[a,b]\\L_2^{n_y}[c,d]\\L_2^{n_2}\bl[[a,b]\times[c,d]\br]}\rightarrow \sbmat{\R^{m_0}\\L_2^{m_x}[a,b]\\L_2^{m_y}[c,d]\\L_2^{m_2}\bl[[a,b]\times[c,d]\br]}$ be a PI operator of the form
{\small
\begin{align}\label{eq:0112PI_standard_form}
    \bl(\mcl{D}\mbf{x}\br)(s)=
    \left[\!\!\begin{array}{llll}
        R_{00}v_0 & \hspace*{-0.2cm}+\ \int_{a}^{b}R_{0x}(x)\mbf{v}_{x}(x)dx & \hspace*{-0.2cm}+\ \int_{c}^{d}R_{0y}(y)\mbf{v}_{y}(y)dy & \hspace*{-0.2cm}+\ \int_{a}^{b}\int_{c}^{d}R_{02}(x,y)\mbf{v}_{2}(x,y)dydx \\
        R_{x0}(x)v_0 & \hspace*{-0.2cm}+\ \bl(\mcl{P}[R_{xx}]\mbf{v}_{x}\br)(x) & \hspace*{-0.2cm}+\ \int_{c}^{d}R_{xy}(x,y)\mbf{v}_{y}(y)dy & \hspace*{-0.2cm}+\ \int_{c}^{d}\bl(\mcl{P}[R_{x2}]\mbf{v}_2\br)(x,y) dy \\
        R_{y0}(y)v_0 & \hspace*{-0.2cm}+\ \int_{a}^{b}R_{yx}(x,y)\mbf{v}_{x}(x)dx & \hspace*{-0.2cm}+\ \bl(\mcl{P}[R_{yy}]\mbf{v}_{y}\br)(y) & \hspace*{-0.2cm}+\ \int_{a}^{b}\bl(\mcl{P}[R_{y2}]\mbf{v}_2\br)(x,y) dx \\
        R_{20}(x,y)v_0 & \hspace*{-0.2cm}+\ \bl(\mcl{P}[R_{2x}]\mbf{v}_{x}\br)(x,y) & \hspace*{-0.2cm}+\ \bl(\mcl{P}[R_{2y}]\mbf{v}_{y}\br)(x,y) & \hspace*{-0.2cm}+\ \bl(\mcl{P}[R_{22}]\mbf{v}_2\br)(x,y) \\
    \end{array}\!\right]
\end{align}
}
for $\mbf{v}=\sbmat{v_0\\\mbf{v}_x\\\mbf{v}_y\\\mbf{v}_2}\in \sbmat{\R^{n_0}\\L_2^{n_x}[a,b]\\L_2^{n_y}[c,d]\\L_2^{n_2}\bl[[a,b]\times[c,d]\br]}$, where $\mcl{P}[R_{xx}]$, $\mcl{P}[R_{yy}]$, $\mcl{P}[R_{x2}]$, $\mcl{P}[R_{y2}]$, $\mcl{P}[R_{2x}]$ and $\mcl{P}[R_{2y}]$ are 3-PI operators, and where $\mcl{P}[R_{22}]$ is a 9-PI operator.
We can represent the operator $\mcl{D}$ as an \texttt{opvar2d} object \texttt{D} with fields as defined in Table~\ref{tab:opvar2d_fields}.

\begin{table}[!ht]
\renewcommand{\arraystretch}{1.0}
\fontsize{11}{13}
 \begin{tabular}{p{1.0cm}p{2.00cm}p{12.75cm}}
 \hline
    \texttt{D.dim}    & \texttt{= [m0,n0; \hspace*{0.4cm} mx,nx; \hspace*{0.4cm} my,ny; \hspace*{0.4cm} m2,n2;]} 
    &  $4\times 2$ array of type \texttt{double} specifying the dimensions of the function spaces $\sbmat{\R^{m_0}\\L_2^{m_x}[a,b]\\L_2^{m_y}[c,d]\\L_2^{m_2}\bl[[a,b]\times[c,d]\br]}$ and $\sbmat{\R^{n_0}\\L_2^{n_x}[a,b]\\L_2^{n_y}[c,d]\\L_2^{n_2}\bl[[a,b]\times[c,d]\br]}$ the operator maps to and from;\\
    \texttt{D.var1} & \texttt{= [s1; s2]}    &  $2\times 1$ \texttt{pvar} (\texttt{polynomial} class) object specifying the spatial variables $(x,y)$; \\
    \texttt{D.var2} & \texttt{= [s1\_dum; \hspace*{0.4cm} s2\_dum]}    &  $2\times 1$ \texttt{pvar} (\texttt{polynomial} class) object specifying the dummy variables $(\theta,\nu)$;   \\
    \texttt{D.I} & \texttt{= [a,b; \hspace*{0.4cm} c,d]}       &  $2\times 2$ array of type \texttt{double}, specifying the domain $[a,b]\times[c,d]$ on which the spatial variables $(x,\theta)$ and $(y,\nu)$ exist; \\
    \texttt{D.R00} & \texttt{= R00} & $m_0\times n_0$ array of type \texttt{double} or \texttt{polynomial} defining the matrix $R_{00}$; \\
    \texttt{D.R0x} & \texttt{= R0x} & $m_0\times n_x$ array of type \texttt{double} or \texttt{polynomial} defining the function $R_{0x}(x)$; \\
    \texttt{D.R0y} & \texttt{= R0y} & $m_0\times n_y$ array of type \texttt{double} or \texttt{polynomial} defining the function $R_{0y}(y)$; \\
    \texttt{D.R02} & \texttt{= R02} & $m_0\times n_2$ array of type \texttt{double} or \texttt{polynomial} defining the function $R_{02}(x,y)$; \\
    \texttt{D.Rx0} & \texttt{= Rx0} & $m_x\times n_0$ array of type \texttt{double} or \texttt{polynomial} defining the function $R_{x0}(x)$; \\
    \texttt{D.Rxx} & \texttt{= Rxx} & $3\times 1$ cell array specifying the 3-PI parameters $R_{xx}$; \\
    \texttt{D.Rxy} & \texttt{= Rxy} & $m_x\times n_y$ array of type \texttt{double} or \texttt{polynomial} defining the function $R_{xy}(x,y)$; \\
    \texttt{D.Rx2} & \texttt{= Rx2} & $3\times 1$ cell array specifying the 3-PI parameters $R_{x2}$; \\
    \texttt{D.Ry0} & \texttt{= Ry0} & $m_y\times n_0$ array of type \texttt{double} or \texttt{polynomial} defining the function $R_{y0}(y)$; \\
    \texttt{D.Ryx} & \texttt{= Ryx} & $m_y\times n_x$ array of type \texttt{double} or \texttt{polynomial} defining the function $R_{yx}(x,y)$; \\
    \texttt{D.Ryy} & \texttt{= Ryy} & $1\times 3$ cell array specifying the 3-PI parameters $R_{yy}$; \\
    \texttt{D.Ry2} & \texttt{= Ry2} & $1\times 3$ cell array specifying the 3-PI parameters $R_{y2}$; \\
    \texttt{D.R20} & \texttt{= R20} & $m_2\times n_0$ array of type \texttt{double} or \texttt{polynomial} defining the function $R_{20}(x,y)$; \\
    \texttt{D.R2x} & \texttt{= R2x} & $3\times 1$ cell array specifying the 3-PI parameters $R_{2x}$; \\
    \texttt{D.R2y} & \texttt{= R2y} & $1\times 3$ cell array specifying the 3-PI parameters $R_{2y}$; \\
    \texttt{D.R22} & \texttt{= R22} & $3\times 3$ cell array specifying the 9-PI parameters $R_{22}$; \\\hline
 \end{tabular}
 \caption{Fields in an \texttt{opvar2d} object \texttt{D}, defining a general PI operator in 2D as in Equation~\eqref{eq:0112PI_standard_form}}
\label{tab:opvar2d_fields}
\end{table}

\part{PIETOOLS Workflow for ODE-PDE and DDE Models}

\chapter{Setup and Representation of PDEs and DDEs}\label{ch:PDE_DDE_representation}

Using PIETOOLS, a wide variety of linear differential equations and time-delay systems can be simulated and analysed by representing them as Partial Integral Equations (PIEs). To facilitate this, PIETOOLS includes several input formats to declare Partial Differential Equations (PDEs) and Delay-Differential Equations (DDEs), which can then be easily converted to equivalent PIEs using the PIETOOLS function \texttt{convert}, as we show in Chapter~\ref{ch:PIE}. In this chapter, we present two of these input formats, discussing in detail how linear PDE and DDE systems can be easily implemented using the Command Line Parser for PDEs and Batch-Based input format for DDEs. We refer to Chapter~\ref{ch:alt_PDE_input} for information on two alternative input formats for PDEs, and we refer to Chapter~\ref{ch:alt_DDE_input} for two alternative input formats for time-delay systems, namely the Neutral Delay System (NDS) and Delay Difference Equation (DDF) formats.



\section{Command Line Parser for Coupled ODE-PDEs}
In PIETOOLS 2025, the simplest and most intuitive format for declaring coupled ODE-PDE systems is the Command Line Parser format. The Command Line Parser format represents ODE-PDE systems in MATLAB as \texttt{pde\_struct} objects, for which a variety of operations (addition, multiplication, substitution) have been defined to allow for easy declaration of a broad class of systems. In this section, we provide an overview on how to declare such systems as \texttt{pde\_struct} objects, referring to Section~\ref{sec:alt_PDE_input:terms_input_PDE} for more background.


\begin{boxEnv}{\textbf{Note}}
In PIETOOLS 2022, a Command Line Parser format for declaring 1D ODE-PDE systems was introduced, generating dependent variables using the \texttt{state} function and representing the system as a \texttt{sys} class object. These functions are still available in PIETOOLS 2025, but do not currently support declaration of 2D PDE systems, and are therefore not discussed here. See Section~\ref{sec:alt_PDE_input:sys} instead.

\end{boxEnv}

\subsection{Defining a coupled ODE-PDE system}
For the purpose of demonstration, consider the following coupled ODE-PDE system in control theory framework 
\begin{align*}
\dot{x}(t) &= -5 x(t)+\int_0^1 \partial_s \mbf x(t,s) ds + u(t),\\
\dot{\mbf x}(t,s) &= 9 \mbf x(t,s)+ \partial_s^2 \mbf x(t,s) +sw(t),\\
\mbf x(t,0) &= 0, \quad \partial_s\mbf x(t,1) + x(t) = 2w(t),\\
z(t) &= \bmat{\int_0^1 \mbf x(t,s) ds\\ u(t)},\\
y(t) &= \mbf x(t,0).
\end{align*}
The following code shows how this system can be declared using the Command Line Parser format, and subsequently converted to a PIE.
\begin{codebox}
\begin{matlab}
\begin{verbatim}
>> pvar t s;
>> x = pde_var();         X = pde_var(s,[0,1]);
>> w = pde_var('in');     z = pde_var('out',2);
>> u = pde_var('control');
>> y = pde_var('sense');
>> out_eq = z==[int(X,s,[0,1]); u];
>> eqns = [diff(x,t)==-5*x+int(diff(X,s,1),s,[0,1])+u;
           diff(X,t)==9*X+diff(X,s,2)+s*w;
           subs(X,s,0)==0;
           subs(diff(X,s),s,1)==-x+2*w;
           y==subs(X,s,0)];
>> odepde = [eqns; out_eq];
>> odepde = initialize(odepde)
>> PIE = convert(odepde,'pie');
\end{verbatim}
\end{matlab}
\end{codebox}

We will break down each step used in the code above and explain the action performed by each line of the code. Specifying any PDE system using the `Command Line Parser' format follows the same three simple steps listed below:
\begin{enumerate}
    \item Define independent variables ($s$, $t$) 
    \item Define dependent variables ($\mbf x$, $x$, $z$, $y$, $w$ and $u$)
    \item Define the equations
\end{enumerate}

\subsubsection{Define independent variables}

To define equations symbolically, first, the independent variables (spatial variable and time variable) and dependent variables (states, inputs, and outputs) have to be declared. For example, if the PDE is defined in terms of spatial variable $s$ and temporal variable $t$, we would start by defining these variables as \texttt{polynomial} objects, using the function \texttt{pvar} as shown below:

\begin{matlab}
>> pvar ~t ~s;  ~\% independent variables are polynomial objects
\end{matlab}

Note that \texttt{t} \textbf{will always be interpreted as the temporal variable} in PIETOOLS. Although we highly recommend always using \texttt{s} or \texttt{s1} as spatial variable (as this is the default used by PIETOOLS), the spatial variable can feasibly be given any name, so long as it is properly assigned to e.g. a PDE state as we show next. 

\subsubsection{Define dependent variables}
After defining independent variables, we need to define dependent variables such as ODE/PDE states, inputs, and outputs (see aslo Chapter~\ref{ch:scope}). Dependent variables are defined as \texttt{pde\_struct} objects, and can be declared using the function \texttt{pde\_var}. For example:

\begin{matlab}
\begin{verbatim}
>> x = pde_var();         X = pde_var(s,[0,1]);
>> w = pde_var('in');     z = pde_var('out',2);
>> u = pde_var('control');
>> y = pde_var('sense');
\end{verbatim}
\end{matlab}

The above code, when executed in MATLAB, creates six symbolic variables, namely \texttt{x, X, w, u, z, y}. Here, the variables \texttt{x} and \texttt{X} are not explicitly assigned a particular type, and will therefore default to be interpreted as state variables. Passing the polynomial variable \texttt{s} as well as the interval \texttt{[0,1]} in declaring \texttt{X}, the variable is interpreted to be a PDE state variable $\mbf{x}(t,s)$ with spatial domain $s\in[0,1]$. For the remaining variables, a type is explicitly specified, declaring \texttt{w} to be an exogenous input, \texttt{z} to be a regulated output, \texttt{u} to be a controlled input, and \texttt{y} to be a sensed output. In addition, the output \texttt{z} is declared to be vector-valued, with length 2, which will be crucial when declaring the equation 
\[z(t) = \bmat{\int_0^1 \mbf x(t,s) ds\\ u(t)}.\] 





\subsubsection{Define the equations}

Having declared the dependent variables that appear in the PDE, we can now use standard algebraic operations such as addition (\texttt{+}) and multiplication (\texttt{*}), as well as operations such as integration (\texttt{int}), differentiation (\texttt{diff}), and substitution (\texttt{subs}) to declare our system. For example, to declare the equation for $z(t)$, we call
\begin{matlab}
\begin{verbatim}
 >> out_eq = z==[int(X,s,[0,1]); u];
\end{verbatim}
\end{matlab}
We can also declare multiple equations together in a column vector, e.g. specifying the remaining $5$ equations and boundary conditions listed below
\begin{align*}
\dot{x}(t) &= -5 x(t)+\int_0^1 \partial_s \mbf x(t,s) ds + u(t)\\
\dot{\mbf x}(t,s) &= 9 \mbf x(t,s)+ \partial_s^2 \mbf x(t,s) +sw(t),\qquad \mbf x(t,0) = 0, \quad \partial_s\mbf x(t,1) + x(t) = 2w(t)\\
y(t) &= \mbf x(t,0).
\end{align*}
by calling
\begin{matlab}
\begin{verbatim}
 >> eqns = [diff(x,t)==-5*x+int(diff(X,s,1),s,[0,1])+u;
            diff(X,t)==9*X+diff(X,s,2)+s*w;
            subs(X,s,0)==0;
            subs(diff(X,s),s,1)==-x+2*w;
            y==subs(X,s,0)];
\end{verbatim}
\end{matlab}
To combine these equations into a single structure, we simply concatenate, and initialize, as
\begin{matlab}
\begin{verbatim}
 >> odepde = [eqns; out_eq];
 >> odepde = initialize(odepde);
\end{verbatim}
\end{matlab}
Here, the \texttt{initialize} function cleans up the PDE structure and checks for errors in the declaration, providing an overview of the variables it encounters as
\begin{matlab}
\begin{verbatim}
Encountered 2 state components: 
 x1(t),      of size 1, finite-dimensional;
 x2(t,s),    of size 1, differentiable up to order (2) in variables (s);

Encountered 1 actuator input: 
 u(t),    of size 1;

Encountered 1 exogenous input: 
 w(t),    of size 1;

Encountered 1 observed output: 
 y(t),    of size 1;

Encountered 1 regulated output: 
 z(t),    of size 2;

Encountered 2 boundary conditions: 
 F1(t) = 0, of size 1;
 F2(t) = 0, of size 1;
\end{verbatim}
\end{matlab}
After initialization, the system can be converted to an equivalent PIE as
\begin{matlab}
\begin{verbatim}
 >> PIE = convert(odepde)
 PIE = 
   pie_struct with properties:

      dim: 1;
     vars: [1×2 polynomial];
      dom: [1×2 double];

        T: [2×2 opvar];     Tw: [2×1 opvar];     Tu: [2×1 opvar]; 
        A: [2×2 opvar];     B1: [2×1 opvar];     B2: [2×1 opvar]; 
       C1: [2×2 opvar];    D11: [2×1 opvar];    D12: [2×1 opvar]; 
       C2: [1×2 opvar];    D21: [1×1 opvar];    D22: [1×1 opvar]; 
\end{verbatim}
\end{matlab}
The output is a \texttt{pie\_struct} object, storing \texttt{opvar} objects representing the PI operators defining the PIE representation of the input system -- see Chapter~\ref{ch:PIE} for more details.
Once the PIE structure is obtained, we can proceed to perform analysis, control, and simulation, as discussed in detail in Chapters~\ref{ch:PIESIM} and~\ref{ch:LPIs}.

\begin{boxEnv}{\textbf{Note on declaring equations}}
The $=$ symbol is not used while defining equations. Instead $==$ is used, since MATLAB uses $=$ as a protected symbol for assignment operation. Thus, any symbolic expression that needs to be added takes the form \texttt{expr==0} or \texttt{exprA==exprB}.
\end{boxEnv}

\begin{boxEnv}{\textbf{Note on PDE display}}
When displaying PDE variables and equations in the MATLAB Command Window, one must keep the following in mind:
\begin{itemize}
    \item PDE variables are always represented by a letter \texttt{x} for states, \texttt{y} for observed outputs, \texttt{z} for regulated outputs, \texttt{u} for controlled inputs, and \texttt{w} for exogenous inputs;

    \item Each variable is displayed with an integer subscript corresponding to the unique ID assigned to this variable. This ID is crucial for PIETOOLS to distinguish different PDE variables, but may become cumbersomely large when declaring multiple systems. To avoid this issue, the ID counter can be reset by calling \texttt{clear stateNameGenerator};

    \item When converting to a PIE, equations are always re-ordered to start with the ODE states, followed by the PDE (PIE) states, observed outputs, regulated outputs, and finally the boundary conditions. As such, the order of the different variables and equations after initialization or conversion may not be the same as initially declared.
\end{itemize}
\end{boxEnv}

\subsection{Declaring 2D PDEs}

The Command Line Input format can also be used to declare PDE systems involving multiple spatial variables. To illustrate, consider the following system of a coupled ODE, 1D PDE, and 2D PDE, with a distributed disturbance $\mbf{w}$ and output $\mbf{y}$:
\begin{align*}
    \frac{d}{dt}x_{1}(t)&=-x_{1}(t) + \mbf{x}_{4}(t,b,d) +u_{1}(t),   &       &t\geq 0,    \\
    \partial_{t}\mbf{x}_{2}(t,s_{1})&=\partial_{s_{1}}^{2}\mbf{x}_{2}(t,s_{1}) +\mbf{w}(t,s_{1}), &   &s_{1}\in[a,b], \\
    \partial_{t}\mbf{x}_{3}(t,s_{2})&=\partial_{s_{2}}^{2}\mbf{x}_{3}(t,s_{2}) +s_{2}u2(t), &   &s_{2}\in[c,d], \\
    \partial_{t}\mbf{x}_{4}(t,s_{1},s_{2})&=\partial_{s_{1}}^2\mbf{x}_{4}(t,s_{1},s_{2}) +\partial_{s_{2}}^2\mbf{x}_{4}(t,s_{1},s_{2}) +4\mbf{x}_{4}(t,s_{1},s_{2}),    \\
    \mbf{y}(t,s_{2})&=\bmat{\mbf{x}_{3}(t,s_{2})\\\mbf{x}_{4}(t,b,s_{2})}, \\
    z(t)&=\int_{a}^{b}\int_{c}^{d}\mbf{x}_{4}(t,s_{1},s_{2}) ds_{2} ds_{1}, \\
    \mbf{x}_{2}(t,a)&=x_{1}(t),\qquad\qquad \partial_{s_{1}}\mbf{x}_{2}(t,b)=0,   \\
    \mbf{x}_{3}(t,c)&=x_{1}(t),\qquad\qquad \mbf{x}_{3}(t,d)=0,   \\
    \mbf{x}_{4}(t,s_{1},c)&=\mbf{x}_{2}(t,s_{1}),\qquad \mbf{x}_{4}(t,s_{1},d)=0,   \\
    \mbf{x}_{4}(t,a,s_{2})&=\mbf{x}_{3}(t,s_{2}),\qquad \partial_{s_{1}}\mbf{x}_{4}(t,b,s_{2})=0.
\end{align*}
In this case, we have an ODE state $x_{1}(t)\in\R$, two 1D PDE states $\mbf{x}_{2}(t)\in L_{2}[a,b]$ and $\mbf{x}_{3}(t)\in L_{2}[c,d]$, and a 2D PDE state $\mbf{x}_{4}(t)\in L_{2}[[a,b]\times[c,d]]$. In addition, we have two controlled inputs $u_{1}(t),u_{2}(t)\in\R$ and a regulated output $z(t)\in\R$, as well as a distributed disturbance $\mbf{w}(t)\in L_{2}[a,b]$, and vector-valued sensed output $\mbf{y}(t)\in L_{2}^{2}[c,d]$ at all times $t\geq 0$. We declare this system for $[a,b]=[0,1]$ and $[c,d]=[-1,1]$ as follows.

\begin{codebox}
\begin{matlab}
\begin{verbatim}
 >> clear stateNameGenerator
 >> pvar t s1 s2
 >> a = 0;   b = 1;   c = -1;   d = 1;
 >> x1 = pde_var();      
 >> x2 = pde_var(s1,[a,b]);
 >> x3 = pde_var(s2,[c,d]);
 >> x4 = pde_var([s1;s2],[a,b;c,d]);
 >> w = pde_var('in',s1,[a,b]);
 >> z = pde_var('out');
 >> u1 = pde_var('control');     u2 = pde_var('control');
 >> y = pde_var('sense',2,s2,[c,d]);
 >> odepde = [diff(x1,t)==-x1+subs(x4,[s1;s2],[b;d])+u1;
              diff(x2,t)==diff(x2,s1,2)+w;
              diff(x3,t)==diff(x3,s2,2)+s2*u2;
              diff(x4,t)==diff(x4,s1,2)+diff(x4,s2,2)+4*x4;
              y==[x3;subs(x4,s1,b)];
              z==int(x3,[s1;s2],[a,b;c,d]);
              subs(x2,s1,a)==x1; subs(diff(x2,s1),s1,b)==0;
              subs(x3,s2,c)==x1; subs(x3,s2,d)==0;
              subs(x4,s2,c)==x2; subs(x4,s2,d)==0;
              subs(x4,s1,a)==x3; subs(diff(x4,s1),s1,b)==0];
 >> odepde = initialize(odepde);
 >> PIE = convert(odepde);
\end{verbatim}
\end{matlab}
\end{codebox}
Note here that the two spatial variables on which $\mbf{x}_{4}$ depends are declared as a column array \texttt{[s1;s2]}, with the corresponding interval on which each variable exists being specified in the respective rows of the second argument \texttt{[a,b;c,d]}. Furthermore, since the disturbance $\mbf{w}$ and output $\mbf{y}$ are distributed as well, the variables on which they depend (as well as the domain of those variables) must be passed in the call to \texttt{pde\_var} when declaring these objects. Finally, the size, \texttt{2}, of the vector-valued output $\mbf{y}$ is specified before the spatial variable \texttt{s2}.


\begin{boxEnv}{\textbf{Note on declaring $N$D PDEs}}
Although \texttt{pde\_struct} objects can be used to represents PDEs in arbitrary numbers of spatial variables, PIETOOLS does not currently offer tools for analysis or simulation of PDEs in three or more variables. Such features will be added in a later release.
\end{boxEnv}

\begin{boxEnv}{\textbf{Note on reordering of variables}}
The order of the different state variables, inputs, and outputs in the PDE will be determined by the order in which they are generated using the \texttt{pde\_var} function. As a result, the order of equations in the final PDE may not match the order in which they are declared.
\end{boxEnv}

\subsection{More examples of command line parser format}
In this subsection, we provide a few more examples to demonstrate the typical use of the command line parser. More specifically, we focus on examples involving inputs, outputs, delays, vector-valued PDEs, etc., to demonstrate the capabilities of command line parser.
\subsubsection{Example: Transport equation}
Consider the Transport equation which is modeled as a PDE with 1$^{st}$ derivatives in time and space given by 
\begin{align*}
    \partial_{t} \mbf{x}(t,s) &= 5\partial_s \mbf x(t,s)+u(t),\quad s\in[0,2]\\
    y(t) &= \mbf{x}(t,2),\\
    \mbf x(t,0) &= 0.
\end{align*}
Here, we use a control input in the domain and an observer at the right boundary with an intention to design an observer based controller. This system can be defined using the Command Line Input format as shown below.

\begin{codebox}
\begin{matlab}
\begin{verbatim}
 >> clear stateNameGenerator
 >> pvar t s;
 >> pde_var state X control u sense y;
 >> X.vars = s;    X.dom = [0,2];
 >> odepde = [diff(X,t)==5*diff(X,s)+u;
              subs(X,s,0)==0; 
              y==subs(X,s,2)];
 >> odepde = initialize(odepde);
 >> PIE = convert(odepde,'pie');
\end{verbatim}
\end{matlab}
\end{codebox}
In this case, we declare the PDE variables in a manner similar to how we declare the independent variables, using the arguments \texttt{state}, \texttt{control}, and \texttt{sense} to declare the subsequent variables to be state, controlled inputs, and sensed output variables. In doing so, the variable $X$ will be initially interpreted as an ODE state, which we resolve by manually setting the variables \texttt{X.vars} and domain \texttt{X.dom}.





\subsubsection{Example: PDE with delay terms}
The Command Line Input format can also be used to declare systems with temporal delay. To illustrate, consider a reaction-diffusion equation coupled to an ODE through a channel that is delayed by an amount $\tau=2$. Specifically, we consider the following equations
\begin{align*}
    \dot{x}(t) &= -5x(t),\\
    \partial_{t}\mbf{x}(t,s) &= 10\mbf x(t,s)+\partial_s^2 \mbf x(t,s)+x(t-2),\\
    \mbf x(t,0) &= 0 = \mbf x(t,1).
\end{align*}
This system can be declared in PIETOOLS and converted to a PIE using the following code. 
\begin{codebox}
\begin{matlab}
\begin{verbatim}
 >> clear stateNameGenerator
 >> pvar s t
 >> x = pde_var();
 >> X = pde_var(s,[0,1]);
 >> odepde = [diff(x,t)==-5*x; 
              diff(X,t)==diff(X,s,2)+subs(x,t,t-2);
              subs(X,s,0)==0;
              subs(X,s,1)==0];
 >> odepde = initialize(odepde);
 >> PIE = convert(odepde,'pie');
\end{verbatim}
\end{matlab}
\end{codebox}
Running this code, and in particular the last line \texttt{PIE=convert(odepde)}, PIETOOLS will display several warnings as
\begin{matlab}
\begin{verbatim}
 Added 1 state components: 
    x3(t,ntau2)  := x1(t-ntau2);

 Variable s has been merged with variable ntau_2.
 All spatial variables have been rescaled to exist on the interval [-1,1].

 The state components have been re-indexed as:
    x1(t)       -->   x1(t)
    x3(t,s1)    -->   x2(t,s1)
    x2(t,s1)    -->   x3(t,s1)
\end{verbatim}
\end{matlab}
This is because the PIE representation does not support temporal delays in any of the variables. Instead, an additional state variable $\mbf{v}(t,r)=x_{1}(t-r)$ is introduced to represent the delayed state. This state variable will be governed by a 1D transport equation, satisfying
\begin{equation}
    \partial_{t}\mbf{v}(t,r)=-\partial_{r}\mbf{v}(t,r),\qquad r\in[0,2],\qquad
    \mbf{v}(t,2)=x(t).
\end{equation}
However, simply adding this equation to our reaction-diffusion PDE would yield a 2D system: existing on $(s,r)\in[0,1]\times[0,2]$. To reduce complexity, therefore, PIETOOLS automatically rescales the variables $s$ and $r$ to both exist on the spatial domain $[-1,1]$, rescaling the PDE variables and equations accordingly to represent the system as a 1D ODE-PDE system
\begin{align*}
    \dot{x}_{1}(t) &= -5x_{1}(t),\\
    \partial_{t} \mbf{x}_{2}(t,s) &= \partial_{s}\mbf{x}_{2}(t,s),  & s&\in[-1,1],  \\
    \partial_{t} \mbf{x}_{3}(t,s) &= 10\mbf{x}_{3}(t,s)+4\partial_s^2 \mbf {x}_{3}(t,s)+\mbf{x}_{2}(t,-1),\\
    \mbf{x}_{3}(t,-1) &= 0 = \mbf{x}_{3}(t,1)\quad \mbf{x}_{2}(t,1) = x_{1}(t),
\end{align*}
where now $x_{1}(t)=x(t)$, $\mbf{x}_{2}(t,s)=\mbf{v}(t,1-s)=x(t+s-1)$ and $\mbf{x}_{3}(t,s)=\mbf{x}(t,0.5(1+s))$. Consequently, the PIE representation will also involve three state variables, $(x_{1}(t),\partial_{s}\mbf{x}_{2}(t),\partial_{s}^2\mbf{x}_{3}(t,s))$, as indicated by the dimensions of e.g. the operator \texttt{PIE.T}
\begin{matlab}
\begin{verbatim}
 >> PIE
 PIE = 
   pie_struct with properties:
      dim: 1;
     vars: [1×2 polynomial];
      dom: [1×2 double];

        T: [3×3 opvar];     Tw: [3×0 opvar];     Tu: [3×0 opvar]; 
        A: [3×3 opvar];     B1: [3×0 opvar];     B2: [3×0 opvar]; 
       C1: [0×3 opvar];    D11: [0×0 opvar];    D12: [0×0 opvar]; 
       C2: [0×3 opvar];    D21: [0×0 opvar];    D22: [0×0 opvar]; 
\end{verbatim}
\end{matlab}




\subsubsection{Example: Wave equation}\label{ex:parser_wave}

Next, consider the 1D wave equation, modeled as a PDE with $2^{nd}$-order derivative in time:
\begin{align*}
&\ddot{\mbf{x}}(t,s) = \partial_{s}^2\mbf{x}(t,s),\qquad s\in[0,1]\\
&\mbf{x}(t,0)=\partial_{s}\mbf{x}(t,1)=0
\end{align*}
This system can be declared using the Command Line Input format as shown below.
\begin{codebox}
\begin{matlab}
\begin{verbatim}
 >> clear stateNameGenerator
 >> pvar s t
 >> x = pde_var(s,[0,1]);
 >> eqns = [diff(x,t,2)==diff(x,s,2); 
            subs(x,s,0)==0;    subs(diff(x,s),s,1)==0];
 >> PDE = initialize(eqns);
 >> PIE = convert(PDE,'pie');
\end{verbatim}
\end{matlab}
\end{codebox}
Running this code, the following message will be displayed:
\begin{matlab}
\begin{verbatim}
 The following fundamental state components have been introduced:
    x1(t,s)   <--          (d/ds)^2 x1(t,s)
    x2(t,s)   <--   (d/dt) (d/ds)^2 x1(t,s)
\end{verbatim}
\end{matlab}
This message warns the user of the fact that, not does the returned PIE representation model the fundamental state variable $\mbf{x}_{\text{f},1}(t):=\partial_{s}^2\mbf{x}(t)$, it also models the first-order temporal derivative of this variable, $\mbf{x}_{\text{f},2}(t):=\dot{\mbf{x}}_{\text{f},1}(t)=\partial_{s}^2\dot{\mbf{x}}(t)$. This is to make sure that the PIE representation is still first-order in time, representing the second-order temporal derivative $\ddot{\mbf{x}}_{\text{f},1}(t)$ as a first-order temporal derivative of the auxiliary state $\mbf{x}_{\text{f},2}$, i.e. $\ddot{\mbf{x}}_{\text{f},1}(t)=\dot{\mbf{x}}_{\text{f},2}(t)$. More details on how this PIE representation is constructed are provided in Subsection~\ref{subsec:PIE:PDE2PIE:pde_2_pie:wave}.

\subsubsection{Example: Beam equation}\label{ex:parser_timoshenko}
Here we consider the Timoshenko Beam equations which are modeled as a PDE with $2^{nd}$-order derivatives in both space and time:
\begin{align*}
&\ddot{w} = \partial_s (w_s-\phi), \quad \ddot{\phi} = \phi_{ss}+(w_s - \phi)\\
&\phi(0)=w(0)=0, \quad \phi_s(1)=0,\quad w_s(1)-\phi(1)=0.
\end{align*}
While this system could be directly declared using the Command Line Input format, we now instead redefine the state variables to convert it to a PDE with only first-order temporal derivatives. In particular, using the PDE state $\mbf x= [\dot{w}, w_s-\phi, \dot{\phi},\phi_s]$, we get
\begin{align*}
    &\dot{\mbf x}(t,s) = \bmat{0&0&0&0\\0&0&-1&0\\0&1&0&0\\0&0&0&0}\mbf x(t,s) + \bmat{0&1&0&0\\1&0&0&0\\0&0&0&1\\0&0&1&0}\partial_s \mbf x(t,s),\\
    &\bmat{1&0&0&0&0&0&0&0\\0&0&1&0&0&0&0&0\\0&0&0&0&0&0&0&1\\0&0&0&0&0&1&0&0}\bmat{\mbf x(t,0)\\\mbf x(t,1)} = 0,
\end{align*}
which is a vector-valued transport equation with a reaction term. We define this system using the Command Line Input format as shown below.

\begin{codebox}
\begin{matlab}
\begin{verbatim}
 >> clear stateNameGenerator
 >> pvar s t
 >> x = pde_var(4,s,[0,1]);
 >> A0 = [0,0,0,0;0,0,-1,0;0,1,0,0;0,0,0,0];
 >> A1 = [0,1,0,0;1,0,0,0;0,0,0,1;0,0,1,0];
 >> B = [1,0,0,0,0,0,0,0;0,0,1,0,0,0,0,0;0,0,0,0,0,0,0,1;0,0,0,0,0,1,0,0];
 >> eqns = [diff(x,t)==A0*x+A1*diff(x,s); 
            B*[subs(x,s,0); subs(x,s,1)]==0];
 >> PDE = initialize(eqns);
 >> PIE = convert(PDE,'pie');
\end{verbatim}
\end{matlab}
\end{codebox}

As seen above, the presence of vector-valued states does not change the typical workflow to define the PDE. As long as the dimensions of the parameters and vectors used in the equations match, the process and steps remain the same.





\section{Alternative Input Formats for PDEs}\label{sec:PDE_DDE_representation:alt_PDEs}

In addition to the command line parser input format, PIETOOLS 2025 offers a graphical user interface (GUI) for declaring 1D PDEs, that allows users to simultaneously visualize the PDE that they are specifying. We briefly introduce this input format in Subsection~\ref{subsec:PDE_DDE_representation:alt_PDEs:GUI}, referring to Chapter~\ref{ch:alt_PDE_input} for more details.


\subsection{A GUI for Declaring PDEs}\label{subsec:PDE_DDE_representation:alt_PDEs:GUI}

Aside from the Command Line Input format, the GUI is the easiest way to declare linear 1D ODE-PDE systems in PIETOOLS, providing a simple, intuitive and interactive visual interface to directly input the model. The GUI can be opened by running \texttt{PIETOOLS\_PIETOOLS\_GUI} from the command line, opening a window like the one displayed in the picture below:
\begin{figure}[H]
	\centering
	\includegraphics[width=0.95\textwidth]{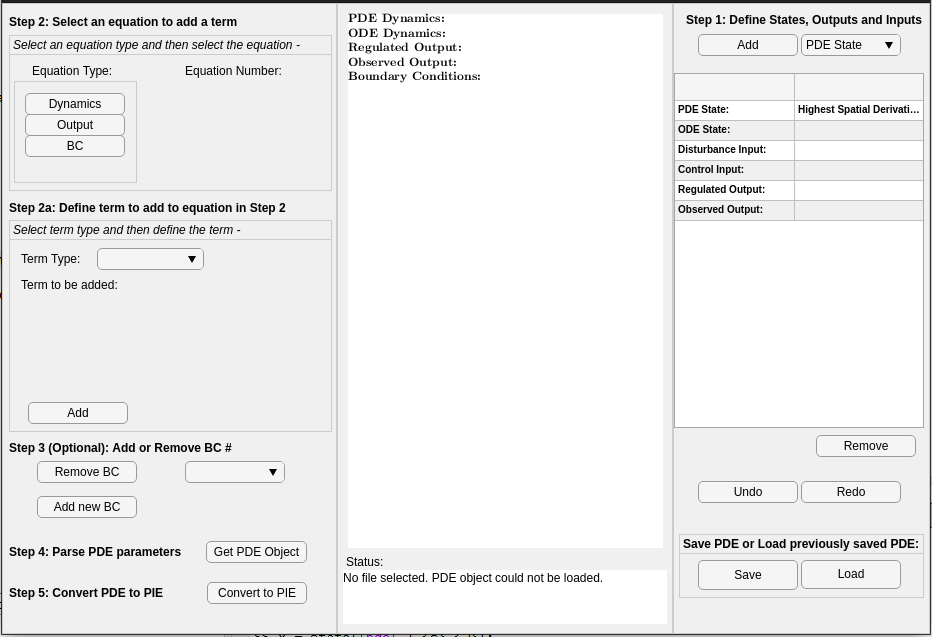}
	\caption{Example of empty GUI window.}
	\label{gui_empty}
\end{figure}
Then, the desired PDE can be declared following steps 1 through 4, first specifying the state variables, inputs and outputs, then declaring the different equations term by term, and finally adding any boundary conditions. It also allows PDE models to be saved and loaded, so that e.g. the system
\begin{align*}
    \dot{\mbf{x}}(t,s)&=\partial_{s}^2 \mbf{x}(t,s) + sw(t), &  s&\in[0,1]\\
    z(t)&=\int_{0}^{1}\mbf{x}(t,s)ds,   \\
    \mbf{x}(t,0)&=0,\qquad \mbf{x}(t,1)=0,
\end{align*}
can be retrieved by simply loading the file\\
\texttt{PIETOOLS\_PDE\_Ex\_Heat\_Eq\_with\_Distributed\_Disturbance\_GUI}\\
from the library of PDE examples, returning a window that looks like
\begin{figure}[H]
	\centering
	\includegraphics[width=0.95\textwidth]{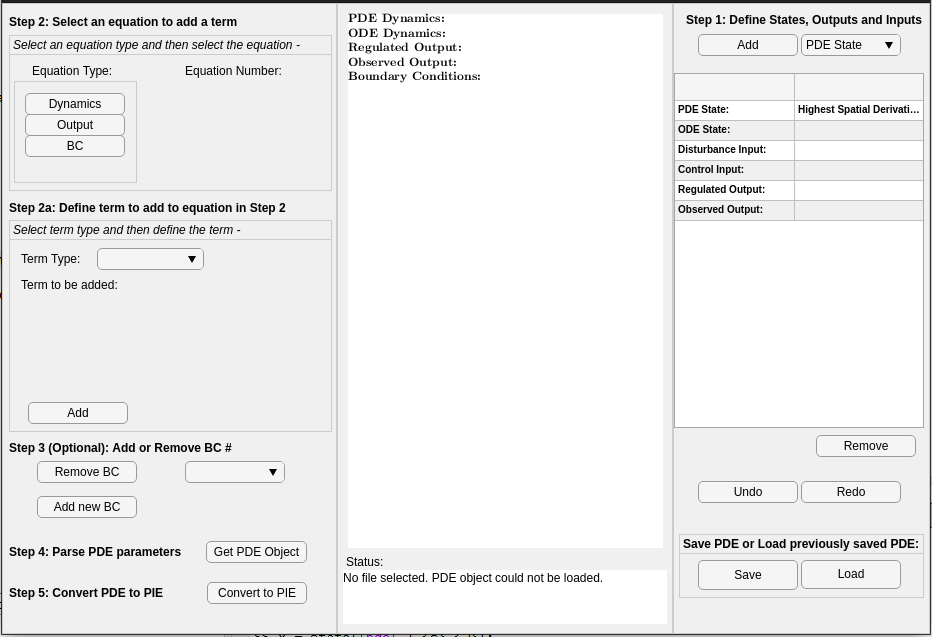}
	\caption{GUI window after loading the file\\ \texttt{PIETOOLS\_PDE\_Ex\_Heat\_Eq\_with\_Distributed\_Disturbance\_GUI} from the library of PDE examples.}
	\label{gui_loaded}
\end{figure}
The declared system can be parsed by clicking \texttt{Get PDE Objects}, returning a structure \texttt{PDE\_GUI} in the MATLAB workspace that can be used for further analysis. For more details on how to use the GUI, we refer to Section~\ref{sec:GUI}.

\section{Batch Input Format for DDEs}\label{sec:PDE_DDE_representation:DDEs}
 
The DDE data structure allows the user to declare any of the matrices in the following general form of delay-differential equation.
\begin{align}
	&\bmat{\dot{x}(t)\\z(t) \\ y(t)}=\bmat{A_0 & B_{1} & B_{2}\\ C_{1} & D_{11} &D_{12}\\ C_{2} & D_{21} &D_{22}}\bmat{x(t)\\w(t)\\u(t)}+\sum_{i=1}^K \bmat{A_i & B_{1i} & B_{2i}\\C_{1i} & D_{11i} & D_{12i}\\C_{2i} & D_{21i} & D_{22i}} \bmat{x(t-\tau_i)\\w(t-\tau_i)\\u(t-\tau_i)}\notag \\
& \hspace{2cm}+\sum_{i=1}^K \int_{-\tau_i}^0\bmat{A_{di}(s) & B_{1di}(s) &B_{2di}(s)\\C_{1di}(s) & D_{11di}(s) & D_{12di}(s)\\C_{2di}(s) & D_{21di}(s) & D_{22di}(s)} \bmat{x(t+s)\\w(t+s)\\u(t+s)}ds \label{eqn:DDE_ch4}
\end{align}
In this representation, it is understood that
\begin{itemize}
\item The present state is $x(t)$.\vspace{-2mm}
\item The disturbance or exogenous input is $w(t)$. These signals are not typically known or alterable. They can account for things like unmodelled dynamics, changes in reference, forcing functions, noise, or perturbations.\vspace{-2mm}
\item The controlled input is $u(t)$. This is typically the signal which is influenced by an actuator and hence can be accessed for feedback control. \vspace{-2mm}
\item The regulated output is $z(t)$. This signal typically includes the parts of the system to be minimized, including actuator effort and states. These signals need not be measured using senors.\vspace{-2mm}
\item The observed or sensed output is $y(t)$. These are the signals which can be measured using sensors and fed back to an estimator or controller.\vspace{-2mm}
\end{itemize}

Note that this input format extends the possibilities of the command-line parser, which does offer support for the user to add the terms with distributed delays $A_{di}(s)$,$B_{1di}(s)$, $B_{2di}(s)$, $C_{1di}(s)$, $D_{11di}(s)$, $D_{12di}(s)$, $C_{2di}(s)$, $D_{21di}(s)$, and $D_{22di}(s)$. To add any term to the DDE structure in this Batch input format, simply declare its value. For example, to represent 
\[
\dot x(t)=-x(t-1),\qquad z(t)=x(t-2)
\]
we use
	\begin{matlab}
		>> DDE.tau = [1 2];\\
		>> DDE.Ai\{1\} = -1;\\
		>> DDE.C1i\{2\} = 1;
	\end{matlab}
All terms not declared are assumed to be zero. The exception is that we require the user to specify the values of the delay in \texttt{DDE.tau}. When you are done adding terms to the DDE structure, use the function \texttt{DDE=PIETOOLS\_initialize\_DDE(DDE)}, which will check for undeclared terms and set them all to zero. It also checks to make sure there are no incompatible dimensions in the matrices you declared and will return a warning if it detects such malfeasance. The complete list of terms and DDE structural elements is listed in Table~\ref{tab:DDE_parameters_ch4}.

\begin{table}[ht!]\vspace{-2mm}
\begin{center}{
\begin{tabular}{c|c||c|c||c|c}
  \multicolumn{6}{c}{\textbf{ODE Terms:}}\\
Eqn.~\eqref{eqn:DDE_ch4}  & \texttt{DDE.}  &   Eqn.~\eqref{eqn:DDE_ch4}  & \texttt{DDE.} &   Eqn.~\eqref{eqn:DDE_ch4}  & \texttt{DDE.}\\
\hline
$A_0$    & \texttt{A0} & $B_{1}$ & \texttt{B1} &$B_{2}$&\texttt{B2}\\ 
$C_{1}$ & \texttt{C1} & $D_{11}$ &\texttt{D11}&$D_{12}$&\texttt{D12}\\ 
$C_{2}$ & \texttt{C2} & $D_{21}$ &\texttt{D21}&$D_{22}$&\texttt{D22} \\ \hline \\
 \multicolumn{6}{c}{ \textbf{Discrete Delay Terms:}}\\
Eqn.~\eqref{eqn:DDE_ch4}  & \texttt{DDE.}  &   Eqn.~\eqref{eqn:DDE_ch4}  & \texttt{DDE.} &   Eqn.~\eqref{eqn:DDE_ch4}  & \texttt{DDE.}\\
\hline
$A_i$    & \texttt{Ai\{i\}} & $B_{1i}$ & \texttt{B1i\{i\}} &$B_{2i}$&\texttt{B2i\{i\}}\\ 
$C_{1i}$ & \texttt{C1i\{i\}} & $D_{11i}$ &\texttt{D11i\{i\}}&$D_{12i}$&\texttt{D12i\{i\}}\\ 
$C_{2i}$ & \texttt{C2i\{i\}} & $D_{21i}$ &\texttt{D21i\{i\}}&$D_{22i}$&\texttt{D22i\{i\}}\\
\hline \\  \multicolumn{6}{c}{\textbf{Distributed Delay Terms: May be functions of \texttt{pvar s}}}\\
Eqn.~\eqref{eqn:DDE_ch4}  & \texttt{DDE.}  &   Eqn.~\eqref{eqn:DDE_ch4}  & \texttt{DDE.} &   Eqn.~\eqref{eqn:DDE_ch4}  & \texttt{DDE.}\\
\hline
$A_{di} $   & \texttt{Adi\{i\}} & $B_{1di}$ & \texttt{B1di\{i\}} &$B_{2di}$&\texttt{B2di\{i\}}\\ 
$C_{1di} $& \texttt{C1di\{i\}} &$ D_{11di} $&\texttt{D11di\{i\}}&$D_{12di}$&\texttt{D12di\{i\}}\\ 
$C_{2di}$ & \texttt{C2di\{i\}} & $D_{21di} $&\texttt{D21di\{i\}}&$D_{22di}$&\texttt{D22di\{i\}}\\
\end{tabular}
}
\end{center}\vspace{-2mm}
\caption{ Equivalent names of Matlab elements of the \texttt{DDE} structure terms for terms in Eqn.~\eqref{eqn:DDE_ch4}. For example, to set term \texttt{XX} to \texttt{YY}, we use \texttt{DDE.XX=YY}.  In addition, the delay $\tau_i$ is specified using the vector element \texttt{DDE.tau(i)} so that if $\tau_1=1, \tau_2=2, \tau_3=3$, then \texttt{DDE.tau=[1 2 3]}. }\label{tab:DDE_parameters_ch4}\end{table}

\subsection{Initializing a DDE data structure} The user need only add non-zero terms to the DDE structure. All terms which are not added to the data structure are assumed to be zero. Before conversion to another representation or data structure, the data structure will be initialized using the command
	\begin{flalign*}
		&\texttt{DDE = initialize\_PIETOOLS\_DDE(DDE)}&
	\end{flalign*}
This will check for dimension errors in the formulation and set all non-zero parts of the \texttt{DDE} data structure to zero. Note that, to make the code robust, all PIETOOLS conversion utilities perform this step internally.

\section{Alternative Input Formats for TDSs}\label{sec:PDE_DDE_representation:alt_TDSs}

Although the delay differential equation (DDE) format is perhaps the most intuitive format for representing time-delay systems (TDS), it is not the only representation of TDS systems, and not every TDS can be represented in this format. For this reason, PIETOOLS includes two additional input format for TDSs, namely the Neutral Type System (NDS) representation, and Differential-Difference Equation (DDF) representation. Here, the structure of a NDS is identical to that of a DDE except for 6 additional terms: 
\begin{align*}
	\bmat{\dot{x}(t)\\z(t) \\ y(t)}&=\bmat{A_0 & B_{1} & B_{2}\\ C_{1} & D_{11} &D_{12}\\ C_{2} & D_{21} &D_{22}}\bmat{x(t)\\w(t)\\u(t)}+\sum_{i=1}^K \bmat{A_i  & B_{1i} & B_{2i}& E_i\\C_{1i}& D_{11i} & D_{12i} & E_{1i}\\C_{2i} & D_{21i} & D_{22i}&E_{2i}} \bmat{x(t-\tau_i)\\w(t-\tau_i)\\u(t-\tau_i)\\ \dot x(t-\tau_i)}\notag\\[-3mm]
& +\hspace{-1mm}\sum_{i=1}^K \hspace{0mm}\int_{-\tau_i}^0\hspace{-1mm}\bmat{A_{di}(s) & \hspace{-1mm}B_{1di}(s) &\hspace{-1mm}B_{2di}(s)& \hspace{-1mm}E_{di}(s)\\C_{1di}(s) & \hspace{-1mm}D_{11di}(s) & \hspace{-1mm}D_{12di}(s)& \hspace{-1mm}E_{1di}(s)\\C_{2di}(s) &\hspace{-1mm}D_{21di}(s) & \hspace{-1mm}D_{22di}(s)& \hspace{-1mm}E_{2di}(s)} \hspace{-2mm}\bmat{x(t+s)\\w(t+s)\\u(t+s)\\ \dot x(t+s)}\hspace{-1mm}ds. 
\end{align*}
These new terms are parameterized by $E_i,E_{1i}$, and $E_{2i}$ for the discrete delays and by $E_{di},E_{1di}$, and $E_{2di}$ for the distributed delays, and should be included in a NDS object as, e.g. \texttt{NDS.E\{1\}=1}. On the other hand, the DDF representation is more compact but less transparent than the DDE and NDS representation, taking the form
\begin{align*}
	\bmat{\dot{x}(t)\\ z(t)\\y(t)\\r_i(t)}&=\bmat{A_0 & B_1& B_2\\C_1 &D_{11}&D_{12}\\C_2&D_{21}&D_{22}\\C_{ri}&B_{r1i}&B_{r2i}}\bmat{x(t)\\w(t)\\u(t)}+\bmat{B_v\\D_{1v}\\D_{2v}\\D_{rvi}} v(t) \notag\\
v(t)&=\sum_{i=1}^K C_{vi} r_i(t-\tau_i)+\sum_{i=1}^K \int_{-\tau_i}^0C_{vdi}(s) r_i(t+s)ds.
\end{align*}
In this representation, the output signal from the ODE part is decomposed into sub-components $r_i$, each of which is delayed by amount $\tau_i$. Identifying these sub-components is often challenging, so in most cases it will be preferable to use the NDS or DDE representation instead. However, the DDF representation is more general than either the DDE or NDS representation, so PIETOOLS also includes an input format for declaring DDF systems. For more information on how to declare systems in the DDF or NDS representation, and how to convert between different representations, we refer to Chapter~\ref{ch:alt_DDE_input}.


\chapter{Conversion of PDEs and DDEs to PIEs and Closing the Loop}\label{ch:PIE}

In the previous chapter, we showed how general linear ODE-PDE and DDE systems can be declared in PIETOOLS. In order to analyze such systems, PIETOOLS represents each of them in a standardized format, as a Partial Integral Equation (PIE). This format is parameterized by partial integral, or PI operators, rather than by differential operators, allowing PIEs to be analysed by solving optimization problems on these PI operators (see Chapter~\ref{ch:LPIs}).

In this chapter, we show how an equivalent PIE representation of PDE and DDE systems can be computed in PIETOOLS. In particular, in Section~\ref{sec:PIE:PDE2PIE:what_is_a_pie}, we first provide a simple illustration of what a PIE is. In Sections~\ref{sec:PIE:PDE2PIE:pde_2_pie} and~\ref{sec:PIE:PDE2PIE:dde_2_pie}, we then show how a PDE and a DDE can be converted to a PIE, and in Section~\ref{sec:PIE:PDE2PIE:io_pde_2_pie}, we show how a PDE with inputs and outputs can be converted to a PIE. To reduce notation, we demonstrate the PDE conversion only for 1D systems, though we note that the same steps also work for 2D PDEs.

\section{What is a PIE?}\label{sec:PIE:PDE2PIE:what_is_a_pie}

To illustrate the concept of partial integral equations, suppose we have a simple 1D PDE
\begin{align}\label{eq:1D_PDE_example}
 \dot{\mbf{x}}(t,s)&=2\partial_{s}\mbf{x}(t,s) + 10\mbf{x}(t,s),    &   s&\in[0,1],  \\
 \mbf{x}(t,0)&=0.   \nonumber
\end{align}
In this system, the PDE state $\mbf{x}(t)$ at any time $t\geq 0$ is a function of $s$, that has to satisfy the boundary condition (BC) $\mbf{x}(t,0)=0$. Moreover, the state must be at least first-order differentiable with respect to $s$, for us to be able to evaluate the derivative $\partial_{s}\mbf{x}(t,s)$. As such, a more fundamental state would actually be this first-order derivative $\partial_{s}\mbf{x}(t)$ of the state, which does not need to be differentiable, nor does it need to satisfy any boundary conditions. We therefore define $\mbf{x}_{\text{f}}(t,s):=\partial_{s}\mbf{x}(t,s)$ as the \textit{fundamental state} associated with this PDE. Using the fundamental theorem of calculus, we can then express the PDE state in terms of the fundamental state as
\begin{align*}
 \mbf{x}(t,s)&=\mbf{x}(t,0)+\int_{0}^{s}\partial_{s}\mbf{x}(t,\theta)d\theta=\mbf{x}(t,0)+\int_{0}^{s}\mbf{x}_{\text{f}}(t,\theta)d\theta = \int_{0}^{s}\mbf{x}_{\text{f}}(t,\theta)d\theta,
\end{align*}
where we invoke the boundary condition $\mbf{x}(t,0)=0$. Substituting this result into the PDE, we arrive at an equivalent representation of the system as
\begin{align}\label{eq:1D_PIE_example}
    \int_{0}^{s}\dot{\mbf{x}}_{\text{f}}(t,\theta)d\theta &= 2\mbf{x}_{\text{f}}(t,s) + \int_{0}^{s}10\mbf{x}_{\text{f}}(t,\theta)d\theta,   &   s&\in[0,1],
\end{align}
in which the fundamental state $\mbf{x}_{\text{f}}(t)$ does not need to satisfy any boundary conditions, nor does it need to be differentiable with respect to $s$. We refer to this representation as the Partial Integral Equation, or PIE representation of the system, involving only partial integrals, rather than partial derivatives with respect to $s$. It can be shown that for any well-posed linear PDE -- meaning that the solution to the PDE is uniquely defined by the dynamics and the BCs -- there exists an equivalent PIE representation. In PIETOOLS, this equivalent representation can be obtained by simply calling \texttt{convert} for the desired PDE structure \texttt{PDE}, returning a structure \texttt{PIE} that corresponds to the equivalent PIE representation.

\section{Converting a PDE to a PIE}\label{sec:PIE:PDE2PIE:pde_2_pie}

Suppose that we have a PDE structure \texttt{PDE}, defining a 1D heat equation with integral boundary conditions:
\begin{align}\label{eq:PDE_ex_PDE2PIE_1}
     \dot{\mbf{x}}(t,s)&=\partial_{s}^{2}\mbf{x}(t,s),  &   s&\in[0,1]  \nonumber\\
     \text{with BCs}\hspace*{1.0cm} \mbf{x}(t,0)&+\int_{0}^{1}\mbf{x}(t,s)ds=0,    \hspace*{1.0cm}
     \mbf{x}(t,1)+\int_{0}^{1}\mbf{x}(t,s)ds=0
\end{align}
In this system, the state $\mbf{x}(t,s)$ at each time $t\geq 0$ must be at least second order differentiable with respect to $s$, so we define the associated fundamental state as $\mbf{x}_{\text{f}}(t,s)=\partial_{s}^2\mbf{x}(t,s)$.
We implement this system in PIETOOLS using the Command Line Input format as follows:
\begin{matlab}
\begin{verbatim}
 >> pvar s t
 >> x = pde_var(s,[0,1]);
 >> PDE_dyn = diff(x,t) == diff(x,s,2);
 >> PDE_BCs = [subs(x,s,0) + int(x,s,[0,1]) == 0;
               subs(x,s,1) + int(x,s,[0,1]) == 0];
 >> PDE = [PDE_dyn; PDE_BCs];
\end{verbatim}
\end{matlab}
Then, we can derive the associated PIE representation by simply calling
\begin{matlab}
\begin{verbatim}
 >> PIE = convert(PDE,`pie')
 PIE = 
   pie_struct with properties:

     dim: 1;
    vars: [1×2 polynomial];
     dom: [0 1];

       T: [1×1 opvar];     Tw: [1×0 opvar];     Tu: [1×0 opvar]; 
       A: [1×1 opvar];     B1: [1×0 opvar];     B2: [1×0 opvar]; 
      C1: [0×1 opvar];    D11: [0×0 opvar];    D12: [0×0 opvar]; 
      C2: [0×1 opvar];    D21: [0×0 opvar];    D22: [0×0 opvar]; 
\end{verbatim}
\end{matlab}
In this structure, the field \texttt{dim} corresponds to the spatial dimensionality of the system, with \texttt{dim=1} indicating that this is a 1D PIE. The fields \texttt{vars} and \texttt{dom} define the spatial variables in the PIE and their domain, with
\begin{matlab}
\begin{verbatim}
 >> PIE.vars
 ans = 
   [ s, s_dum]

 >> PIE.dom
 ans = 
      0     1
\end{verbatim}
\end{matlab}
indicating that \texttt{s} is the primary variable, \texttt{s\_dum} the dummy variable (used for integration), and both exist on the domain $s,s_{\text{dum}} \in[0,1]$.
We note that the remaining fields in the \texttt{PIE} structure are all \texttt{opvar} objects, representing PI operators in 1D. Moreover, most of these operators are empty, being of dimension $1\times 0$, $0\times 1$ or $0\times 0$. This is because the PDE~\eqref{eq:PDE_ex_PDE2PIE_1} does not involve any inputs or outputs, and therefore its associated PIE has the simple structure
\begin{align*}
    \bl(\mcl{T}\dot{\mbf{x}}_{\text{f}}\br)(t,s)&=\bl(\mcl{A}\mbf{x}_{\text{f}}\br)(t,s),
\end{align*}
where the operator $\mcl{T}$ maps the fundamental state $\mbf{x}_{\text{f}}$ back to the PDE state $\mbf{x}$ as
\begin{align*}
    \bl(\mcl{T}\mbf{x}_{\text{f}}\br)(t,s)&=\mbf{x}(t,s).
\end{align*}
For the PDE~\eqref{eq:PDE_ex_PDE2PIE_1}, we know that $\mbf{x}_{\text{f}}(t,s)=\partial_{s}^2\mbf{x}(t,s)$. The associated operators $\mcl{T}$ and $\mcl{A}$ are represented by the \texttt{opvar} objects \texttt{T} and \texttt{A} in the \texttt{PIE} structure, for which we find that
\begin{matlab}
\begin{verbatim}
 >> T = PIE.T
 T =
     [] | [] 
     ---------
     [] | T.R 
     
 T.R = 
   [0] | [s*s_dum-0.25*s_dum^2-0.75*s_dum] | [s*s_dum-0.25*s_dum^2-s+0.25*s_dum]        
 >> A = PIE.A
 A =
     [] | [] 
     ---------
     [] | A.R 
 
 A.R = 
          [1] | [0] | [0]
\end{verbatim}
\end{matlab}
We conclude that the PDE~\eqref{eq:PDE_ex_PDE2PIE_1} is equivalently represented by the PIE
\begin{align*}
    \int_{0}^{s}\underbrace{\bbl(\bbl(s-\frac{1}{4}\theta-\frac{3}{4}\bbr)\theta
    \bbr)}_{\texttt{PIE.T.R.R1}}\dot{\mbf{x}}_{\textnormal{f}}(t,\theta)d\theta + \int_{s}^{1}\underbrace{\bbl(\bbl(s-\frac{1}{4}\theta+\frac{1}{4}\bbr)\theta-s\bbr)}_{\texttt{PIE.T.R.R2}}\dot{\mbf{x}}_{\textnormal{f}}(t,\theta)d\theta = \underbrace{1}_{\texttt{PIE.A.R.R0}}\mbf{x}_{\textnormal{f}}(t,s).
\end{align*}

\subsection{PDEs with Higher-Order Temporal Derivatives}\label{subsec:PIE:PDE2PIE:pde_2_pie:wave}

PDE to PIE conversion is also supported for systems involving higher-order temporal derivatives. For example, suppose we have a wave equation with the same boundary conditions as in~\eqref{eq:PDE_ex_PDE2PIE_1}:
\begin{align}\label{eq:PDE_ex_PDE2PIE_wave}
     \ddot{\mbf{x}}(t,s)&=\partial_{s}^{2}\mbf{x}(t,s),  &   s&\in[0,1]  \nonumber\\
     \text{with BCs}\hspace*{1.0cm} \mbf{x}(t,0)&+\int_{0}^{1}\mbf{x}(t,s)ds=0,    \hspace*{1.0cm}
     \mbf{x}(t,1)+\int_{0}^{1}\mbf{x}(t,s)ds=0
\end{align}
We can declare this PDE in PIETOOLS as
\begin{matlab}
\begin{verbatim}
 >> PDE_dyn = diff(x,t,2) == diff(x,s,2);
 >> PDE = [PDE_dyn; PDE_BCs];
\end{verbatim}
\end{matlab}
For this system, the state $\mbf{x}(t)$ at each time $t\geq 0$ is again second-order differentiable in space, and we can use the same PI operator $\mcl{T}$ as before to express $\mbf{x}(t)=\mcl{T}\mbf{x}_{\text{f}}(t)$ for $\mbf{x}_{f}(t)=\partial_{s}^2\mbf{x}(t)$. As such, the associated PIE representation is given by
\begin{equation*}
    \left(\mcl{T}\ddot{\mbf{x}}_{\text{f}}\right)(t,s)=\mbf{x}_{\text{f}}(t,s).
\end{equation*}
However, this PIE now involves a second-order derivative in time, which cannot be represented using the \texttt{pie\_var} structure. To resolve this, when calling \texttt{PIE=convert(PDE,'pie')} for the PDE structure, PIETOOLS automatically expands the second-order temporal derivative by introducing a secondary state variable, $\mbf{x}_{\text{f},2}(t)=\dot{\mbf{x}}_{\text{f}}(t)$. The resulting structure \texttt{PIE} then represents the PIE modeling the dynamics of the augmented fundamental state $\bmat{\mbf{x}_{\text{f}}(t)\\\mbf{x}_{\text{f},2}(t)}$, which satisfy
\begin{equation*}
    \overbrace{\bmat{I\\&\mcl{T}}}^{\texttt{PIE.T}}\bmat{\dot{\mbf{x}}_{\text{f}}(t)\\\dot{\mbf{x}}_{\text{f},2}(t)}
    =\overbrace{\bmat{0&I\\I&0}}^{\texttt{PIE.A}}\bmat{\mbf{x}_{\text{f}}(t)\\\mbf{x}_{\text{f},2}(t)},
\end{equation*}
so that, in this case,
\begin{matlab}
\begin{verbatim}     
 >> A = PIE.A
 A.R = 
          [0,1] | [0,0] | [0,0] 
          [1,0] | [0,0] | [0,0] 
\end{verbatim}
\end{matlab}
Similarly, for systems involving an $N$th-order temporal derivative of the PDE state, the PIE will model the dynamics of the augmented state defined by $N$ fundamental state variables:
\begin{equation*}
    (\mbf{x}_{\text{f}}(t),\partial_{t}\mbf{x}_{\text{f}}(t),\partial_{t}^{2}{\mbf{x}}_{\text{f}}(t),\cdots,\partial_{t}^{N}\mbf{x}_{\text{f}}(t)).
\end{equation*}
An alternative approach for expanding higher-order temporal derivatives is presented in Subsection~\ref{subsec:alt_PDE_input:terms_input_PDE:expand_tderivatives}, introducing auxiliary PDE state variables (e.g. $\mbf{x}_{2}(t)=\dot{\mbf{x}}(t)$) rather than auxiliary fundamental state variables to obtain a first-order in time representation.

\section{Converting a DDE to a PIE}\label{sec:PIE:PDE2PIE:dde_2_pie}

Just like PDEs, DDEs (and other delay-differential equations) can also be equivalently represented as PIEs. For example, consider the following DDE
\begin{align*}
    \dot{x}(t)=\bmat{-1.5&0\\0.5&-1}x(t) + \int_{-1}^{0} \bmat{3 & 2.25\\0 &0.5}x(t+s)ds + \int_{-2}^{0}\bmat{-1&0\\0&-1}x(t+s)ds,
\end{align*}
where $x(t)\in\R^2$ for $t\geq 0$. We declare this system as a structure \texttt{DDE} in PIETOOLS as
\begin{matlab}
\begin{verbatim}
 >> DDE.A0 = [-1.5, 0; 0.5, -1];
 >> DDE.Adi{1} = [3, 2.25; 0, 0.5];     DDE.tau(1) = 1;
 >> DDE.Adi{2} = [-1, 0; 0, -1];        DDE.tau(2) = 2;
\end{verbatim}
\end{matlab}
We can then convert the DDE to a PIE by calling
\begin{matlab}
\begin{verbatim}
 >> PIE = convert_PIETOOLS_DDE(DDE,`pie')
 PIE = 
   pie_struct with properties:

     dim: 1;
    vars: [1×2 polynomial];
     dom: [1x2 double];
       T: [6×6 opvar];     Tw: [6×0 opvar];     Tu: [6×0 opvar]; 
       A: [6×6 opvar];     B1: [6×0 opvar];     B2: [6×0 opvar]; 
      C1: [0×6 opvar];    D11: [0×0 opvar];    D12: [0×0 opvar]; 
      C2: [0×6 opvar];    D21: [0×0 opvar];    D22: [0×0 opvar]; 
\end{verbatim}
\end{matlab}
In this structure, we note that \texttt{dim=1}, indicating that the PIE is 1D, even though the state $x(t)\in\R^2$ in the DDE is finite-dimensional. This is because, in order to incorporate the delayed signals, the state is augmented to $\mbf{x}(t)=\sbmat{x(t)\\\mbf{x}_1(t)\\\mbf{x}_2(t)}\in\sbmat{\R^2\\ L_2^2[-1,0]\\ L_2^2[-1,0]}$, where 
\begin{align*}
    \mbf{x}_1(t,s)&=x(t+\tau_{1}s)=x(t+s),  &   &\text{and,}    &
    \mbf{x}_{2}&=x(t+\tau_2 s)=x(t-2s)
\end{align*}
for $s\in[-1,0]$. Here, the artificial states $\mbf{x}_1(t)$ and $\mbf{x}_2(t)$ will have to satisfy
\begin{align*}
    \dot{\mbf{x}}_1(t,s)&=\dot{x}(t+s)=\partial_{r}x(r)=\partial_{s}x(t+s)=\partial_{s}\mbf{x}_1(t,s), \\
    \dot{\mbf{x}}_2(t,s)&=\dot{x}(t+2s)=\partial_{r}x(r)=\frac{1}{2}\partial_{s}x(t+2s)=\frac{1}{2}\partial_{s}\mbf{x}_2(t,s) &   s&\in[-1,0]
\end{align*}
and we can equivalently represent the DDE as a PDE
\begin{align*}
    \dot{x}(t)&=\bmat{-1.5&0\\0.5&-1}x(t) + \int_{-1}^{0} \bmat{3 & 2.25\\0 &0.5}\mbf{x}_1(t,s)ds + \int_{-2}^{0}\bmat{-1&0\\0&-1}\mbf{x}_2(t,s)dt, \\
    \dot{\mbf{x}}_1(t,s)&=\partial_{s}\mbf{x}_1(t,s)    \\
    \dot{\mbf{x}}_2(t,s)&=\frac{1}{2}\partial_{s}\mbf{x}_2(t,s) \\
    \text{with BCs}\qquad &\mbf{x}_1(t,-1)=x(t), \qquad  \mbf{x}_1(t,-2)=x(t).
\end{align*}
In this system, $\mbf{x}_1$ and $\mbf{x}_2$ must be first-order differentiable with respect to $s$, suggesting that the fundamental state associated to this PDE is given by $\mbf{x}_{\text{f}}(t)=\sbmat{x_{\text{f},0}(t)\\\mbf{x}_{\text{f},1}(t)\\\mbf{x}_{\text{f},2}(t)}=\sbmat{x(t)\\\partial_{s}\mbf{x}_1(t)\\\partial_{s}\mbf{x}_2(t)}$ for $t\geq 0$. The \texttt{PIE} structure derived from the \texttt{DDE} will describe the dynamics in terms of this fundamental state $\mbf{x}_{\text{f}}$, where we note that, indeed, the objects \texttt{T} and \texttt{A} are of dimension $6\times 6$. In particular, we find that
\begin{matlab}
\begin{verbatim}
 >> T = PIE.T
 T =
      [1,0] | [0,0,0,0] 
      [0,1] | [0,0,0,0] 
      ------------------
      [1,0] | T.R 
      [0,1] |   
      [1,0] |   
      [0,1] |   

 T.R =
     [0,0,0,0] | [0,0,0,0] | [-1,0,0,0] 
     [0,0,0,0] | [0,0,0,0] | [0,-1,0,0] 
     [0,0,0,0] | [0,0,0,0] | [0,0,-1,0] 
     [0,0,0,0] | [0,0,0,0] | [0,0,0,-1] 
\end{verbatim}
\end{matlab}
where \texttt{T.P} is simply a $2\times 2$ identity operator, as the first two state variables of the augmented state $\mbf{x}$ and the fundamental state $\mbf{x}_{\text{f}}$ are both identical, and equal to the finite-dimensional state $x(t)$. More generally, we find that the augmented state can be retrieved from the associated fundamental state as
\begin{align*}
    \mbf{x}(t,s)=\bl(\mcl{T}\mbf{x}_{\text{f}}\br)(t,s)
    =\left[\begin{array}{ll}
         I_{2\times 2} & \smallint_{-1}^{0}d\theta \sbmat{0_{2\times 2} & 0_{2\times 2}}  \\
         \sbmat{I_{2\times 2}\\I_{2\times 2}} & \smallint_{s}^{0}d\theta \sbmat{-I_{2\times 2}&0_{2\times 2}\\ 0_{2\times 2}&-I_{2\times 2}} 
    \end{array}\right]
    \left[\begin{array}{l}
        x_{\text{f},0}(t)\\ \mbf{x}_{\text{f},1}(t,\theta)\\ \mbf{x}_{\text{f},2}(t,\theta)
    \end{array}\right]
    =\left[\begin{array}{l}
        x_{\text{f},0}(t)\\ x_{\text{f},0}(t)-\int_{s}^{0}\mbf{x}_{\text{f},1}(t,\theta)d\theta\\ x_{\text{f},0}(t)-\int_{s}^{0}\mbf{x}_{\text{f},2}(t,\theta)d\theta
    \end{array}\right]
\end{align*}
Then, studying the value of the object \texttt{A}
\begin{matlab}
\begin{verbatim}
 >> A = PIE.A
 A =
      [-0.5,2.25] | [-3*s-3,-2.25*s-2.25,2*s+2,0] 
       [0.5,-2.5] | [0,-0.5*s-0.5,0,2*s+2] 
      --------------------------------------------
            [0,0] | A.R 
            [0,0] |   
            [0,0] |   
            [0,0] |   

 A.R =
          [1,0,0,0] | [0,0,0,0] | [0,0,0,0] 
          [0,1,0,0] | [0,0,0,0] | [0,0,0,0] 
     [0,0,0.5000,0] | [0,0,0,0] | [0,0,0,0] 
     [0,0,0,0.5000] | [0,0,0,0] | [0,0,0,0] 
\end{verbatim}
\end{matlab}
we find that the DDE can be equivalently represented by the PIE
\begin{align*}
    &\bl(\mcl{T}\dot{\mbf{x}}_{\text{f}}\br)(t,s)=
    \left[\begin{array}{l}
        \dot{x}_{\text{f},0}(t)\\ \dot{x}_{\text{f},0}(t)-\int_{s}^{0}\dot{\mbf{x}}_{\text{f},1}(t,\theta)d\theta\\ \dot{x}_{\text{f},0}(t)-\int_{s}^{0}\dot{\mbf{x}}_{\text{f},2}(t,\theta)d\theta
    \end{array}\right]  \\
    &=\left[\begin{array}{l}
        \sbmat{-0.5&2.25\\0.5&-2.5}x_{\text{f},0}(t) + \int_{-1}^{0}(s+1)\bbl(\sbmat{-3&-2.25\\0&-0.5}\mbf{x}_{\text{f},1}(t,s) + 2\mbf{x}_{\text{f},2}(t,s) \bbr)ds\\ 
        \mbf{x}_{\text{f},1}(t,s)\\
        \frac{1}{2}\mbf{x}_{\text{f},2}(t,s)
    \end{array}\right]
    =\bl(\mcl{A}\mbf{x}_{\text{f}}\br)(t,s)
\end{align*}

\section{Converting an Input-Output System to a PIE}\label{sec:PIE:PDE2PIE:io_pde_2_pie}

In addition to simple differential systems, systems with inputs and outputs can also be represented as PIEs. In this case, the PIE takes a more general form
\begin{align}\label{eq:ioPIE_structure}
    \mcl{T}_u\dot{u}(t)+\mcl{T}_w\dot{w}(t)+\mcl{T}\dot{\mbf{x}}_{\text{f}}(t)&=\mcl{A}\mbf{x}_{\text{f}}(t)+\mcl{B}_1 w(t)+\mcl{B}_2 u(t),   \nonumber\\
    z(t)&=\mcl{C}_1\mbf{x}_{\text{f}}(t) + \mcl{D}_{11}w(t) + \mcl{D}_{12}u(t), \nonumber\\
    y(t)&=\mcl{C}_2\mbf{x}_{\text{f}}(t) + \mcl{D}_{21}w(t) + \mcl{D}_{22}u(t),
\end{align}
where $w$ denotes the exogenous inputs, $u$ the actuator inputs, $z$ the regulated outputs, and $y$ the observed outputs. Here, the operator $\mcl{T}_{u}$, $\mcl{T}_{w}$ and $\mcl{T}$ define the map from the fundamental state $\mbf{x}_{\text{f}}(t)$ back to the PDE state as
\begin{align*}
    \mbf{x}(t)=\mcl{T}_{u}u(t)+\mcl{T}_{w}w(t)+\mcl{T}\mbf{x}_{\text{f}}(t),
\end{align*}
where the operators $\mcl{T}_{u}$ and $\mcl{T}_{w}$ will be nonzero only if the inputs $u$ and $w$ contribute to the boundary conditions enforced upon the PDE state $\mbf{x}$. As such, the temporal derivatives $\dot{u}$ and $\dot{w}$ will also contribute to the PIE only if these inputs appear in the boundary conditions, which may be the case when performing e.g. boundary or delayed control.

In PIETOOLS, systems with inputs and outputs can be converted to PIEs in the same manner as autonomous systems. For example, consider a 1D heat equation with distributed disturbance $w$, and boundary control $u$, where we can observe the state at the upper boundary, and we wish to regulate the integral of the state over the entire domain:
\begin{align}\label{eq:ex_PDE2PIE_io_PDE}
    \dot{\mbf{x}}(t,s)&=\frac{1}{2}\partial_{s}^2\mbf{x}(t,s)+s(2-s)w(t),  &   s&\in[0,1]  \nonumber\\
    z(t)&=\int_{0}^{1}\mbf{x}(t,s)ds,    \nonumber\\
    y(t)&=\mbf{x}(t,1),  \nonumber\\
    \text{with BCs} \hspace*{1.0cm} 
    \mbf{x}(t,0)&=u(t), \hspace*{1.5cm} \partial_{s}\mbf{x}(t,1)=0,
\end{align}
This system too can be represented as a partial integral equation, describing the dynamics of the fundamental state $\mbf{x}_{f}=\partial_{s}^2\mbf{x}$. To arrive at this PIE representation, we once more implement the PDE using the Command Line Input format as
\begin{matlab}
\begin{verbatim}
 >> pvar s t
 >> x = pde_var(s,[0,1]);    
 >> w = pde_var('in');    u = pde_var('control');
 >> z = pde_var('out');   y = pde_var('observe');
 >> PDE_dyn = diff(x,t) == 0.5*diff(x,s,2) + s*(2-s)*w;
 >> PDE_z   = z == int(x,s,[0,1]);
 >> PDE_y   = y == subs(x,s,1);
 >> PDE_BCs = [subs(x,s,0) == u; subs(diff(x,s),s,1) == 0];
 >> PDE = [PDE_dyn; PDE_z; PDE_y; PDE_BCs];
\end{verbatim}
\end{matlab}
Then, we can convert this system to an equivalent PIE as before, finding a structure
\begin{matlab}
\begin{verbatim}
 >> PIE = convert(PDE,`pie')
 PIE = 
   pie_struct with properties:

     dim: 1;
    vars: [1×2 polynomial];
     dom: [0 1];
     
       T: [1×1 opvar];     Tw: [1×1 opvar];     Tu: [1×1 opvar]; 
       A: [1×1 opvar];     B1: [1×1 opvar];     B2: [1×1 opvar]; 
       C1: [1×1 opvar];    D11: [1×1 opvar];    D12: [1×1 opvar]; 
       C2: [1×1 opvar];    D21: [1×1 opvar];    D22: [1×1 opvar]; 
\end{verbatim}
\end{matlab}
In this structure, the fields \texttt{T} through \texttt{D22} describe the PI operators $\mcl{T}$ through $\mcl{D}_{22}$ in the PIE~\eqref{eq:ioPIE_structure}. Here, since the exogenous input $w$ does not contribute to the boundary conditions, it also will not contribute to the map $\mbf{x}=\mcl{T}_{u}u+\mcl{T}_{w}w+\mcl{T}\mbf{x}_{\text{f}}$ from the fundamental state $\mbf{x}_{\text{f}}$ to the PDE state $\mbf{x}$. As such, we also find that the associated \texttt{opvar} object \texttt{Tw} has all parameters equal to zero, whereas \texttt{Tu} and \texttt{T} are distinctly nonzero
\begin{matlab}
\begin{verbatim}
 >> Tw = PIE.Tw
 Tw = 
       [] | [] 
      -----------
      [0] | Tw.R 
      
 >> Tu = PIE.Tu
 Tu = 
       [] | [] 
      -----------
      [1] | Tu.R 

 >> T = PIE.T
 T = 
      [] | [] 
      ---------
      [] | T.R 
 T.R = 
       [0] | [-s_dum] | [-s] 
\end{verbatim}
\end{matlab}
Note here that only the parameter \texttt{Tu.Q2} is non-empty for \texttt{Tu}, and only \texttt{T.R} is nonempty for \texttt{T}, as $\mcl{T}_u$ maps a finite-dimensional state $u\in\R$ to an infinite dimensional state $\mbf{x}\in L_2[0,1]$, whilst $\mcl{T}$ maps an infinite-dimensional state $\mbf{x}_{\text{f}}\in L_2[0,1]$ to an infinite-dimensional state $\mbf{x}\in L_2[0,1]$. Studying the values of \texttt{Tu} and \texttt{T}, we find that we can retrieve the PDE state as
\begin{align*}
    \mbf{x}=\mcl{T}_{u}u+\mcl{T}\mbf{x}_{\text{f}}=u -\int_{0}^{s}\theta\mbf{x}_{\text{f}}(\theta)d\theta - \int_{s}^{1}s\mbf{x}_{\text{f}}(\theta)d\theta = u-\int_{0}^{s}\theta\partial_{\theta}^2\mbf{x}(\theta)d\theta - \int_{s}^{1}s\partial_{\theta}^2\mbf{x}(\theta)d\theta
\end{align*}
Next, we look at the operators $\mcl{A}$, $\mcl{B}_1$ and $\mcl{B}_{2}$. Here, $\mcl{B}_2$ will be zero, as the input $u$ does not appear in the equation for $\dot{\mbf{x}}$, nor does the value of $\mbf{x}_{\text{f}}=\partial_{s}^2\mbf{x}$ depend on $u$. For the remaining operators, we find that they are equal to
\begin{matlab}
\begin{verbatim}      
 >> A = PIE.A
 A = 
      [] | [] 
      ---------
      [] | T.R 

 A.R = 
       [0.5] | [0] | [0]  

 >> B1 = PIE.B1
 B1 =
              [] | [] 
      ------------------
      [-s^2+2*s] | B1.R   
\end{verbatim}
\end{matlab}
suggesting that the fundamental state $\mbf{x}_{\text{f}}$ must satisfy
\begin{align*}
    \dot{u}(t) -\int_{0}^{s}\theta\dot{\mbf{x}}_{\text{f}}(t,\theta)d\theta - \int_{s}^{1}s\dot{\mbf{x}}_{\text{f}}(t,\theta)d\theta &= \frac{1}{2}\mbf{x}_{\text{f}}(t) + [-s^2 + 2s]w(t),    &
    s&\in[0,1]
\end{align*}
This leaves only the output equations. Here, since there is no feedthrough from $w$ into $z$ or $y$, the operators $\mcl{D}_{11}$ and $\mcl{D}_{21}$ will both be zero. However, despite the actuator input $u$ not appearing in the PDE equations for $z$ and $y$, the contribution of $u$ to the BCs means that the value of the PDE state $\mbf{x}=\mcl{T}\mbf{x}_{\text{f}}+\mcl{T}_u u$ also depends on the value of $u$, and therefore $\mcl{D}_{12}$ and $\mcl{D}_{22}$ are nonzero. In particular, we find that
\begin{matlab}
\begin{verbatim}      
 >> C1 = PIE.C1
 C1 = 
      [] | [0.5*s^2-s] 
      -----------------
      [] | C1.R 

 >> D12 = PIE.D12
 D12 = 
      [1] | [] 
      -----------
      []  | D12.R 

 >> D22 = PIE.D22
 D22 =
      [1] | [] 
      -----------
      []  | D22.R  
     
 >> C2 = PIE.C2
 C2 =
      [] | [-s] 
      ----------
      [] | C2.R                
\end{verbatim}
\end{matlab}
Here, only the parameters \texttt{Q1} of \texttt{C1} and \texttt{C2} are non-empty, as the operators $\mcl{C}_1$ and $\mcl{C}_2$ map infinite-dimensional states $\mbf{x}_{\text{f}}\in L_2[0,1]$ to finite-dimensional outputs $z,y\in\R$. Similarly, only the parameters \texttt{P} of \texttt{D12} and \texttt{D22} are non-empty, as $\mcl{D}_{12}$ and $\mcl{D}_{22}$ map the finite-dimensional input $u\in\R$ to finite-dimensional outputs $z,y\in\R$. Combining with the earlier results, we find that the PDE~\eqref{eq:ex_PDE2PIE_io_PDE} may be equivalently represented by the PIE
\begin{align}\label{eq:ex_PDE2PIE_io_PIE}
    \dot{u}(t)-\int_{0}^{s}\theta\dot{\mbf{x}}_{\text{f}}(t,\theta)d\theta - \int_{s}^{1}s\dot{\mbf{x}}_{\text{f}}(t,\theta)d\theta &= \frac{1}{2}\mbf{x}_{\text{f}}(t,s) + s(2-s)w(t),    \qquad   s\in[0,1],~t\geq 0,  \\
    z(t)&=\int_{0}^{1}\bbl(\frac{1}{2}s^2-s\bbr)\mbf{x}_{\text{f}}(t,s)ds + u(t),   \notag\\
    y(t)&=-\int_{0}^{1}s\mbf{x}_{\text{f}}(t,s)ds + u(t),   \qquad \text{where $\mbf{x}_{\text{f}}(t,s)=\partial_{s}^{2}\mbf{x}(t,s)$}. \notag
\end{align}

\begin{boxEnv}{\textbf{Note on infinite-dimensional inputs in the PIE}}
For multivariate PDEs with infinite-dimensional input signals (e.g. $\mbf{w}(t,s_{1})$), if the input appears in the boundary conditions of the PDE (e.g. $\mbf{x}(t,s_{1},0)=\mbf{w}(t,s_{1})$), then the resulting PIE representation may involve boundary values and/or spatial derivatives of the input (e.g. $\mbf{w}(t,0)$ and/or $\partial_{s_{1}}\mbf{w}(t,s_{1})$). PIETOOLS will display a message specifying how different inputs signals in the PIE relate to input signals in the PDE.
\end{boxEnv}

\section{Declaring and Manipulating PIEs}\label{sec:PIE:piess}

In the previous sections, we showed how equivalent PIE representations of PDE and DDE systems can be easily computed by simply calling the command \texttt{convert}. However, it is of course also possible to construct PIE systems directly, which can be convenient when e.g. building the closed-loop representation after performing PIE estimator or controller synthesis. In this section, we show how such PIEs can be generated for given PI operators using the function \texttt{piess}, focusing on systems without inputs and outputs in Subsection~\ref{subsec:PIE:piess:autonomous}, systems with disturbances and regulated outputs in Subsection~\ref{subsec:PIE:piess:wz}, and adding controlled inputs and observed outputs in Subsection~\ref{subsec:PIE:piess:uy}. We will also show how feedback interconnections of PIE systems can be performed to construct closed-loop PIE representations in Subsection~\ref{subsec:PIE:piess:pielft}.

\subsection{Declaring a simple PIE}\label{subsec:PIE:piess:autonomous}

In its simplest form, a PIE governing the dynamics of a fundamental state $\mbf{x}_{\text{f}}(t)\in L_{2}^{n}[a,b]$ is defined by only two PI operators: $\mcl{T}$ and $\mcl{A}$. For example, the PIE representation of the transport equation in~\eqref{eq:1D_PDE_example} at the start of this section is given by
\begin{equation*}
    \partial_{t}(\mcl{T}\mbf{x}_{\text{f}})(t,s)
    =\overbrace{\int_{0}^{s}\partial_{t}\mbf{x}_{\text{f}}(t,\theta)d\theta}^{(\mcl{T}\mbf{x}_{\text{f}})(t,s)}
    =\overbrace{2\mbf{x}_{\text{f}}(t,s) +\int_{0}^{s} 10\mbf{x}_{\text{f}}(t,\theta)d\theta}^{(\mcl{A}\mbf{x}_{\text{f}})(t,s)}
    =(\mcl{A}\mbf{x}_{\text{f}})(t,s),\qquad s\in[0,1],~t\geq 0,
\end{equation*}
with fundamental state $\mbf{x}_{\text{f}}(t)\in L_{2}[0,1]$.
Here, the operators $\mcl{T}$ and $\mcl{A}$ defining this representation can be declared as \texttt{opvar} objects as
\begin{matlab}
\begin{verbatim}
 >> opvar T A;
 >> T.I = [0,1];    A.I = [0,1];
 >> T.R.R1 = 1;  
 >> A.R.R0 = 2;     A.R.R1 = 10;
\end{verbatim}
\end{matlab}
See Chapter~\ref{ch:PIs} for more details. Now, to construct the PIE representation defined by these operators, we call the function \texttt{piess} as
\begin{matlab}
\begin{verbatim}
 >> PIE1 = piess(T,A)
 PIE1 = 
   pie_struct with properties:

      dim: 1;
      vars: [1×2 polynomial];
      dom: [1×2 double];

        T: [1×1 opvar];     Tw: [1×0 opvar];     Tu: [1×0 opvar]; 
        A: [1×1 opvar];     B1: [1×0 opvar];     B2: [1×0 opvar]; 
       C1: [0×1 opvar];    D11: [0×0 opvar];    D12: [0×0 opvar]; 
       C2: [0×1 opvar];    D21: [0×0 opvar];    D22: [0×0 opvar];
\end{verbatim}
\end{matlab}
This returns a \texttt{pie\_struct} object \texttt{PIE} with the fields \texttt{PIE.T} and \texttt{PIE.A} set to the input operators \texttt{T} and \texttt{A}, respectively. Note that the domain and variables defining the PIE are automatically set equal to those of the operators \texttt{T} and \texttt{A}, and PIETOOLS will throw an error if the variables or domains of these operators don't match.

\subsection{Declaring a PIE with outputs and disturbances}\label{subsec:PIE:piess:wz}

Suppose now that we have a more elaborate PIE of the form,
\begin{align}
    \partial_{t}(\mcl{T}\mbf{x}_{\text{f}})(t)&=\mcl{A}\mbf{x}_{\text{f}}(t)+\mcl{B} w(t),   \nonumber\\
    z(t)&=\mcl{C}\mbf{x}_{\text{f}}(t) + \mcl{D}w(t),
\end{align}
involving a disturbance $w(t)$ and regulated output $z(t)$, both finite or infinite-dimensional, at each time $t\geq 0$. Now, the dynamics are represented by five PI operators, so declaring this system with \texttt{piess} also requiring passing all five operators as e.g.
\begin{matlab}
 >> PIE\_zw = piess(T,A,B,C,D);
\end{matlab}
To illustrate, suppose we have the same system as in the previous subsection, but now with a finite-dimensional disturbance and regulated output  as
\begin{align*}
    \partial_{t} \int_{0}^{s}\mbf{x}_{\text{f}}(t,\theta)d\theta &=2\mbf{x}_{\text{f}}(t,s) +\int_{0}^{s}10\mbf{x}_{\text{f}}(t,\theta)d\theta + 5s w(t),  \\
    z(t)&=\int_{0}^{1}(1-s)\mbf{x}_{\text{f}}(t,s)ds.
\end{align*}
In this example, the PI operators $\mcl{A}$ and $\mcl{T}$ are the same as in the previous subsection, and the multiplier operator $\mcl{B}:\R\to L_{2}[0,1]$ and integral operator $\mcl{C}:L_{2}[0,1]\to \R$ are given by
\begin{align}
    (\mcl{B}w)(s):=5sw,\quad \forall w\in\R,\qquad
    (\mcl{C}\mbf{v}):=\int_{0}^{1}(1-s)\mbf{v}(s)ds,\quad \forall\mbf{v}\in L_{2}[0,1].
\end{align}
We can declare the integral operator as an \texttt{opvar} object \texttt{C} using the field \texttt{C.Q2} as
\begin{matlab}
\begin{verbatim}
 >> opvar C
 >> C.I = [0,1];    s = C.var1;
 >> C.Q1 = 1-s;
\end{verbatim}
\end{matlab}
Here, we set \texttt{s=C.var1} to use the default spatial variable that PIETOOLS uses for \texttt{opvar} objects, which is always a safe choice. Formally, we should then also specify the multiplier operator $\mcl{B}:\R\to L_{2}[0,1]$ as an \texttt{opvar} object \texttt{B} with \texttt{B.Q1=5*s}. However, passing polynomial functions to \texttt{piess}, these functions are automatically interpreted as corresponding multiplier operators, so that we can simply declare our input-output PIE as
\begin{matlab}
\begin{verbatim}
 >> PIE2 = piess(T,A,5*s,C,0)
 PIE2 = 
   pie_struct with properties:

      dim: 1;
     vars: [1×2 polynomial];
      dom: [1×2 double];

        T: [1×1 opvar];     Tw: [1×1 opvar];     Tu: [1×0 opvar]; 
        A: [1×1 opvar];     B1: [1×1 opvar];     B2: [1×0 opvar]; 
       C1: [1×1 opvar];    D11: [1×1 opvar];    D12: [1×0 opvar]; 
       C2: [0×1 opvar];    D21: [0×1 opvar];    D22: [0×0 opvar]; 
\end{verbatim}
\end{matlab}
passing \texttt{0} as fifth argument to declare a zero feedthrough $\mcl{D}=0$. Note that the resulting operator \texttt{PIE2.B1} indeed represents the desired multiplier operator,
\begin{matlab}
\begin{verbatim}
 >> PIE2.B1
 ans =
          [] | [] 
      ---------------
      [5*s1] | ans.R 

 ans.R =
    [] | [] | [] 
\end{verbatim}
\end{matlab}
Note also that the field \texttt{PIE2.Tw}, representing the term $\mcl{T}_{w}\dot{w}(t)$ in the left-hand side of the PIE dynamics is automatically populated with a zero operator of appropriate dimensions,
\begin{matlab}
\begin{verbatim}
 >> PIE2.Tw
 ans =
       [] | [] 
      -----------
      [0] | ans.R 

 ans.R =
    [] | [] | [] 
\end{verbatim}
\end{matlab}

\subsection{Declaring a PIE with sensing and control}\label{subsec:PIE:piess:uy}

Finally, let us consider a PIE in its most general form as in~\eqref{eq:ioPIE_structure}, involving not only disturbance and regulated outputs, but also controlled inputs and observed outputs. This representation is parameterized by 12 PI operators: three operators $\{\mcl{T},\mcl{T}_{w},\mcl{T}_{u}\}$ defining the left-hand side of the PIE dynamics, a state operator $\mcl{A}$, two input operators $\{\mcl{B}_{1},\mcl{B}_{2}\}$, two output operators $\{\mcl{C}_{1},\mcl{C}_{2}\}$, and four feedthrough operators $\{\mcl{D}_{11},\mcl{D}_{12},\mcl{D}_{21},\mcl{D}_{22}\}$. To declare such a PIE for given operators, we pass the values of the different operator to \texttt{piess} using cell structures as follows:
\begin{matlab}
\begin{verbatim}
 >> PIE = piess({T,Tw,Tu}, A, {B1,B2}, {C1;C2}, {D11,D12;D21,D22});
\end{verbatim}
\end{matlab}
Note here that the operators $\mcl{B}_{1}$ and $\mcl{B}_{2}$ must be declared using a $1\times 2$ cell array, whereas the operators $\mcl{C}_{1}$ and $\mcl{C}_{2}$ must be declared using a $2\times 1$ cell array. The operators $\mcl{D}_{ij}$ should be declared accordingly as a $2\times 2$ cell array. For example, consider the PIE representation of the PDE in~\eqref{eq:ex_PDE2PIE_io_PDE}, given in~\eqref{eq:ex_PDE2PIE_io_PIE}, defined by the operators
\begin{align*}
    &(\mcl{T}\mbf{v})(s):=-\int_{0}^{s}\theta\mbf{v}(\theta)d\theta -\int_{s}^{1}s\mbf{v}(\theta)d\theta,    &
    &(\mcl{T}_{u}u)(s):=u,  &   
    &\forall s\in[0,1].\\
    &(\mcl{A}\mbf{v})(s):=\frac{1}{2}\mbf{v}(s), &    &(\mcl{B}_{1}w)(s)=s(2-s),   \\
    &\mcl{C}_{1}\mbf{v}:=\int_{0}^{1}\left(\frac{1}{2}s^2-s\right)\mbf{v}(s)ds, &
    &\mcl{D}_{12}u:=u,  \\
    &\mcl{C}_{2}\mbf{v}:=-\int_{0}^{1}s\mbf{v}(s)ds, &
    &\mcl{D}_{22}u:=u.
\end{align*}
for $\mbf{v}\in L_{2}[0,1]$ and $u,w\in\R$. We can declare these operators as \texttt{opvar} objects as
\begin{matlab}
\begin{verbatim}
 >> opvar T C1 C2;
 >> T.I = [0,1];    C1.I = [0,1];    C2.I = [0,1];
 >> s = T.var1;         theta = T.var2;
 >> T.R.R1 = -s;        T.R.R2 = -theta;
 >> C1.Q1 = 0.5*s^2-s;  C2.Q1 = -s; 
\end{verbatim}
\end{matlab}
where again, we use \texttt{s=T.var1} and \texttt{theta=T.var2} to ensure that the primary spatial variable and dummy variable match the default variables used by PIETOOLS. Of course, we can also declare the multiplier operators $\mcl{A}:L_{2}[0,1]\to L_{2}[0,1]$, $\mcl{B}_{1},\mcl{T}_{u}:\R\to L_{2}[0,1]$, and $\mcl{D}_{12},\mcl{D}_{22}:\R\to\R$ as \texttt{opvar} operators, but for the purpose of passing them to \texttt{piess} it suffices to declare them as just
\begin{matlab}
\begin{verbatim}
 >> A = 0.5;    Tu = 1;     B1 = s*(2-s);
 >> D12 = 1;    D22 = 1;
\end{verbatim}
\end{matlab}
Then, calling
\begin{matlab}
\begin{verbatim}
 >> PIE3 = piess({T,0,Tu},A,{B1,0},{C1;C2},{0,D12;0,D22})
 PIE3 = 
   pie_struct with properties:

      dim: 1;
     vars: [1×2 polynomial];
      dom: [1×2 double];

        T: [1×1 opvar];     Tw: [1×1 opvar];     Tu: [1×1 opvar]; 
        A: [1×1 opvar];     B1: [1×1 opvar];     B2: [1×1 opvar]; 
       C1: [1×1 opvar];    D11: [1×1 opvar];    D12: [1×1 opvar]; 
       C2: [1×1 opvar];    D21: [1×1 opvar];    D22: [1×1 opvar]; 
\end{verbatim}
\end{matlab}
we obtain a \texttt{pie\_struct} object representing our desired PIE, where we note that e.g. the input parameter \texttt{A} is automatically augmented to a suitable \texttt{opvar} object as
\begin{matlab}
\begin{verbatim}
 >> PIE3.A
 ans =
       [] | [] 
       -----------
       [] | ans.R 

 ans.R =
     [0.5000] | [0] | [0]
\end{verbatim}
\end{matlab}
representing a multiplier operator $\mcl{A}:L_{2}[0,1]\to L_{2}[0,1]$. Note also that in the call to \texttt{piess}, we can always use \texttt{0} or \texttt{[]} for an argument to indicate that the corresponding operator is a zero operator, and the function will automatically generate an associated \texttt{opvar} object will all zero parameters. Of course, users should be careful that this only works if \texttt{piess} is able to deduce the row and column dimensions of the operator from the remaining input arguments.

\subsection{Taking interconnections of PIEs}\label{subsec:PIE:piess:pielft}


A crucial property of the PIE representation of a system is that, due to the lack of boundary conditions and the algebraic nature of PI operators, we can easily take (feedback) interconnections of PIEs. To facilitate the computation of such interconnections of \texttt{pie\_struct} objects, PIETOOLS currently offers two functions: \texttt{closedLoopPIE}, for imposing a simple feedback law $u=\mcl{K}\mbf{v}$ in a PIE representation, and \texttt{pielft}, for taking the linear fractional transformation of two PIE systems. To illustrate, consider again the PIE in~\eqref{eq:ex_PDE2PIE_io_PIE}, declared as a \texttt{pie\_struct} object in the previous subsection. This PIE involves a finite-dimensional scalar input $u(t)\in\R$, and finite-dimensional scalar output $y(t)$. Suppose now that, e.g. by using the controller synthesis LPI (see Chapter~\ref{sec:LPI_examples:control}, we find that the control law
\begin{equation*}
    u(t)=\mcl{K}\mbf{x}_{\text{f}}(t):=\int_{0}^{1}\mbf{x}_{\text{f}}(t,s)ds
\end{equation*}
for PIE state $\mbf{x}_{\text{f}}(t)\in L_{2}[0,1]$ is a good control law to e.g. stabilize the system. We can declare the operator $\mcl{K}:L_{2}[0,1]\to\R$ representing this feedback law as
\begin{matlab}
\begin{verbatim}
 >> opvar K;
 >> K.I = [0,1];    K.Q1 = 1;
\end{verbatim} 
\end{matlab}
Then, to impose this feedback law in our PIE, we can call \texttt{closedLoopPIE} as
\begin{matlab}
\begin{verbatim}
 >> PIE3_CL1 = closedLoopPIE(PIE3,K)
 PIE3_CL1 = 
   pie_struct with properties:

      dim: 1;
     vars: [1×2 polynomial];
      dom: [1×2 double];

        T: [1×1 opvar];     Tw: [1×1 opvar];     Tu: [1×0 opvar]; 
        A: [1×1 opvar];     B1: [1×1 opvar];     B2: [1×0 opvar]; 
       C1: [1×1 opvar];    D11: [1×1 opvar];    D12: [1×0 opvar]; 
       C2: [1×1 opvar];    D21: [1×1 opvar];    D22: [1×0 opvar]; 
\end{verbatim}
\end{matlab}
returning a \texttt{pie\_struct} object representing the closed-loop PIE
\begin{align*}
    \int_{0}^{s}[1-\theta]\dot{\mbf{x}}_{\text{f}}(t,\theta)d\theta + \int_{s}^{1}[1-s]\dot{\mbf{x}}_{\text{f}}(t,\theta)d\theta &= \frac{1}{2}\mbf{x}_{\text{f}}(t,s) + s(2-s)w(t),    \qquad   s\in[0,1],~t\geq 0,  \\
    z(t)&=\int_{0}^{1}\bbl(1+\frac{1}{2}s^2-s\bbr)\mbf{x}_{\text{f}}(t,s)ds,   \notag\\
    y(t)&=\int_{0}^{1}(1-s)\mbf{x}_{\text{f}}(t,s)ds. \notag
\end{align*}
Note that this PIE representation no longer involves any controlled input $u(t)$, as we have imposed the feedback law $u(t)=\int_{0}^{1}\mbf{x}_{\text{f}}(t,s)ds$.

Now, suppose that we instead want to synthesize a Luenberger-type estimator for our PIE, simulating an estimate $\hat{\mbf{x}}_{\text{f}}(t)$ of the PIE state and $\hat{z}(t)$ of the output state as
\begin{align*}
    \dot{u}(t)-\int_{0}^{s}\theta\dot{\hat{\mbf{x}}}_{\text{f}}(t,\theta)d\theta - \int_{s}^{1}s\dot{\hat{\mbf{x}}}_{\text{f}}(t,\theta)d\theta &= \frac{1}{2}\hat{\mbf{x}}_{\text{f}}(t,s) +\mcl{L}(\hat{y}(t)-y(t)),    \qquad   s\in[0,1],~t\geq 0,  \\
    \hat{z}(t)&=\int_{0}^{1}\bbl(\frac{1}{2}s^2-s\bbr)\hat{\mbf{x}}_{\text{f}}(t,s)ds + u(t),   \notag\\
    \hat{y}(t)&=-\int_{0}^{1}s\hat{\mbf{x}}_{\text{f}}(t,s)ds + u(t), \notag
\end{align*}
for some gain $\mcl{L}:R\to L_{2}[0,1]$. Let a suitable gain $\mcl{L}$ (computed using e.g. the LPI from Section~\ref{sec:LPI_examples:estimation}) be given by $\mcl{L}y=s(1-s)y$ for $y\in\R$. Then, we can construct the ``closed-loop'' system for this value of the gain as
\begin{matlab}
\begin{verbatim}
 >> opvar L;        s = L.var1;
 >> L.I = [0,1];    L.Q2 = s*(1-s);
 >> PIE3_CL2 = closedLoopPIE(PIE3,L,'observer')
 PIE3_CL2 = 
   pie_struct with properties:

        T: [2×2 opvar];     Tw: [2×1 opvar];     Tu: [2×1 opvar]; 
        A: [2×2 opvar];     B1: [2×1 opvar];     B2: [2×1 opvar]; 
       C1: [2×2 opvar];    D11: [2×1 opvar];    D12: [2×1 opvar]; 
       C2: [0×2 opvar];    D21: [0×1 opvar];    D22: [0×1 opvar]; 
\end{verbatim}
\end{matlab}
representing the augmented system of our original PIE with the resulting estimator as
\begin{align*}
    \dot{u}(t)-\int_{0}^{s}\theta\dot{\mbf{x}}_{\text{f}}(t,\theta)d\theta - \int_{s}^{1}s\dot{\mbf{x}}_{\text{f}}(t,\theta)d\theta &= \frac{1}{2}\mbf{x}_{\text{f}}(t,s) + s(2-s)w(t),    \qquad   s\in[0,1],~t\geq 0,  \\
    \dot{u}(t)-\int_{0}^{s}\theta\dot{\hat{\mbf{x}}}_{\text{f}}(t,\theta)d\theta - \int_{s}^{1}s\dot{\hat{\mbf{x}}}_{\text{f}}(t,\theta)d\theta &= \frac{1}{2}\hat{\mbf{x}}_{\text{f}}(t,s) -s(1-s)\int_{0}^{1}\theta\hat{\mbf{x}}_{\text{f}}(t,s)d\theta +s(1-s)\int_{0}^{1}\theta\mbf{x}_{\text{f}}(t,s)d\theta,    \\
    z(t)&=\int_{0}^{1}\bbl(\frac{1}{2}s^2-s\bbr)\mbf{x}_{\text{f}}(t,s)ds + u(t),   \notag\\
    \hat{z}(t)&=\int_{0}^{1}\bbl(\frac{1}{2}s^2-s\bbr)\hat{\mbf{x}}_{\text{f}}(t,s)ds + u(t).
\end{align*}
Note that this system involves no more observed output $y$, instead implicitly feeding this output back into the estimator dynamics. Further note that the PIE now involves two state variables, $(\mbf{x}_{\text{f}}(t),\hat{\mbf{x}}_{\text{f}})$, and two regulated outputs, $(z(t),\hat{z}(t))$, corresponding to the true values and their estimates.

Finally, aside from imposing simple feedback laws, we can also take a linear fractional transformation of two complete PIE systems using \texttt{pielft}. To illustrate, suppose that we combine our earlier estimator and controller, to generate a control effort $u(t)$ for our PIE using the estimate $\hat{\mbf{x}}_{\text{f}}(t)$ of the state as
\begin{align*}
    -\int_{0}^{s}\theta\dot{\hat{\mbf{x}}}_{\text{f}}(t,\theta)d\theta - \int_{s}^{1}s\dot{\hat{\mbf{x}}}_{\text{f}}(t,\theta)d\theta &= \frac{1}{2}\hat{\mbf{x}}_{\text{f}}(t,s) -s(1-s)\int_{0}^{1}\theta\hat{\mbf{x}}_{\text{f}}(t,s)d\theta -s(1-s)\hat{y}(t), \notag\\
    \hat{z}(t)&=\int_{0}^{1}\bbl(\frac{1}{2}s^2-s\bbr)\hat{\mbf{x}}_{\text{f}}(t,s)ds,  \\
    \hat{u}(t)&=\int_{0}^{1}\hat{\mbf{x}}_{\text{f}}(t,s)ds.
\end{align*}
Note now that the output $y(t)$ of our original PIE is used as input $\hat{y}(t)$ to this estimator PIE, and we will use the output $\hat{u}(t)$ to this estimator PIE as input $u(t)$ to our original system. To declare this estimator system, we again use \texttt{piess} as
\begin{matlab}
\begin{verbatim}
 >> opvar T A C1;
 >> T.I = [0,1];    A.I = [0,1];    C1.I = [0,1];
 >> s = T.var1;         theta = T.var2;
 >> T.R.R1 = -s;        T.R.R2 = -theta;
 >> A.R.R0 = 0.5;       A.R.R1 = -s*(1-s)*theta;    A.R.R2 = A.R.R1;
 >> C1.Q1 = 0.5*s^2-s;
 >> PIE3_est = piess(T,A,{[],-L},{C1;K});
\end{verbatim}


\end{matlab}
where we use the same operators \texttt{K} and \texttt{L} as before, passing \texttt{{[],-L}} as argument to declare a controlled input $\mcl{L}\hat{u}(t)$, and passing \texttt{{C1;K}} as argument to declare an observed output $\hat{y}(t)=\mcl{K}\hat{\mbf{x}}_{\text{f}}(t)$, in addition to the regulated output $\hat{z}(t)=\mcl{C}_{1}\hat{\mbf{x}}(t)$. Then, to take the linear fractional transformation of our PIE with this estimator system, using the interconnection signals $\hat{u}(t)=y(t)$ and $u(t)=\hat{y}(t)$, we simply call \texttt{pielft} as
\begin{matlab}
\begin{verbatim}
 >> PIE3_CL3 = pielft(PIE3,PIE3_est)
 PIE3_CL3 = 
   pie_struct with properties:

        T: [2×2 opvar];     Tw: [2×1 opvar];     Tu: [2×0 opvar]; 
        A: [2×2 opvar];     B1: [2×1 opvar];     B2: [2×0 opvar]; 
       C1: [2×2 opvar];    D11: [2×1 opvar];    D12: [2×0 opvar]; 
       C2: [0×2 opvar];    D21: [0×1 opvar];    D22: [0×0 opvar];
\end{verbatim}    
\end{matlab}
yielding a \texttt{pie\_struct} object representing the system
\begin{align*}
    &\int_{0}^{1}\dot{\hat{\mbf{x}}}_{\text{f}}(t)-\int_{0}^{s}\theta\dot{\mbf{x}}_{\text{f}}(t,\theta)d\theta - \int_{s}^{1}s\dot{\mbf{x}}_{\text{f}}(t,\theta)d\theta = \frac{1}{2}\mbf{x}_{\text{f}}(t,s) + s(2-s)w(t),    \qquad   s\in[0,1],~t\geq 0,  \\
    &-\int_{0}^{s}\theta\dot{\hat{\mbf{x}}}_{\text{f}}(t,\theta)d\theta - \int_{s}^{1}s\dot{\hat{\mbf{x}}}_{\text{f}}(t,\theta)d\theta= \frac{1}{2}\hat{\mbf{x}}_{\text{f}}(t,s) -s(1-s)\int_{0}^{1}\theta\hat{\mbf{x}}_{\text{f}}(t,s)d\theta +s(1-s)\int_{0}^{1}\theta\mbf{x}_{\text{f}}(t,s)d\theta,    \\
    &z(t)=\int_{0}^{1}\bbl(\frac{1}{2}s^2-s\bbr)\mbf{x}_{\text{f}}(t,s)ds +\int_{0}^{1}\hat{\mbf{x}}_{\text{f}}(t,s)ds,   \notag\\
    &\hat{z}(t)=\int_{0}^{1}\bbl(\frac{1}{2}s^2-s\bbr)\hat{\mbf{x}}_{\text{f}}(t,s)ds.
\end{align*}
This PIE has no more controlled inputs or observed outputs, as the controlled inputs of the first PIE have been set equal to the observed outputs of the second PIE, and vice versa. However, the closed-loop PIE system is expressed in terms of the states, disturbances, and regulated outputs of both PIE systems.

For additional examples on using \texttt{piess} and \texttt{pielft} to construct closed-loop PIE representations after controller and observer synthesis, see also demos 5 through 7 in Chapter~\ref{ch:demos}.

\begin{boxEnv}{\textbf{Note}}
When imposing feedback laws using \texttt{closedLoopPIE}, \textbf{all} controlled inputs are assumed to be governed by the feedback. Thus, the output dimensions of the specified gain must match the dimensions of the controlled input in the PIE. Similarly, when constructing an estimator using \texttt{closedLoopPIE}, the input dimensions of the Luenberger gain must match the dimensions of the observed output. Finally, taking the linear fractional transformation of two PIEs, the dimensions of the controlled input of the first PIE must match those of the observed output of the second PIE, and vice versa.
\end{boxEnv}

\begin{boxEnv}{\textbf{Warning}}
The structure of \texttt{opvar} objects restricts them to map only the function space $\R^{m}\times L_{2}^{n}[a,b]$, not e.g. $L_{2}[a,b]\times \R^{m}$. As such, the state, input, and output variables in the PIE representation will always be reordered to place the finite-dimensional variables (e.g ODE states) ahead of the infinite-dimensional variables (e.g. PDE states). Similarly, state, input, and output variables on lower-dimensional domains (e.g. 1D PDE states) will always be placed ahead of variables on higher-dimensional domains (e.g. 2D PDE states). Consequently, both when converting ODE-PDE or DDE systems to PIEs and when taking e.g. linear fractional transformations of PIE systems, the order of the states, inputs, and outputs may be changed. 
\end{boxEnv}

\chapter{PIESIM: A General-Purpose Simulation Tool based on PIETOOLS}\label{ch:PIESIM}


PIESIM is a general-purpose, versatile, high-fidelity simulation tool, distributed with the\\ PIETOOLS package. It numerically solves PDEs, as well as coupled PDE/ODEs, and DDE systems, in one and two dimensions. PIESIM utilizes high-order methods based on Chebyshev polynomial approximation for spatial discretization and offers temporal discretization schemes of the order one through four based on implicit backward-difference formulas (BDF). PIESIM is based on the PDE-to-PIE transformation framework native to PIETOOLS and seamlessly applies high-order methods to PDE problems while exactly satisfying any arbitrary set of admissible boundary conditions. 

\section{Organization of PIESIM}

PIESIM is based on a set of MATLAB functions that perform the following tasks:
\begin{enumerate}
\item Discretization of the PDE/DDE/PIE solution domain and the PI operators.
\item Temporal integration  using the BDF scheme.
\item Transformation of solution from the PIE (fundamental) state back to the PDE (primary) state and plotting the solution.  
\end{enumerate}

\section{\texttt{PIESIM.m} - the Main Routine of PIESIM}\label{sec:PIESIM:main}

The main routine of the PIESIM solver is the function \texttt{PIESIM.m} located in the folder \\
\texttt{PIESIM/PIESIM\_routines}. This function can be called using the syntax
\vspace{2mm}
\begin{matlab}
    >> [solution, grid] = PIESIM(system,opts,uinput);
\end{matlab}
where the inputs are characterized by:
\begin{enumerate}
\item \texttt{system} - a PDE, DDE, or a PIE structure  - \textbf{required}\\ (define \texttt{system}=PDE, \texttt{system}=DDE, or \texttt{system}=PIE);
\item \texttt{opts} - simulation options structure - \textbf{optional};
\item \texttt{uinput} - structure that describes initial conditions and other options - \textbf{optional};
\end{enumerate}
and the outputs:
\begin{enumerate}
\item \texttt{solution} - returns the variables associated with the solution field; 
\item \texttt{grid} - returns the coordinates of the spatial grid associated with the numerical solution. 
\end{enumerate}

Detailed description of the inputs (and their default values if not specified by user) is provided below.
\begin{enumerate}
\item \texttt{system} is the \emph{required} input that declares the structure of the problem to be numerically solved. It can take a value of PDE, DDE, or a PIE (supplying PIE directly into PIESIM is currently supported in 1D only, for 2D simulations use PDE input). The user is responsible for declaring this variable using available options in PIETOOLS prior to calling PIESIM. The \texttt{system} variable does not have a default value. The simulation will not run and an error will be issued if the variable \texttt{system} is not declared.
\item \texttt{opts} is the \emph{optional} structure with the fields:
\begin{itemize}
\item \texttt{N} - integer array of size \texttt{dim} (dimension of the problem) - polynomial order used in a discretization of a \textbf{primary} state solution along each spatial dimension; each entry should be an integer $\ge 2$+$n$, where $n$ is the order of the highest spatial derivative along the given dimension - default \texttt{N}=8 in 1D or \texttt{N}=$[8,8]$ in 2D. If array of size 1 is specified for 2D problems, the same polynomial order will be used in both dimensions.
\item \texttt{tf} (real $\ge 0$) - final time of simulation - default \texttt{tf}=1
\item  \texttt{Norder} (integer $\{1,2,3,4\}$) - order of time integration scheme - default \texttt{Norder}=2
\item  \texttt{dt} (real $> 0$) - time step value used in numerical time integration - default \texttt{dt}=0.01
\item  \texttt{plot} (`yes' or `no', `no' default) - option for plotting solution: returns solution plots if `yes', no plots otherwise
\item  \texttt{ploteig} (`yes' or `no', `no' default) - option for plotting discrete eigenvalues of a temporal propagator: returns eigenvalue plot if `yes', no plot otherwise
\item  \texttt{ifexact} (true or false, false default) - a flag indicating whether a comparison with exact solution will be performed.
If \texttt{ifexact}=true is chosen, exact solution for all the states should be provided under \texttt{uinput.exact}. Otherwise, \texttt{opts.ifexact} will be defaulted to false. If \texttt{opts.ifexact=true} and \texttt{opts.plot=`yes'}, the exact solution will be plotted together with the numerical solution.  \texttt{uinput.exact} is provided automatically if PIESIM built-in examples are used.
\item  \texttt{dist} - an \emph{optional} flag that can take a value of `constant', `sin', or `sinc' - specifies a form of disturbance signal if disturbance is not defined by user. This flag will be ignored if disturbance is defined under \texttt{uinput.w}
\item  \texttt{control} - an \emph{optional} flag that can take a value of `constant', `sin', or `sinc' - specifies a form of control signal if control signal is not defined by user. This flag will be ignored if control signal is defined under \texttt{uinput.u}
\end{itemize}
\item \texttt{uinput} is the \emph{optional} structure with the fields:
\begin{itemize}
\item \texttt{ic} - initial conditions on the states. Initial conditions should be specified for the \textbf{primary} states if the system type is PDE, \textbf{primary} states history if the system type is DDE, and for the \textbf{fundamental} states if the system type is PIE. Initial conditions must be defined as scalars (double) or symbolic objects in \texttt{sx} (1D), or \texttt{sx}, \texttt{sy} (2D) (use \texttt{syms sx sy} to declare). The order of entries in \texttt{uinput.ic} should correspond to the order of declared states. Thus, initial conditions on ODE states  usually precede initial conditions on PDE/DDE/PIE states. To avoid problems related to a declaration of mixed double and symbolic variables, it is recommended to use an array syntax for \texttt{uinput.ic}. For example, for 2 ODE states (with initial conditions as 1, 1) and 2 PDE states (with initial conditions as $s_1$, $\sin(\pi s_1)\sin(\pi s_2)$), use 
\begin{matlab}
>> uinput.ic=[1,1,sx,sin(pi*sx)*sin(pi*sy)]; 
\end{matlab}
\item \texttt{w} - external disturbances defined in one of the following 3 ways:
\begin{itemize}
\item As a MATLAB symbolic object in \texttt{st}, \texttt{sx} (1D), or \texttt{st}, \texttt{sx}, \texttt{sy} (2D) (use \texttt{syms st sx sy} to declare) 
\item As a double array of size 2 $\times$ \texttt{nt} (\texttt{nt} - number of temporal samples) and form $[t_1, t_2,\ldots, t_{nt};w_1, w_2,\ldots ,w_{nt}]$. The first row stores discrete values of the time samples, and the second row stores  discrete values of the disturbance at the time samples, $w_j=w(t_j)$. The time samples need not be uniform (spline interpolation will be used). The time interval for the active disturbance can be a subset of the total simulation time (disturbance outside of the provided time interval will assume the value of zero). The double format for disturbances is supported only for \emph{finite-dimensional} disturbances.
\item As a type of a disturbance signal, defined under \texttt{opts.dist} flag (can be `constant', `sin', or `sinc'). The specified type of signal will be used for \emph{all} disturbances, with the amplitude of 1. In this case, \texttt{uinput.w} need not be provided. If \texttt{uinput.w} is provided, the flag  \texttt{opts.dist} will be ignored.  
\end{itemize}
\item \texttt{u} - external control inputs defined in one of the following 3 ways:
\begin{itemize}
\item
As a MATLAB symbolic object in \texttt{st}, \texttt{sx} (1D), or \texttt{st}, \texttt{sx}, \texttt{sy} (2D) (use \texttt{syms st sx sy} to declare)  
\item As a double array of size 2 $\times$ \texttt{nt} (\texttt{nt} - number of temporal samples) and form $[t_1, t_2,\ldots, t_{nt};u_1, u_2,\ldots ,u_{nt}]$. The first row stores discrete values of the time samples, and the second row stores  discrete values of the control input at the time samples, $u_j=u(t_j)$. The time samples need not be uniform (spline interpolation will be used). The time interval for the active control input can be a subset of the total simulation time (input outside of the provided time interval will assume the value of zero). The double format for control inputs is supported only for \emph{finite-dimensional} inputs.
\item As a type of a control input signal, defined under \texttt{opts.control} flag (can be `constant', `sin', or `sinc'). The specified type of signal will be used for \emph{all} disturbances, with the amplitude of 1. In this case, \texttt{uinput.u} need not be provided. If \texttt{uinput.u} is provided, the flag  \texttt{opts.control} will be ignored.  
\end{itemize}
\item \texttt{exact} - used only if \texttt{opts.ifexact==true} and contains exact solution to the declared problem, specified as a symbolic object in \texttt{sx} (1D), or \texttt{sx, sy} (2D).
\item \texttt{weight} - integer specifying the polynomial weight $k$ if the weighted formulation is desired. Applicable if the right-hand side of the unweighted PDE contains terms like $s^{-k},\, k>0$, such as in cylindrical coordinates.  In this case, the user should multiply the right-hand side of the PDE by $s^{k}$ before entering it into PIETOOLS and provide the weight $k$ as \texttt{uinput.weight}=k. The left-hand side of the PDE (consequently, PIE) will be multiplied by $s^{k}$ internally. Currently supported only in 1D (axisymmetric cases).
\end{itemize}
\end{enumerate}

Outputs of the \texttt{PIESIM.m} function contain the following variables: 
\begin{enumerate}
\item 
\texttt{solution} structure with the fields:
\begin{itemize}
\item \texttt{tf} - scalar - actual final time of the solution
\item  \texttt{final.primary} - value of all the primary state solutions at the final time 
\begin{itemize}
\item In 1D problems - cell array \texttt{final.primary}\{1,2\}\\
\texttt{final.primary}\{1\} - array of size \texttt{n0} containing final value of finite-dimensional (ODE) states, \texttt{n0} - number of finite-dimensional (ODE) states

\texttt{final.primary\{2\}}  - array of size \texttt{(N+1)} $\times $ \texttt{nx} - final value of spatially-varying (PDE) states,  \texttt{nx} - number of spatially-varying (PDE) states
\item In 2D problems - cell array \texttt{final.primary}\{1,2,3,4\}\\
\texttt{final.primary}\{1\} - array of size \texttt{n0} containing final value of finite-dimensional (ODE) states, \texttt{n0} - number of finite-dimensional (ODE) states

\texttt{final.primary\{2\}}  - array of size (\texttt{N(1)+1}) $\times $ \texttt{nx} - final value of PDE states that are only the functions of variable $s_1$, \texttt{nx} - number of PDE states depending only on $s_1$

\texttt{final.primary\{3\}}  - array of size (\texttt{N(2)+1}) $\times $ \texttt{ny} - final value of PDE states that are only the functions of variable $s_2$, \texttt{ny} - number of PDE states depending only on $s_2$

\texttt{final.primary\{4\}}  - array of size (\texttt{N(1)+1}) $\times$ (\texttt{N(2)+1}) $\times $ \texttt{n2} - final value of PDE states that are the functions of both $s_1$ and $s_2$, \texttt{n2} - number of PDE states depending on $s_1$ and $s_2$
\end{itemize}

\item  \texttt{final.observed} - value of observed outputs at the final time 
\begin{itemize}
\item In 1D problems - cell array \texttt{final.observed}\{1,2\}\\
\texttt{final.observed}\{1\} - array of size \texttt{no0} containing final value of finite-dimensional observed outputs, \texttt{no0} - number of finite-dimensional observed outputs

\texttt{final.observed\{2\}}  - array of size \texttt{(N+1)} $\times $ \texttt{nox} - final value of infinite-dimensional observed outputs,  \texttt{nox} - number of infinite-dimensional observed outputs 
\item In 2D problems - cell array \texttt{final.observed}\{1,2,3,4\}\\
\texttt{final.observed}\{1\} - array of size \texttt{no} containing final value of finite-dimensional observed outputs, \texttt{no0} - number of finite-dimensional observed outputs

\texttt{final.observed\{2\}}  - array of size (\texttt{N(1)+1}) $\times $ \texttt{nox} - final value of observed outputs that are only the functions of variable $s_1$, \texttt{nox} - number of observed outputs depending only on $s_1$

\texttt{final.observed\{3\}}  - array of size (\texttt{N(2)+1}) $\times $ \texttt{noy} - final value of observed outputs that are only the functions of variable $s_2$, \texttt{noy} - number of observed outputs depending only on $s_2$

\texttt{final.observed\{4\}}  - array of size (\texttt{N(1)+1}) $\times$ (\texttt{N(2)+1}) $\times $ \texttt{no2} - final value of observed outputs that are the functions of both $s_1$ and $s_2$, \texttt{no2} - number of observed outputs depending on $s_1$ and $s_2$
\end{itemize}

\item  \texttt{final.regulated} - value of regulated outputs at the final time 
\begin{itemize}
\item In 1D problems - cell array \texttt{final.regulated}\{1,2\}\\
\texttt{final.regulated}\{1\} - array of size \texttt{nr0} containing final value of finite-dimensional regulated outputs, \texttt{nr0} - number of finite-dimensional regulated outputs

\texttt{final.regulated\{2\}}  - array of size \texttt{(N+1)} $\times $ \texttt{nrx} - final value of infinite-dimensional regulated outputs,  \texttt{nrx} - number of infinite-dimensional regulated outputs 
\item In 2D problems - cell array \texttt{final.regulated}\{1,2,3,4\}\\
\texttt{final.regulated}\{1\} - array of size \texttt{nr0} containing final value of finite-dimensional regulated outputs, \texttt{nr0} - number of finite-dimensional regulated outputs

\texttt{final.regulated\{2\}}  - array of size (\texttt{N(1)+1}) $\times $ \texttt{nrx} - final value of regulated outputs that are only the functions of variable $s_1$, \texttt{nrx} - number of regulated outputs depending only on $s_1$

\texttt{final.regulated\{3\}}  - array of size (\texttt{N(2)+1}) $\times $ \texttt{nry} - final value of regulated outputs that are only the functions of variable $s_2$, \texttt{nry} - number of regulated outputs depending only on $s_2$

\texttt{final.regulated\{4\}}  - array of size (\texttt{N(1)+1}) $\times$ (\texttt{N(2)+1}) $\times $ \texttt{nr2} - final value of regulated outputs that are the functions of both $s_1$ and $s_2$, \texttt{nr2} - number of regulated outputs depending on $s_1$ and $s_2$
\end{itemize}
\item \texttt{timedep.dtime} - array of size $1 \times$ \texttt{Nsteps} - array of discrete time values at which the time-dependent solution is computed. \texttt{Nsteps} - number of time steps - is calculated as \texttt{Nsteps= floor(tf/dt)} 
\item  \texttt{timedep.primary} - time-dependent value of all the primary state solutions 
\begin{itemize}
\item In 1D problems - cell array \texttt{timedep.primary}\{1,2\}\\
\texttt{timedep.primary}\{1\} - array of size \texttt{n0}$\times$ \texttt{Nsteps} containing time-dependent value of finite-dimensional (ODE) states, \texttt{n0} - number of finite-dimensional (ODE) states

\texttt{timedep.primary\{2\}}  - array of size \texttt{(N+1)} $\times $ \texttt{nx} $\times $ \texttt{Nsteps} - time-dependent value of spatially-varying (PDE) states,  \texttt{nx} - number of spatially-varying (PDE) states
\item In 2D problems - cell array \texttt{timedep.primary}\{1,2,3,4\}\\
\texttt{timedep.primary}\{1\} - array of size \texttt{no} $\times$ \texttt{Nsteps} containing time-dependent value of finite-dimensional (ODE) states, \texttt{no} - number of finite-dimensional (ODE) states

\texttt{timedep.primary\{2\}}  - array of size (\texttt{N(1)+1}) $\times $ \texttt{nx} $\times$ \texttt{Nsteps} - time-dependent value of PDE states that are only the functions of variable $s_1$, \texttt{nx} - number of PDE states depending only on $s_1$

\texttt{timedep.primary\{3\}}  - array of size (\texttt{N(2)+1}) $\times $ \texttt{ny} $\times$ \texttt{Nsteps} - time-dependent value of PDE states that are only the functions of variable $s_2$, \texttt{ny} - number of PDE states depending only on $s_2$

\texttt{timedep.primary\{4\}}  - array of size (\texttt{N(1)+1}) $\times$ (\texttt{N(2)+1}) $\times $ \texttt{n2} $\times$ \texttt{Nsteps}  - time-dependent value of PDE states that are the functions of both $s_1$ and $s_2$, \texttt{n2} - number of PDE states depending on $s_1$ and $s_2$
\end{itemize}
\item  \texttt{timedep.observed} - time-dependent value of observed outputs
\begin{itemize}
\item In 1D problems - cell array \texttt{timedep.observed}\{1,2\}\\
\texttt{timedep.observed}\{1\} - array of size \texttt{no0} $\times$ \texttt{Nsteps} containing time-dependent value of finite-dimensional observed outputs, \texttt{no0} - number of finite-dimensional observed outputs

\texttt{timedep.observed\{2\}}  - array of size \texttt{(N+1)} $\times $ \texttt{nox} $\times $ \texttt{Nsteps} - time-dependent value of infinite-dimensional observed outputs,  \texttt{nox} - number of infinite-dimensional observed outputs 
\item In 2D problems - cell array \texttt{timedep.observed}\{1,2,3,4\}\\
\texttt{timedep.observed}\{1\} - array of size \texttt{no0} $\times$ \texttt{Nsteps} containing time-dependent value of finite-dimensional observed outputs, \texttt{no0} - number of finite-dimensional observed outputs

\texttt{timedep.observed\{2\}}  - array of size (\texttt{N(1)+1}) $\times $ \texttt{nox} $\times $ \texttt{Nsteps} - time-dependent value of observed outputs that are only the functions of variable $s_1$, \texttt{nox} - number of observed outputs depending only on $s_1$

\texttt{timedep.observed\{3\}}  - array of size (\texttt{N(2)+1}) $\times $ \texttt{noy} $\times $ \texttt{Nsteps} - time-dependent value of observed outputs that are only the functions of variable $s_2$, \texttt{noy} - number of observed outputs depending only on $s_2$

\texttt{timedep.observed\{4\}}  - array of size (\texttt{N(1)+1}) $\times$ (\texttt{N(2)+1}) $\times $ \texttt{no2} $\times $ \texttt{Nsteps}- time-dependent value of observed outputs that are the functions of both $s_1$ and $s_2$, \texttt{no2} - number of observed outputs depending on $s_1$ and $s_2$
\end{itemize}

\item  \texttt{timedep.regulated} - time-dependent value of regulated outputs 
\begin{itemize}
\item In 1D problems - cell array \texttt{timedep.regulated}\{1,2\}\\
\texttt{timedep.regulated}\{1\} - array of size \texttt{nr0} $\times$ \texttt{Nsteps} containing time-dependent value of finite-dimensional regulated outputs, \texttt{nr0} - number of finite-dimensional regulated outputs

\texttt{timedep.regulated\{2\}}  - array of size \texttt{(N+1)} $\times $ \texttt{nrx} $\times $ \texttt{Nsteps} - time-dependent value of infinite-dimensional regulated outputs,  \texttt{nrx} - number of infinite-dimensional regulated outputs 
\item In 2D problems - cell array \texttt{timedep.regulated}\{1,2,3,4\}\\
\texttt{timedep.regulated}\{1\} - array of size \texttt{nr0} $\times$ \texttt{Nsteps} containing time-dependent value of finite-dimensional regulated outputs

\texttt{timedep.regulated\{2\}}  - array of size (\texttt{N(1)+1}) $\times $ \texttt{nrx} $\times$ \texttt{Nsteps} - time-dependent value of regulated outputs that are only the functions of variable $s_1$, \texttt{nrx} - number of regulated outputs depending only on $s_1$

\texttt{timedep.regulated\{3\}}  - array of size (\texttt{N(2)+1}) $\times $ \texttt{nry} $\times$ \texttt{Nsteps} - time-dependent value of regulated outputs that are only the functions of variable $s_2$, \texttt{nry} - number of regulated outputs depending only on $s_2$

\texttt{timedep.regulated\{4\}}  - array of size (\texttt{N(1)+1}) $\times$ (\texttt{N(2)+1}) $\times $ \texttt{nr2} $\times$ \texttt{Nsteps} - time-dependent value of regulated outputs that are the functions of both $s_1$ and $s_2$, \texttt{nr2} - number of regulated outputs depending on $s_1$ and $s_2$
\end{itemize}
\end{itemize}
\item \texttt{grid} - spatial coordinates of the collocation (grid) points on which spatially-dependent solutions are evaluated:
\begin{itemize}
\item In 1D -  an array of size \texttt{(N+1)} $\times$ 1 containing spatial coordinates of the physical grid for
  the primary solution
\item In 2D - a cell array  \texttt{grid\{1,2\}} 

\texttt{grid\{1\}} - array of size
\texttt{(N(1)+1)} $\times$ 1 containing spatial coordinates of the physical grid for
  the primary solution along $s_1$ direction

\texttt{grid\{2\}} - array of size
\texttt{(N(2)+1)} $\times$ 1 containing spatial coordinates of the physical grid for
  the primary solution along $s_2$ direction
  
    \end{itemize}
\end{enumerate}

\begin{boxEnv}{Note on final time}
    If \texttt{(tf/dt)} is not an integer value, i.e \texttt{Nsteps} = \texttt{floor(tf/dt)}$\ne$ \texttt{(tf/dt)}, the final time \texttt{tf} is adjusted as \texttt{tf=Nsteps}$\times$ \texttt{dt}. The time step value \texttt{dt} is always used as specified by the user and is never changed intrinsically. 
\end{boxEnv}

\begin{boxEnv}{Note on solution output}
    Irrespective of the input (PDE or PIE), both \texttt{final.primary} and \\ \texttt{solution.timedep.primary} fields always store the value of the \textbf{primary state} (reconstructed) solution to the original PDE/DDE problem, i.e. $\mcl T\mbf v+\mcl T_w w+\mcl T_u u$ as the value.
\end{boxEnv}

\begin{boxEnv}{Note on disturbances}
    If all disturbances  are symbolic, they can be entered either as a cell array -- \texttt{uinput.w}\{j\}, or as a numeric array -- \texttt{uinput.w}(j) for the j$^{th}$ input (same for the control inputs). If disturbances (control inputs) in the double format are present, or if there is a mix of symbolic and double entries, the cell structure must be used. \underline{Important}: disturbances (and control inputs) will be numbered in the order they are declared, not in the order they enter the dynamics and/or boundary conditions. 
\end{boxEnv}

\begin{boxEnv}{Note on warnings}
    If required inputs (such as boundary conditions, initial conditions, disturbance values etc.) are not provided, they will be substituted by the default values (typically, zero). If a simulation output is not as expected, make sure to double-check the warnings.
\end{boxEnv}

\section{Running PIESIM}

PIESIM can be run from any location within PIETOOLS, as long as the desired PDE, DDE or PIE structure has already been defined, by calling the \texttt{PIESIM.m} function as
\vspace{3mm}
\begin{matlab}
    >> [solution, grid] = PIESIM(system);
\end{matlab}
when the default options for \texttt{opts} and \texttt{uinput} will be used, or via 
\vspace{3mm}
\begin{matlab}
    >> [solution, grid] = PIESIM(system,opts,uinput);
\end{matlab}
when \texttt{opts} and \texttt{uinput} will be declared by the user prior to calling the \texttt{PIESIM.m} function. 

Note that the user may choose to declare either \texttt{opts} or  \texttt{uinput} independently while leaving the other input argument as default, with calling
\vspace{1mm}
\begin{matlab}
    >> [solution, grid] = PIESIM(system,opts);
\end{matlab}
or
\vspace{1mm}
\begin{matlab}
    >> [solution, grid] = PIESIM(system,uinput);
\end{matlab}
respectively. 


Finally, the two outputs of \texttt{PIESIM.m} are optional. If the knowledge of grid is not required, a command can be executed as 
\vspace{3mm}
\begin{matlab}
    >> solution = PIESIM(system);
\end{matlab}
or
\vspace{3mm}
\begin{matlab}
    >> solution = PIESIM(system,opts,uinput).
\end{matlab}
If only one output is requested, the function will return \texttt{solution} by default; if \texttt{grid} output is desired, two outputs need to be requested. The code will also execute without any outputs, and the plots will be produced if the corresponding plotting options are chosen under \texttt{opts} input. 


PIETOOLS distribution has built-in examples and demonstrations of how to run PIESIM. These examples can be found in: 
\begin{enumerate}
\item PIESIM demonstrations accessible from \texttt{PIETOOLS\_Code\_Illustrations\_Ch6\_PIESIM} located in the folder \texttt{PIETOOLS\_demos/snippets\_from\_manual}. These are described in details in Chapter~\ref{sec:PIESIM:demonstrations} of this manual.
\item PIETOOLS demonstrations located in the folder \texttt{PIETOOLS\_demos} of PIETOOLS (see Chapter~\ref{ch:demos} of this manual). Some of the demonstrations (specifically, demonstrations 1, 5, 6, 7 and 9) feature PIESIM.
\item PIETOOLS examles in the file \texttt{PIETOOLS\_PDE.m}. The script \texttt{PIETOOLS\_PDE.m} allows user to select predefined examples from the PIETOOLS example library (located in the folder \texttt{PIETOOLS\_examples}), featuring both 1D and 2D problems. It then offers the user an option to run stability or a controller synthesis script (based on the example) by selecting the `y' option. Regardless of whether `y' or `n' is selected, PIESIM is executed afterwards.
\item PIESIM examples located in the files \texttt{examples\_pde\_library\_PIESIM\_1D} and \\\texttt{examples\_pde\_library\_PIESIM\_2D} within the \texttt{PIESIM} folder, for 1D and 2D problems respectively. To run examples from the \texttt{PIESIM} folder, a MATLAB executable \texttt{solver\_PIESIM.m} is provided in the same folder. To execute the examples using\\  \texttt{solver\_PIESIM.m}, the user should first adjust the desired problem dimension by setting \texttt{dim=1} for 1D problems or \texttt{dim=2} for 2D problems within \texttt{solver\_PIESIM.m} file. The user then must choose the example number by setting the variable \texttt{example} to the  corresponding example number which can range between 1 and 41 for 1D problems and between 1 and 32 for 2D problems. All examples in the \texttt{PIESIM} folder come with the provided analytical solution that will be plotted alongside the numerical solution if \texttt{opts.ifexact=true} and \texttt{opts.plot=`yes'} are selected. This feature provides a good option for benchmarking PIESIM numerical solutions and performing convergence checks, if needed. 
\end{enumerate}

\begin{boxEnv}{Note on Numerical Stability}
The backward-difference formula is an implicit scheme whose region of stability lies outside of a small circle in the right half of the complex plane. This means that numerical simulations may be unstable if a discrete system has eigenvalues with small positive real parts. These eigenvalues can arise if the underlying physical problem is unstable or if the discretization errors move otherwise stable eigenvalues into the right half-plane. This situation may occur especially if the actual system has eigenvalues  close to purely imaginary, as in systems with very little dissipation.  PIESIM has a built-in check of numerical stability and issues a warning if the scheme is unstable. It also provides a recommendation on the time step size when the scheme may stabilize, if the recommended value is less than one. Increasing time step to a value higher than one is not recommended, since it will lead to large numerical errors. If this situation occurs, it is likely that the underlying physical problem is unstable, and the corresponding warning is issued.
\end{boxEnv}

\section{PIESIM Demonstrations}\label{sec:PIESIM:demonstrations}
This section presents several illustrative examples on how PIESIM can be simulated systems in PDE, DDE and PIE representation. Note that, although the figures displayed in this section have been modified slightly from the default figures returned by PIESIM, the script
\texttt{PIETOOLS\_Code\_Illustrations\_Ch6\_PIESIM} located in the folder \\\texttt{PIETOOLS\_demos/snippets\_from\_manual} contains full codes to reproduce each of these plots. 

\subsection{PIESIM Demonstration A: 1D PDE example}\label{subsec:PIESIM:examples:PDE_1D}
In this section, we will demonstrate the standard process involved in simulation of 1D PDEs, simulating Example 4 from \texttt{examples\_pde\_library\_PIESIM\_1D} using the following code.

\begin{codebox}
\begin{matlab}
    >> syms sx st;\\
    >> [PDE,uinput]=examples\_pde\_library\_PIESIM\_1D(4);\\
    >> uinput.exact(1) = -2*sx*st-sx\^{}2;\\
    >> uinput.w\{1\} = 0;\\
    >> uinput.w\{2\} = -4*st-4;\\
    >> uinput.ic = -sx\^{}2;\\
    >> opts.ifexact=true;\\
    >> opts.plot = `yes';\\
    >> opts.N = 8;\\
    >> opts.tf = 1;\\
    >> opts.Norder = 2;\\
    >> opts.dt=1e-3;\\
    >> solution = PIESIM(PDE,opts,uinput);\\
    >> tval = solution.timedep.dtime;\\
    >> xval = reshape(solution.timedep.primary\{2\}(:,1,:),opts.N+1,[]);
\end{matlab}
\end{codebox}

We will explain each line used in the above code. First, we load an example from the PIESIM examples library. An example can be selected by specifying the example number (between 1 and 41 for 1D problems) to load the example.
\begin{matlab}[PDE,uinput]=examples\_pde\_library\_PIESIM\_1D(4);
\end{matlab}

\noindent In this demonstration, we choose the example
\begin{align}\label{eq:Ch6_ExA_PDE}
&\dot{\mbf x}(t,s) = s \partial_s^2 \mbf x(t,s), \qquad s\in[0,2], t\ge 0\\
&\mbf x(t,0) = w_{1}(t), \qquad \mbf x(t,2) = w_2(t), \qquad \mbf x(0,s) = -s^2.\notag
\end{align}
where $w_{1}(t)=0$ and $w_2(t) = -4t-4$. For this PDE, the exact solution is known and is given by the expression $\mbf x(t,s) = -2st-s^2$ which can be specified under \texttt{uinput} structure for verification as shown below. 
\begin{matlab}
    >> uinput.exact(1) = -2*sx*st-sx\^{}2;
\end{matlab}

\noindent Likewise, other input parameters such as initial conditions and inputs at the boundary are specified as
\begin{matlab}
    >> uinput.w(1) = 0;\\
    >> uinput.w(2) = -4*st-4;\\
    >> uinput.ic = -sx\^{}2;
\end{matlab}

\noindent where \texttt{sx, st} are MATLAB symbolic objects. However, the example automatically defines the \texttt{uinput} structure and the above expressions are provided for demonstration only and not necessary when using a PDE from example library. Once the PDE and system inputs are defined, we may choose to specify simulation parameters under \texttt{opts} structure. 
\begin{matlab}
\begin{verbatim}
 >> opts.ifexact = true;
 >> opts.plot = `yes';
 >> opts.N = 8;
 >> opts.tf = 1;
 >> opts.Norder = 2;
 >> opts.dt = 1e-3;
\end{verbatim}
\end{matlab}
\noindent 
First, we turn on the plotting by setting the plotting option to `yes'. Next, we specify the order of spatial discretization \texttt{N} (order of Chebyshev polynomials to be used in the approximation of the PDE solution) and the time of simulation, \texttt{tf}. Then, we select the order of the temporal integration scheme  (backward-difference scheme) as 2 and the time step as 1e-3. If these parameters are not specified, they will be set at their default values as listed in Section~\ref{sec:PIESIM:main}. However, the user can modify these parameters as needed.
\noindent Now that we have defined all necessary and optional parameters, we can run the simulation using the command for the PDE example: 
\begin{matlab}
    >> solution = PIESIM(PDE,opts,uinput); \\
    >> tval = solution.timedep.dtime; \\
    >> xval = reshape(solution.timedep.primary\{2\}(:,1,:),opts.N+1,[]);
\end{matlab}
We obtain the time steps at which the solution is computed from \texttt{soution.timedep.dtime}, and the value of the PDE state $x(t,s)$ at these times from \texttt{solution.timedep.primary\{2\}}. Note that the number of columns of this latter field matches the number of state variables in our system, so we extract the first (and in this case only) column of this array. Element \texttt{xval(i,j)} will then specify the value of our PDE state at grid point \texttt{i} and time step \texttt{j}. An isosurface plot of these values is shown in Figure~\ref{fig:piesimdemoA}, along with the simulated and exact value of the PDE state at the final time $t=1$. These plots are produced by PIESIM if \texttt{opts.plot=`yes'} is chosen.

\begin{figure}[ht]
	\centering
	\includegraphics[width=0.98\textwidth]{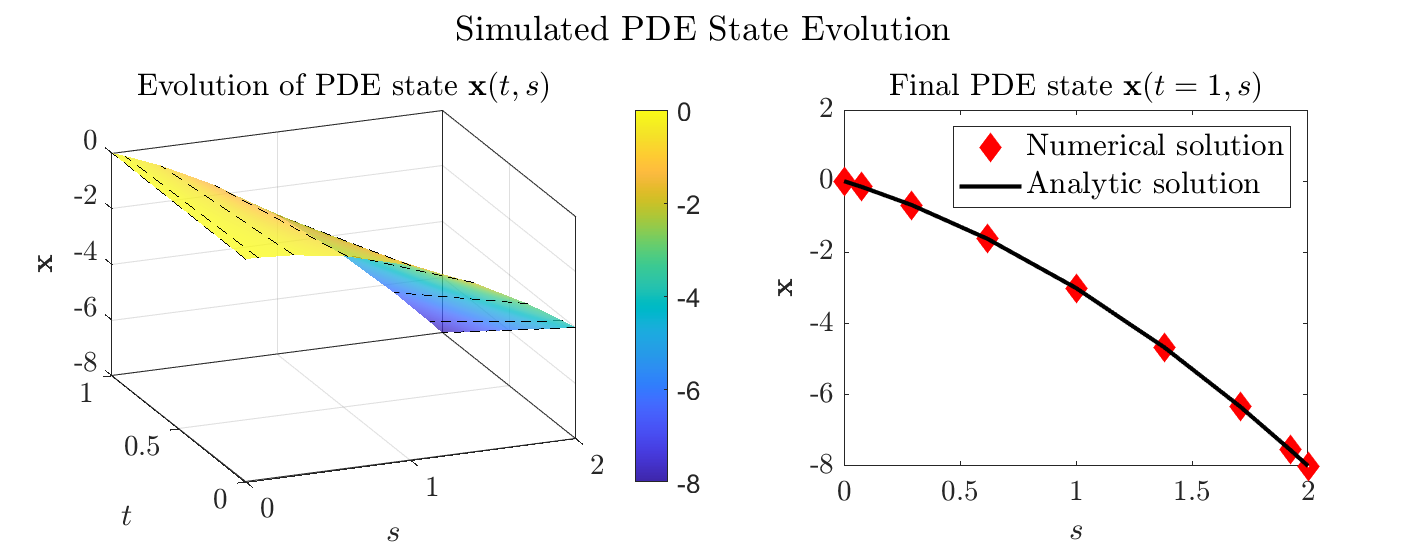}
	\caption{Simulated evolution of the PDE state $\mbf{x}(t,s)$ (left) and final value $\mbf{x}(t=1,s)$ (right) for the PDE~\eqref{eq:Ch6_ExA_PDE}, along with the analytic solution at the final time.}\label{fig:piesimdemoA}
\end{figure}

\subsection{PIESIM Demonstration B: 1D Cylindrical PDE Example}\label{subsec:PIESIM:examples:PDE_1D}
This section presents a simulation of a 1D cylindrical (axisymmetric) PDE, corresponding to Example~39 from \texttt{examples\_pde\_library\_PIESIM\_1D}. \noindent In this demonstration, we consider the axisymmetric diffusion equation
\begin{align}\label{eq:PIESIM_Ex39_PDE}
&\dot{\mbf x}(t,s)
=
\alpha\!\left(
\frac{1}{s}\partial_s \mbf x(t,s)
+
\partial_s^2 \mbf x(t,s)
\right),
\qquad s \in [0,1],\; t \ge 0, \\
&\partial_s \mbf x(t,0) = 0,
\qquad
\mbf x(t,1) = 0,
\qquad
\mbf x(0,s) = J_0(j_{0,1}s).\notag
\end{align}
Here, $\alpha>0$ denotes the diffusion coefficient, $J_0(\cdot)$ is the Bessel function of the first kind of order zero, and $j_{0,1}$ denotes its first positive root. In this example, $\alpha = 4$ and $j_{0,1} = 2.4048$. For this PDE, the exact solution is known and is given by the expression
$\mbf x(t,s) = J_0(j_{0,1} s)\,e^{-\alpha j_{0,1}^2 t}$, which can be specified under the \texttt{uinput.exact} entry for verification.

Note that Eq.~(\ref{eq:PIESIM_Ex39_PDE}) contains singularity and can not be entered into PIETOOLS as given. To avoid this, the entire equation needs to be multiplied by $s^k$, where $k$ is the highest order of the singularity present in the PDE; here $k=1$. The resulting PDE then becomes \begin{align}
\label{eq:Ch6_Ex_Cylindrical_PDE}
s\,\dot{\mbf x}(t,s)
=
\alpha\!\left(
\partial_s \mbf x(t,s)
+
s\,\partial_s^2 \mbf x(t,s)
\right).
\end{align}

This is implemented as `weighted formulation' in PIESIM. With the weighted formulation, the user needs to multiply the right-hand side of the equation by the weight prior to entering the PDE into PIETOOLS, and specify the weight as: 
\begin{matlab}
    >> uinput.weight = 1;
\end{matlab}
PIESIM will then intrinsically multiply the left-hand side of the equation by $s^k$ and construct the corresponding PIE operators.

The corresponding code implementation is shown below.
\begin{codebox}
\begin{matlab}
    >> syms sx st;\\
    >> [PDE,uinput]=examples\_pde\_library\_PIESIM\_1D(39);\\
    >> uinput.exact(1)=besselj(0,j01*sx)*exp(-alpha*j01\^{}2*st);\\
    >> uinput.w(1) = 0;\\
    >> uinput.w(2) = 0;\\
    >> uinput.ic = besselj(0,j01*sx);\\
    >> uinput.weight = 1;\\
    >> opts.ifexact=true;\\
    >> opts.plot = `yes';\\
    >> opts.N = 8;\\
    >> opts.tf = 0.1;\\
    >> opts.Norder = 2;\\
    >> opts.dt=1e-3;\\
    >> solution = PIESIM(PDE,opts,uinput);\\
    >> tval = solution.timedep.dtime;\\
    >> xval = reshape(solution.timedep.primary\{2\}(:,1,:),opts.N+1,[]);
\end{matlab}
\end{codebox}

\noindent As in the previous example, the problem can be loaded from the PIESIM examples library by specifying the corresponding example number (between 1 and 41 for 1D problems).
\begin{matlab}[PDE,uinput]=examples\_pde\_library\_PIESIM\_1D(39);
\end{matlab}

\noindent Input parameters such as the exact solution, initial condition, boundary inputs, and weighting option are specified as
\begin{matlab}
>> uinput.exact(1)=besselj(0,j01*sx)*exp(-alpha*j01\^{}2*st);\\
>> uinput.w(1)=0;\\
>> uinput.w(2)=0;\\
>> uinput.ic=besselj(0,j01*sx);\\
>> uinput.weight=1;
\end{matlab}

\noindent where \texttt{sx, st} are MATLAB symbolic objects. However, the example automatically defines the \texttt{uinput} structure, and the above expressions are provided for demonstration only and not necessary when using a PDE from example library. Once the PDE and system inputs are defined, we may choose to specify simulation parameters under \texttt{opts} structure. 
\begin{matlab}
\begin{verbatim}
 >> opts.ifexact=true;
 >> opts.plot = `yes';
 >> opts.N = 8;
 >> opts.tf = 0.1;
 >> opts.Norder = 2;
 >> opts.dt = 1e-3;
\end{verbatim}
\end{matlab}
\noindent 
First, we turn on the plotting by setting the plotting option to `yes'. Next, we specify the order of spatial discretization \texttt{N} (order of Chebyshev polynomials to be used in the approximation of the PDE solution) and the time of simulation, \texttt{tf}. Then, we select the order of the temporal integration scheme  (backward-difference scheme) as 2 and the time step as 1e-3. If these parameters are not specified, they will be set at their default values as listed in Section~\ref{sec:PIESIM:main}. However, the user can modify these parameters as needed.
\noindent Now that we have defined all necessary and optional parameters, we can run the simulation using the command for the PDE example: 
\begin{matlab}
    >> solution = PIESIM(PDE,opts,uinput); \\
    >> tval = solution.timedep.dtime; \\
    >> xval = reshape(solution.timedep.primary\{2\}(:,1,:),opts.N+1,[]);
\end{matlab}
We obtain the time steps at which the solution is computed from \texttt{soution.timedep.dtime}, and the value of the PDE state $x(t,s)$ at these times from \texttt{solution.timedep.primary\{2\}}. Note that the number of columns of this latter field matches the number of state variables in our system, so we extract the first (and in this case only) column of this array. Element \texttt{xval(i,j)} will then specify the value of our PDE state at grid point \texttt{i} and time step \texttt{j}. An isosurface plot of these values is shown in Figure~\ref{fig:piesimdemoB}, along with the simulated and exact value of the PDE state at the final time $t=0.1$. These plots are produced by PIESIM if \texttt{opts.plot=`yes'} is chosen.

\begin{figure}[ht]
	\centering
	\includegraphics[width=0.98\textwidth]{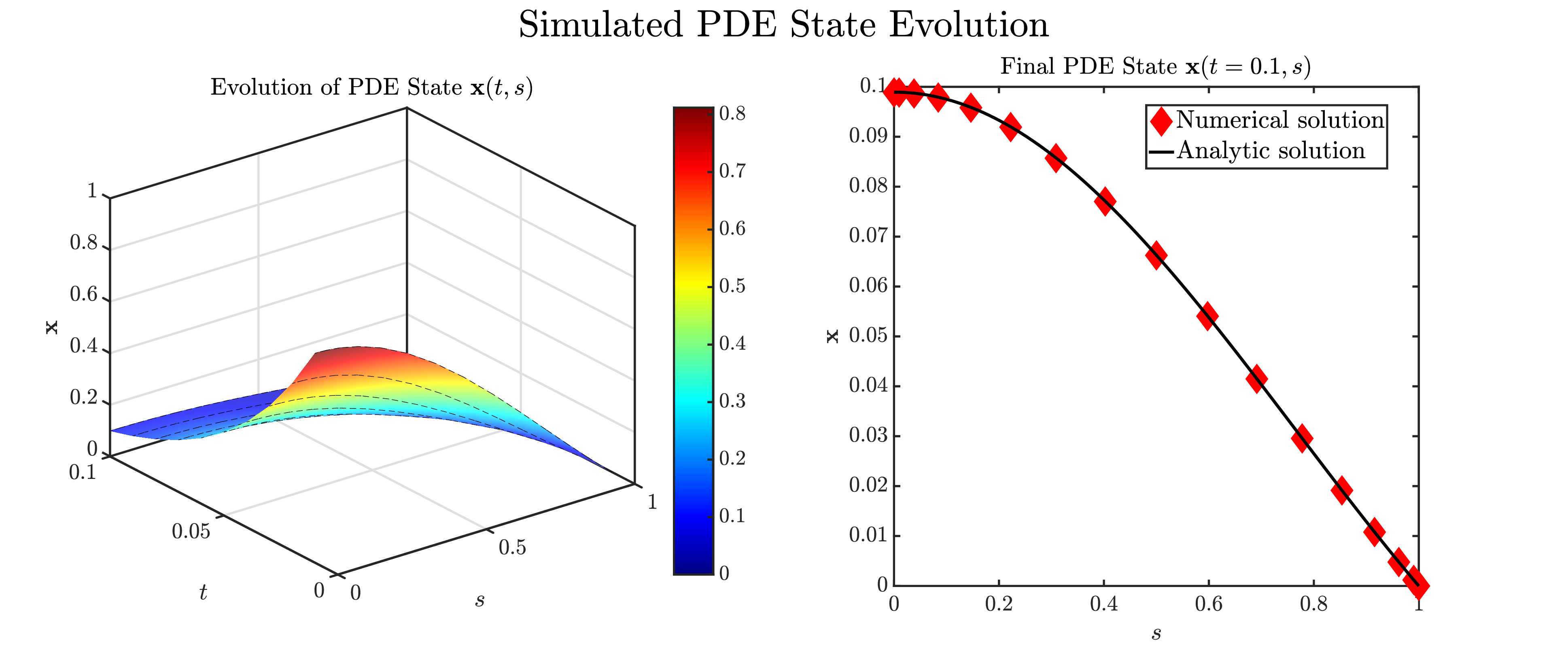}
	\caption{Simulated evolution of the PDE state $\mbf{x}(t,s)$ (left) and final value $\mbf{x}(t=1,s)$ (right) for the PDE~\eqref{eq:Ch6_Ex_Cylindrical_PDE}, along with the analytic solution at the final time.}\label{fig:piesimdemoB}
\end{figure}

\subsection{PIESIM Demonstration C: 2D PDE example}\label{subsec:PIESIM:examples:PDE_2D}

Simulation of 2D ODE-PDE systems can be done with PIESIM in much the same way to that of 1D systems. To illustrate, consider the following system, consisting of a 2D PDE coupled to 1D PDEs and an ODE at the boundaries,
\begin{align}\label{eq:Ch6_ExB_PDE}
    \dot{x}_{1}(t)&=-x_1(t),  &   &t\geq 0,\\
    \partial_{t}\mbf{x}_{1}(t,s_{1})&=\frac{1}{\pi^2}\partial_{s_{1}}^2\mbf{x}_{2}(t,s_{1}),    &   &s_{1}\in[0,1], \notag\\
    \partial_{t}\mbf{x}_{3}(t,s_{2})&=\frac{1}{\pi^2}\partial_{s_{2}}^2\mbf{x}_{3}(t,s_{2}),    &   &s_{2}\in[0,5], \notag\\
    \partial_{t}\mbf{x}_{4}(t,s_{1},s_{2})&=\frac{1}{2\pi^2}\bl(\partial_{s_{1}}^2\mbf{x}_{4}(t,s_{1})+\partial_{s_{2}}^2\mbf{x}_{4}(t,s_{1})\br),    \notag\\
    \mbf{x}_{2}(t,0)&=x(t),\qquad\qquad \partial_{s_{1}}\mbf{x}_{2}(t,1)=0, \notag\\
    \mbf{x}_{3}(t,0)&=x(t),\qquad\qquad \partial_{s_{2}}\mbf{x}_{3}(t,5)=0,\notag\\
    \mbf{x}_{4}(t,0,s_{2})&=\mbf{x}_{3}(t,s_{2})\qquad \partial_{s_{1}}\mbf{x}_{4}(t,1,s_{2})=0,    \notag\\
    \mbf{x}_{4}(t,s_{1},0)&=\mbf{x}_{2}(t,s_{1})\qquad \partial_{s_{2}}\mbf{x}_{4}(t,s_{1},5)=0,        \notag
\end{align}
This system can be readily declared using the Command Line Input format as shown in Chapter~\ref{ch:PDE_DDE_representation}, but is also included in the PIESIM 2D example library as example number 10. Starting with initial conditions
\begin{align}\label{eq:Ch6_ExB_ICs}
    x_{1}(0)&=10,\qquad 
    \mbf{x}_{2}(0,s_{1})=10\cos(\pi s_{1}),\quad
    \mbf{x}_{3}(0,s_{2})=10\cos(\pi s_{2}), \notag\\
    \mbf{x}_{4}(0,s_{1},s_{2})&=10\cos(\pi s_{1})\cos(\pi s_{2}),
\end{align}
the exact solution to the PDE, $\mbf{x}(t)=(x_1(t),\mbf x_2(t),\mbf x_3(t),\mbf x_4(t)$, is given by $\mbf{x}(t)=e^{-t}\mbf{x}(0)$. We can  numerically approximate this solution through simulation with PIESIM by running the following code.

\begin{codebox}
First, extract our desired example PDE:
\begin{matlab}
\begin{verbatim}
 >> [PDE,~] = examples_pde_library_PIESIM_2D(10);
\end{verbatim}
\end{matlab}
Next, set the initial conditions and exact solution:
\begin{matlab}
\begin{verbatim}
 >> syms sx sy st real
 >> u_ex = 10*cos(sym(pi)*sx)*cos(sym(pi)*sy)*exp(-st);
 >> uinput.exact(1) = subs(subs(u_ex,sy,0),sx,0);
 >> uinput.exact(2) = subs(u_ex,sy,0);
 >> uinput.exact(3) = subs(u_ex,sx,0);
 >> uinput.exact(4) = u_ex;
 >> uinput.ic = subs(uinput.exact,st,0);
\end{verbatim}
\end{matlab}
Also, set the options for simulation
\begin{matlab}
\begin{verbatim}
 >> opts.ifexact = true;
 >> opts.plot = 'yes';
 >> opts.N = 16;
 >> opts.tf = 1;
 >> opts.dt = 1e-3;
\end{verbatim}
\end{matlab}
Finally, simulate and extract solution
\begin{matlab}
\begin{verbatim}
 >> [solution,grid] = PIESIM(PDE,opts,uinput);
 >> x1val = solution.timedep.primary{1}; % x1(t)
 >> x2fin = solution.final.primary{2};   % x2(t=tf,s1);
 >> x3fin = solution.final.primary{3};   % x3(t=tf,s2);
 >> x4fin = solution.final.primary{4};   % x4(t=tf,s1,s2);
 >> s1_grid = grid{1};
 >> s2_grid = grid{2};
\end{verbatim}
\end{matlab}
\end{codebox}
Running this code, the solution to the PDE~\eqref{eq:Ch6_ExB_PDE} is simulated up to time $t=1$, using a time step of $\Delta t=10^{-3}$. Spatial discretization is performed using $16\times 16$ Chebyshev polynomials in the spatial variables $s_{1}$ and $s_{2}$. 

Note that, since we are simulating a 2D PDE, the output field \texttt{solution.final.primary} (as well as \texttt{solution.timedep.primary}) will be a cell with four elements, with the first element containing  the solution of the ODE states at the final time, the second and third elements containing the solution of the 1D PDE states at the final time and each grid point, and the fourth element containing the solution of the 2D PDE states at the final time and each grid point. In this case, the simulated value of $\mbf{x}_{2}(1,s_{1})$ at all grid points in $s_{1}\in[0,1]$ is stored in the field \texttt{solution.final.primary\{2\}}, and the simulated value of $\mbf{x}_{3}(1,s_{2})$ is stored in \texttt{solution.final.primary\{3\}}. The output \texttt{x4fin=solution.final.primary\{4\}} will be an $(16+1)\times(16+1)$ array, with each element \texttt{x4fin(i,j)} specifying the value of $\mbf{x}_{4}(t,s_{1},s_{2})$ at the \texttt{i}th grid point along $s_{1}\in[0,1]$, and the \texttt{j}th gridpoint along $s_{2}\in[0,5]$. 
The values of these grid points are stored in the output cell array \texttt{grid}, with the first element specifying grid points along the first spatial direction $s_{1}$, and the second element specifying grid points along the second spatial direction $s_{2}$. 

The simulated value of the 1D and 2D PDE states at the final time $t=1$ are displayed in Fig.~\ref{fig:Ch6:ExB_1DState} and Fig.~\ref{fig:Ch6:ExB_2DState}, respectively, along with the value of the exact solution at that time.

\begin{figure}[H]
	\centering
	\includegraphics[width=0.98\textwidth]{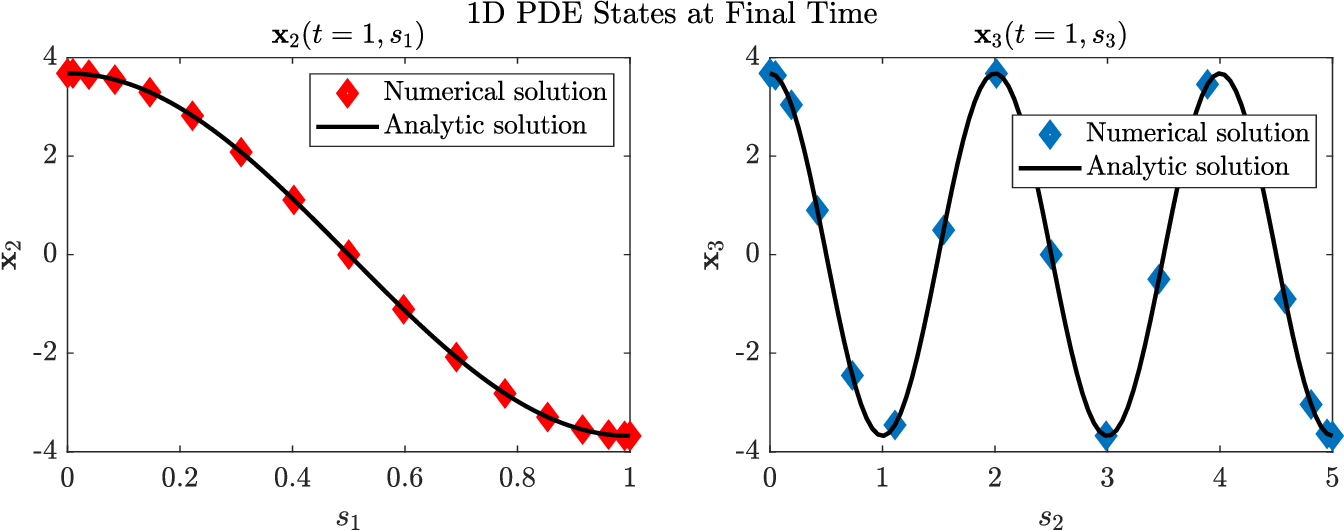}
	\caption{Numerical and true values of the 1D PDE states $\mbf{x}_{2}(t)$ and $\mbf{x}_{3}(t)$ from the ODE-PDE~\eqref{eq:Ch6_ExB_PDE} at $t=1$, simulated with initial conditions as in~\eqref{eq:Ch6_ExB_ICs}.}\label{fig:Ch6:ExB_1DState}
\end{figure}

\begin{figure}[H]
	\centering
	\includegraphics[width=0.98\textwidth]{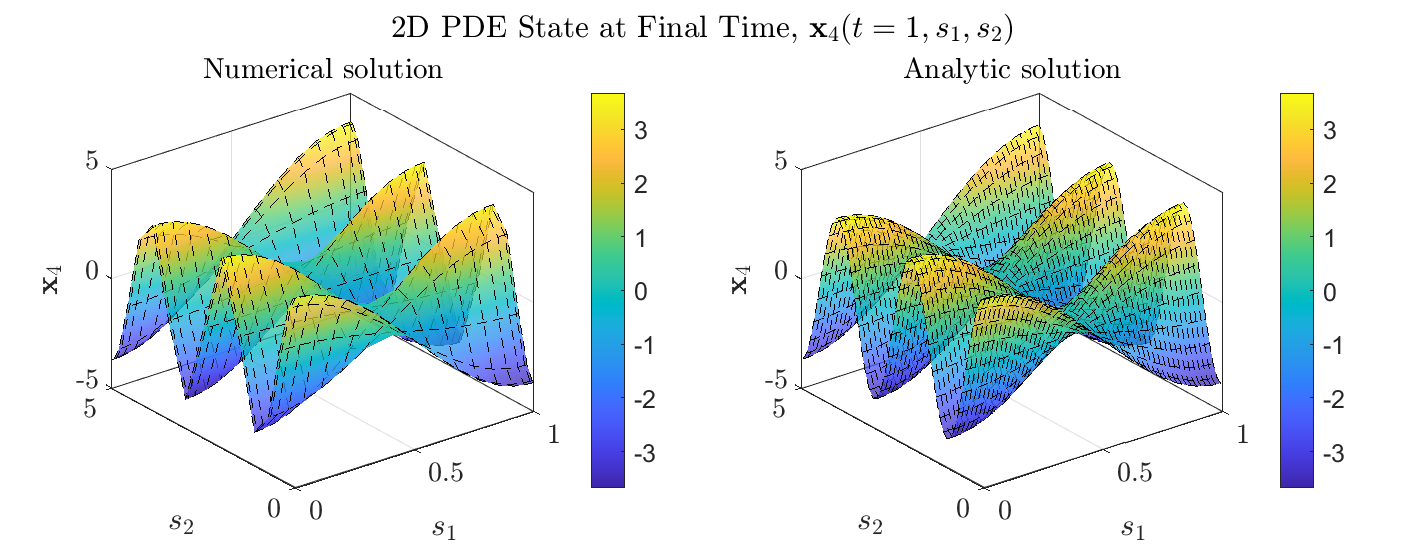}
	\caption{Numerical and true values of the 2D PDE state $\mbf{x}_{4}(t)$ from the ODE-PDE~\eqref{eq:Ch6_ExB_PDE} at $t=1$, simulated with initial conditions as in~\eqref{eq:Ch6_ExB_ICs}.}\label{fig:Ch6:ExB_2DState}
\end{figure}

\subsection{PIESIM Demonstration D: DDE example}\label{subsec:PIESIM:examples:DDE}
Simulation of DDEs can be performed using the same steps as the simulation of PDEs, only now passing a DDE structure rather than PDE structure as first argument to the \texttt{PIESIM()} function. To illustrate, consider a DDE system,
\begin{align}\label{eq:Ch6_ExC_DDE}
    \dot{x}(t) &= \bmat{-1&2\\0&1}x(t) + \bmat{0.6&-0.4\\0&0}x(t-\tau_a) + \bmat{0&0\\0&-0.5}x(t-\tau_b)+\bmat{1\\1}w(t)+\bmat{0\\1}u(t)\\
    z(t) &= \bmat{1&0\\0&1\\0&0} x(t)+\bmat{0\\0\\0.1} u(t) \notag
\end{align}
where $\tau_a = 1$ and $\tau_b = 2$. As shown in Chapter~\ref{ch:PDE_DDE_representation}, this system can be represented in PIETOOLS as a structure \texttt{DDE} as follows:
\begin{matlab}
\begin{verbatim}
 >> DDE.A0=[-1 2;0 1];        DDE.Ai{1}=[.6 -.4; 0 0]; 
 >> DDE.Ai{2}=[0 0; 0 -.5];   DDE.B1=[1;1];
 >> DDE.B2=[0;1];             DDE.C1=[1 0;0 1;0 0];
 >> DDE.D12=[0;0;.1];         DDE.tau=[1,2];
\end{verbatim}
\end{matlab}
We simulate the system with an initial state $x_{1}(0)=x_{2}(0)=1$, and with a disturbance $w(t)=-4t-4$, which we declare uas simply
\begin{matlab}
\begin{verbatim}
 >> uinput.w(1) = -4*st-4;
 >> uinput.u(1) = 0;
 >> uinput.ic=[1,1];
\end{verbatim}
\end{matlab}
Note that since the initial conditions on history are not provided, they will be defaulted to zero. Now, for the remaining simulation options, we use the same settings as in Subsection~\ref{subsec:PIESIM:examples:PDE_1D}:
\begin{matlab}
\begin{verbatim}
 >> opts.plot = 'yes';
 >> opts.N = 8;
 >> opts.tf = 1;
 >> opts.Norder = 2;
 >> opts.dt=1e-3;
\end{verbatim}
\end{matlab}
Then, we can simulate the system and extract the solution as simply
\begin{matlab}
\begin{verbatim}
 >> solution = PIESIM(DDE,opts,uinput);
 >> tval = solution.timedep.dtime;
 >> xval = solution.timedep.primary{1};
 >> zval = solution.timedep.regulated{1};
\end{verbatim}
\end{matlab}
Note here that the simulated value of the DDE state is stored in \texttt{solution.timedep.primary\{1\}} -- there is no separate field for DDE solutions. Element \texttt{xval(i,j)} will then give the simulated value of the \texttt{i}th state at time \texttt{tval(j)}. The simulated evolution of the DDE states is displayed in Figure~\ref{fig:Ch6:ExC_DDE}
\begin{figure}[H]
	\centering
	\includegraphics[width=0.98\textwidth]{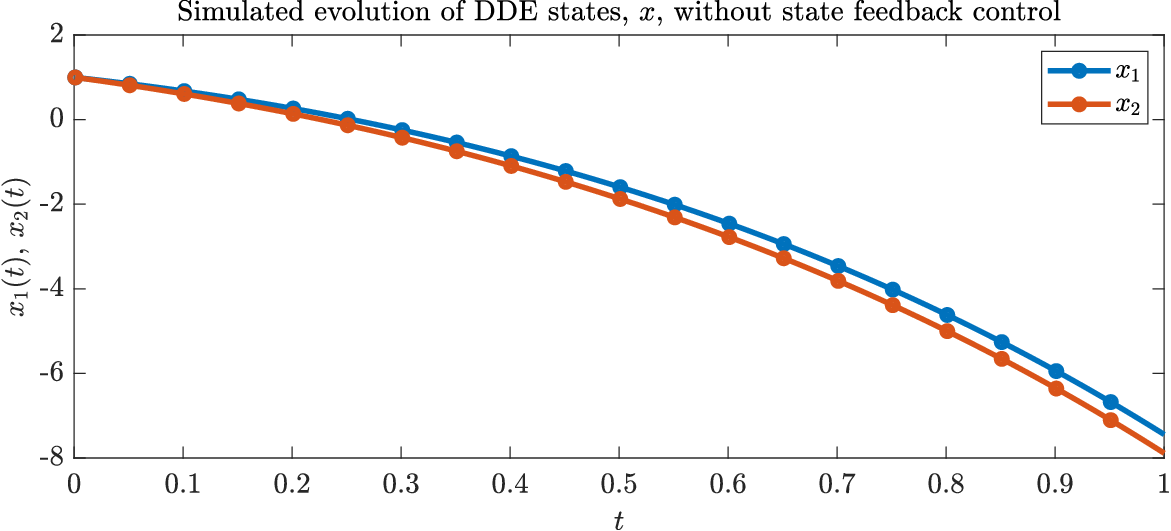}
	\caption{State solutions $x_{1}(t)$ and $x_{2}(t)$ to the DDE~\eqref{eq:Ch6_ExC_DDE} simulated up to $t=1$, with unit initial conditions and disturbance $w(t)=-4t-4$.}\label{fig:Ch6:ExC_DDE}
\end{figure}




\subsection{PIESIM Demonstration E: PIE example}\label{subsec:PIESIM:examples:PIE}
Simulating the DDE from Section~\ref{subsec:PIESIM:examples:DDE}, it appears that this system is unstable. To resolve this, we can use PIETOOLS to design a stabilizing controller for the PIE, and then visualize the behaviour of the system under the action of the controller by simulating the closed-loop system. However, controller synthesis in PIETOOLS is performed in terms of the PIE representation, and conversion of a PIE back to DDE/PDE format is often tricky. Fortunately, \texttt{PIESIM()} also supports simulation of systems in the PIE representation, by passing a \texttt{pie\_struct} object as first argument to the function. To illustrate, consider again the DDE~\eqref{eq:Ch6_ExC_DDE}, which we declared as a \texttt{DDE} structure in the previous subsection. We can synthesize a controller for this system by first converting it to a PIE, and then calling the controller synthesis LPI (see also Sec.~\ref{sec:LPI_examples:control}) as follows
\begin{matlab}
\begin{verbatim}
 >> PIE = convert_PIETOOLS_DDE(DDE,'pie');
 >> [~, K, gam] = lpiscript(PIE,'hinf-controller','light');
 >> PIE_CL = closedLoopPIE(PIE,K);
\end{verbatim}
\end{matlab}
Here, \texttt{PIE} will be a \texttt{pie\_struct} objecting representing the PIE representation of our DDE, expressed in terms of a fundamental state $\mbf{x}_{\text{f}}(t)$. The output \texttt{K} will then be an \texttt{opvar} object representing the optimal controller gain $\mcl{K}$ defining the control law $u(t)=\mcl{K}\mbf{x}_{\text{f}}(t)$. Finally, \texttt{PIE\_CL} will be another \texttt{pie\_struct} object, representing the PIE dynamics of the system with the feedback law $u(t)=\mcl{K}\mbf{x}_{\text{f}}(t)$ imposed. Now, to simulate the behaviour of this system, we can simply call PIESIM for the object \texttt{PIE\_CL}. Here, we will use the same options for numerical integration and discretization as in the previous subsection, setting
\begin{matlab}
\begin{verbatim}
 >> opts.plot = 'yes';
 >> opts.N = 8;
 >> opts.tf = 1;
 >> opts.Norder = 2;
 >> opts.dt = 1e-3;
\end{verbatim}
\end{matlab}
However, for the field \texttt{uinput}, we note that in this case we have no exact solution, nor any controlled input (as \texttt{PIE\_CL} represents the closed-loop system), so we declare only the disturbance $w(t)=-4t-4$ and the initial conditions on the initial state $x_1(0)=x_2(0)=1$,  as in the previous example,
\begin{matlab}
\begin{verbatim}
 >> clear uinput;    syms st
 >> uinput.w(1) = -4*st-4;
 >> uinput.ic=[1,1];
\end{verbatim}
\end{matlab}
Note that, since we are now simulating a PIE, additional initial conditions would have to be declared for the distributed PIE states. However, since we want to model the response to a zero PIE initial state, we declare no extra initial conditions in this case, so that the PIE initial state will default to zero. 
Then, we can finally simulate the closed-loop response of the system using PIESIM as
\begin{matlab}
\begin{verbatim}
 >> solution = PIESIM(PIE,opts,uinput);
 >> tval = solution.timedep.dtime;
 >> xval = solution.timedep.primary{1};
 >> zval = solution.timedep.regulated{1};
\end{verbatim}
\end{matlab}
Here, even though we pass a PIE system to the PIESIM, the values \texttt{solution.timedep.primary\{1\}} and \texttt{solution.timedep.primary\{2\}} will still correspond to simulated values of the primary state, not the fundamental state. The simulated evolution of the DDE state variables is displayed in Figure~\ref{fig:Ch6:ExD_DDE}.

\begin{figure}[H]
\centering
\includegraphics[width=0.98\textwidth]{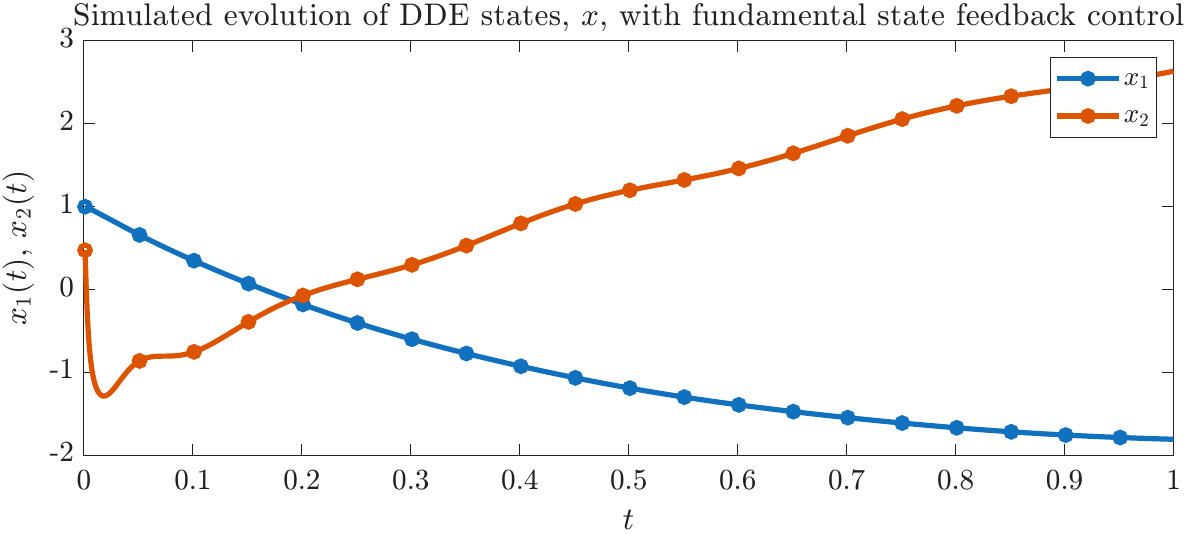}
\caption{State solutions $x_{1}(t)$ and $x_{2}(t)$ to the DDE~\eqref{eq:Ch6_ExC_DDE} with optimal feedback synthesized with PIETOOLS, simulated up to $t=1$, with unit initial conditions and disturbance $w(t)=-4t-4$.}\label{fig:Ch6:ExD_DDE}
\end{figure}

\begin{boxEnv}{Note}
    PIESIM does not currently support simulation of 2D PIEs directly. This feature will be added in a later release. To run PIESIM in 2D, use PDE input.
\end{boxEnv}

\chapter{Declaring and Solving Convex Optimization Programs on PI Operators}\label{ch:LPIs}

In Chapter~\ref{ch:PIE} we showed how PIETOOLS can be used to derive an equivalent PIE representation of any well-posed system of linear partial differential and delay-differential equations. This PIE representation is free of the boundary conditions and continuity constraints that appear in the PDE representation, allowing analysis of PIEs to be performed without having to explicitly account for such additional constraints. In addition, PIEs are parameterized by PI operators, which can be added and multiplied, and for which concepts of e.g. positivity are well-defined. This allows us to impose positivity and negativity constraints on PI operators, referred to as Linear PI Inequalities (LPIs), to define convex optimization programs for testing properties (such as stability) of PIEs. 

In this chapter, we show how these convex optimization problems can be implemented in PIETOOLS. In particular, in Sections~\ref{sec:LPIs:prog_struct} and~\ref{sec:LPI:decvars}, we show how an LPI optimization program can be initialized, and how (PI operator) decision variables can be added to this program structure. Next, in Section~\ref{sec:LPIs:constraints} we show how PI operator equality and inequality constraints can be specified, followed by how an objective function can be set for the program in Section~\ref{sec:LPIs:obj_fun}. In Sections~\ref{sec:LPIs:solve} and~\ref{sec:LPIs:getsol}, we then show how the optimization program can be solved, and how the obtained solution can be extracted, respectively. Finally, in Section~\ref{sec:executives-settings}, we show how pre-defined executive files can be used to solve standard LPI optimization programs for PIEs, and how properties in these optimization programs can be modified using the \texttt{settings} files.

To illustrate each of these functions, we consider the problem of constructing an $H_{\infty}$-optimal estimator for a simple diffusive 2D PDE. This can be done in the PIE representation by solving an LPI (see Sec.~\ref{sec:LPI_examples:estimation}), and we will illustrate how this LPI can be declared and solved in PIETOOLS in the following sections. 
The full codes presented throughout this section are also included in the script ``PIETOOLS\_Code\_Illustrations\_Ch7\_LPI\_Programming''.

\newpage

\begin{Statebox}{\textbf{Example}}
Consider the problem of designing an $H_{\infty}$-optimal estimator of the form 
\begin{align}\label{eq:LPIs:PIE_example}
	\partial_t(\mcl T \hat{\mbf{x}}_{\text{f}})(t) &=\mcl A\mbf{\hat{x}_{\text{f}}}(t)+\mathcal{L}\bl(y(t)-\hat{y}(t)\br), & & &
    \partial_t(\mcl T \mbf{x}_{\text{f}}(t))&=\mcl A\mbf{x}_{\text{f}}(t)+\mcl{B}_1w(t), \notag\\
	\hat{z}(t) &= \mcl{C}_1\mbf{\hat{x}_{\text{f}}}(t),   &   &\text{for a PIE},  &
    z(t) &= \mcl{C}_1\mbf x_{\text{f}}(t) + \mcl{D}_{11}w(t),  \notag\\
	\hat{y}(t) &= \mcl{C}_2\mbf{\hat{x}_{\text{f}}}(t),   & & &
    y(t) &= \mcl{C}_2\mbf{x}_{\text{f}}(t) + \mcl{D}_{21}w(t),
\end{align}
aiming to find an operator $\mcl{L}$ that minimize the gain $\sup_{w\in L_{2}[0,\infty),~w\neq 0}\frac{\|\hat{z}-z\|_{L_2}}{\|w\|_{L_2}}$. 
To construct such an operator, we solve the LPI,
\begin{align}\label{eq:LPIs:LPI_example}
	&\min\limits_{\gamma,\mcl{P},\mcl{Z}} ~~\gamma&\notag\\[-1.0em]
	&\mcl{P}\succ0, &
	&\hspace{-1ex}Q:=\bmat{-\gamma I& -\mcl D_{11}^{\top}&-(\mcl P\mcl B_1+\mcl Z\mcl D_{21})^*\mcl T\\(\cdot)^*&-\gamma I&\mcl C_1\\(\cdot)^*&(\cdot)^*&(\mcl P\mcl A+\mcl Z\mcl C_2)^*\mcl T+(\cdot)^*}\preccurlyeq 0&
\end{align}
so that, for any solution $(\gamma,\mcl{P},\mcl{Z})$ to this problem, letting $\mcl{L}:=\mcl{P}^{-1} \mcl{Z}$, the estimation error will satisfy $\norm{z-\hat{z}}_{L_{2}} \leq \gamma \norm{w}_{L_{2}}$. For more details on this LPI, and additional examples of LPIs and their applications, we refer to Chapter~\ref{ch:LPI_examples}. 

We will solve the LPI~\eqref{eq:LPIs:LPI_example} for the PIE corresponding to the PDE
\begin{align}\label{eq:LPIs:PDE_example}
    && \partial_{t}\mbf{x}(t,s_{1},s_{2})&=\partial_{s_{1}}^2\mbf{x}(t,s_{1},s_{2})+\partial_{s_{1}}^2\mbf{x}(t,s_{1},s_{2}) + 4\mbf{x}(t,s_{1},s_{2}) &    s_{1}&\in[0,1],\nonumber\\
    && &\qquad + s_{1}(1-s_{1})(s_{2}+1)(3-s_{2})w(t),  & s_{2}&\in[-1,1],  \notag\\
    \text{with BCs}& & 0&=\mbf{x}(t,0,s_{2})=\mbf{x}(t,1,s_{2}),  \notag\\
    &&  0&=\mbf{x}(t,s_{1},-1)=\partial_{s_{2}}\mbf{x}(t,s_{1},1),  \notag\\
    \text{and outputs}& & z(t)&=\int_{0}^{1}\int_{-1}^{1}\mbf{x}(t,s_{1},s_{2})ds_{2}ds_{1} + w(t),   \notag\\
    &&y(t,s_{1})&=\mbf{x}(t,s_{1},1). 
\end{align}
We declare this PDE using the Command Line Input format as
\begin{matlab}
\begin{verbatim}
>> pvar s1 s2 t
>> pde_var state x input w output z sense y;
>> x.vars = [s1;s2];   x.dom = [0,1;-1,1];
>> y.vars = s1;        y.dom = [0,1];
>> PDE = [diff(x,t)==diff(x,s1,2)+diff(x,s2,2)+4*x +s1*(1-s1)*(s2-1)*(3-s2)*w;
                  z==int(x,[s1;s2],[0,1;-1,1]) +w;
                  y==subs(x,s2,1);
               subs(x,s1,0)==0;    subs(x,s1,1)==0;
               subs(x,s2,-1)==0;   subs(diff(x,s2),s2,1)==0];
\end{verbatim}  
\end{matlab}
We convert the PDE to an equivalent PIE using \texttt{convert}, and extract the defining PI operators so that these can be used to declare the LPI~\eqref{eq:LPIs:LPI_example}
\begin{matlab}
\begin{verbatim}
>> PIE = convert(PDE,'pie');
>> T = PIE.T;    
>> A = PIE.A;      C1 = PIE.C1;      C2 = PIE.C2;
>> B1 = PIE.B1;    D11 = PIE.D11;    D21 = PIE.D21;
\end{verbatim}
\end{matlab}

\end{Statebox}

\section{Initializing an Optimization Problem Structure}\label{sec:LPIs:prog_struct}

In PIETOOLS, optimization programs are stored as program structures \texttt{prog}. These structures keep track of the free variables in the optimization program, the decision variables in the optimization program, the constraints imposed upon these decision variables, and the objective function in terms of these decision variables. Such an optimization program structure must always be initialized with the function \texttt{lpiprogram} as
\begin{matlab}
\begin{verbatim}
 >> prog = lpiprogram(vars, dum_vars, dom, dvars, free_vars);
\end{verbatim}
\end{matlab}
This function takes 5 inputs, only two of which are mandatory:
\begin{itemize}
    \item \texttt{vars}: A $n\times 1$ array specifying the $n$ spatial variables that appear in the optimization program. These variables will be independent variables in the optimization program.

    \item \texttt{dum\_vars}: (optional) A $n\times 1$ array specifying for each of the spatial variables an associated dummy variable used for integration in the PI operators. In most cases, this argument need not be declared, in which case the function automatically generates a dummy variable for each of the specified spatial variables by adding \texttt{\_dum} to the variable name (e.g. \texttt{s} yields \texttt{s\_dum}). This matches the default dummy variables used throughout PIETOOLS. However, if you explicitly used different dummy variable names to define the PI operators in your optimization program (e.g. if \texttt{T.var1=s} but \texttt{T.var2=theta}), it is crucial that you specify these dummy variables when initializing the optimization program as well.

    \item \texttt{dom}: A $n\times 2$ array specifying for each of the $n$ variables the lower boundary (first column) and upper boundary (second column) of the interval on which this variable exists.

    \item \texttt{dvars}: (optional) A $q\times 1$ array of decision variables that appear in the optimization program. Decision variables can also be added to the optimization program structure later using the functions described in Section~\ref{sec:LPI:decvars}.

    \item \texttt{freevars}: (optional) A $m\times 1$ array of additional independent variables in the optimization program, that are not necessary spatial variables, and therefore are not restricted to a particular domain. This field is usually left empty.
\end{itemize}

The output is a structure representing an optimization program, to which (additional) decision variables and constraints can be added as shown in the following sections.

\begin{boxEnv}{\textbf{Note:}}
To represent LPI optimization programs, PIETOOLS utilizes the \texttt{sosprogram} optimization program structure from SOSTOOLS 4.00~\cite{sostools}. For additional options allowed by SOSTOOLS not mentioned here, we refer to the SOSTOOLS 4.00 manual.
\end{boxEnv}

\newpage

\begin{Statebox}{\textbf{Example}}

For the LPI~\eqref{eq:LPIs:LPI_example}, the relevant spatial variables are $(s_{1},s_{2})\in[0,1]\times[-1,1]$. We initialize the optimization program structure as
\begin{matlab}
\begin{verbatim}
 >> prog = lpiprogram([s1;s2], [0,1;-1,1])
 prog = 

   struct with fields:

             var: [1×1 struct]
            expr: [1×1 struct]
        extravar: [1×1 struct]
       objective: []
         solinfo: [1×1 struct]
        vartable: [4×1 polynomial]
          varmat: [1×1 struct]
     decvartable: {}
             dom: [2x2 double]
\end{verbatim}
\end{matlab}
This initializes an empty optimization program in the spatial variables \texttt{s1} and \texttt{s2}, existing on the intervals \texttt{[0,1]} and \texttt{[-1,1]}, respectively, and stored in the field \texttt{vartable}
\begin{matlab}
\begin{verbatim}
 >> prog.vartable
 ans =
   [     s1]
   [     s2]
   [ s1_dum]
   [ s2_dum]
\end{verbatim}
\end{matlab}
Note that \texttt{lpiprogram} automatically adds dummy variables \texttt{s1\_dum} and \texttt{s2\_dum} to the program, which match the dummy variable defining the PI operators in our PIE:
\begin{matlab}
\begin{verbatim}
 >> vars = PIE.vars
 vars = 
   [ s1, s1_dum]
   [ s2, s2_dum]
\end{verbatim}
\end{matlab}
If, for whatever reason, you explicitly declared different dummy variables to define your PI operators (e.g. \texttt{pvar theta nu}), it is vital that you pass these dummy variables to \texttt{lpiprogram} instead (e.g. using \texttt{lpiprogram([s1;s2],[theta;nu],[0,1;-1,1]]}).
\end{Statebox}

\section{Declaring Decision Variables}\label{sec:LPI:decvars}
Having shown how to initialize an optimization program structure \texttt{prog}, in this section, we show how decision variables can be added to the optimization program structure. For the purposes of implementing LPIs, we distinguish three types of decision variables: standard scalar decision variables (Subsection~\ref{sec:lpidecvar}), positive semidefinite PI operator decision variables (Subsection~\ref{sec:poslpivar}), and indefinite PI operator decision variables (Subsection~\ref{sec:lpivar}).

\subsection{\texttt{lpidecvar}}\label{sec:lpidecvar}

The simplest decision variables in LPI programs are represented by scalar \texttt{dpvar} objects, and can be declared using \texttt{lpidecvar}, as e.g.
\begin{matlab}
\begin{verbatim}
 >> [prog,d1] = lpidecvar(prog,'d1');    % Generate decision variable with name d1
\end{verbatim}
\end{matlab}
This function takes an optimization program structure \texttt{prog}, and returns the same structure but with the decision variable \texttt{d1} added to it, where the output \texttt{d1} is a \texttt{dpvar} object which can then be manipulated using standard operations such as addition and multiplication. Note that the second argument \texttt{'d1'} is an optional input that merely specifies the desired name of the output decision variable, so that \texttt{d1.dvarname='d1'}. If the name of the decision variable is not of importance, an $m\times n$ array of decision variables can also be declared as
\begin{matlab}
\begin{verbatim}
 >> [prog,d_arr] = lpidecvar(prog,[m,n]);  % Generate mxn decision variable array
\end{verbatim}
\end{matlab}
where now \texttt{d\_arr} is an $m\times n$ \texttt{dpvar} object, with each element \texttt{d\_arr(i,j)} being a decision variable named \texttt{coeff\_k} for some $k$.

\begin{Statebox}{\textbf{Example}}

In the LPI~\eqref{eq:LPIs:LPI_example}, $\gamma$ is a scalar decision variable, that appears both in the constraints and the objective function. To declare this variable, we simply call
\begin{matlab}
\begin{verbatim}
 >> [prog,gam] = lpidecvar(prog, 'gam')
 prog = 

   struct with fields:
 
             var: [1×1 struct]
            expr: [1×1 struct]
        extravar: [1×1 struct]
       objective: 0
         solinfo: [1×1 struct]
        vartable: [4×1 polynomial]
          varmat: [1×1 struct]
     decvartable: {'gam'}
             dom: [2x2 double]
\end{verbatim}
\end{matlab}
returning an optimization program structure with the variable \texttt{'gam'} added to the field \texttt{decvartable}. The output \texttt{gam} is a \texttt{dpvar} object representing this variable, which we will use to declare constraints in Section~\ref{sec:LPIs:constraints}.

\end{Statebox}

\subsection{\texttt{poslpivar}}\label{sec:poslpivar}

In PIETOOLS, positive semidefinite PI operator decision variables $\mcl{P}\succcurlyeq 0$ are parameterized by positive matrices $P\succcurlyeq 0$ and positive scalar-valued functions $g(s)\geq 0$ (for $s$ in the domain) as $\mcl{P}=\mcl{Z}_d^* (gP)\mcl{Z}_d\succeq 0$, where $\mcl{Z}_{d}$ is a PI operator parameterized by monomials of degree of at most $d$ (see Theorem~\ref{th:positivity} in Appendix~\ref{appx:PI_theory}). Such PI operators can be declared using the function \texttt{poslpivar}:
\begin{matlab}
 >> [prog,P] = poslpivar(prog,n,d,opts);
\end{matlab}
This function takes four inputs, two of which are mandatory:
\begin{itemize}
	\item \texttt{prog}: An LPI program structure to which to add the PI operator decision variable.
    
	\item \texttt{n}: A $2\times 1$ vector \texttt{[n0;n1]} specifying the dimensions of $\mcl{P}:\sbmat{\R^{n_0}\\L_2^{n_1}[a,b]}\rightarrow\sbmat{\R^{n_0}\\L_2^{n_1}[a,b]}$ for a 1D operator, or a $4\times 1$ vector \texttt{[n0;n1;n2;n3]} specifying the dimensions for a 2D operator $\mcl{P}:\sbmat{\R^{n_0}\\L_2^{n_1}[a,b]\\L_2^{n_2}[c,d]\\L_2^{n_3}[[a,b]\times[c,d]]}\rightarrow \sbmat{\R^{n_0}\\L_2^{n_1}[a,b]\\L_2^{n_2}[c,d]\\L_2^{n_3}[[a,b]\times[c,d]]}$.
	\item \texttt{d} (optional): 
    \begin{itemize}
     \item 1D: A cell structure of the form $\texttt{\{a,[b0,b1,b2]\}}$, specifying the degree \texttt{a} of $s$ in $Z_1(s)$, the degree \texttt{b0} of $s$ in $Z_2(s,\theta)$, the degree \texttt{b1} of $\theta$ in $Z_2(s,\theta)$, and the degree \texttt{b2} of $s$ and $\theta$ combined in $Z_2(s,\theta)$ (see Thm.~\ref{th:positivity}).

     \item 2D: A structure with fields  \texttt{d.dx,d.dy,d.d2}, specifying degrees for operators along $x\in[a,b]$, along $y\in[c,d]$, and along both $(x,y)\in[a,b]\times[c,d]$;\\ call \texttt{help poslpivar\_2d} for more information.
    \end{itemize}

	\item \texttt{opts} (optional): A structure with fields 
	\begin{itemize}
		\item \texttt{exclude}: $4\times 1$ vector with 0 and 1 values. Excludes the block $T_{ij}$ (see Thm.~\ref{th:positivity}) if $i$-th value is 1. Binary $16\times1$ array in 2D; call \texttt{help poslpivar\_2d} for more information.
		\item \texttt{psatz}: Sets $g(s)=1$ if set to 0, and $g(s) = (b-s)(s-a)$ if set to 1 in 1D. In 2D, sets $g(x,y)=(b-x)(x-a)(d-y)(y-c)$ if \texttt{psatz=1}, or $g(x,y)=((x-\frac{b+a}{2})-\frac{b-a}{2})^2 ((y-\frac{d+c}{2})-\frac{d-c}{2})^2$.
		\item \texttt{sep}: Binary scalar value, constrains $\texttt{P.R.R1=P.R.R2}$ if set to 1. In 2D, this field is a $6\times 1$ array; call \texttt{help poslpivar\_2d} for more information.
  \end{itemize}
\end{itemize}
The output is a \texttt{dopvar} class object \texttt{P} representing a positive semidefinite PI operator decision variable, and an updated program structure \texttt{prog} including this decision variable. The function \texttt{poslpivar} has other experimental features to impose sparsity constraints on the $T$ matrix of Thm.~\ref{th:positivity}, which should be used with caution. Call \texttt{help poslpivar} for more information. 

Note that, to enforce $\mcl{P}\succeq 0$ only for $s\in[a,b]$ (or $(x,y)\in[a,b]\times[c,d]$), the option \texttt{psatz} can be used as
\begin{matlab}
\begin{verbatim}
 >> [prog,P] = poslpivar(prog,n,d);
 >> opts.psatz = 1;
 >> [prog,P2] = poslpivar(prog,n,d,opts);
 >> P = P+P2;
\end{verbatim}
\end{matlab}
This will allow $P$ to be nonpositive outside of the specified domain \texttt{I}, allowing for more freedom in the optimization problem. However, since it involves declaring a second PI operator decision variable \texttt{P2}, it may also substantially increase the computational complexity associated with setting up and solving the optimization problem.

Note also that the output operator \texttt{P} of \texttt{poslpivar} is only positive semidefinite, i.e. $\mcl{P}\succeq 0$. To ensure positive definiteness, we can add a strictly positive operator $\epsilon I$
for (small) $\epsilon>0$ to $\mcl{P}$ as e.g.
\begin{matlab}
\begin{verbatim}
 >> P = P + 1e-5*eye(size(P));
\end{verbatim}
\end{matlab}

\begin{Statebox}{\textbf{Example}}
In the LPI~\eqref{eq:LPIs:LPI_example}, we have one positive definite PI operator decision variable $\mcl{P}\succ 0$. 
This operator $\mcl{P}$ should have the same dimensions $\sbmat{n_0\\n_1\\n_{2}\\n_{3}}$ as the operator $\mcl{T}$, which we declare as follows:
 \begin{matlab}
\begin{verbatim}
 >> Pdim = T.dim(:,1)
 Pdim =
        0
        0
        0
        1      
\end{verbatim}
\end{matlab}
With this, we indicate that the operator $\mcl{P}$ should map $L_{2}[[0,1]\times[-1,1]]\to L_{2}[[0,1]\times[-1,1]]$. In its most general form, such an operator may be defined using multiplier and integral operators, but to minimize computational cost, we will simply choose $\mcl{P}$ here to be a (scalar-valued) matrix, excluding any of the integral terms by using the settings
\begin{matlab}
\begin{verbatim}
 >> P_opts.exclude = [1;
                      1;1;1;
                      1;1;1;
                      0;1;1;1;1;1;1;1;1];
\end{verbatim}
\end{matlab}
Here, the first three lines are just to exclude all operator components mapping $\R\to\R$, $L_{2}[0,1]\to L_{2}[0,1]$, and $L_{2}[-1,1]\to L_{2}[-1,1]$, respectively -- which will also be excluded automatically by the specified dimensions of the operator. However, with the last line, we indicate that we also wish to exclude all integral-type operators mapping $L_{2}[[0,1]\times[-1,1]]\to L_{2}[[0,1]\times[-1,1]]$, only including a multiplier-type operator. Although this multiplier could still be defined by a polynomial, we will set the degree of this polynomial to just 0 as
\begin{matlab}
\begin{verbatim}   
 >> Pdeg = 0;
\end{verbatim}
\end{matlab}
We declare a positive semidefinite PI operator $\mcl{P}$ with these specifications as
\begin{matlab}
\begin{verbatim}
 >> [prog,P] = poslpivar(prog,Pdim,Pdeg,P_opts)
 prog = 

   struct with fields:

             var: [1×1 struct]
            expr: [1×1 struct]
        extravar: [1×1 struct]
       objective: [2×1 double]
         solinfo: [1×1 struct]
        vartable: [4×1 polynomial]
          varmat: [1×1 struct]
     decvartable: {2×1 cell}
             dom: [2x2 double]
\end{verbatim}
\end{matlab}
The output \texttt{P} here is a \texttt{dopvar2d} object, representing a PI operator decision variable rather than a fixed PI operator. As such, the parameters (e.g \texttt{P.R00}, \texttt{P.R0x}, \texttt{P.R0y}, etc.) defining this PI operator are \texttt{dpvar} class objects, parameterized by decision variables \texttt{coeff\_i}. These decision variables are collected in the field \texttt{decvartable} of the program structure, and represent the matrix $T$ in the expansion $\mcl{P}=\mcl{Z}_{d}^* T\mcl{Z}_{d}$, constrained to satisfy $T\succcurlyeq 0$. With our settings, the variable $\mcl{P}$ will just be a multiplier operator mapping $L_{2}[[0,1]\times[-1,1]]\to L_{2}[[0,1]\times[-1,1]]$, with the only non-zero parameter being stored in \texttt{P.R22\{1,1\}}:
\begin{matlab}
\begin{verbatim}
 >> P.R22{1,1}
 ans = 
   coeff_1
\end{verbatim}
\end{matlab}
In the LPI~\eqref{eq:LPIs:LPI_example}, the operator $\mcl{P}$ is required to be strictly positive definite, satisfying $\mcl{P}\succeq\epsilon I$ for some $\epsilon>0$. To enforce this, we let $\epsilon=10^{-4}$, and ensure strict positivity of $\mcl{P}$ by calling
\begin{matlab}
\begin{verbatim}
 >> eppos = 1e-4;
 >> P = P + eppos*eye(size(P));
\end{verbatim}
\end{matlab}
\end{Statebox}

\subsection{\texttt{lpivar}}\label{sec:lpivar}
A general (indefinite) PI operator decision variable $\mcl{Z}$ can be declared in PIETOOLS using the \texttt{lpivar} function as shown below.
\begin{matlab}
 >> [prog,Z] = lpivar(prog,n,d);
\end{matlab}
This function takes three inputs, two of which are mandatory:
\begin{itemize}
	\item \texttt{prog}: An LPI program structure to which to add the PI operator decision variable.
	\item \texttt{n}: a $2\times 2$ array \texttt{[m0,n0;m1,n1]} specifying the dimensions of $\mcl{Z}:\sbmat{\R^{n_0}\\L_2^{n_1}[a,b]}\rightarrow\sbmat{\R^{m_0}\\L_2^{m_1}[a,b]}$ for a 1D operator, or $4\times 2$ array \texttt{[m0,n0;m1,n1;m2,n2;m3,n3]} specifying the dimensions for a 2D operator $\mcl{Z}:\sbmat{\R^{n_0}\\L_2^{n_1}[a,b]\\L_2^{n_2}[c,d]\\L_2^{n_3}[[a,b]\times[c,d]]}\rightarrow \sbmat{\R^{m_0}\\L_2^{m_1}[a,b]\\L_2^{m_2}[c,d]\\L_2^{m_3}[[a,b]\times[c,d]]}$.
	\item \texttt{d} (optional):
    \begin{itemize}
	    \item 1D: An array  structure of the form $\texttt{[b0,b1,b2]}$, specifying the degree \texttt{b0} of $s$ in \texttt{Q1, Q2, R0}, the degree \texttt{b1} of $\theta$ in \texttt{Z.R.R1, Z.R.R2}, and the degree \texttt{b2} of $s$ in \texttt{Z.R.R1, Z.R.R2}.
        \item 2D: A structure with fields  \texttt{d.dx,d.dy,d.d2}, specifying degrees for operators along $x\in[a,b]$, along $y\in[c,d]$, and along both $(x,y)\in[a,b]\times[c,d]$; call \texttt{help lpivar\_2d} for more information.
	\end{itemize}
\end{itemize}
The output is a \texttt{dopvar} (or \texttt{dopvar2d}) object \texttt{Z} representing an indefinite PI decision variable, and an updated program structure to which the decision variable has been added. Note that, since PI operator decision variables $\mcl{Z}$ need not be symmetric, the second argument \texttt{n} to the function \texttt{lpivar} must specify both the output (row) dimensions of the operator $\mcl{Z}$ (\texttt{n(:,1)}), and the input (column) dimensions of the operator $\mcl{Z}$ (\texttt{n(:,2)}). 


\begin{Statebox}{\textbf{Example}}

For the LPI~\eqref{eq:LPIs:LPI_example}, we need to declare a PI operator decision variable $\mcl{Z}$ which need not be positive or negative definite. Here, given that $\mcl{C}_{2}:L_{2}[[0,1]\times[-1,1]]\to L_{2}[0,1]$, the operator $\mcl{Z}$ should map $L_{2}[0,1]\to L_{2}[[0,1]\times [-1,1]]$, so we specify the dimensions of $\mcl{Z}$ as
\begin{matlab}
\begin{verbatim}
 >> Zdim = C2.dim(:,[2,1])
 Zdim =

      0     0
      0     1
      0     0
      1     0
\end{verbatim}
\end{matlab}
so that in our case $\mcl{Z}:L_{2}[0,1]\to L_{2}[[0,1]\times [-1,1]]$. As such, only the parameter \texttt{Z.R2x} will be non-empty, but this parameter may still involve both a multiplier term (\texttt{Z.R2x\{1\}}) and partial integral terms (\texttt{Z.R2x\{2\}} and \texttt{Z.R2x\{3\}}). We will allow each of these terms to be defined by polynomials of degree at most 2 in each of the variables, specifying degrees
\begin{matlab}
\begin{verbatim}
 >> Zdeg = 2;
\end{verbatim}
\end{matlab}
Using the function \texttt{lpivar}, we then declare an indefinite PI operator decision variable as
\begin{matlab}
\begin{verbatim}
 >> [prog,Z] = lpivar(prog,Zdim,Zdeg)
 prog = 

   struct with fields:

             var: [1×1 struct]
            expr: [1×1 struct]
        extravar: [1×1 struct]
       objective: [33×1 double]
         solinfo: [1×1 struct]
        vartable: [4×1 polynomial]
          varmat: [1×1 struct]
     decvartable: {33×1 cell}
             dom: [2x2 double]
\end{verbatim}
\end{matlab}
returning another \texttt{dopvar2d} object \texttt{Z}. We can verify that only the field \texttt{Z.R2x} is defined by non-zero parameters, with e.g.
\begin{matlab}
\begin{verbatim}
 >> Z.R2x{1}
 ans = 
   coeff_02 + coeff_04*s2 + coeff_07*s2^2 + coeff_03*s1 + coeff_06*s1*s2 
   + coeff_09*s1*s2^2 + coeff_05*s1^2 + coeff_08*s1^2*s2 + coeff_10*s1^2
   *s2^2
\end{verbatim}
\end{matlab}
indicating that the multiplier term is defined by a polynomial of degree 2, as desired. In total, the object \texttt{Z} is defined by 32 decision variables \texttt{coeff}, increasing the total number of decision variables in \texttt{prog.decvartable} to 33.
\end{Statebox}

\section{Imposing Constraints}\label{sec:LPIs:constraints}
Constraints form a crucial aspect of almost every optimization problem. In LPIs, constraints often appear as inequality or equality conditions on PI operators (e.g. $\mcl{Q}\preceq 0$ or $\mcl{Q}=0$). These constraints can be set up using the functions \texttt{lpi\_ineq} and \texttt{lpi\_eq} respectively, as we show in the next subsections.
	
\subsection{\texttt{lpi\_ineq}}\label{sec:ineq_pi}
Given a program structure \texttt{prog} and \texttt{dopvar} (or \texttt{dopvar2d}) object \texttt{Q} (representing a PI operator variable $\mcl{Q}$), an inequality constraint $\mcl{Q}\succeq 0$ can be added to the program by calling
\begin{matlab}
 >> prog = lpi\_ineq(prog,Q,opts);
\end{matlab}
This function takes three input arguments
\begin{itemize}
	\item \texttt{prog}: An LPI program structure to which to add the constraint $\mcl{Q}\succeq 0$.
	\item \texttt{Q}: A \texttt{dopvar} or \texttt{dopvar2d} object representing the PI operator $\mcl{Q}$. 
	\item \texttt{opts} (optional): A structure with fields 
	\begin{itemize}
		\item \texttt{psatz}: Binary scalar indicating whether to enforce the constraint only locally. If \texttt{psatz=0} (default), a constraint $\mcl{Q}=\mcl{R}_1$ will be enforced, where $\mcl{R}\succeq 0$ will be declared as a \texttt{dopvar} or \texttt{dopvar2d} object
        \begin{matlab}
         >> R1=poslpivar(prog,n,d,opts1)
        \end{matlab}
        with \texttt{opts1.psatz=0}. If \texttt{psatz=1} (or \texttt{psatz=2}), a constraint $\mcl{Q}=\mcl{R}_1+\mcl{R}_2$ will be enforced, with $\mcl{R}_1$ as before, and $\mcl{R}_2\succeq 0$ declared as a \texttt{dopvar} or \texttt{dopvar2d} object 
        \begin{matlab}
         >> R2=poslpivar(prog,n,d,opts2)
        \end{matlab}
        with \texttt{opts2.psatz=psatz}. Using \texttt{psatz=1} (or alternatively \texttt{psatz=2} in 2D) allows the constraint $\mcl{Q}\succeq 0$ to be violated outside of the spatial domain \texttt{I}, but may also substantially increase the computational effort in solving the problem.
  \end{itemize}
\end{itemize}
Calling \texttt{lpi\_ineq}, a modified optimization structure \texttt{prog} is returned with the constraints \texttt{Q>=0} included. Note that the constraint imposed by \texttt{lpi\_ineq} is always \textbf{non-strict}. For strict positivity, an offset $\epsilon>0$ may be introduced, enforcing $\mcl{Q}-\epsilon I\succeq 0$ to ensure $\mcl{Q}\succeq\epsilon I\succ 0$. This may be implemented as e.g.
\begin{matlab}
\begin{verbatim}
 >> prog = lpi_ineq(prog,Q-ep*eye(size(Q)));
\end{verbatim}
\end{matlab}
where \texttt{ep} is a small positive number.

\newpage

\begin{Statebox}{\textbf{Example}}
For the LPI~\eqref{eq:LPIs:LPI_example}, we impose the constraint $\mcl{Q}\preccurlyeq 0$ by calling
\begin{matlab}
\begin{verbatim}
 >> Iw = eye(size(B1,2));    Iz = eye(size(C1,1));
 >> Q = [-gam*Iw,           -D11',       -(P*B1+Z*D21)'*T;
         -D11,              -gam*Iz,     C1;
         -T'*(P*B1+Z*D21),  C1',         (P*A+Z*C2)'*T+T'*(P*A+Z*C2)];
 >> Q_opts.psatz = 1;
 >> prog = lpi_ineq(prog,-Q,Q_opts);
 prog = 
 
   struct with fields:

             var: [1×1 struct]
            expr: [1×1 struct]
        extravar: [1×1 struct]
       objective: [102185×1 double]
         solinfo: [1×1 struct]
        vartable: [4×1 polynomial]
          varmat: [1×1 struct]
     decvartable: {102185×1 cell}
             dom: [2x2 double]
\end{verbatim}
\end{matlab}
We note that \texttt{lpi\_ineq} enforces the constraint $-\mcl{Q}\succcurlyeq 0$ by introducing a new PI operator decision variable $\mcl{R}\succcurlyeq 0$, and imposing the equality constraint $-\mcl{Q}=\mcl{R}$. In doing so, \texttt{lpi\_ineq} tries to ensure that the degrees of the polynomial parameters defining $\mcl{R}$ match those of the parameters defining $-\mcl{Q}$, in this case parameterizing $\mcl{R}$ by $102185-33 = 102152$ decision variables (check the size of \texttt{decvartable}). An operator $\mcl{R}$ involving more or fewer decision variables can also be declared manually, at  which point the constraint $-\mcl{Q}=\mcl{R}$ can be enforced using \texttt{lpi\_eq}, as we show next.
\end{Statebox}

\subsection{\texttt{lpi\_eq}}\label{sec:eq_pi}
Given a program structure \texttt{prog} and \texttt{dopvar} (or \texttt{dopvar2d}) object \texttt{Q} (representing a PI operator decision variable $\mcl{Q}$), an equality constraint $\mcl{Q}=0$ can be added to the program by calling
\begin{matlab}
 >> prog = lpi\_eq(prog,Q);
\end{matlab}
This returns a modified optimization structure \texttt{prog} with the constraints \texttt{Q=0} included. Since \texttt{opvar} objects are defined by parameters (\texttt{Q.P}, \texttt{Q.Q1}, \texttt{Q.Q2}, etc.), each of these parameters will be set to zero in the optimization program. Here, if \texttt{Q} is symmetric, we note that e.g. \texttt{Q.Q1}$=$\texttt{Q.Q2}, and therefore only \texttt{Q.Q1=0} needs to be enforced. To exploit this fact, an argument \texttt{'symmetric'} can be passed to \texttt{lpi\_eq} as
\begin{matlab}
 >> prog = lpi\_eq(prog,Q,'symmetric');
\end{matlab}
This argument is not mandatory, but can reduce the number of equality constraints in the optimization program, and therefore the computational cost of solving it. Of course, this argument should only be passed if \texttt{Q} is indeed symmetric!

\begin{Statebox}{\textbf{Example}}
 For the LPI~\eqref{eq:LPIs:LPI_example}, we can enforce the constraint $\mcl{Q}\precceq 0$ by declaring a new positive semidefinite PI operator decision variable $\mcl{R}\succeq 0$, and enforcing $\mcl{Q}=-\mcl{R}\preccurlyeq 0$. In particular, as an alternative to using \texttt{lpi\_ineq} as in the previous subsection, we can call
\begin{matlab}
\begin{verbatim}
 >> Rdeg = 3;       R_opts.psatz = 1;
 >> [prog_alt,R1] = poslpivar(prog,Q.dim(:,1),Rdeg);
 >> [prog_alt,R2] = poslpivar(prog_alt,Q.dim(:,1),Rdeg-1,R_opts);
 >> prog_alt = lpi_eq(prog_alt,Q+R1+R2,'symmetric');
 prog_alt = 

   struct with fields:

             var: [1×1 struct]
            expr: [1×1 struct]
        extravar: [1×1 struct]
       objective: [105670×1 double]
         solinfo: [1×1 struct]
        vartable: [4×1 polynomial]
          varmat: [1×1 struct]
     decvartable: {105670×1 cell}
             dom: [2x2 double]
\end{verbatim}
\end{matlab}
Here, we chose the degrees of \texttt{R} to be slightly larger than those we used to define \texttt{P}, hoping to make sure that all monomials appearing in \texttt{Q} can then be matched by those in \texttt{R}.
The new program structure includes the constraints $\mcl{R}\succcurlyeq 0$ and $\mcl{Q}+\mcl{R}=0$. The resulting total number of decision variables (105670) is somewhat larger than in the program obtained using \texttt{lpi\_ineq} in the previous subsection, as the specified degrees \texttt{Rdeg} used to define \texttt{R} are quite large. Naturally, we can reduce the number of decision variables and thus the computational cost of solving by specifying smaller degrees \texttt{Rdeg}, but this may come at the cost of greater conservatism.
\end{Statebox}

\section{Defining an Objective Function}\label{sec:LPIs:obj_fun}
Aside from constraints, many optimization problems also involve an objective function, aiming to minimize or maximize some function of the decision variables. To \textbf{minimize} the value of a \textbf{linear} objective function $f(d_1,\hdots,d_1)$, where $d_1,\hdots,d_n$ are decision variables, call \texttt{lpisetobj} with as first argument the program structure, and as second argument the function $f(d)$:
\begin{matlab}
\begin{verbatim}
 >> prog = lpisetobj(prog, f);
\end{verbatim}
\end{matlab}
where \texttt{f} must be a \texttt{dpvar} object representing the objective function. For example, a function $f(\gamma_1,\gamma_2)=\gamma_1+5\gamma_2$ could be specified as objective function using
\begin{matlab}
\begin{verbatim}
 >> [prog,gam] = lpidecvar(prog,[1,2]);
 >> f = gam(1) + 5*gam(2);
 >> prog = lpisetobj(prog, f);
\end{verbatim}
\end{matlab}
\begin{boxEnv}{\textbf{Note:}}
    \begin{itemize}
        \item The objective function must always be linear in the decision variables.
        \item Only one (scalar) objective function can be specified.
        \item In solving the optimization program, the value of the objective function will always be minimized. Thus, to maximize the value of the objective function $f(d)$, specify $-f(d)$ as objective function. 
    \end{itemize}
\end{boxEnv}

\begin{Statebox}{\textbf{Example}}
 In the LPI~\eqref{eq:LPIs:LPI_example}, the value of the decision variable $\gamma$ is minimized. As such, the objective function in this problem is just $f(\gamma)=\gamma$, which we declare as
\begin{matlab}
\begin{verbatim}
 >> prog = lpisetobj(prog, gam);
 prog = 

   struct with fields:

             var: [1×1 struct]
            expr: [1×1 struct]
        extravar: [1×1 struct]
       objective: [102185×1 double]
         solinfo: [1×1 struct]
        vartable: [4×1 polynomial]
          varmat: [1×1 struct]
     decvartable: {102185×1 cell}
             dom: [2x2 double]
\end{verbatim} 
\end{matlab}
The field \texttt{objective} in the resulting structure \texttt{prog} will be a vector with all elements equal to zero, except the first element equal to 1, indicating that the objective function is equal to 1 times the first decision variable in \texttt{decvartable}, which is $\gamma$.
\end{Statebox}

\section{Solving the Optimization Problem}\label{sec:LPIs:solve}

Once an optimization program has been specified as a program structure \texttt{prog}, it can be solved by calling \texttt{lpisolve}
\begin{matlab}
\begin{verbatim}
 >> prog_sol = lpisolve(prog,opts);
\end{verbatim}
\end{matlab}
Here \texttt{opts} is an optional argument to specify settings in solving the optimization program, with fields
\begin{itemize}
    \item \texttt{opts.solve}: \texttt{char} type object specifying which semidefinite programming (SDP) solver to use. Options include `\texttt{sedumi}' (default), `\texttt{mosek}', `\texttt{sdpt3}', and `\texttt{sdpnalplus}'. Note that these solvers must be separately installed in order to use them.

    \item \texttt{opts.simplify}: Binary value indicating whether the solver should attempt to simplify the SDP before solving. The simplification process may take additional time, but may reduce the time of actually solving the SDP.
\end{itemize}
After calling \texttt{lpisolve}, it is important to check whether the problem was actually solved. Using the solver SeDuMi, this can be established looking at e.g. the value of \texttt{feasratio}, which will be close to $+1$ if the problem was successfully solved, and the values of \texttt{pinf} and \texttt{dinf}, which should both be zero if the problem is primal and dual solvable. If \texttt{lpisolve} is unsuccessful in solving the problem, the problem may be infeasible, or different settings must be used in declaring the decision variables and constraints (e.g. include higher degree monomials in the positive operators).

\begin{Statebox}{\textbf{Example}}
Having declared the full LPI~\eqref{eq:LPIs:LPI_example}, we can finally solve the problem as
\begin{matlab}
\begin{verbatim}
 >> solve_opts.solver = 'sedumi';
 >> solve_opts.simplify = true;
 >> prog_sol = lpisolve(prog,solve_opts);

 Residual norm: 0.00031666
 
          iter: 18
     feasratio: -0.0725
          pinf: 0
          dinf: 0
        numerr: 2
        timing: [0.9840 223.2360 0.0320]
       wallsec: 224.2520
        cpusec: 114.1406
\end{verbatim}
\end{matlab}
The returned value of \texttt{feasratio} is close to zero, and \texttt{numerr} is \texttt{2}, indicating that SeDuMi has run into serious numerical problems. This is not uncommon when solving LPI optimization programs where we are simultaneously testing feasibility of some (large) LPI constraints and minimizing an objective. In such cases, it is often a good idea to run the optimization program for a fixed a value of the objective function, only testing feasibility of the LPI, and manually performing bisection on the value of the objective function coefficients if necessary. For example, rerunning this same test but now fixing a value \texttt{gam==2} a priori, \texttt{lpisolve} returns an output structure
\begin{matlab}
\begin{verbatim}
         iter: 15
    feasratio: 0.8673
         pinf: 0
         dinf: 0
       numerr: 1
       timing: [0.3380 2.7187e+02 0.0850]
      wallsec: 2.7227e+02
       cpusec: 133
\end{verbatim}
\end{matlab}
Here, we have \texttt{pinf=0} and \texttt{dinf=0}, indicating that the proposed problem was not found to be primal or dual infeasible, and the feasratio is quite close to 1. Thus, it appears that the optimization program was successfully solved, and the LPI~\eqref{eq:LPIs:LPI_example} is feasible for $\gamma=2$.

\end{Statebox}

\section{Extracting the Solution}\label{sec:LPIs:getsol}

Calling \texttt{prog\_sol=lpisolve(prog)}, a program structure \texttt{prog\_sol} is returned that is very similar to the input structure \texttt{prog}, defining the solved optimization problem. 
Given this solved LPI program structure, we can retrieve the (optimal) value of any decision variable appearing in the LPI -- be it a function specified as an object of type \texttt{dpvar}, or an operator specified as an object of type \texttt{dopvar} -- using the function \texttt{lpigetsol}, passing the solved optimization program structure \texttt{prog\_sol} as first argument, and the considered \texttt{dpvar} or \texttt{dopvar} (\texttt{dopvar2d}) decision variable as second argument:
\begin{matlab}
\begin{verbatim}
 >> f_val = lpigetsol(prog_sol,f);
 >> Pop_val = lpigetsol(prog_sol,Pop);
\end{verbatim}
\end{matlab}
Note that the lpiprogram \texttt{prog\_sol} must be in a solved state (\texttt{lpisolve} must have been called) to retrieve the solution for the input \texttt{dpvar}, \texttt{dopvar}, or \texttt{dopvar2d} object. The output \texttt{f\_val} will then by a \texttt{polynomial} class object, and \texttt{Pop\_val} will then be \texttt{opvar} or \texttt{opvar2d} class object, representing a fixed PI operator, with the solved (optimal) values of the decision variables substituted into the associated parameters.

\begin{Statebox}{\textbf{Example}}
To extract the optimal value of $\gamma$ found when solving the LPI~\eqref{eq:LPIs:LPI_example}, we call
\begin{matlab}
\begin{verbatim}
 >> gam_val = lpigetsol(prog_sol,gam)
 gam_val = 
   1.1344
\end{verbatim}
\end{matlab}
We find that, through proper choice of the operator $\mcl{L}$, the estimator can achieve an $H_{\infty}$-norm $\frac{\|\tilde{z}\|_{L_2}}{\|w\|_{L_2}}\leq 1.1344$. Of course, since the optimization program was not found to be feasible, this bound may not actually be accurate, and we should instead continue with the solved program structure obtained when fixing $\gamma=2$. 
To extract the values of the operators $\mcl{P}$ and $\mcl{Z}$ achieving this gain, we call
\begin{matlab}
\begin{verbatim}
 >> Pval = lpigetsol(prog_sol,P);  
 >> Zval = lpigetsol(prog_sol,Z);
\end{verbatim}
\end{matlab}    
The resulting object \texttt{Pval} and \texttt{Zval} are \texttt{opvar2d} objects, representing the values of the operator $\mcl{P}$ and $\mcl{Z}$ for which the LPI~\eqref{eq:LPIs:LPI_example} holds. Using these values, we can compute an operator $\mcl{L}$ such that the Estimator~\eqref{eq:LPIs:PIE_example} satisfies $\frac{\|\tilde{z}\|_{L_2}}{\|w\|_{L_2}}\leq \gamma=2$, as we show next.
\end{Statebox}

When performing estimator or controller synthesis (see also Chapter~\ref{ch:LPI_examples}), the optimal estimator or controller associated with a solved problem has to be constructed from the solved PI operator decision variables. For example, for the estimator in~\eqref{eq:LPIs:PIE_example}, the value of $\mcl{L}$ is determined by the values of $\mcl{P}$ and $\mcl{Z}$ in the LPI~\eqref{eq:LPIs:LPI_example} as $\mcl{L}=\mcl{P}^{-1}\mcl{Z}$. To facilitate this post-processing of the solution, PIETOOLS includes several utility functions.

\subsection{\texttt{getObserver}}
For a solution $(\mcl{P},\mcl{Z})$ to the optimal estimator LPI~\eqref{eq:LPIs:LPI_example}, the operator $\mcl{L}$ in the Estimator~\eqref{eq:LPIs:PIE_example} can be computed as
\begin{matlab}
\begin{verbatim}
 >> Lval = getObserver(Pval,Zval);
\end{verbatim}
\end{matlab}
where \texttt{Pval} and \texttt{Zval} are \texttt{opvar} (\texttt{opvar2d}) objects representing the (optimal) values of $\mcl{P}$ and $\mcl{Z}$ in the LPI, and \texttt{Lval} is an \texttt{opvar}  (\texttt{opvar2d}) object representing the associated (optimal) value of $\mcl{L}$ in the estimator.

\begin{Statebox}{\textbf{Example}}
Given the \texttt{opvar2d} objects \texttt{Pval} and \texttt{Zval}, we can finally construct an optimal observer operator $\mcl{L}$ for the System~\eqref{eq:LPIs:PIE_example} by calling
\begin{matlab}
\begin{verbatim}
 >> Lval = getObserver(Pval,Zval)
 Lval =

     [] |       [] |       [] |       [] 
     -----------------------------------
     [] | Lval.Rxx |       [] | Lval.Rx2 
     -----------------------------------
     [] |       [] | Lval.Ryy | Lval.Ry2 
     -----------------------------------
     [] | Lval.R2x | Lval.R2y | Lval.R22 

 Lval.R2x =
 
     Too big to display | Too big to display | Too big to display
\end{verbatim}
\end{matlab}
where the expression for \texttt{Lval.Rx2} is rather complicated. Nevertheless, using this value for the operator $\mcl{L}$, an $H_{\infty}$ norm $\frac{\|\tilde{z}\|_{L_2}}{\|w\|_{L_2}}\leq \gamma=2$ can be achieved. Performing bisection on the value of $\gamma$, and perhaps increasing the freedom in our optimization problem (e.g. by increasing the degrees of the monomials defining $\mcl{P}$ and $\mcl{Z}$), we may be able to achieve a tighter bound on the $H_{\infty}$ norm for the obtained operator $\mcl{L}$, or find another operator $\mcl{L}$ achieving a smaller value of the norm.
\end{Statebox}

\subsection{\texttt{getController}}
For a solution $(\mcl{P},\mcl{Z})$ to the optimal control LPI~\eqref{eq:cont_lpi} (Section~\ref{sec:LPI_examples:control}), the operator $\mcl{K}=\mcl{Z}\mcl{P}^{-1}$ defining the feedback law $u=\mcl{K}\mbf{v}$ for optimal control of the PIE~\eqref{eq:LPI_examples:control_PIE} can be computed as
\begin{matlab}
\begin{verbatim}
 >> Kval = getController(Pval,Zval);
\end{verbatim}
\end{matlab}
where \texttt{Pval} and \texttt{Zval} are \texttt{opvar} (\texttt{opvar2d}) objects representing the (optimal) values of $\mcl{P}$ and $\mcl{Z}$ in the LPI, and \texttt{Kval} is an \texttt{opvar} (\texttt{opvar2d}) object representing the associated (optimal) value of $\mcl{K}$ in the feedback law $u=\mcl{K}\mbf{v}$. Note that this feedback law is described in terms of the PIE state $\mbf{v}$, not the state of the associated PDE or TDS. Deriving an optimal controller for the associated PDE or TDS system will require careful consideration of how the PIE state relates to the PDE or TDS state.

\subsection{\texttt{closedLoopPIE}}
For a PIE~\eqref{eq:LPI_examples:control_PIE} and an operator $\mcl{K}$ defining a feedback law $u=\mcl{K}\mbf{v}$, a PIE corresponding to the closed-loop system for the given feedback law can be computed as
\begin{matlab}
\begin{verbatim}
 >> PIE_CL = closedLoopPIE(PIE,Kval);
\end{verbatim}
\end{matlab}
where \texttt{Kval} is an \texttt{opvar} (\texttt{opvar2d}) object representing the value of $\mcl{K}$ in the feedback law $u=\mcl{K}\mbf{v}$, \texttt{PIE} is a \texttt{pie\_struct} object representing the PIE system without feedback, and \texttt{PIE\_CL} is a \texttt{pie\_struct} object representing the closed-loop PIE system with the feedback law $u=\mcl{K}\mbf{v}$ enforced. Note that the resulting system takes no more actuator inputs $u$, so that operators \texttt{PIE\_CL.Tu}, \texttt{PIE\_CL.B2}, \texttt{PIE\_CL.D12}, and \texttt{PIE\_CL.D22} are all empty.

In addition, for a PIE and an operator $\mcl{L}$ defining a Luenberger estimator gain,
a PIE for the closed-loop observed system can be computed as
\begin{matlab}
\begin{verbatim}
 >> PIE_CL = closedLoopPIE(PIE,Lval,'observer');
\end{verbatim}
\end{matlab}
where \texttt{Lval} is an \texttt{opvar} (\texttt{opvar2d}) object representing the value of $\mcl{L}$ in the estimator in~\eqref{eq:LPIs:PIE_example}, \texttt{PIE} is a \texttt{pie\_struct} object representing the PIE system in~\eqref{eq:LPIs:PIE_example}, and \texttt{PIE\_CL} is a \texttt{pie\_struct} object representing the closed-loop PIE system of the form
\begin{align*}
    \partial_t\left(\bmat{\mcl{T}&0\\0&\mcl{T}}\bmat{\mbf{x}_{\text{f}}\\\hat{\mbf{x}}_{\text{f}}}\right)(t,s)&=\left(\bmat{\mcl{A}&0\\-\mcl{L}\mcl{C}_{2}&\mcl{A}+\mcl{L}\mcl{C}_{2}}\bmat{\mbf{x}_{\text{f}}\\\hat{\mbf{x}}_{\text{f}}}\right)(t,s)+\left(\bmat{\mcl{B}_{1}\\\mcl{L}\mcl{D}_{21}}w\right)(t) \\
    \bmat{z\\\hat{z}}(t)&=\left(\bmat{\mcl{C}_1&0\\0&\mcl{C}_1}\bmat{\mbf{x}_{\text{f}}\\\hat{\mbf{x}}_{\text{f}}}\right)(t)+\left(\bmat{\mcl{D}_{11}\\0}w\right)(t)
\end{align*}
so that e.g. \texttt{PIE\_CL.A} represents the operator $\sbmat{\mcl{A}&0\\-\mcl{L}\mcl{C}_{2}&\mcl{A}+\mcl{L}\mcl{C}_{2}}$. Note that, for a PIE with state $\mbf{x}_{\text{f}}(t)\in L_2^{n}$ and output $z(t)\in\R$, the state and output of the closed-loop observed system are respectively given by $\sbmat{\mbf{x}_{\text{f}(t)}\\\hat{\mbf{x}}_{\text{f}}(t)}\in L_2^{2n}$ and $\sbmat{z(t)\\\hat{z}(t)}$, where $\hat{\mbf{x}}_{\text{f}}(t)$ and $\hat{z}(t)$ are estimated values of the PIE state and regulated output, respectively. However, for PIEs with coupled finite- and infinite-dimensional states $\sbmat{x(t)\\\mbf{x}_{\text{f}}(t)}\in\sbmat{\R^{m}\\L_2^{n}}$, the state of the closed-loop system will be of the form $\sbmat{x(t)\\\hat{x}(t)\\\mbf{x}_{\text{f}}(t)\\\hat{\mbf{x}}_{\text{f}}(t)}\in\sbmat{\R^{2m}\\L_2^{2n}}$. This is because \texttt{opvar} objects can only represent maps of states in $\R\times L_2$, not e.g. states in $L_2\times\R$, and so the state components and their estimates will be organized accordingly.

\begin{Statebox}{\textbf{Example}}
 A full code declaring and solving the optimal estimator LPI~\eqref{eq:LPIs:LPI_example} for the PDE~\eqref{eq:LPIs:PDE_example} has been included as a script ``PIETOOLS\_Code\_Illustrations\_Ch7\_LPI\_Programming'' in PIETOOLS. An alternative demonstration for how an optimal estimator can be constructed and simulated is also given in Section~\ref{sec:demos:estimator}.
\end{Statebox}

\section{Running Pre-Defined LPIs: Executives and Settings}\label{sec:executives-settings}

Combining the steps from the previous sections, we find that the $H_{\infty}$-optimal estimator LPI~\eqref{eq:LPIs:LPI_example} can be declared and solved for any given PIE structure \texttt{PIE} using roughly the same code. Therefore, to facilitate solving the $H_{\infty}$-optimal estimator problem, the code has been implemented in an \texttt{executive} file \texttt{PIETOOLS\_Hinf\_estimator}, that is structured roughly as
\begin{matlab}
\begin{verbatim}
 >> % Extract the operators and initialize the LPI program
 >> T = PIE.T;     A = PIE.A;       B1 = PIE.B1;
 >> C1 = PIE.C1;   D11 = PIE.D11;   C2 = PIE.C2;   D12 = PIE.D12;
 >> prog = lpiprogram(PIE.vars,PIE.dom);

 >> % Declare the objective function min{gamma}
 >> [prog,gam] = lpidecvar(prog, 'gam');
 >> prog = lpisetobj(prog, gam);

 >> % Declare the positive operator P>=0
 >> [prog,P] = poslpivar(prog,T.dim,dd1,options1);
 >> if override1==0
 >>     % Allow P<=0 outside domain PIE.dom
 >>     [prog,P2] = poslpivar(prog,T.dim,dd12,options12);
 >>     P = P + P2;
 >> end
 >> % Enforce strict positivity P>0
 >> P.P = P.P + eppos*eye(size(P.P));
 >> P.R.R0 = eppos2*eye(size(P.R.R0));
 
 >> % Declare the indefinite operator Z
 >> [prog,Z] = lpivar(prog,C2.dim(:,[2,1]),PIE.dom,ddZ);

 >> % Enforce the negativity constraint Q<=0
>> Iw = eye(size(B1,2));    Iz = eye(size(C1,1));
 >> Q = [-gam*Iw,           -D11’,    -(P*B1+Z*D21)’*T;
         -D11,              -gam*Iz,  C1;
         -T’*(P*B1+Z*D21),  C1’,      (P*A+Z*C2)’*T+T’*(P*A+Z*C2)+epneg*T'*T];
 >> if use_ineq
 >>     % Enforce using lpi_ineq
 >>     prog = lpi_ineq(prog,-Q,opts);
 >> else
 >>     % Enforce using lpi_eq
 >>     [prog,R] = poslpivar(prog,Q.dim(:,1),Q.I,dd2,options2);
 >>     if override2==0
 >>         % Allow R<=0 outside of domain Q.I
 >>         [prog,R2] = poslpivar(prog,Q.dim(:,1),Q.I,dd3,options3);
 >>         R = R+R2;
 >>     end
 >>     % Enforce Q=-R<=0
 >>     prog = lpi_eq(prog,Q+R,'symmetric');
 >> end
 
 >> % Solve the optimization program and extract the solution
 >> prog_sol = lpisolve(prog, sos_opts);
 >> gam_val = lpigetsol(prog_sol, gam);
 >> Pval = lpigetsol(prog_sol, P);     
 >> Zval = lpigetsol(prog_sol, Z);
 >> Lval = getObserver(Pval, Zval);
\end{verbatim}
\end{matlab}
We note that, in this code, there are several parameters that can be set, including what degrees to use for the PI operator decision variables (\texttt{dd1}, \texttt{ddZ}, \texttt{dd2}, etc.), whether or not to enforce positivity/negativity strictly and/or locally (\texttt{eppos}, \texttt{epneg}, \texttt{override1}, etc.), and what options to use in calling each of the different functions (\texttt{options1}, \texttt{opts}, \texttt{sos\_opts}, etc.). To specify each of the options, the executive files such as \texttt{PIETOOLS\_Hinf\_estimator} can be called with an optional argument \texttt{settings}, as we describe in the following subsection.

\subsection{Settings in PIETOOLS Executives}

When calling an executive function such as \texttt{PIETOOLS\_Hinf\_estimator} in PIETOOLS, a second (optional) argument can be used to specify settings to use in declaring the LPI program. This argument should be a MATLAB structure with fields as defined in Table~\ref{tab:executive_settings_fields}, specifying a value for each of the different options to be used in declaring the LPI program.

\begin{table}[ht]
\begin{tabular}{l|l}
 \texttt{settings} Field  &  \textbf{Application} \\\hline
 \texttt{eppos} & Nonnegative (small) scalar to enforce strict positivity of \texttt{Pop.P}\\
 \texttt{eppos2} & Nonnegative (small) scalar to enforce strict positivity of \texttt{Pop.R0}\\
 \texttt{epneg} & Nonnegative (small) scalar to enforce strict negativity of \texttt{Dop}\\
 \texttt{sosineq} & Binary value, set 1 to use \texttt{sosineq}\\
 \texttt{override1} & Binary value, set 1 to let $\texttt{P2op}=0$\\
 \texttt{override2} & Binary value, set 1 to let $\texttt{De2op}=0$\\
 \texttt{dd1} & 1x3 cell structure defining monomial degrees for \texttt{Pop}\\
 \texttt{dd12} & 1x3 cell structure defining monomial degrees for \texttt{P2op}\\
 \texttt{dd2} & 1x3 cell structure defining monomial degrees for \texttt{Deop}\\
 \texttt{dd3} & 1x3 cell structure defining monomial degrees for \texttt{De2op}\\
 \texttt{ddZ} & 1x3 array defining monomial degrees for \texttt{Zop}\\
 \texttt{options1} & Structure of \texttt{poslpivar} options for \texttt{Pop}\\
 \texttt{options12} & Structure of \texttt{poslpivar} options for \texttt{P2op}\\
 \texttt{options2} & Structure of \texttt{poslpivar} options for \texttt{Deop}\\
 \texttt{options3} & Structure of \texttt{poslpivar} options for \texttt{De2op}\\
 \texttt{opts} & Structure of \texttt{lpi\_ineq} options for enforcing $\texttt{Dop}\leq0$\\
 \texttt{sos\_opts} & Structure of \texttt{sossolve} options for solving the LPI
\end{tabular}
\caption{Fields of settings structure passed on to PIETOOLS executive files}
\label{tab:executive_settings_fields}
\end{table}

To help in declaring settings for the executive files, PIETOOLS includes several pre-defined \texttt{settings} structures, allowing LPI programs of varying complexity to be constructed. In particular, we distinguish \texttt{extreme}, \texttt{stripped}, \texttt{light}, \texttt{heavy} and \texttt{veryheavy} settings, corresponding to LPI programs of increasing complexity. A \texttt{settings} structure associated to each can be extracted by calling the function \texttt{lpisettings}, using
\begin{matlab}
\begin{verbatim}
 >> settings = lpisettings(complexity, epneg, simplify, solver);
\end{verbatim}
\end{matlab}
This function takes the following arguments:
\begin{itemize}
    \item \texttt{complexity}: A \texttt{char} object specifying the complexity for the settings. Can be one of `\texttt{extreme}', `\texttt{stripped}', `\texttt{light}', `\texttt{heavy}', `\texttt{veryheavy}' or `\texttt{custom}'.

    \item epneg: (optional) Positive scalar $\epsilon$ indicating how strict the negativity condition $Q\precceq \epsilon\|\mcl{T}\|^2$ would need to be in e.g. the LPI for stability. Defaults to 0, enforcing $Q\precceq 0$.

    \item simplify: (optional) A \texttt{char} object set to `\texttt{psimplify}' if the user wishes to simplify the SDP produced in the executive before solving it, or set to `\texttt{}' if not. Defaults to `\texttt{}'.

    \item solver: (optional) A \texttt{char} object specifying which solver to use to solve the SDP in the executive. Options include `\texttt{sedumi}' (default), `\texttt{mosek}', `\texttt{sdpt3}', and `\texttt{sdpnalplus}'. Note that these solvers must be separately installed in order to use them.
\end{itemize}
Note that, using higher-complexity settings, the number of decision variables in the optimization problem will be greater. This offers more freedom in solving the optimization program, thereby allowing for (but not guaranteeing) more accurate results, but also (substantially) increasing the computational effort. We therefore recommend initially trying to solve with e.g. \texttt{stripped} or \texttt{light} settings, and only using heavier settings if the executive fails to solve the problem. Note also that PIETOOLS includes a \texttt{custom} settings file, which can be used to declare custom settings for the executives. 

Once settings have been specified, the desired LPI can be declared and solved for a PIE represented by a structure \texttt{PIE} by simply calling the corresponding \texttt{executive} file, solving e.g. the $H_\infty$-optimal estimator LPI~\eqref{eq:LPIs:LPI_example} by calling
\begin{matlab}
\begin{verbatim}
 >> settings = lpisettings('light');
 >> [prog, Lop, gam, P, Z] = PIETOOLS_Hinf_estimator(PIE, settings);
\end{verbatim}
\end{matlab}
Alternatively, this executive can also be called using the function \texttt{lpiscript} as
\begin{matlab}
\begin{verbatim}
 >> [prog, Lop, gam, P, Z] = lpiscript(PIE,'hinf-observer','light');
\end{verbatim}
\end{matlab}
If successful, this returns the program structure \texttt{prog} associated to the solved problem, as well as an \texttt{opvar} object \texttt{Lop} and scalar \texttt{gam} xcorresponding to the operator $\mcl{L}$ in the estimator~\eqref{eq:LPIs:PIE_example} and associated estimation error gain $\gamma$, respectively. The function also returns \texttt{dopvar} objects \texttt{P} and \texttt{Z}, corresponding to the unsolved PI operator decision variables in the LPI.

\subsection{Executive Functions Available in PIETOOLS}

In addition to the $H_{\infty}$-optimal estimation LPI, PIETOOLS includes several other executive files to run standard LPI programming tests for a provided PIE. For example, stability of the PIE defined by \texttt{PIE} when $w=0$ can be tested by calling
\begin{matlab}
\begin{verbatim}
 >> [prog] = lpiscript(PIE,'stability','light');
\end{verbatim}
\end{matlab}
returning the optimization program structure \texttt{prog} associated to the solved program, and displaying a message of whether the system was found to be stable or not in the command window.

Table~\ref{tab:executive_functions} lists the different executive functions that have already been implemented in PIETOOLS. For each executive, a brief description of its purpose is provided, along with a mathematical description of the LPI that is solved. Each executive can be called for a \texttt{PIE} structure with fields
\begin{matlab}
\begin{verbatim}
     dim: N;
    vars: [N×2 polynomial];
     dom: [N×2 double];
     
       T:  [nx × nx opvar];     Tw: [nx × nw opvar];     Tu: [nx × nu opvar]; 
       A:  [nx × nx opvar];     B1: [nx × nw opvar];     B2: [nx × nu opvar]; 
       C1: [nz × nx opvar];    D11: [nz × nw opvar];    D12: [nz × nu opvar]; 
       C2: [ny × nx opvar];    D21: [ny × nw opvar];    D22: [ny × nu opvar]; 
\end{verbatim}
\end{matlab}
representing a PIE of the form
\begin{align*}
    \mcl{T}_u\dot{u}(t)+\mcl{T}_w\dot{w}+\mcl{T}\dot{\mbf{x}}_{\text{f}}(t)&=\mcl{A}\mbf{x}_{\text{f}}(t)+\mcl{B}_1 w(t)+\mcl{B}_2 u(t),   \nonumber\\
    z(t)&=\mcl{C}_1\mbf{x}_{\text{f}}(t) + \mcl{D}_{11}w(t) + \mcl{D}_{12}u(t), \nonumber\\
    y(t)&=\mcl{C}_2\mbf{x}_{\text{f}}(t) + \mcl{D}_{21}w(t) + \mcl{D}_{22}u(t).
\end{align*}
For more information on the origin and application of each LPI, see the references provided in the table, as well as Chapter~\ref{ch:LPI_examples}.

\begin{table}[ht]
\hspace*{-1.0cm}
\setlength{\abovedisplayskip}{-6pt}
\setlength{\belowdisplayskip}{-2pt}
\begin{tabular}{m{6.5cm} | m{10cm}}
    \textbf{Problem} & \textbf{LPI} \\
    \hline\hline
    \multicolumn{2}{l}{\texttt{[prog, P] = lpiscript(PIE,'stability',settings)
    }} \\\hline
    Test stability of the PIE for $w=0$ and $u=0$, by verifying feasibility of the primal LPI~\cite{shivakumar2022dualPIE}.  &
    {\begin{flalign*} 
    &   \mcl{T}^*\mcl{Q}=\mcl{Q}^*\mcl{T}\succcurlyeq \alpha\mcl{T}^*\mcl{T} & \\
	& \mcl{Q}^*\mcl{A}+\mcl{A}^*\mcl{Q}\preccurlyeq -\beta\mcl{P}&
    \end{flalign*} } \\
    \hline\hline
    \multicolumn{2}{l}{\texttt{[prog, P] = lpiscript(PIE,'stability-dual',settings)
    }} \\\hline
    Test stability of the PIE for $w=0$ and $u=0$ by verifying feasibility of the dual LPI~\cite{shivakumar2022dualPIE}.         & 
    {\begin{flalign*} 
    &   \mcl{T}\mcl{Q}=\mcl{Q}^*\mcl{T}^*\succcurlyeq \alpha\mcl{T}\mcl{T}^* & \\
	& \mcl{Q}^*\mcl{A}^*+\mcl{A}\mcl{Q}\preccurlyeq -\beta\mcl{P}&
    \end{flalign*} }\\
    \hline\hline
    \multicolumn{2}{l}{\texttt{[prog, P, gam] = lpiscript(PIE,'l2gain',settings)
    }} \\\hline
    Determine an upper bound $\gamma$ on the $\mcl{H}_{\infty}$-norm $\sup_{w,z\in L_2}\frac{\|z\|_{L_2}}{\|w\|_{L_2}}$ of the PIE for $u=0$, by solving the primal LPI~\cite{shivakumar_2019CDC}. &
    {\begin{flalign*}
    &\min\limits_{\gamma,\mcl{P}} ~~\gamma&\\
	&\mcl{P}\succ0&\\
	&\bmat{-\gamma I & \mcl{D}_{11}^*&\mcl{B}_1^*\mcl{PT}\\(\cdot)^*&-\gamma I&\mcl{C}_1\\(\cdot)^*&(\cdot)^*&\mcl{T}^*\mcl{P}\mcl{A}+\mcl{A}^*\mcl{P}\mcl{T}}\preccurlyeq 0&\end{flalign*}}\\
    \hline\hline    
    \multicolumn{2}{l}{\texttt{[prog, P, gam] = lpiscript(PIE,'l2gain-dual',settings)
    }} \\\hline
    Determine an upper bound $\gamma$ on the $\mcl{H}_{\infty}$-norm $\sup_{w,z\in L_2}\frac{\|z\|_{L_2}}{\|w\|_{L_2}}$ of the PIE for $u=0$, by solving the dual LPI~\cite{shivakumar2022dual}. &
    {\begin{flalign*}
	&\min\limits_{\gamma,\mcl{P}} ~~\gamma&\\
	&\mcl{P}\succ0&\\
	&\bmat{-\gamma I & \mcl{D}_{11}&\mcl{T}\mcl{P}\mcl{C}_1\\(\cdot)^*&-\gamma I&\mcl{B}_1^*\\(\cdot)^*&(\cdot)^*&\mcl{T}\mcl{P}\mcl{A}^*+\mcl{A}\mcl{P}\mcl{T}^*}\preccurlyeq 0&\end{flalign*}}\\
    \hline\hline  
    \multicolumn{2}{l}{\texttt{[prog, L, gam, P, Z] = lpiscript(PIE,'hinf-observer',settings)
    }} \\\hline
    Establish an $\mcl{H}_{\infty}$-optimal observer $\mcl{T}\dot{\hat{\mbf{x}}}_{\text{f}}=A\hat{\mbf{x}}_{\text{f}}+\mcl{L}(\mcl{C}_1\hat{\mbf{x}}_{\text{f}}-y)$ for the PIE with $u=0$ by solving the LPI and returning $\mcl{L}=\mcl{P}^{-1}\mcl{Z}$~\cite{das_2019CDC}.  &
    {\begin{flalign*}
    &\min\limits_{\gamma,\mcl{P},\mcl{Z}} ~~\gamma&\\
	&\mcl{P}\succ0&\\
	&\bmat{-\gamma I & \mcl{D}_{11}^*&(\mcl{P}\mcl{B}_1+\mcl{Z}\mcl{D}_{21})^*\mcl{T}\\(\cdot)^*&-\gamma I&\mcl{C}_1\\(\cdot)^*&(\cdot)^*& (\cdot)^*+(\mcl{P}\mcl{A}+\mcl{Z}\mcl{C}_{2})^*\mcl{P}\mcl{T}}\preccurlyeq 0&\end{flalign*}}\\
    \hline\hline  
    \multicolumn{2}{l}{\texttt{[prog, K, gam, P, Z] = lpiscript(PIE,'hinf-controller',settings)
    }} \\\hline
    Establish an $\mcl{H}_{\infty}$-optimal controller $u=\mcl{K}\mbf{x}_{\text{f}}$ for the PIE by solving the LPI and returning $\mcl{K}=\mcl{Z}\mcl{P}^{-1}$~\cite{shivakumar2022dual}.  &
    {\begin{flalign*}
	&\min\limits_{\gamma,\mcl{P},\mcl{Z}} ~~\gamma&\\
	&\mcl{P}\succ0&\\
	&\bmat{-\gamma I & \mcl{D}_{11}&\mcl{T}(\mcl{P}\mcl{C}_1+\mcl{Z}\mcl{D}_{12})\\(\cdot)^*&-\gamma I&\mcl{B}_1^*\\(\cdot)^*&(\cdot)^*& (\cdot)^*+(\mcl{A}\mcl{P}+\mcl{B}_2\mcl{Z})\mcl{T}^*}\preccurlyeq 0&\end{flalign*}}\\
    \hline\hline  
\end{tabular}
\caption{List of pre-defined executives for analysis and control of PIEs. See also Chapter~\ref{ch:LPI_examples}.}
\label{tab:executive_functions}
\end{table}


\part{Additional PIETOOLS Functionality}

\chapter{Input Formats for ODE-PDE Systems}\label{ch:alt_PDE_input}

In PIETOOLS, there are three main methods for declaring coupled ODE-PDE systems: The Command Line Input format (via command line or MATLAB scripts), the graphical user interface (a MATLAB app), and using the \texttt{sys} structure. The first two of these methods have been presented in Chapter~\ref{ch:PDE_DDE_representation}, and are the recommended input formats. The last of these methods, using the \texttt{sys} structure, was introduced in PIETOOLS 2022 as a Command Line Parser format, and functions similarly to the current Command Line Input format. Although the \texttt{sys} structure functionality is still available in PIETOOLS 2025, it supports at most 1D ODE-PDE systems, and it is therefore recommended to use \texttt{pde\_var} instead. 

In this chapter, we provide a bit more background on each of the three input formats for ODE-PDE systems. In particular, in Section~\ref{sec:GUI}, we show how any well-posed, linear, coupled 1D ODE-PDE system can be declared in PIETOOLS using the GUI. In Section~\ref{sec:alt_PDE_input:terms_input_PDE}, we then shift focus to the Command Line Input format, providing a detailed description of how general ODE-PDE systems can be declared using this format, and how this format actually works ``behind the scenes''. Finally, in Section~\ref{sec:alt_PDE_input:sys}, we show how the \texttt{sys} structure can be used to declare 1D ODE-PDE models as well.

\section{A GUI for Defining PDEs}\label{sec:GUI}

In addition to the Command Line Input format, PIETOOLS 2025 also allows PDEs to be delcared using a graphical user interface (GUI), that provides a simple, intuitive and interactive visual interface to directly input the model. It also allows declared PDE models to be saved and loaded, so that the same system can be used in different sessions without having to declare the model from scratch each time.

To open the GUI, simply call \texttt{PIETOOLS\_PDE\_GUI} from the command line. You will see something similar to the picture below:
\begin{figure}[H]
	\centering
	\includegraphics[width=0.95\textwidth]{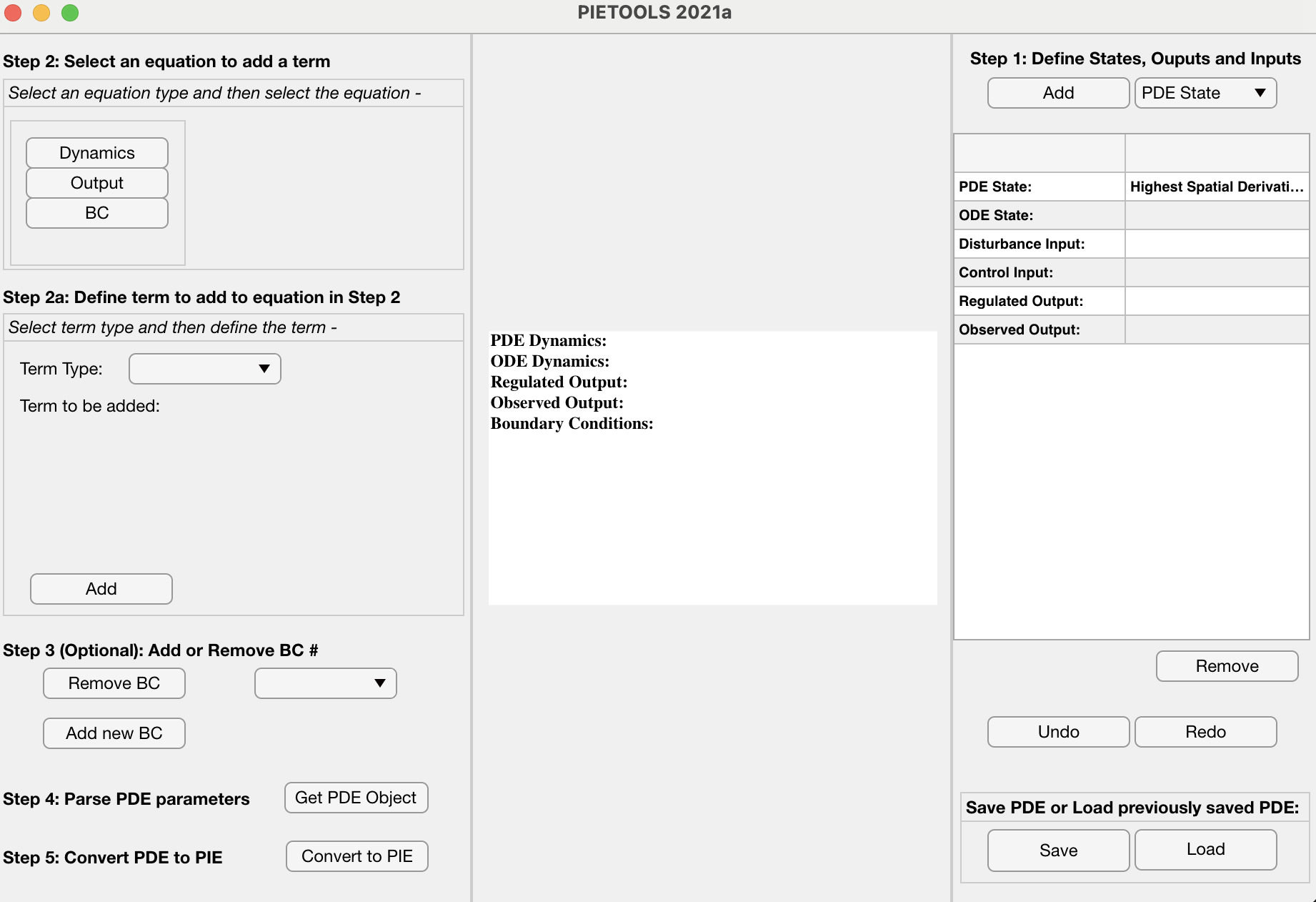} 
	\caption{GUI overview.}
	\label{gui1}
\end{figure}

Now we will go over the GUI step-by-step to demonstrate how to define your own linear, 1D ODE-PDE model in PIETOOLS. 

\subsection{Step 1: Define States, Outputs and Inputs}
First, we start with the right side of the screen as follows:
	\begin{figure}[H]
	\centering
	\includegraphics[width=0.95\textwidth]{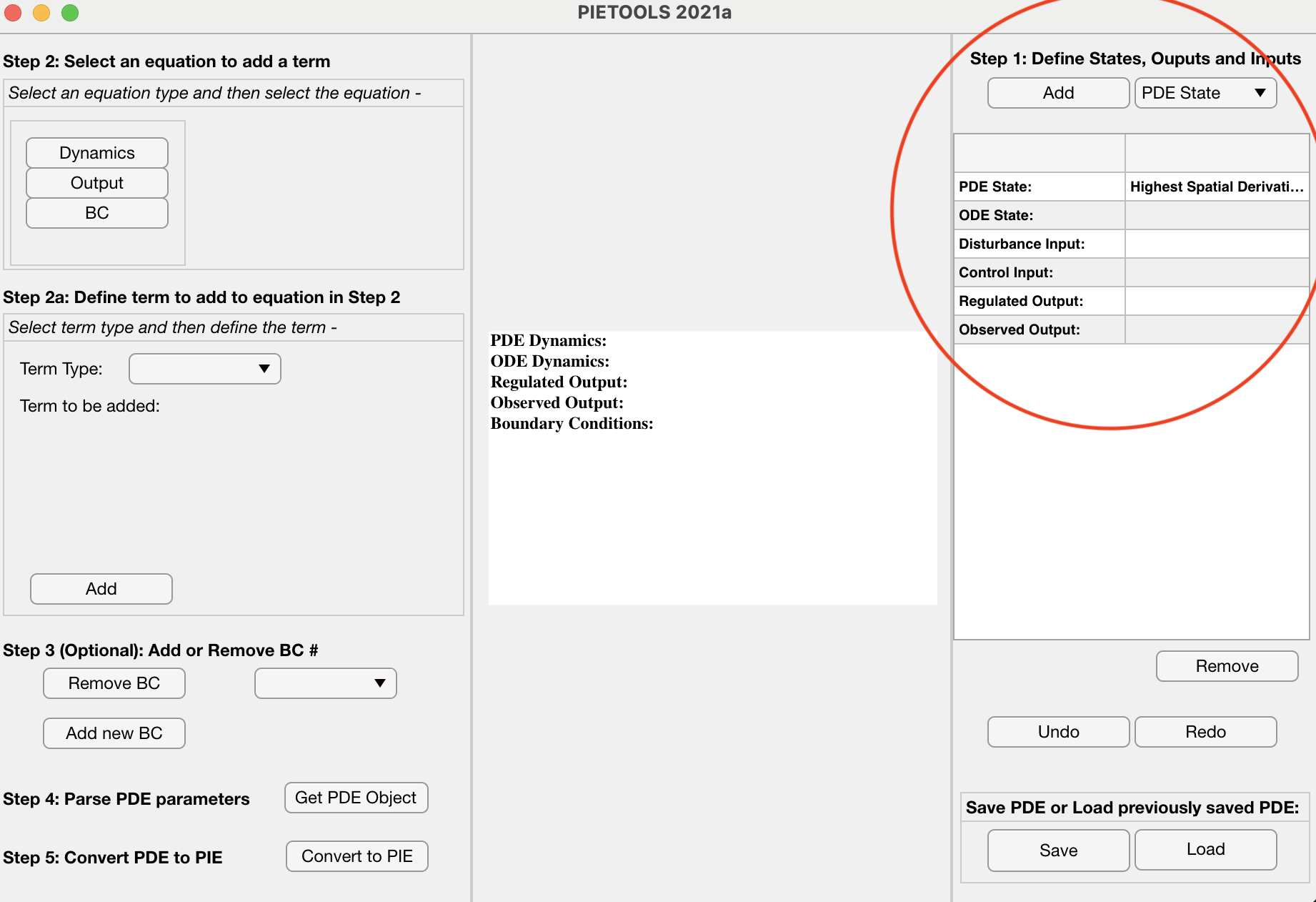} 
	\caption{Step 1: Define States, Outputs and Inputs}
	\label{gui2}
\end{figure}

\begin{enumerate}
    \item The drop-down menu \texttt{PDE State} provides a list of all the possible variables to be defined on your model. Clicking on the \texttt{PDE State} menu reveals the list
    
    	\begin{figure}[H]
	\centering
	\includegraphics[width=0.30\textwidth]{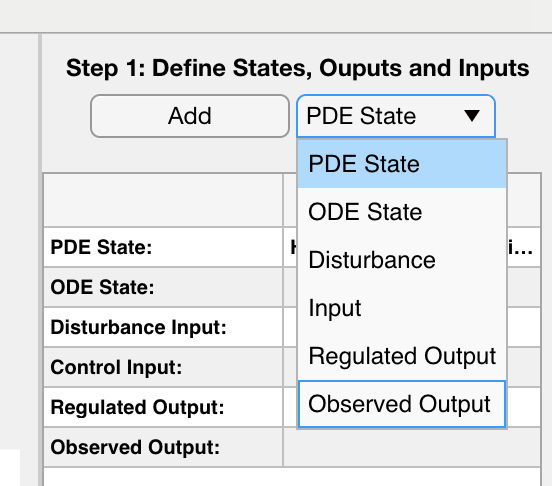} 
	\caption{Adding variables your model}
	\label{gui3}
\end{figure}

\item After selecting your intended variable, you can add it by clicking on the \texttt{Add} button.
 Note that ehen you select \texttt{PDE State} from the drop-down menu and attempt to a PDE state, you also have to specify the highest order of spatial derivative that the particular state admits.

	\begin{figure}[H]
	\centering
	\includegraphics[width=0.30\textwidth]{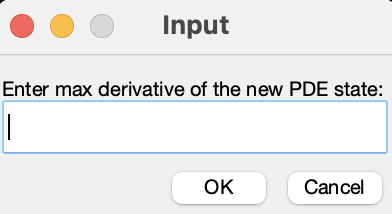} 
	\caption{Enter the highest order of derivative the particular state admits}
	\label{gui4}
\end{figure}

\item Once the variables are added, they automatically get displayed in the display panel in the middle. Since no dynamics have been specified for the model so far, all the variables are set to the default setting temporarily. 
\begin{figure}[H]
	\centering
	\includegraphics[width=0.80\textwidth]{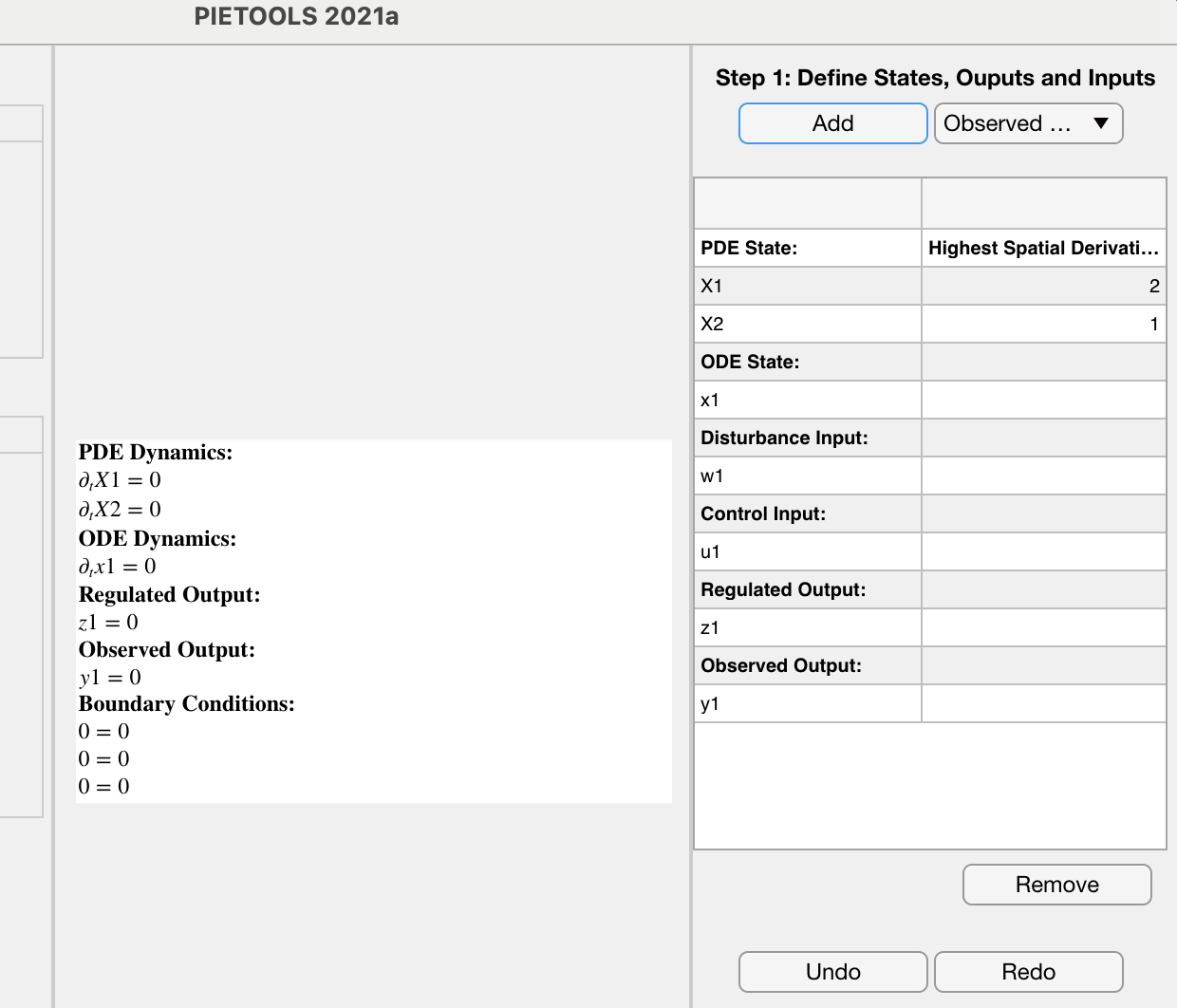} 
	\caption{After adding the variables to the model}
	\label{gui5}
\end{figure}

\item At the bottom there are options of \texttt{Remove}, \texttt{Undo}, \texttt{Redo} to delete or recover variables. 
\end{enumerate}

\subsection{Step 2: Select an Equation to Add a Term}
Now we specify the dynamics and the terms corresponding to each of the variables defined in Step 1. This is located on the left hand side of the GUI.

    	\begin{figure}[H]
	\centering
	\includegraphics[width=0.80\textwidth]{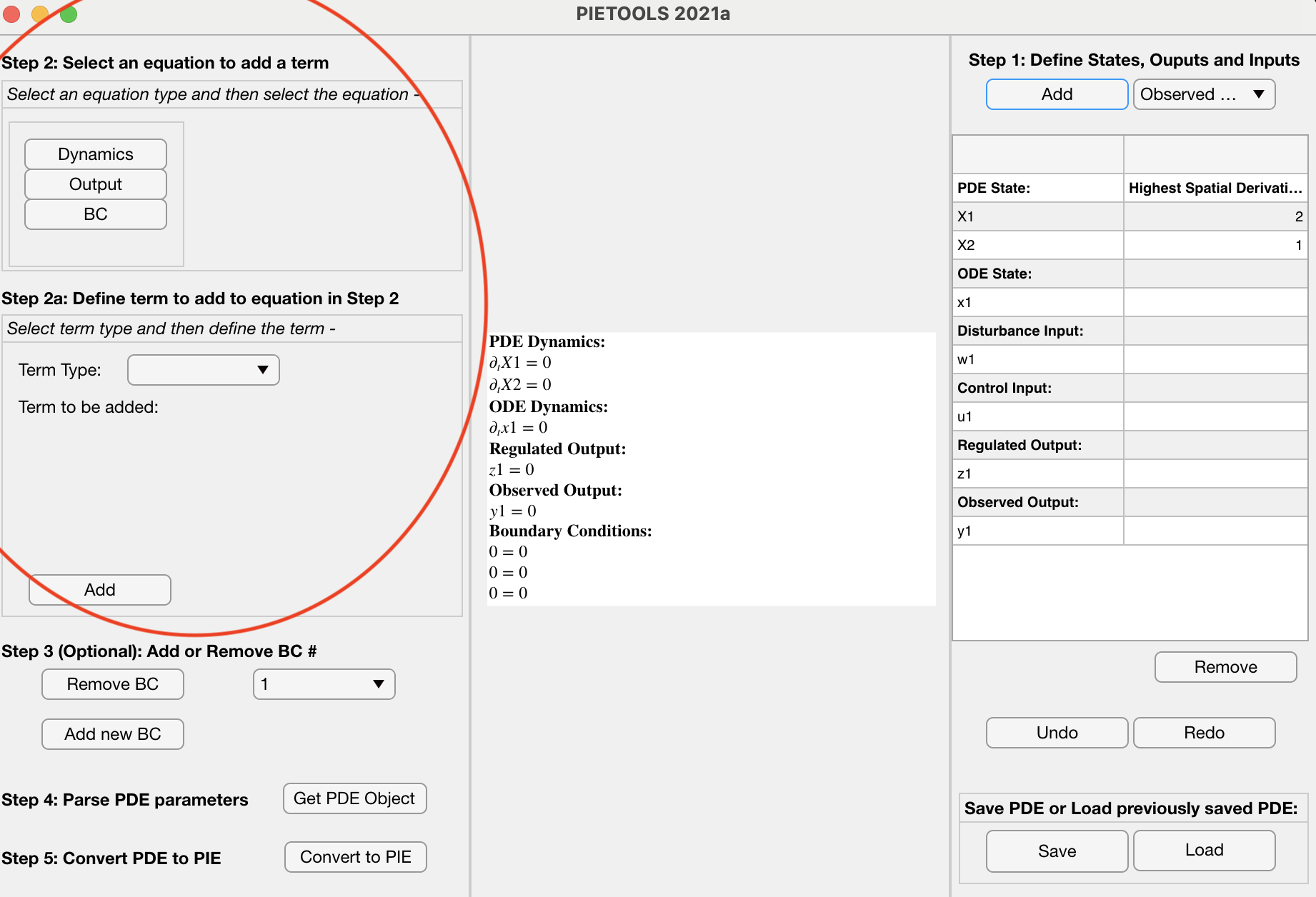} 
	\caption{Step 2: Select an Equation to Add a Term}
	\label{gui2_1}
\end{figure}
This has two parts. On the top, we have a panel for \texttt{Select an Equation type and then select the equation-} to choose which part of the model to be defined. Below that panel, there is another panel that has to be used to \texttt{Select term type and then define the term-}.

\begin{enumerate}
    \item In the panel titled \texttt{Select an Equation type and then select the equation-}, select either \texttt{Dynamics}, \texttt{Output} or \texttt{BC} (i.e. Boundary Conditions). 
    
\item If you select \texttt{Dynamics}, all the PDE and ODE states that you specified in Step 1 appear. 
    	\begin{figure}[H]
	\centering
	\includegraphics[width=0.30\textwidth]{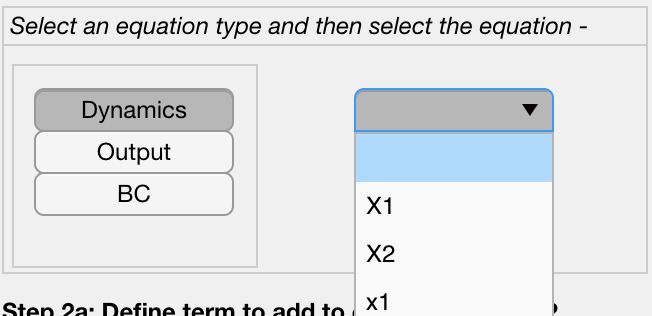} 
	\label{gui2_2}
\end{figure}

\item Once you select the desired state for which to add terms to the dynamics, go down to Step 2.a \texttt{Select term type and then define the term-}.
	\begin{figure}[H]
	\centering
	\includegraphics[width=0.30\textwidth]{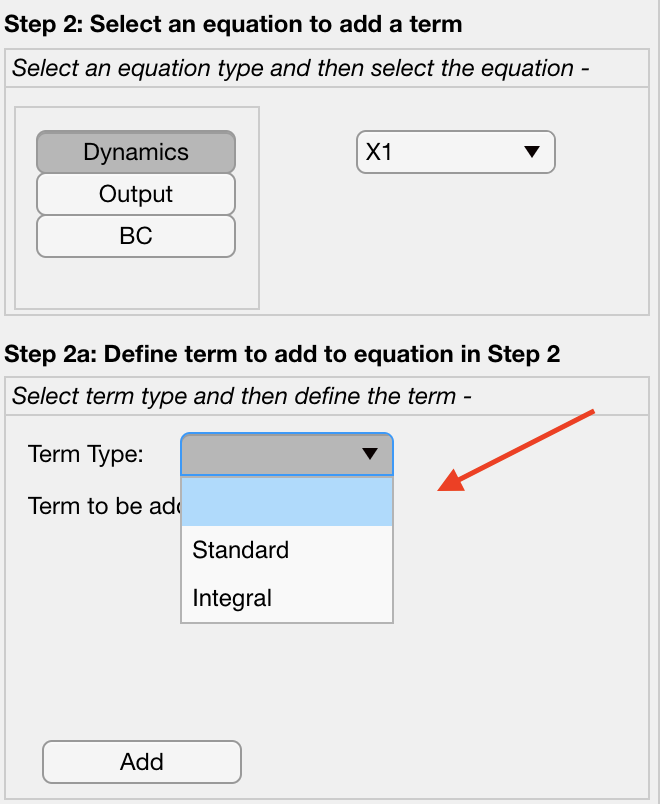} 
	\label{gui2_3}
\end{figure}

\item For a term involving an individual PDE state, you may have two kinds of terms, a \texttt{Standard} term, pre-multiplying the state with some coefficients, or a \texttt{Integral} term, taking some (partial) integral of the state. The \texttt{Standard} option can be used to define both terms involving states and terms involving inputs. On the other hand, the \texttt{Integral} option is only available for the PDE states. 
	\begin{figure}[H]
	\centering
	\includegraphics[width=0.30\textwidth]{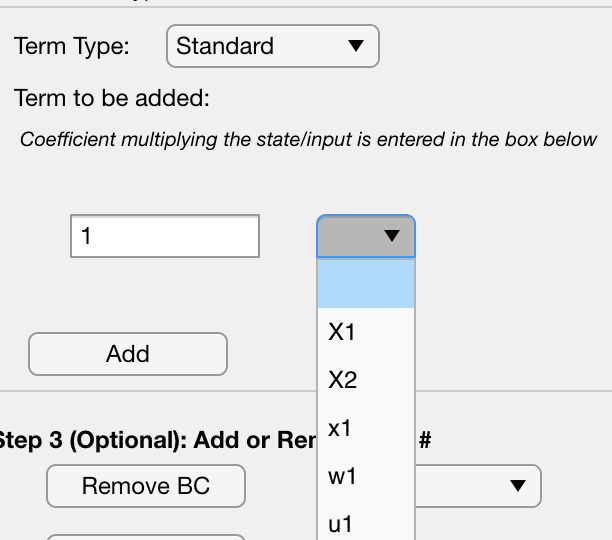}
	\includegraphics[width=0.35\textwidth]{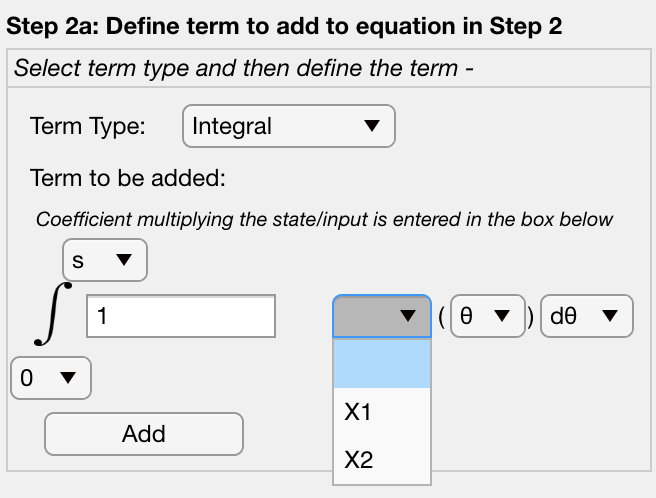}
\end{figure}

\begin{itemize}
\item Now in \texttt{Standard} option, one has to select the variable and add the coefficient in the adjacent panel. Moreover, the PDE states may also contain its derivatives. If you select a PDE state, you can input the order of derivative (from $0$ up to the highest order derivative for that state), the independent variable with respect to which the function is defined (it is $s$ for in-domain, $0, 1$ for boundary), and the corresponding coefficient terms. Then, by clicking on the \texttt{Add} button, we can add that term to the model and it gets shown in the display panel. 

	\begin{figure}[H]
	\centering
	\includegraphics[width=0.70\textwidth]{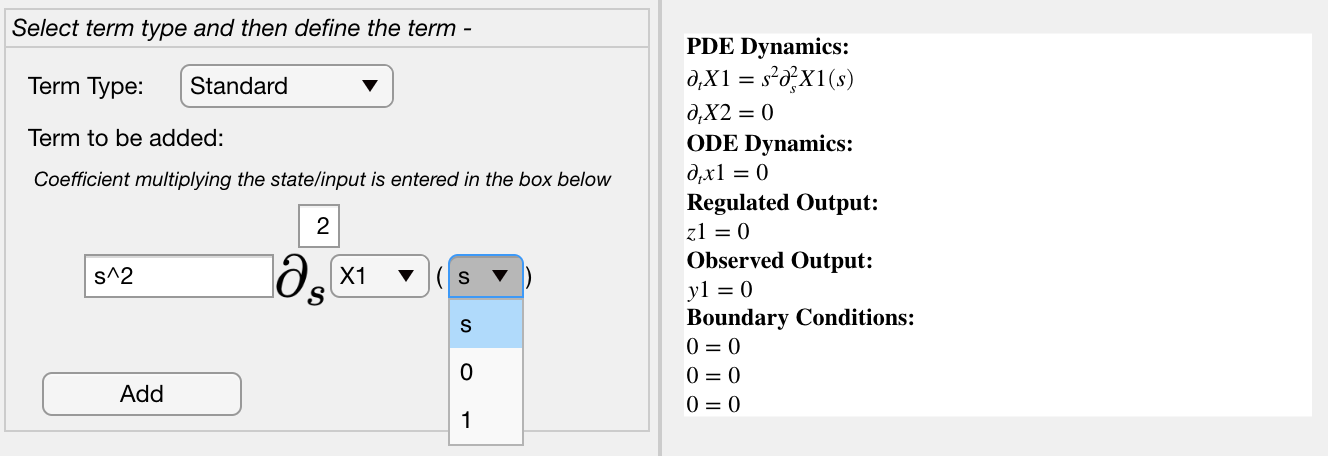}
\end{figure}

\begin{boxEnv}{\textbf{Note:}}
Only the PDE states can be a function of `$s$'. For other terms, the option of adding $'s'$ as an independent variable is not available.
\end{boxEnv} 

The order of derivative can not exceed the highest order derivative for that state. If the input value exceeds that, the following error will be displayed

	\begin{figure}[H]
	\centering
	\includegraphics[width=0.30\textwidth]{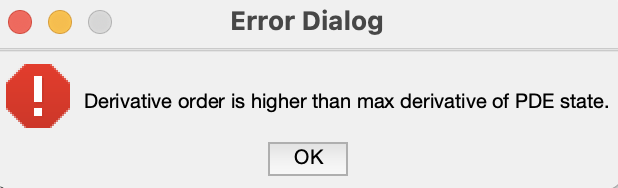}
\end{figure}

\item One can also add an integral term by selecting \texttt{Integral}. Here, identical to the \texttt{Standard} option, we can define the order or derivative, the coefficients and the limits of integral which can be $0$ or $s$ for lower limit and $s$ or $1$ for upper limit. The functions are always with respect to $\theta$. 

	\begin{figure}[H]
	\centering
	\includegraphics[width=0.80\textwidth]{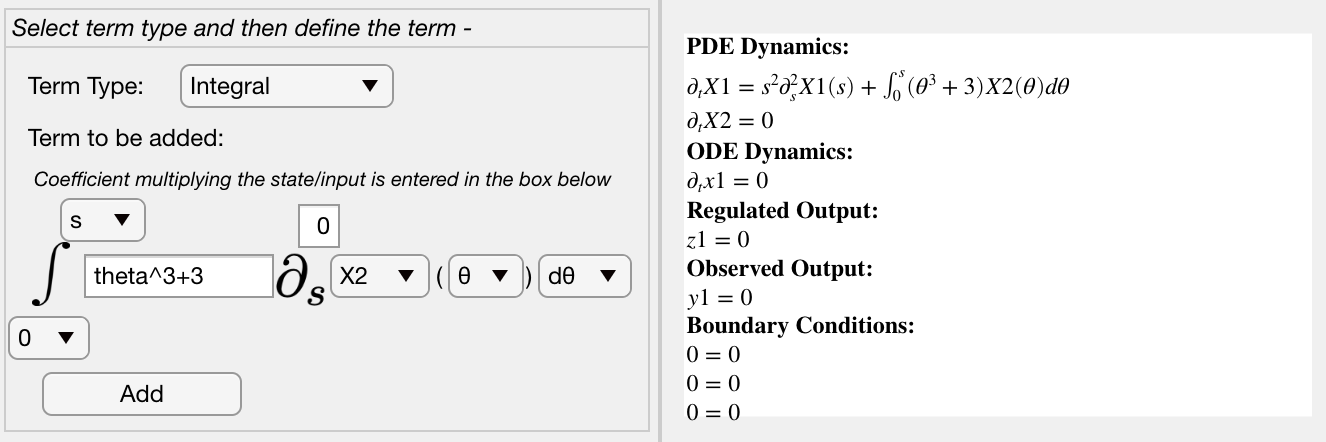}
\end{figure}

	\begin{figure}[H]
	\centering
	\includegraphics[width=0.80\textwidth]{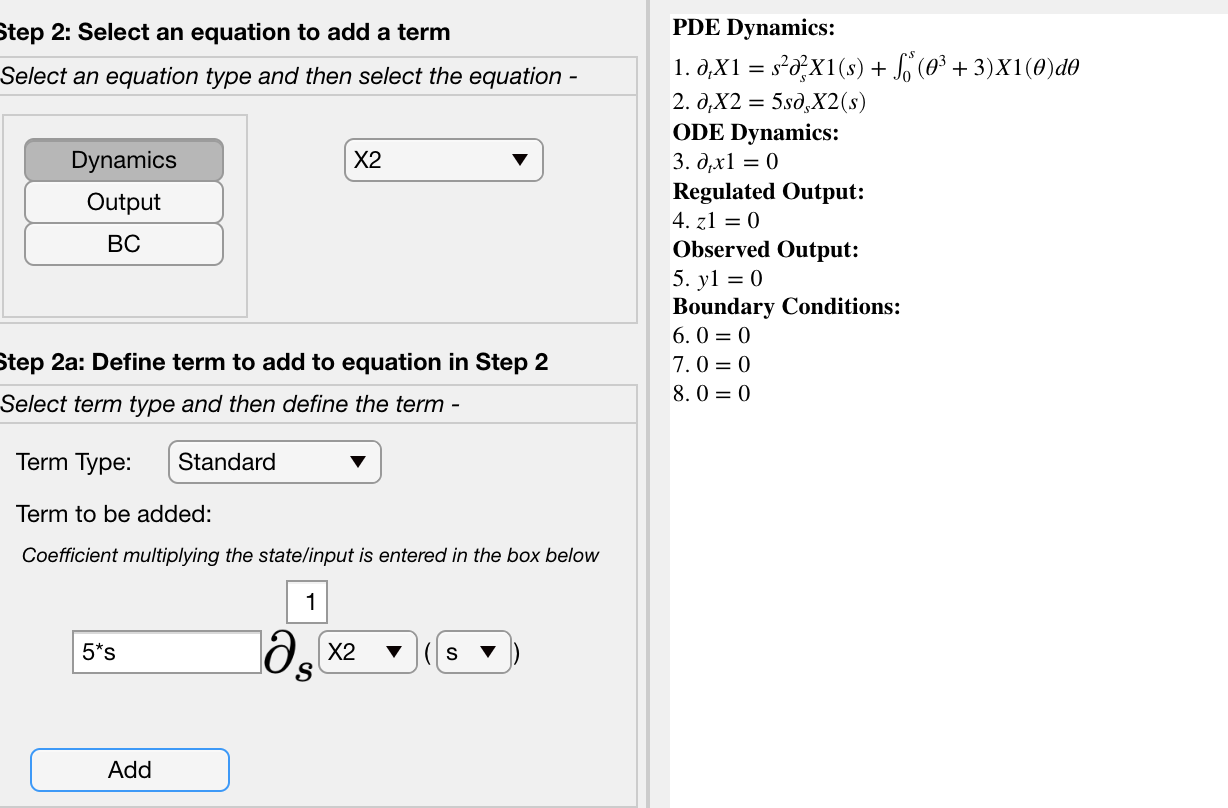}
\end{figure}

\end{itemize}

\item To define the outputs and boundary conditions, one must follow the same steps as above. 
	\begin{figure}[H]
	\centering
	\includegraphics[width=0.35\textwidth]{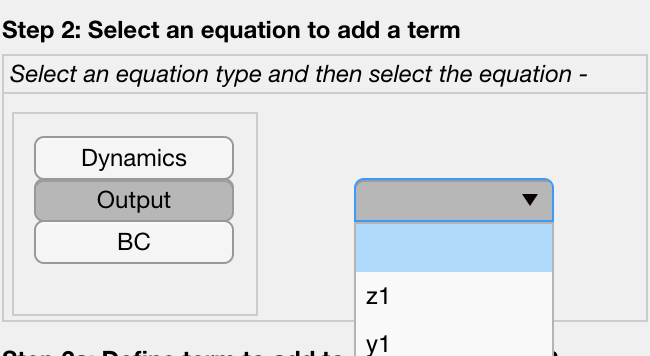}
	\includegraphics[width=0.35\textwidth]{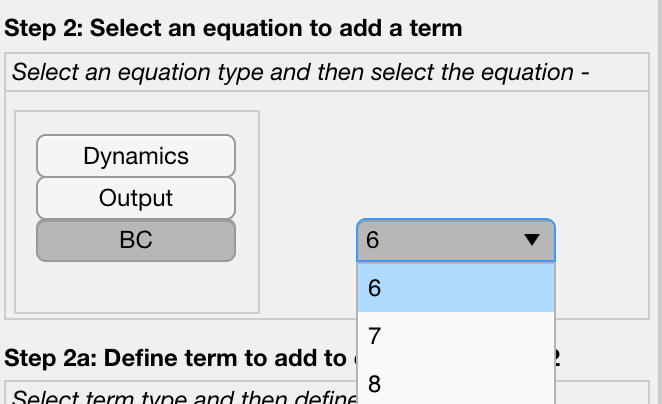}
\end{figure}
\end{enumerate}

\begin{boxEnv}{Additional Remarks:}
A) The integral term of PDE states can only be a function of $'\theta'$.

\noindent B) Terms can be specified and added only for one variable at a time. Once a desired variable and one of the options (\texttt{Dynamics}, \texttt{Output}, \texttt{BC}) has been selected at the top, the term can be added following the instructions at the bottom (Step 2a). In order to select another variable for which to add a term, the above steps must be repeated.
\end{boxEnv}

After adding all the desired terms for dynamics, outputs and boundary conditions, the complete description of an example model looks something like below:
	\begin{figure}[H]
	\centering
	\includegraphics[width=0.40\textwidth]{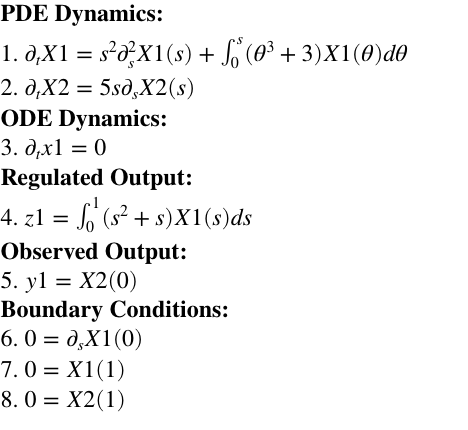}
	\caption{An example of a complete model as displayed in the GUI}
\end{figure}

\subsection{Step 3: (Optional) Add or Remove BC}
If desired, the user can also add a new boundary condition or remove one. Note that for a PDE state variable differentiable up to order $N$, a well-posed PDE must impose exactly $N$ boundary conditions on this state variable. 
	\begin{figure}[H]
	\centering
	\includegraphics[width=0.50\textwidth]{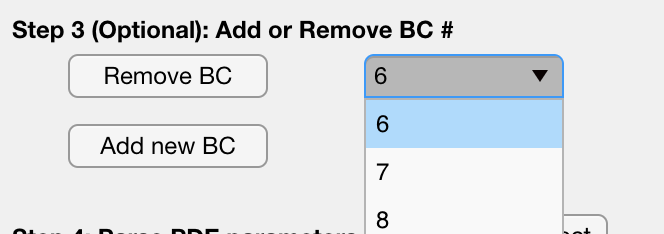}
\end{figure}

\subsection{Step 4-5: Parse PDE Parmeters and Convert Them to PIE}
\begin{enumerate}
    \item Once the desired model has been declared, you can extract all the parameters defining this model by pressing \texttt{Get PDE Object}, storing these parameters in an object called \texttt{PDE\_GUI} which directly gets loaded into the MATLAB workspace. 
    
   \item In addition,  pressing \texttt{convert to PIE}, you can convert your model to a PIE and store it in an object called \texttt{PIE\_GUI} which directly gets loaded into the MATLAB workspace. 
\end{enumerate}

	\begin{figure}[H]
	\centering
	\includegraphics[width=0.45\textwidth]{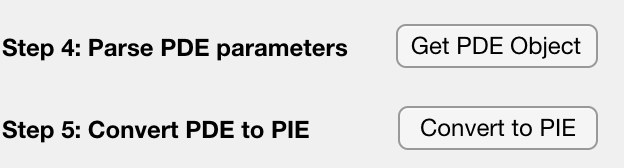}
\end{figure}

\section{The Command Line Input Format}\label{sec:alt_PDE_input:terms_input_PDE}

In PIETOOLS 2025, the easiest format for declaring ODE-PDE systems is using the Command Line Input format presented in Chapter~\ref{ch:PDE_DDE_representation}. Using this format, \texttt{pde\_struct} class objects can be declared using the \texttt{pde\_var} function, and subsequently manipulated using e.g. algebraic operations to declare a broad class of linear ODE-PDEs. In this section, we show how the \texttt{pde\_struct} object actually stores the information necessary to represent such systems, and how we can use this structure to define a broad class of linear systems. To illustrate, we will use the following example of a wave equation throughout this section.

\begin{Statebox}{\textbf{Example}}
\vspace*{-0.5cm}
\begin{align}\label{eq:alt_PDE_input:terms_input_PDE:example_PDE}
    \partial_{t}\mbf{x}_{1}(t,s_{1},s_{2})&=\mbf{x}_{2}(t,s_{1},s_{2}),  &  t\geq 0,&\notag\\
    \partial_{t}\mbf{x}_{2}(t,s_{1},s_{2})&=5(\partial_{s_{1}}^2\mbf{x}_{1}(t,s_{1},s_{2})+\partial_{s_{2}}^2\mbf{x}_{1}(t,s_{1},s_{2}))+(3-s_{1})s_{2}u_{1}(t),    &   s_{1}\in[0,3],&    \notag\\
    \mbf{y}_{1}(t,s_{1})&=\mbf{x}_{1}(t,s_{1},1)+s_{1} w_{1}(t),   & s_{2}\in[-1,1],&  \notag\\
    \mbf{y}_{2}(t,s_{2})&=\mbf{x}_{2}(t,3,s_{2})+w_{2}(t),  \notag\\
    z(t)&=\bmat{10\int_{0}^{3}\int_{-1}^{1}\mbf{x}_{1}(t,s_{1},s_{2}) ds_{2} ds_{1}\\\partial_{s_{1}}\partial_{s_{2}}\mbf{x}_{2}(t,3,1)},    \notag\\
    \mbf{x}_{1}&(t,s_{1},-1)=\partial_{s_{2}}\mbf{x}_{1}(t,s_{1},1)=0,  \notag\\
    \mbf{x}_{1}&(t,0,s_{2})=\mbf{u}_{2}(t-0.5,s_{2}),\qquad \partial_{s_{1}}\mbf{x}_{1}(t,3,s_{2})=0,   \notag\\
    \mbf{x}_{2}&(t,s_{1},-1)=\partial_{s_{2}}\mbf{x}_{2}(t,s_{1},1)=0,  \notag\\
    \mbf{x}_{2}&(t,0,s_{2})=\partial_{s_{1}}\mbf{x}_{1}(t,3,s_{2})=0.
\end{align}
\end{Statebox}

\subsection{Representing State, Input, and Output Variables}

In Chapter~\ref{ch:PDE_DDE_representation}, we showed that a state, input, or output variable can be declared in PIETOOLS 2025 using the function \texttt{pde\_var}. A general call to this function takes the form
\begin{matlab}
\begin{verbatim}
 >> obj = pde_var(type,size,vars,dom,diff);
\end{verbatim}
\end{matlab}
and may involve up to five inputs:
\begin{itemize}
    \item \texttt{type}: A character array specifying the desired type of variable. Must be set to \texttt{'state'} to declare an ODE or PDE state variable, $x(t)$, \texttt{'input'} (or \texttt{'in'}) to declare an exogenous input variable, $w(t)$, \texttt{'control'} to declare an actuator input variable, $u(t)$, \texttt{'output'} (or \texttt{'out'}) to declare a regulated output variable, $z(t)$, and \texttt{'sense'} the declared a sensed output variable, $y(t)$. Defaults to \texttt{'state'};

    \item \texttt{size}: An integer specifying the size of the object in case it is vector-valued. Defaults to \texttt{1};

    \item \texttt{vars}: A $p\times 1$ array of type \texttt{'polynomial'}, specifying the spatial variables on which the object depends. Defaults to an empty array \texttt{[]}, indicating that the object does not vary in space;

    \item \texttt{dom}: A $p\times 2$ numeric array specifying for each of the spatial variables the interval on which it is defined, with the first column specifying the lower limit of the interval, and the second column the upper limit. Defaults to an array with the same number of rows as \texttt{vars}, with the first column all zeros, and the second column all ones, indicating that all spatial variables exist on $[0,1]$;

    \item \texttt{diff}: A $p\times 1$ numeric array, specifying the order of differentiability of the object with respect to each of the spatial variables on which it depends. This can only be declared for PDE variables, i.e. \texttt{'state'} type objects, and is optional.
\end{itemize}
Note that, in general, the fifth input can be omitted, as PIETOOLS can usually infer the order of differentiability of PDE states from the declared PDE dynamics.
In addition, most of the remaining inputs admit a default value that is used when no value is specified, \textbf{so long as the arguments are passed in the correct order}. For example, a scalar-valued PDE state $\mbf{x}(t,s_{1},s_{2})$ on $s_{1},s_{2}\in[0,1]$ can be declared as
\begin{matlab}
\begin{verbatim}
 >> x = pde_var('state',1,[s1;s2],[0,1;0,1]);
\end{verbatim}
\end{matlab}
but can also be declared as e.g.
\begin{matlab}
\begin{verbatim}
 >> x = pde_var(1,[s1;s2]);
\end{verbatim}
\end{matlab}
or even just
\begin{matlab}
\begin{verbatim}
 >> x = pde_var([s1;s2]);
\end{verbatim}
\end{matlab}
However, we can simplify this call any further, as the spatial variables on which a PDE state depends \textbf{must always be specified}.

Calling \texttt{obj=pde\_var(...)}, the output object \texttt{obj} is a \texttt{pde\_struct} object representing the desired state, input, or output variable. To represent these variables, a \texttt{pde\_struct} object \texttt{obj} has the following fields:

\begin{Statebox}{\texttt{pde\_struct}}
\begin{flalign*}
 &\texttt{obj.x}	&	&\text{a cell with each element $i$ specifying a state component $\mbf{x}_i$ in the system;}	\\
 &\texttt{obj.w}	&	&\text{a cell with each element $i$ specifying an exogenous input $\mbf{w}_{i}$;}	\\
 &\texttt{obj.u}	&	&\text{a cell with each element $i$ specifying an actuator input $\mbf{u}_{i}$;}	\\
 &\texttt{obj.z}	&	&\text{a cell with each element $i$ specifying a regulated output $\mbf{z}_{i}$;}	\\
 &\texttt{obj.y}	&	&\text{a cell with each element $i$ specifying an observed output $\mbf{y}_{i}$;}	\\
 &\texttt{obj.BC}	&	&\text{a cell with each element $i$ specifying a boundary condition for the PDE.}	&	& \\	
 &\texttt{obj.free}	&	&\text{a cell with each element $i$ specifying a free PDE variable or set of terms.}
\end{flalign*}
\end{Statebox}

Depending on the specified \texttt{type} in the call to \texttt{pde\_var}, one of these fields will be populated with a single element representing the desired state, input, or output variable. To this end, each of the cell elements of the fields will again be a structure, with the following fields:

\begin{flalign*}
&\texttt{size}	&	&\text{an integer specifying the size of the state component, input or output;} \\
&\texttt{vars}	&	&\text{a $p\times 1$ pvar (polynomial) array (for $p\leq 2$), specifying the spatial variables} \\
&   &   &\text{of the state component, input, or output;}	&	&\\
&\texttt{dom}	&	&\text{a $p\times 2$ array specifying the interval on which each spatial variable exists;}	&	&\\
&\texttt{term}	&	&\text{a cell defining the (differential) equation associated with the state component,}\\
&   &   &\text{output, or boundary condition;}	&	& \\
&\texttt{ID} &    &\text{a unique integer value distinguishing the state, input, or output component} & &\\
&      &   &\text{from all others;} \\
&\texttt{diff} &     &\text{a $1\times p$ array specifying the order of differentiability of the state variable}\\
& & &\text{with respect to each spatial variable;}
\end{flalign*}

Naturally, the values of these fields will be populated with the inputs passed to \texttt{pde\_var} (where the order of differentiability will be converted from column to row array). For example, our PDE state $\mbf{x}(t,s_{1},s_{2})$ will be represented as an object \texttt{x} with all fields empty, except the field \texttt{x.x}, which will be a $1\times 1$ cell structure with elements
\begin{matlab}
\begin{verbatim}
 >> x.x{1}
 ans = 

   struct with fields:

     size: 1
     vars: [2×1 polynomial]
      dom: [2×2 double]
       ID: 1
\end{verbatim}
\end{matlab}
where \texttt{x.x\{1\}.vars=[s1;s2]}, and \texttt{x.x\{1\}.dom=[0,1;0,1]}. In addition, a single free term representing the state variable will be stored in the field \texttt{x.free\{1\}.term\{1\}}. We will show how exactly such terms are represented in the structure in the next subsection.

\begin{Statebox}{\textbf{Example}}
Consider the wave equation example from~\eqref{eq:alt_PDE_input:terms_input_PDE:example_PDE}. The system is represented as a 2D PDE, in spatial variables $(s_{1},s_{2})\in[0,3]\times[-1,1]$. We initialize these variables as
\begin{matlab}
\begin{verbatim}
 >> pvar s1 s2
\end{verbatim}
\end{matlab}
Now, the system is defined in terms of two PDE state variables, $\mbf{x}_{1}(t),\mbf{x}_{2}(t)\in L_{2}[[0,3]\times[-1,1]$, two actuator inputs, $u_{1}(t)\in\R$ and $\mbf{u}_{2}(t)\in L_{2}[-1,1]$, two exogenous inputs, $w_{1}(t),w_{2}(t)\in \R$, two sensed outputs, $\mbf{y}_{1}(t)\in L_{2}[0,3]$ and $\mbf{y}_{2}(t)\in L_{2}[-1,1]$, and a vector-valued regulated output $z(t)\in\R^{2}$. We can declare these variables as
\begin{matlab}
\begin{verbatim}
 >> x1 = pde_var('state',1,[s1;s2],[0,3;-1,1],[2;2]);
 >> x2 = pde_var('state',1,[s1;s2],[0,3;-1,1],[2;2]);
 >> u1 = pde_var('control',1,[],[]);
 >> u2 = pde_var('control',1,s2,[-1,1]);
 >> w1 = pde_var('input',1,[],[]);
 >> w2 = pde_var('input',1,[],[]);
 >> y1 = pde_var('sense',1,s1,[0,3]);
 >> y2 = pde_var('sense',1,s2,[-1,1]);
 >> z = pde_var('output',2,[],[]);
\end{verbatim}
\end{matlab}
or, equivalently, as
\begin{matlab}
\begin{verbatim}
 >> x1 = pde_var([s1;s2],[0,3;-1,1]);
 >> x2 = pde_var([s1;s2],[0,3;-1,1],[2;2]);
 >> u1 = pde_var('control');
 >> u2 = pde_var('control',s2,[-1,1]);
 >> w1 = pde_var('in');
 >> w2 = pde_var('in');
 >> y1 = pde_var('sense',s1,[0,3]);
 >> y2 = pde_var('sense',s2,[-1,1]);
 >> z = pde_var('out',2);
\end{verbatim}
\end{matlab}
Here, we do not specify the order of differentiability of state variable $\mbf{x}_{1}(t)$ with respect to the spatial variables. This is because, in the PDE~\eqref{eq:alt_PDE_input:terms_input_PDE:example_PDE}, a second-order derivative of $\mbf{x}_{1}$ is taken with respect to both variables in the dynamics for $\mbf{x}_{2}(t)$, from which PIETOOLS will be able to infer the order of differentiability automatically. However, no second-order spatial derivative is taken of the state variable $\mbf{x}_{2}(t)$, so we manually specify that it is second-order differentiable with respect to both spatial variables in the call to \texttt{pde\_var}. Failing to specify this may result in issues when converting the system to a PIE.

Displaying e.g. the output variable \texttt{y2}, we get something like
\begin{matlab}
\begin{verbatim}
 >> y2
       y8(t,s2);
\end{verbatim}
\end{matlab}
Here, the index $8$ corresponds to the ID assigned to the output: \texttt{y2.y\{1\}.ID=8}.
This ID is uniquely generated by the function \texttt{stateNameGenerator}, and is vital for PIETOOLS to distinguish between the different PDE objects. However, since the value of the IDs tends to increase quite quickly, the user may consider calling \texttt{clear stateNameGenerator} whenever declaring a new PDE, to reset the counter for the IDs.
\end{Statebox}

\subsection{Declaring Terms}

After declaring a PDE variable using \texttt{pde\_var}, we can perform a variety of operations on this variable, including multiplying it with desired coefficients, performing differentiation or integration, and evaluating it at a particular position. The resulting term involving the PDE variable is stored in the field \texttt{free}, which is a cell structure with each element \texttt{free\{i\}} specifying a separate string of free terms to be used to declare equations. In particular, each element \texttt{free\{i\}} has a field \texttt{term}, which is again a cell structure with each element \texttt{free\{i\}.term\{j\}} representing a single term through the fields
\begin{flalign*}
& \texttt{term\{j\}.x;}	 & &\text{integer specifying which state component,}	&	\\
& \hspace*{0.1cm} \textbf{or} \quad \texttt{term\{j\}.w;}	&	&\hspace*{4.1cm} \text{\textbf{or} exogenous input,}	\\
& \hspace*{0.2cm} \textbf{or} \quad \texttt{term\{j\}.u;}	&	&\hspace*{4.2cm} \text{\textbf{or} actuator input,}	\\
& \hspace*{0.3cm} \textbf{or} \quad \texttt{term\{j\}.y;}	&	&\hspace*{4.3cm} \text{\textbf{or} sensed output,}	\\
& \hspace*{0.4cm} \textbf{or} \quad \texttt{term\{j\}.z;}	&	&\hspace*{4.4cm} \text{\textbf{or} regulated output appears in the term;}	\\
& \texttt{term\{j\}.D}	& &\text{$1\times p$ integer array specifying the order of the derivative of the \textbf{state}} \\
&   &   &\text{\textbf{component} \texttt{term\{j\}.x} in each variable;}	\\
& \texttt{term\{j\}.loc} & &\text{$1\times p$ polynomial or ``double'' array specifying the spatial position}\\
&   &   &\text{at which to evaluate the \textbf{state component} \texttt{term\{j\}.x};}	\\
& \texttt{term\{j\}.I} & &\text{$p\times 1$ cell array specifying the lower and upper limits of the integral}\\
&   &   &\text{to take of the state or input;}	\\
& \texttt{term\{j\}.C}	& &\text{(polynomial) factor with which to multiply the state or input;} \\
& \texttt{term\{j\}.delay}	& &\text{scalar integer specifying the temporal delay in the state or input;}
\end{flalign*}
where $p$ denotes the number of spatial variables on which the component depends. Combined, these fields can represent a general term involving e.g. a state component $\mbf{x}_k(t,s)$ or input $\mbf{u}_{k}(t,s)$ of the form
{\small
\begin{align*}
&\underbrace{\int_{L_1}^{U_1}}_{\texttt{I}}\left(
\underbrace{C(s,\theta)}_{\texttt{C}}\
\overbrace{\partial_{\theta}^{d}}^{\texttt{D}}\thinspace
\underbrace{\mbf{x}_{k}}_{\texttt{x}}(t-\overbrace{\tau}^{\texttt{delay}},\overbrace{\theta}^{\texttt{loc}})\right) d\theta   & &\text{or}  &
&\underbrace{\int_{L}^{U}}_{\texttt{I}}\left(
\underbrace{C(s,\theta)}_{\texttt{C}}\thinspace
\underbrace{\mbf{u}_{k}}_{\texttt{u}}(t-\overbrace{\tau}^{\texttt{delay}},\overbrace{\theta})\right) d\theta.
\end{align*}}
Declaring e.g. a PDE state \texttt{x2=pde\_var([s1;s2],[0,3;-1,1],[2;2]} as in the previous subsection, this state will be immediately assigned a single term in the field \texttt{x2.free\{1\}.term\{1\}}, representing just the variable itself as
\begin{matlab}
\begin{verbatim}
 >> x2.free{1}.term{1}
 ans = 

   struct with fields:

       x: 1
       C: 1
     loc: [1×2 polynomial]
       D: [0 0]
       I: {2×1 cell}
\end{verbatim}
\end{matlab}
Here, \texttt{x2.free\{1\}.term\{1\}.loc=[s1,s2]} and \texttt{x2.free{1}.term{1}.I=\{[];[]}\}, so that the state variable $\mbf{x}_{2}$ is multiplied with $1$, evaluated at $(s_{1},s_{2})=(s_{1},s_{2})$, and not differentiated or integrated. Note that the index \texttt{x2.free\{1\}.term\{1\}.x} is \texttt{1}, which is not the same as the ID \texttt{x2.x\{1\}.ID=2}, but rather matches the element $i$ of the field \texttt{x2.x\{i\}} in which the information on the state variable is stored -- no need to worry about this though. Although all the fields in the structure could be adjusted manually to specify a desired term involving $\mbf{x}_{2}$, the \texttt{pde\_struct} class comes with a variety of overloaded functions for declaring these values much more easily.

\subsubsection{Multiplication}

Given a \texttt{pde\_struct} object representing some term, we can pre-multiply it with a desired factor using the standard multiplication operation \texttt{*}. The result is the same term, but now with coefficients \texttt{term\{j\}.C} set to the specified factor. For example, to represent $5\mbf{x}_{1}(t,s_{1},s_{2})$, we simply call
\begin{matlab}
\begin{verbatim}
 >> trm1 = 5*x1
 
  5 * x1(t,s1,s2);
\end{verbatim}
\end{matlab}
The result will again be a single term, stored in the field \texttt{trm1.free\{1\}.term\{1\}}, where now the coefficients are set to $5$:
\begin{matlab}
\begin{verbatim}
 >> trm1.free{1}.term{1}.C
 ans = 
       5
\end{verbatim}
\end{matlab}
Similarly, we can also pre-multiply input variables with desired coefficients, and these coefficients can also be polynomial. For example, to represent the term $(3-s_{1})s_{2}u_{1}(t)$, we call
\begin{matlab}
\begin{verbatim}
 >> trm2 = (3-s1)*s2*u1
 
       C11(s1,s2) * u3(t);
\end{verbatim}
\end{matlab}
Again, the result is only a single term, in this case represented as
\begin{matlab}
\begin{verbatim}
 >> trm2.free{1}.term{1}
 ans = 

   struct with fields:

     u: 1
     C: [1×1 polynomial]
     I: {0×1 cell}
\end{verbatim}
\end{matlab}
where now \texttt{trm2.free\{1\}.term\{1\}=-s1*s2+3*s2}. We can also check the value of the coefficients defining this term by calling the field \texttt{C} directly, with row and column index as displayed,
\begin{matlab}
\begin{verbatim}
 >> trm2.C{1,1}
 ans = 
    -s1*s2 + 3*s2
\end{verbatim}
\end{matlab}

Note that \texttt{pde\_struct} objects can only represent linear systems, and thus multiplication of state, input, or output variables with one another is not supported. Also, declaring output equations wherein the output is multiplied with a factor is not supported.

\subsubsection{Differentiation}

Naturally, any PDE will involve partial derivatives of a state variable with respect to spatial variables. This operation can be performed in the Command Line Input format using the function \texttt{diff} as
\begin{matlab}
\begin{verbatim}
 >> trm_new = diff(trm_old,vars,order);
\end{verbatim}
\end{matlab}
This function takes three arguments:
\begin{itemize}
    \item \texttt{trm\_old}: A \texttt{pde\_struct} object representing some term of which to take a derivative;

    \item \texttt{vars}: A $p\times 1$ array of type \texttt{'polynomial'}, specifying the spatial variables with respect to which to take the derivative;

    \item \texttt{order}: A $p\times 1$ array of integers specifying for each of the variables the desired order of the derivative of the term with respect to that variable. Defaults to $1$;
\end{itemize}
The output \texttt{trm\_new} is then another \texttt{pde\_struct} object, representing the derivative of the term defined by \texttt{trm\_old} with respect to each of the variables in \texttt{vars}, up to the order specified in \texttt{order}. Specifically, the value of the field \texttt{D} in the term will be set to the specified order. For example, to represent the term $5\partial_{s_{1}}^2\mbf{x}_{1}(t,s_{1},s_{2})$, we can take the derivative of $5\mbf{x}_{1}(t,s_{1},s_{2})$ as
\begin{matlab}
\begin{verbatim}
 >> trm3 = diff(trm1,s1,2)
 
     5 * (d/ds1)^2 x1(t,s1,s2);
\end{verbatim}
\end{matlab}
The result again represents a single term, stored as
\begin{matlab}
\begin{verbatim}
 >> trm3.free{1}.term{1}
 ans = 

   struct with fields:

       x: 1
       C: [1×1 polynomial]
     loc: [1×2 polynomial]
       D: [2 0]
       I: {2×1 cell}
\end{verbatim}
\end{matlab}
Here, the field \texttt{D} is set to \texttt{[2 0]} to indicate that a second-order derivative is taken of the state with respect to its first variable. Similarly, we can take the derivative $\partial_{s_{1}}\partial_{s_{2}}\mbf{x}_{2}(t,s_{1},s_{2})$ as
\begin{matlab}
\begin{verbatim}
 >> trm4 = diff(x2,[s1;s2])

   (d/ds1)(d/ds2) x2(t,s1,s2);
\end{verbatim}
\end{matlab}
with the derivative now stored as \texttt{trm4.free\{1\}.term\{1\}.D=[1,1]}.

Note, however, that it is not possible to take derivatives of input or output signals. In addition, keep in mind that the order of multiplication and differentiation matters, so that e.g. \texttt{s1*diff(x2,s1)} yields a different result than \texttt{diff(s1*x2,s1)}.

\subsubsection{Substitution}

Aside from taking derivatives of PDE states, we can also evaluate these PDE states at the boundary of the domain, using the function \texttt{subs} as
\begin{matlab}
\begin{verbatim}
 >> trm_new = subs(trm_old,vars,loc);
\end{verbatim}
\end{matlab}
This function takes three arguments:
\begin{itemize}
    \item \texttt{trm\_old}: A \texttt{pde\_struct} object representing some term which to evaluate at some boundary position;

    \item \texttt{vars}: A $p\times 1$ array of type \texttt{'polynomial'}, specifying the spatial variables to be substituted for a position;

    \item \texttt{loc}: A $p\times 1$ numeric array specifying the value for which each of the variables are to be substituted. Each value must correspond to either the lower or upper limit of the interval on which the corresponding variable is defined;
\end{itemize}
The function returns a \texttt{pde\_struct} object \texttt{trm\_new}, representing the same term as \texttt{trm\_old} but now evaluated with respect to each of the variables in \texttt{vars} at the position specified in \texttt{loc}. In particular, the value of the field \texttt{term\{j\}.loc} will be set to the specified value. For example, to evaluate our derivative $\partial_{s_{1}}\partial_{s_{2}}\mbf{x}_{2}(t,s_{1},s_{2})$ (represented by \texttt{trm4}) at $(s_{1},s_{2})=(3,1)$, we call
\begin{matlab}
\begin{verbatim}
 >> trm5 = subs(trm4,[s1;s2],[3;1])

     (d/ds1)(d/ds2) x2(t,3,1);
\end{verbatim}
\end{matlab}
where now
\begin{matlab}
\begin{verbatim}
 >> trm5.free{1}.term{1}.loc
 ans = 
 
  [ 3, 1]
\end{verbatim}
\end{matlab}
Similarly, we can declare the term $\mbf{x}_{2}(t,s_{1},-1)$ as
\begin{matlab}
\begin{verbatim}
 >> trm6 = subs(x2,s2,-1)

    x2(t,s1,-1);
\end{verbatim}    
\end{matlab}
where now
\begin{matlab}
\begin{verbatim}
 >> trm6.free{1}.term{1}.loc
 ans = 
 
   [ s1, -1]
\end{verbatim}
\end{matlab}
indicating that the term is evaluated as $(s_{1},s_{2})=(s_{1},-1)$. 

Note that only state variables can be evaluated at a boundary -- substitution of input or output variables is not supported. In addition, note that the order of substitution and multiplication is important, so that e.g. \texttt{s1*subs(x2,s1,0)} yields a different result than \texttt{subs(s1*x2,s1,0)}.

\subsubsection{Integration}

Regulated outputs for PDE systems frequently involve an integral of the PDE state. Such an integral for \texttt{pde\_struct} objects can be declared with the function \texttt{int} as
\begin{matlab}
\begin{verbatim}
 >> trm_new = int(trm_old,vars,dom);
\end{verbatim}
\end{matlab}
This function too takes three arguments:
\begin{itemize}
    \item \texttt{trm\_old}: A \texttt{pde\_struct} object representing some term which to integrate;

    \item \texttt{vars}: A $p\times 1$ array of type \texttt{'polynomial'}, specifying the spatial variables with respect to which to integrate;

    \item \texttt{loc}: A $p\times 2$ numeric or \texttt{'polynomial'} array specifying the domain over which integration is to be performed for each variable, with the first column specifying the lower limit and the second column the upper limit of the integral. Lower and upper limits must either correspond to the boundaries of the domain on which the variable is defined, or, for the purpose of declaring a partial integral, the spatial variable itself;
\end{itemize}
The returned \texttt{trm\_new} is again a \texttt{pde\_struct} object representing the same term as \texttt{trm\_old}, but now integrated over the desired domain with respect to the desired variable. Specifically, the limits of the integral with respect to each variable will be stored in the field \texttt{term\{j\}.I}. For example, to declare the integral $\int_{0}^{3}\int_{-1}^{1}\mbf{x}_{1}(t,s_{1},s_{2})ds_{2}ds_{1}$, we call
\begin{matlab}
\begin{verbatim}
 >> trm7 = 10*int(x1,[s2;s1],[-1,1;0,3])

    int_0^3 int_-1^1 [10 * x1(t,s1,s2)] ds2 ds1;
\end{verbatim}
\end{matlab}
where now
\begin{matlab}
\begin{verbatim}
 >> trm7.free{1}.term{1}.I{1}
 ans = 
   [ 0, 3]
   
 >> trm7.free{1}.term{1}.I{2}
 ans = 
   [ -1, 1]
\end{verbatim}
\end{matlab}
If desired, it is also possible to declare partial integrals, for which substitution must be performed as well as integration. For example, to declare an integral $\int_{0}^{s_{1}}(s_{1}-\theta)\mbf{x}_{2}(t,\theta,s_{2})d\theta$, we can call
\begin{matlab}
\begin{verbatim}
 >> pvar s1_dum
 >> trm_alt = int((s1-s1_dum)*subs(x2,s1,s1_dum),s1_dum,[0,s1])

     int_0^s1 [C11(s1,s1_dum) * x2(t,s1_dum,s2)] ds1_dum;
\end{verbatim}
\end{matlab}
Here, we first need to introduce a dummy variable for integration, and substitute $s_{1}$ for this dummy variable, before taking the integral. Although not strictly necessary, we highly recommend to always give the dummy variable the same name as its associated primary spatial variable, but with \texttt{\_dum} added, as this is the default used by PIETOOLS.

Note that, unlike differentiation and substitution, integration \textbf{is supported} for input variables, \textbf{but not} for output variables.

\subsubsection{Delay}

Although we recommend using the DDE or DDF format for declaring systems with delay, \texttt{pde\_struct} objects do also allow signals with temporal delay to be declared. This can be done using the \texttt{subs} function, by substituting the temporal variable $t$ for $t-\tau$, for some desired value $\tau$. For example, to declare the term $\mbf{u}(t-0.5,s_{2})$, we call
\begin{matlab}
\begin{verbatim}
 >> pvar t
 >> trm8 = subs(u2,t,t-0.5)

     u4(t-0.5,s2);
\end{verbatim}
\end{matlab}
Here, we must first declare the temporal variable \texttt{t}, and then we can substitute \texttt{t} for \texttt{t-0.5}. Note that the variable \texttt{t} \textbf{will always be interpreted as temporal variable} in PIETOOLS, and therefore should not be used for any other purpose. The value of the delay will be stored in the aptly named field \texttt{delay}, so that e.g.
\begin{matlab}
\begin{verbatim}
 >> trm8.free{1}.term{1}
 ans = 

   struct with fields:

         u: 1
         C: 1
         I: {[]}
       loc: [1×1 polynomial]
     delay: 0.5000
\end{verbatim}
\end{matlab}
Delay can be added to state variables or exogenous inputs in a similar manner, but is not supported for output signals.

\subsubsection{Addition}

Having seen how we can perform a variety of operations on PDE variables to declare a single term, naturally, we will want to add these terms to create an equation. This can be readily done with the overloaded function for addition, \texttt{+}. In particular, given two \texttt{pde\_struct} objects, \texttt{trm\_1} and \texttt{trm\_2}, each specifying their own terms, we can compute the sum of these terms as \texttt{trm\_3=trm\_1+trm\_2}, where now \texttt{trm\_3.free\{i\}.term} includes all elements from \texttt{trm\_1.free\{i\}.term} as well as \texttt{trm\_2.free\{i\}.term}. For example, having declared the term $(3-s_{1})s_{2}u_{1}(t)$  as \texttt{trm2}, and $5\partial_{s_{1}}^2\mbf{x}_{1}(t,s_{1},s_{2})$ as \texttt{trm3}, we can take their sum as
\begin{matlab}
\begin{verbatim}
 >> trms9 = trm2+trm3

     C11(s1,s2) * u3(t) + 5 * (d/ds1)^2 x1(t,s1,s2);
\end{verbatim}
\end{matlab}
where now \texttt{trm9.free\{1\}.term} is a $1\times 2$ cell array, with each element representing a separate term, combining the terms from \texttt{trm2} and \texttt{trm3}. Note that we can determine the coefficients appearing in this sum of terms by calling field \texttt{C\{1,j\}}, where \texttt{j} is the desired term number, so that e.g.
\begin{matlab}
\begin{verbatim}
 >> trms9.C{1,1}
 ans = 
   -s1*s2 + 3*s2
\end{verbatim}
\end{matlab}
and
\begin{matlab}
\begin{verbatim}
 >> trms9.C{1,2}
 ans = 
   5
\end{verbatim}
\end{matlab}
Naturally, we can also add more terms to the structure, adding e.g. $5\partial_{s_{2}}^2\mbf{x}_{1}(t,s_{1},s_{2})$ as
\begin{matlab}
\begin{verbatim}
 >> trms10 = 5*diff(x1,s2,2) + trms9
 
     5 * (d/ds2)^2 x1(t,s1,s2) + C12(s1,s2) * u3(t) + 5 * (d/ds1)^2 x1(t,s1,s2);
\end{verbatim}
\end{matlab}
Note that the coefficients in \texttt{trms10} have been shifted compared to \texttt{trms9}, so that now e.g. \texttt{trms10.C\{1,2\}=trms9.C\{1,1\}=-s1*s2+3*s2}.

\subsubsection{Concatenation}

Finally, \texttt{pde\_struct} objects can also be concatenated vertically, using the standard MATLAB function \texttt{[ ; ]}, to combine separate sums of terms into a single structure. For example, we can concatenate the variable $\mbf{x}_{2}(t,s_{1},s_{2})$, declared as \texttt{x2}, with the terms $5\partial_{s_{1}}^2\mbf{x}_{2}(t,s_{1},s_{2})+5\partial_{s_{2}}^2\mbf{x}_{2}(t,s_{1},s_{2})=(3-s_{1})s_{2}u_{1}(t)$, declared as \texttt{trms10} as
\begin{matlab}
\begin{verbatim}
 >> RHS_x = [x2;trms10]

    x2(t,s1,s2);
    5 * (d/ds2)^2 x1(t,s1,s2) + C12(s1,s2) * u3(t) + 5 * (d/ds1)^2 x1(t,s1,s2);
\end{verbatim}
\end{matlab}
Here, each row of terms in \texttt{RHS1} is stored in a separate element of the cell \texttt{RHS1.free}, so that \texttt{RHS1.free\{1\}.term} only has a single element representing the variable $\mbf{x}_{2}(t,s_{1},s_{2})$, and \texttt{RHS1.free\{2\}.term} has three elements, representing the three terms $5\partial_{s_{1}}^2\mbf{x}_{2}(t,s_{1},s_{2})$, $(3-s_{1})s_{2}u_{1}(t)$, and $5\partial_{s_{2}}^2\mbf{x}_{2}(t,s_{1},s_{2})$. Similarly, we can concatenate the corner value $\partial_{s_{1}}\partial_{s_{2}}\mbf{x}_{2}(t,3,1)$, represented by \texttt{trm5}, and the integral $10\int_{0}^{3}\int_{-1}^{1}\mbf{x}_{1}(t,s_{1},s_{2})ds_{2} ds_{1}$, represented by \texttt{trm7}, as
\begin{matlab}
\begin{verbatim}
 >> RHS_z = [trm7; tmr5]
 
     int_0^3 int_-1^1 [10 * x1(t,s1,s2)] ds2 ds1;
     (d/ds1)(d/ds2) x2(t,3,-1);
\end{verbatim}
\end{matlab}

\subsection{Declaring Equations}

Having seen how terms for PDEs can be declared and stored using \texttt{pde\_struct} objects, all that remains is to use these terms to declare actual equations. Here, we distinguish three types of equations, namely state equations (ODEs and PDEs), output equations, and boundary conditions, all of which can be declared using the function \texttt{==}.

\subsubsection{State equations}

In PIETOOLS, all ordinary and partial differential equations are assumed to involve some temporal variable, distinct from all other (spatial) variables in the dynamics. That is, any PDE is expected to model the evolution of some state $\mbf{x}(t)\in L_{2}[\Omega]$ for $t\geq 0$, governed by an equation
\begin{equation*}
    \partial_{t} \mbf{x}(t,s)=f(\mbf{x}(t,s),\mbf{u}(t,s),\mbf{w}(t,s),s),\qquad s\in\Omega
\end{equation*}
where $f$ is some (polynomial) function of $\mbf{x}(t,s)$ and its partial derivatives with respect to $s$, as well as any input signals $\mbf{u}(t,s)$ and $\mbf{w}(t,s)$. In the previous subsection, we have seen how we can declare such a function $f$ in the \texttt{pde\_struct} format using multiplication, addition, differentiation, etc.. Now, to declare the left-hand side of such an equation, $\partial_{t}\mbf{x}(t,s)$, we can use the function \texttt{diff} just as we did to declare spatial derivatives. For example, to declare the derivative $\partial_{t}\mbf{x}_{1}(t,s_{1},s_{2})$, we call
\begin{matlab}
\begin{verbatim}
 >> pvar t
 >> LHS1 = diff(x1,t);
 
        d_t x1(t,s1,s2);
\end{verbatim}
\end{matlab}
Here, we first need to declare the temporal variable as \texttt{pvar t}, noting that PIETOOLS will always interpret this object as temporal variable. Alternatively, we can also use \texttt{'t'} to represent the temporal variable in this case, declaring e.g. $\partial_{t}\mbf{x}_{2}(t,s_{1},s_{2})$ as
\begin{matlab}
\begin{verbatim}
 >> LHS2 = diff(x2,'t');

        d_t x2(t,s1,s2);
\end{verbatim}
\end{matlab}
The result is again a single term, stored in \texttt{LHS2.free\{1\}.term\{1\}}, where now the order of the temporal derivative is represented through the field \texttt{tdiff}:
\begin{matlab}
\begin{verbatim}
 >> LHS2.free{1}.term{1}.tdiff
 ans =

      1
\end{verbatim}
\end{matlab}
indicating that a first-order temporal derivative is taken of the state variable. Although it is also possible to declare higher-order temporal derivatives -- calling e.g. \texttt{diff(x2,'t',2)} to declare a second-order temporal derivative of $\mbf{x}_{2}$ -- keep in mind that PIETOOLS will always convert the system to a format involving only first-order temporal derivatives when constructing the PIE representation. In doing so, PIETOOLS will add state components without adding boundary conditions, so that the resulting representation may not be entirely equivalent, and results of e.g. stability analysis may be conservative -- see also Subsection~\ref{subsec:alt_PDE_input:terms_input_PDE:expand_tderivatives}.

Given a temporal derivative of a state variable, an equation defining the dynamics of this state can be easily declared using the function \texttt{==}, equating this derivative to some desired set of terms. For example, to declare the equation $\partial_{t}\mbf{x}_{1}(t,s_{1},s_{2})=\mbf{x}_{2}(t,s_{1},s_{2})$, we simply call
\begin{matlab}
\begin{verbatim}
 >> x_eq1 = diff(x1,t)==x2

  d_t x1(t,s1,s2) = x2(t,s1,s2);
\end{verbatim}
\end{matlab}
Similarly, having already declared the terms $5\partial_{s_{1}}^2\mbf{x}_{2}(t,s_{1},s_{2})+(3-s_{1})s_{2}u_{1}(t)+5\partial_{s_{2}}^2\mbf{x}_{2}(t,s_{1},s_{2})$ through \texttt{trms10}, we can declare the dynamics for $\mbf{x}_{2}$ as
\begin{matlab}
\begin{verbatim}
 >> x_eq2 = diff(x2,'t')==trms10

 d_t x2(t,s1,s2) = 5*(d/ds2)^2x1(t,s1,s2)+C12(s1,s2)*u3(t)+5*(d/ds1)^2x1(t,s1,s2);
\end{verbatim}
\end{matlab}
In doing so, the result \texttt{x\_eq2} will store all terms from \texttt{trms10.free\{1\}.term} in a corresponding element of \texttt{x\_eq2.x}. In particular, since \texttt{x\_eq2} defines only a single state equation, the terms will be stored in \texttt{x\_eq2.x\{1\}}, so that \texttt{x\_eq2.x\{1\}.term=trms10.free\{1\}.term}. However, we can also concatenate the declared equations as
\begin{matlab}
\begin{verbatim}
 >> x_eqs = [x_eq1; x_eq2]

  d_t x1(t,s1,s2) = x2(t,s1,s2);
  d_t x2(t,s1,s2) = 5*(d/ds2)^2x1(t,s1,s2)+C22(s1,s2)*u3(t)+5*(d/ds1)^2x1(t,s1,s2);
\end{verbatim}
\end{matlab}
In this output, the two state equations are stored in separate elements of \texttt{x\_eqs.x}, so that in particular \texttt{x\_eqs.x\{2\}.term=x\_eq2.x\{1\}.term=trms10.free\{1\}.term}.

Note that, when concatenating equations, the order of the equations may not reflect the order in which they are specified. For the purposes of e.g. simulation and analysis, the order of the state components will be determined by the order in which their governing equations appear, so it is important to always check the final order of the state (as well as input and output) variables once the full system of equations has been declared.

\subsubsection{Output equations}

Aside from differential equations, we can also declare output equations, using the same function \texttt{==}. For example, to declare the observed output equations $\mbf{y}_{1}(t,s_{1})=\mbf{x}_{1}(t,s_{1})+s_{1}w_{1}(t)$ and $\mbf{y}_{2}(t,s_{2})=\mbf{x}_{2}(t,s_{2})+w_{2}(t)$, we can call
\begin{matlab}
\begin{verbatim}
 >> y_eqs = [y1==subs(x1,s2,1)+s1*w1;
             y2==subs(x2,s1,3)+w2]
             
  y7(t,s1) = x1(t,s1,1) + C32(s1) * w5(t);
  y8(t,s2) = x2(t,3,s2) + w6(t);
\end{verbatim}
\end{matlab}
In doing so, the output object \texttt{y\_eqs} will have the two declared equations stored in \texttt{y\_eqs.y\{1\}} and \texttt{y\_eqs.y\{2\}}, so that e.g. \texttt{y\_eqs.y\{1\}.term\{2\}} represents the term $s_{1}w_{1}$ as
\begin{matlab}
\begin{verbatim}
 >> y_eqs.y{1}.term{2}
 ans = 

   struct with fields:

     w: 1
     C: [1×1 polynomial]
     I: {0×1 cell}
\end{verbatim}
\end{matlab}
where \texttt{y\_eqs.y\{1\}.term\{2\}.C=s1}. Similarly, having declared $\bmat{10\int_{0}^{3}\int_{-1}^{1}\mbf{x}_{1}(t,s_{1},s_{2})ds_{2}ds_{1}\\\partial_{s_{1}}\partial_{s_{2}}\mbf{x}_{2}(t,3,1)}$ as \texttt{RHS\_z}, we can set the regulated output equation as
\begin{matlab}
\begin{verbatim}
 >> z_eqs = z==RHS_z

 z9(t) = int_0^3 int_-1^1[C31*x1(t,s1,s2)]ds2 ds1 + C32*(d/ds1)(d/ds2) x2(t,3,1);
\end{verbatim}
\end{matlab}
Note here that, although the output $z(t)$ is vector-valued, we declared it as only a single variable \texttt{z}. As such, it will also be assigned only one equation, where now the coefficients \texttt{C31} and \texttt{C32} are arrays mapping the scalar-valued terms to the vector-valued outputs. In particular,
\begin{matlab}
\begin{verbatim}
 >> z_eqs.C{3,1}
 ans = 
   [ 10]
   [  0]
\end{verbatim}
\end{matlab}
and
\begin{matlab}
\begin{verbatim}
 >> z_eqs.C{3,2}
 ans = 
      0
      1
\end{verbatim}
\end{matlab}
More generally, the full equation for $z$ will be stored in the field \texttt{z\_eqs.z\{1\}}. 
Thus, although concatenation of vector-valued objects is supported, keep in mind that only one equation will be assigned for each declared variable. Note also that, naturally, equations can only be set if the number of elements on the left-hand side and right-hand side match.

\subsubsection{Boundary conditions}

Any well-posed PDE also involves a number of boundary conditions on the state. These too can be declared as \texttt{pde\_struct} objects using the function \texttt{==}, by setting different terms equal to one another, or setting a term equal to zero. For example, to declare the boundary condition $\mbf{x}_{2}(t,s_{1},-1)=0$, we can call
\begin{matlab}
\begin{verbatim}
 >> BC1 = subs(x2,s2,-1)==0

  0 = x2(t,s1,-1);
\end{verbatim}
\end{matlab}
Here, the zero will always be displayed on the left-hand side of the equation. The actual term $\mbf{x}_{2}(t,s_{1},-1)$ that is set equal to zero will be stored in the field \texttt{BC1.BC\{1\}.term}. Similarly, we can declare the condition $\mbf{x}_{1}(t,0,s_{2})=\mbf{u}_{2}(t-0.5,s_{2})$ as
\begin{matlab}
\begin{verbatim}
 >> BC2 = subs(x1,s1,0)==subs(u2,t,t-0.5)

  0 = x1(t,0,s2) - u2(t-0.5,s2);
\end{verbatim}
\end{matlab}
Again, the equation is represented in the form $0=F(\mbf{x},\mbf{u},\mbf{w},s)$, where the function $F$ is stored in the field \texttt{BC}. In this case, the field \texttt{BC2.BC\{1\}.term} will have two elements, representing the terms $\mbf{x}_{1}(t,0,s_{2})$ and $-\mbf{u}_{2}(t-0.5,s_{2})$. 

Note that any equation that does not involve a temporal derivative or an output will be interpreted as a boundary condition, and added to the field \texttt{BC}.

\newpage

\begin{Statebox}{\textbf{Example}}
Consider the wave equation example from~\eqref{eq:alt_PDE_input:terms_input_PDE:example_PDE}. Having declared the different state, input, and output variables, the full PDE system can be declared as

\begin{matlab}
\begin{verbatim}
 >> pvar t
 >> PDE = [diff(x1,t)==x2;
           diff(x2,t)==5*(diff(x1,s1,2)+diff(x1,s2,2))+(3-s1)*s2*u1;
           y1==subs(x1,s2,1)+s1*w1;
           y2==subs(x2,s1,3)+w2;
           z ==[10*int(x1,[s1;s2],[0,3;-1,1]);
                subs(diff(x2,[s1;s2]),[s1;s2],[3;1])];
           subs(x1,s1,0)==subs(u2,t,t-0.5);
           subs(x2,s1,0)==0;
           subs([x1;x2],s2,-1)==0;
           subs(diff([x1;x2],s1),s1,3)==0;
           subs(diff([x1;x2],s2),s2,1)==0]
           
  dt x1(t,s1,s2) = x2(t,s1,s2);
  dt x2(t,s1,s2) = 5 * (d/ds1)^2 x1(t,s1,s2) + 5 * (d/ds2)^2 x1(t,s1,s2) 
                            + C23(s1,s2) * u3(t);
  y7(t,s1) = x1(t,s1,1) + C32(s1) * w5(t);
  y8(t,s2) = x2(t,3,s2) + w6(t);
  z9(t) = int_0^3 int_-1^1 [C51 * x1(t,s1,s2)]ds2 ds1 
                            + C52 * (d/ds1)(d/ds2) x2(t,3,1);
  0 = x1(t,0,s2) - u4(t-0.5,s2);
  0 = x2(t,0,s2);
  0 = x1(t,s1,-1);
  0 = x2(t,s1,-1);
  0 = (d/ds1) x1(t,3,s2);
  0 = (d/ds1) x2(t,3,s2);
  0 = (d/ds2) x1(t,s1,1);
  0 = (d/ds2) x2(t,s1,1);
\end{verbatim}
\end{matlab}

\end{Statebox}

\subsection{Post-Processing of PDE Structures}

After you have declared a full system of equations and boundary conditions as a \texttt{pde\_struct} object, you are generally ready to convert the system to a PIE with the function \texttt{convert}, for the purpose of e.g. stability analysis or estimator or controller synthesis. However, before converting the system to a PIE, PIETOOLS will first express the system in a particular manner, e.g. re-ordering the state, input and output variables, accounting for any temporal delays, and expanding any higher-order temporal derivatives in a manner that involves only first-order derivatives. Although all of this is done automatically when calling \texttt{convert}, and the user will generally be informed of these changes in the Command Window, it is important that the user be aware of what exactly is happening. Moreover, in some cases, it may be useful for the user to perform these operations themselves. In this subsection, therefore, we list several functions for post-processing of completed PDE structures that PIETOOLS generally runs when converting the system to a PIE, and that the user may also benefit from themselves.

\subsubsection{Declaring controlled inputs and observed outputs}

Although we highly recommend distinguishing between actuator inputs and disturbances when first declaring the PDE variables, it is possible to convert disturbances to controlled inputs after the PDE has been declared as well, using the function \texttt{setControl} as
\begin{matlab}
\begin{verbatim}
 >> PDE_new = setControl(PDE_old,w);
\end{verbatim}
\end{matlab}
This function takes as input a PDE system \texttt{PDE} declared as \texttt{pde\_struct}, and a desired exogenous input variable \texttt{w}, and returns a structure \texttt{PDE\_new} representing the same as the input, but now with the variable \texttt{w} converted to an actuator input \texttt{u}. For example to convert the disturbance $w_{1}$ in our output equations defined by \texttt{y\_eqs} to a controlled input, we call
\begin{matlab}
\begin{verbatim}
 >> y_eqs_u = setControl(y_eqs,w1)
 1 inputs were designated as controlled inputs

  y7(t,s1) = x1(t,s1,1) + C32(s1) * u5(t);
  y8(t,s2) = x2(t,3,s2) + w6(t);
\end{verbatim}
\end{matlab}
In the resulting system, the disturbance $w_{1}$ has been converted to a controlled input $u_{3}(t)$, and the equation $y_{1}(t,s_{1})=\mbf{x}_{1}(t,s_{1},1)+s_{1}w_{1}(t)$ has been updated accordingly to $y_{1}(t,s_{1})=\mbf{x}_{1}(t,s_{1},1)+s_{1}u_{3}(t)$ (although different subscripts are used in the display). Similarly, we can also convert regulated outputs to observed outputs using the function \texttt{setObserve} as
\begin{matlab}
\begin{verbatim}
 >> PDE_new = setObserve(PDE_old,z);
\end{verbatim}
\end{matlab}
Rather than a disturbance, this function takes a regulated output variable as second argument, and converts this output to an observed output in the specified structure. For example, we can convert the regulated output $z(t)$, represented by \texttt{z\_eqs}, to an observed output by calling
\begin{matlab}
\begin{verbatim}
 >> z_eqs_y = setObserve(z_eqs,z)
 1 outputs were designated as observed outputs

  y9(t) = int_0^3 int_-1^1[C31*x1(t,s1,s2)]ds2 ds1 + C32*(d/ds1)(d/ds2) x2(t,3,1);
\end{verbatim}
\end{matlab}
Keep in mind that the functions \texttt{setControl
} and \texttt{setObserve} should only be called once the full PDE system has been declared, to avoid e.g. an input signal appearing as both an exogenous input and controlled input in the system.

\subsubsection{Initializing a PDE structure}

After declaring any PDE structure \texttt{PDE}, the user is highly recommended to initialize the structure by calling \texttt{PDE=initialize(PDE)}. This function checks that the PDE has been properly specified, throwing an error if any terms have not been properly declared, and warning the user of e.g. missing equations or insufficient boundary conditions. As such, it is important that the user \textbf{initializes the system only once all equations have been declared}. The function then displays a summary of the encountered state variables, inputs, outputs, and boundary conditions, and returns a cleaner display of the system, so that the user can check whether PIETOOLS has properly interpreted the system. In this summary and the display, the different variables are also assigned new indices separate from their IDs, instead matching the order in which PIETOOLS will consider them. However, these variables may be reordered again when converting to a PIE, using the function \texttt{reorder\_comps} presented in the next subsection.

\begin{Statebox}{\textbf{Example}}
Consider again the wave equation example from~\eqref{eq:alt_PDE_input:terms_input_PDE:example_PDE}, declared as a \texttt{pde\_struct} object \texttt{PDE} in the previous subsection. We initialize the function as follows:
\begin{matlab}
\begin{verbatim}
 >> PDE = initialize(PDE);
 
 Encountered 2 state components: 
  x1(t,s1,s2), of size 1, differentiable up to order (2,2) in variables (s1,s2);
  x2(t,s1,s2), of size 1, differentiable up to order (2,2) in variables (s1,s2);

 Encountered 2 actuator inputs: 
  u1(t),          of size 1;
  u2(t,s2),       of size 1;

 Encountered 2 exogenous inputs: 
  w1(t),    of size 1;
  w2(t),    of size 1;

 Encountered 2 observed outputs: 
  y1(t,s1),       of size 1;
  y2(t,s2),       of size 1;

 Encountered 1 regulated output: 
  z(t),    of size 2;

 Encountered 8 boundary conditions: 
  F1(t,s2) = 0,   of size 1;
  F2(t,s2) = 0,   of size 1;
  F3(t,s1) = 0,   of size 1;
  F4(t,s1) = 0,   of size 1;
  F5(t,s2) = 0,   of size 1;
  F6(t,s2) = 0,   of size 1;
  F7(t,s1) = 0,   of size 1;
  F8(t,s1) = 0,   of size 1;
\end{verbatim}
\end{matlab}
Note that the inputs and outputs have been re-indexed as \texttt{u1} and \texttt{u2}, \texttt{w1} and \texttt{w2}, and \texttt{y1} and \texttt{y2}.
\end{Statebox}

\subsubsection{Reordering components}

In order to convert any PDE to a PIE, the state variables, inputs, and outputs have to be reordered to support representation in terms PI operators (or more accurately, in terms the \texttt{opvar} and \texttt{opvar2d} representing these operators numerically). This ordering is done primarily based on the dimension of the spatial domain on which the state, input, or output variable is defined, starting with finite-dimensional (ODE states) variables, followed by variables on a 1D domain (1D PDE states), followed by variables on a 2D domain (2D PDE states), etc. This reordering is done using the function \texttt{reorder\_comps}, taking a PDE structure and returning an equivalent representation wherein the different variables have been re-indexed to match the order necessary for representation as a PIE. 

\begin{Statebox}{\textbf{Example}}
Consider again the wave equation example from~\eqref{eq:alt_PDE_input:terms_input_PDE:example_PDE}, declared as an initialized \texttt{pde\_struct} object \texttt{PDE\_i}. Calling \texttt{reorder\_comps}, we get the following message
\begin{matlab}
\begin{verbatim}
 >> PDE_r = reorder_comps(PDE_i)

 The order of the state components x has not changed.
 The order of the exogenous inputs w has not changed.
 The order of the actuator inputs u has not changed.
 The order of the regulated outputs z has not changed.
 The order of the observed outputs y has not changed.
 The boundary conditions have been re-indexed as:
    BC3(t,s1)   -->   BC1(t,s1)
    BC4(t,s1)   -->   BC2(t,s1)
    BC7(t,s1)   -->   BC3(t,s1)
    BC8(t,s1)   -->   BC4(t,s1)
    BC1(t,s2)   -->   BC5(t,s2)
    BC2(t,s2)   -->   BC6(t,s2)
    BC5(t,s2)   -->   BC7(t,s2)
    BC6(t,s2)   -->   BC8(t,s2)
\end{verbatim}
\end{matlab}
Since the state components, inputs, and outputs have already been ordered correctly when declaring the system, they are not reordered here. However, the boundary conditions have been reordered, to first give the boundary conditions depending on the first spatial variable, \texttt{s1}, follows by the boundary conditions defined in terms of the second spatial variable, \texttt{s2}. Now, suppose we introduce an ODE state component $x_{3}(t)$, with dynamics $\dot{x}_{3}(t)=-x_{3}(t)$, added to the PDE structure as
\begin{matlab}
\begin{verbatim}
 >> x3 = pde_var();
 PDE_2 = initialize([PDE;diff(x3,'t')==-x3]);

 Encountered 3 state components: 
 x1(t,s1,s2), of size 1, differentiable up to order (2,2) in variables (s1,s2);
 x2(t,s1,s2), of size 1, differentiable up to order (2,2) in variables (s1,s2);
 x3(t),       of size 1, finite-dimensional;
\end{verbatim}
\end{matlab}
In this system, the finite-dimensional state variable appears last, whereas it must appear first in the PIE representation. As such, calling now the function \texttt{reorder\_comps}, we get a message
\begin{matlab}
\begin{verbatim}
 >> PDE_r2 = reorder_comps(PDE_2);

The state components have been re-indexed as:
   x3(t)          -->   x1(t)
   x1(t,s1,s2)    -->   x2(t,s1,s2)
   x2(t,s1,s2)    -->   x3(t,s1,s2)
\end{verbatim}
\end{matlab}
indicating that in the updated system structure, the state variables have been reordered to place the finite-dimensional state first.
\end{Statebox}

\subsubsection{Combining spatial variables}

Although \texttt{pde\_struct} can be used to declare systems involving an arbitrary number of spatial variables, PIETOOLS 2025 does not support conversion of systems involving more than 2 spatial variables to PIEs. However, in some cases, it is possible to express a higher-dimensional PDE using fewer spatial variables, by rescaling the domain on which each variable is defined. For \texttt{pde\_struct} objects, this can be achieved using the function \texttt{combine\_vars} as
\begin{matlab}
\begin{verbatim}
 >> PDE_new = combine_vars(PDE_old,[a,b]);
\end{verbatim}
\end{matlab}
taking a PDE structure representing some desired equation, and returning a structure defining an equivalent representation of the system, but now with all variables rescaled to exist on the interval \texttt{[a,b]}, and merged where possible. To illustrate, consider the following simple system of three coupled transport equations
\begin{align*}
    \partial_{t}\phi_{1}(t,s_{1})&=\partial_{s_{1}}\phi_{1}(t,s_{1}), &   s_{1}&\in[0,1],~t\geq 0 \\
    \partial_{t}\phi_{2}(t,s_{2})&=\partial_{s_{2}}\phi_{2}(t,s_{1}), &   s_{2}&\in[0,2], \\
    \partial_{t}\phi_{3}(t,s_{3})&=\partial_{s_{3}}\phi_{3}(t,s_{1}), &   s_{3}&\in[0,3], \\
    \mbf{x}_{1}(t,0)&=\mbf{x}_{2}(t,0)=\mbf{x}_{3}(t,0)=0.
\end{align*}
We can declare this system using three different spatial variables as
\begin{matlab}
\begin{verbatim}
 >> pvar s1 s2 s3
 >> phi1 = pde_var(s1,[0,1]); phi2 = pde_var(s2,[0,2]); phi3 = pde_var(s3,[0,3]);
 >> PDE_alt = [diff(phi1,'t')==diff(phi1,s1);
               diff(phi2,'t')==diff(phi2,s2);
               diff(phi3,'t')==diff(phi3,s3);
               subs(phi1,s1,0)==0; subs(phi2,s2,0)==0; subs(phi3,s3,0)==0];
 >> PDE_alt = initialize(PDE_alt)
 Warning: Currently, PIETOOLS supports only problems with at most two distinct
 spatial variables. Analysis of the returned PDE structure will not be possible. 
 Try running "combine_vars" to reduce the dimensionality of your problem. 

 Encountered 3 state components: 
  x1(t,s1), of size 1, differentiable up to order (1) in variables (s1);
  x2(t,s2), of size 1, differentiable up to order (1) in variables (s2);
  x3(t,s3), of size 1, differentiable up to order (1) in variables (s3);

 Encountered 3 boundary conditions: 
  F1(t) = 0, of size 1;
  F2(t) = 0, of size 1;
  F3(t) = 0, of size 1;

   d_t x1(t,s1) = d_s1 x1(t,s1);
   d_t x2(t,s2) = d_s2 x2(t,s2);
   d_t x3(t,s3) = d_s3 x3(t,s3);

   0 = x1(t,0);
   0 = x2(t,0);
   0 = x3(t,0);
\end{verbatim}
\end{matlab}
Upon initializing this system, PIETOOLS already throws a warning that the system cannot be analyzed with PIETOOLS in this form, and recommends running \texttt{combine\_vars} to reduce the dimensionality of the problem. Obeying our computer overlords, we run this function to combine the variables to exist on the domain $[0,1]$ as
\begin{matlab}
\begin{verbatim}
 >> PDE_alt = combine_vars(PDE_alt,[0,1])

 Variables (s1,s2) have been merged with variables (s3,s3) respectively.

 All spatial variables have been rescaled to exist on the interval [0,1].

   d_t x1(t,s) = d_s x1(t,s);
   d_t x2(t,s) = 0.5 * d_s x2(t,s);
   d_t x3(t,s) = 0.33333 * d_s x3(t,s);

   0 = x1(t,0);
   0 = x2(t,0);
   0 = x3(t,0);
\end{verbatim}
\end{matlab}
In doing so, the spatial variables $s_{1}$ and $s_{2}$ are both merged with the spatial variable $s_{3}$, all now converted to a single variable $s$ on the domain $[0,1]$. Accordingly, the differential equations are now also expressed in terms of this single spatial variable, introducing e.g. $\mbf{x}_{2}(t,s)=\phi_{2}(t,2s)$ for $s\in[0,1]$, so that $\partial_{s_{2}}\phi_{2}(t,s_{2})$ becomes $\frac{1}{2}\partial_{s}\mbf{x}_{2}(t,s)$. In this manner, the returned system offers an equivalent representation of the original PDE, through suitable state transformation. Of course, if the system involves distributed inputs or outputs, those will be transformed accordingly as well. 

Note that, even for systems that do not allow for spatial variables to be merged, the function \texttt{combine\_vars} can still be used to rescale all variables to exist on the same domain, which may reduce numerical issues in e.g. stability analysis later on.

\subsubsection{Expanding delays}

Unlike our PDE representation, the PIE representation does not allow for temporal delay to occur in the state variables or inputs. To resolve this, any delay in the PDE is instead modeled using a transport equation. Specifically, for any variable $\mbf{v}(t-\tau,s)$, be it a state or input variable, we can introduce a state variable $\mbf{x}(t,r,s)=\mbf{v}(t-r\tau,s)$ for $r\in[0,1]$, so that we can express $\mbf{v}(t-\tau,s)=\mbf{x}(t,1,s)$ wherever it appears in the PDE dynamics. Here, the state variable $\mbf{x}(t)$ is governed by a transport equation
\begin{equation*}
    \partial_{t}\mbf{x}(t,r,s)=-\frac{1}{\tau}\partial_{r}\mbf{v}(t,r,s),\qquad
    \mbf{x}(t,0,s)=\mbf{v}(t,s).
\end{equation*}
This equation is just a standard PDE with a standard boundary conditions, that can be readily added to the PDE structure and converted to a PIE. In this manner, any temporal delay in a linear PDE can be equivalently represented by adding a suitable transport equation to the system, a process which can be performed for a \texttt{pde\_struct} object \texttt{PDE} by simply calling \texttt{PDE=expand\_delays(PDE)}. This function returns a new structure that offers an equivalent representation of the input system, but now with an added state component modeled by a transport equation for each delay present in the system.

Note that, using the function \texttt{expand\_delay} to expand a PDE state $\mbf{x}(t-\tau,s)$ with delay, the boundary conditions on the state $\mbf{x}(t)$ do not get automatically imposed on the newly introduces state $\hat{\mbf{x}}(t,r,s)=\mbf{x}(t-r\tau,s)$ modeling the delay. This may introduce conservatism when e.g. analyzing stability of the system.

\begin{Statebox}{\textbf{Example}}
Consider again the wave equation example from~\eqref{eq:alt_PDE_input:terms_input_PDE:example_PDE}, declared as an initialized \texttt{pde\_struct} object \texttt{PDE\_i}. As declared, the system involves a delayed input $\mbf{u}_{2}(t-0.5,s_{2})$ in the boundary conditions, which we can expand to get
\begin{matlab}
\begin{verbatim}
 >> PDE_d = expand_delays(PDE_i)

 Added 1 state components: 
    x3(t,s2,ntau_3)  := u2(t-ntau_3,s2);

 Variable s1 has been merged with variable ntau_3.

 All spatial variables have been rescaled to exist on the interval [-1,1].

  d_t x1(t,s1,s2) = x2(t,s1,s2);
  d_t x2(t,s1,s2) = 2.2222 * d_s1^2 x1(t,s1,s2) + 5 * d_s2^2 x1(t,s1,s2)
                        + C23(s1,s2) * u1(t);
  d_t x3(t,s1,s2) = 4 * d_s2 x3(t,s1,s2);

  y1(t,s1) = x1(t,s1,1) + C42(s1) * w1(t);
  y2(t,s2) = x2(t,1,s2) + w2(t);

  z(t) = int_-1^1 int_-1^1[C61 * x1(t,s1,s2)]ds2 ds1 
                        + C62 * d_s1 d_s2 x2(t,1,1);

  0 = x1(t,-1,s2) - x3(t,-1,s2);
  0 = x2(t,-1,s2);
  0 = x1(t,s1,-1);
  0 = x2(t,s1,-1);
  0 = 0.66667 * d_s1 x1(t,1,s2);
  0 = 0.66667 * d_s1 x2(t,1,s2);
  0 = d_s2 x1(t,s1,1);
  0 = d_s2 x2(t,s1,1);
  0 = u2(t,s2) - x3(t,1,s2);
\end{verbatim}
\end{matlab}
In the returned system, a new variable $\mbf{x}_{3}(t,s_{2},r)$ (where $r$ is \texttt{ntau\_3}) is introduced to model the delayed input $\mbf{u}(t-0.5,s_{2})$, adding a transport equation to model $\mbf{x}_{3}$, as well as a boundary condition $\mbf{u}_{2}(t,s_{2})=\mbf{x}_{3}(t,1,s_{2})$. However, since this also increases the number of spatial variables in the system, the function \texttt{combine\_vars} (presented in the previous subsection) is called automatically to rescale all variables to exist on the domain $[-1,1]$, and merge the new variable $r$ with the variable $s_{1}$ (since it cannot be merged with $s_{2}$). Accordingly, all the dynamics have been rescaled as well, so that e.g. the term $5\partial_{s_{1}}^2\mbf{x}(t,s_{1})$ becomes $5(\frac{2}{3})^2=\partial_{s_{1}}^2\mbf{x}_{1}(t,s_{1})$ to account for the fact that $s_{1}$ has been rescaled from the domain $[0,3]$ to the domain $[-1,1]$.

\end{Statebox}

\subsubsection{Expanding higher-order temporal derivatives}\label{subsec:alt_PDE_input:terms_input_PDE:expand_tderivatives}

As briefly illustrated earlier, the Command Line Input format can also be used to declare systems with higher-order temporal derivatives. However, the PIE representation does not actually support such higher-order temporal derivatives. Nevertheless, we can easily get around this issue by introducing additional state variables. For example, consider the following wave equation
\begin{align*}
    &\partial_{t}^2\mbf{x}_{1}(t,s_{1},s_{2})=5\bl(\partial_{s_{1}}^2\mbf{x}_{1}(t,s_{1},s_{2}) +\partial_{s_{2}}^2\mbf{x}_{1}(t,s_{1},s_{2})\br),\qquad (s_{1},s_{2})\in[0,3]\times[-1,1],~t\geq 0,  \\
    &\mbf{x}_{1}(t,0,s_{2})=\partial_{s_{1}}\mbf{x}_{1}(t,3,s_{2})=0,    \\
    &\mbf{x}_{1}(t,s_{1},-1)=\partial_{s_{2}}\mbf{x}_{1}(t,s_{1},1)=0,
\end{align*}
which we can declare as
\begin{matlab}
\begin{verbatim}
 >> PDE_w = [diff(x1,'t',2)==5*(diff(x1,s1,2)+diff(x1,s2,2));
             subs(x1,s1,0)==0;  subs(diff(x1,s1),s1,3)==0;
             subs(x1,s2,-1)==0; subs(diff(x1,s2),s2,1)==0];
 >> PDE_w = initialize(PDE_w)

 Encountered 1 state component: 
  x(t,s1,s2),    of size 1, differentiable up to order (2,2) in variables (s1,s2);

 Encountered 4 boundary conditions: 
  F1(t,s2) = 0,   of size 1;
  F2(t,s2) = 0,   of size 1;
  F3(t,s1) = 0,   of size 1;
  F4(t,s1) = 0,   of size 1;

   d_t^2 x(t,s1,s2) = 5 * d_s1^2 x(t,s1,s2) + 5 * d_s2^2 x(t,s1,s2);

   0 = x(t,0,s2);
   0 = d_s1 x(t,3,s2);
   0 = x(t,s1,-1);
   0 = d_s2 x(t,s1,1);
\end{verbatim}
\end{matlab}
Now, to get rid of the second-order temporal derivative in this system, we can introduce a new state variable $\mbf{x}_{2}(t)=\partial_{t}\mbf{x}_{1}(t)$. Then, the wave equation can be equivalently expressed in a format involving only first-order temporal derivatives as
\begin{align*}
    &\partial_{t}\mbf{x}_{1}(t,s_{1},s_{2})=\mbf{x}_{2}(t,s_{1},s_{2}),     \hspace*{4.5cm}(s_{1},s_{2})\in[0,3]\times[-1,1],~t\geq 0,  \\
    &\partial_{t}^2\mbf{x}_{1}(t,s_{1},s_{2})=5\bl(\partial_{s_{1}}^2\mbf{x}_{1}(t,s_{1},s_{2}) +\partial_{s_{2}}^2\mbf{x}_{1}(t,s_{1},s_{2})\br)    \\
    &\mbf{x}_{1}(t,0,s_{2})=\partial_{s_{1}}\mbf{x}_{1}(t,3,s_{2})=\mbf{x}_{2}(t,0,s_{2})=\partial_{s_{1}}\mbf{x}_{2}(t,3,s_{2})=0,    \\
    &\mbf{x}_{1}(t,s_{1},-1)=\partial_{s_{2}}\mbf{x}_{1}(t,s_{1},1)=\mbf{x}_{2}(t,s_{1},-1)=\partial_{s_{2}}\mbf{x}_{2}(t,s_{1},1)=0.
\end{align*}
Given the \texttt{pde\_struct} object \texttt{PDE\_w} representing our wave equation, we can similarly expand this system in a format involving only first-order temporal derivatives, using the function \texttt{expand\_tderivatives} as
\begin{matlab}
\begin{verbatim}
 >> PDE_w2 = expand_tderivatives(PDE_w)

 Added 1 state component: 
    x2(t,s1,s2)   := d_t x1(t,s1,s2)

  d_t x1(t,s1,s2) = x2(t,s1,s2);
  d_t x2(t,s1,s2) = 5 * d_s1^2 x1(t,s1,s2) + 5 * d_s2^2 x1(t,s1,s2);

  0 = x1(t,0,s2);
  0 = d_s1 x1(t,3,s2);
  0 = x1(t,s1,-1);
  0 = d_s2 x1(t,s1,1);
  0 = x2(t,0,s2);
  0 = d_s1 x2(t,3,s2);
  0 = x2(t,s1,-1);
  0 = d_s2 x2(t,s1,1);
\end{verbatim}
\end{matlab}
The resulting structure represents the expected (expanded) PDE dynamics in terms of $\mbf{x}_{1}(t)$ and $\mbf{x}_{2}(t)=\partial_{t}\mbf{x}_{1}(t)$. Here, the same (homogeneous) boundary conditions will be imposed upon the new state variable $\mbf{x}_{2}(t)$ as imposed on the original state variable, noting that if e.g. $\mbf{x}_{1}(t,0,s_{2})=0$ then also $\partial_{t}\mbf{x}_{1}(t,0,s_{2})=0$ for all $t$ and $s_{2}$. If desired, an alternative version of this function is also available, \texttt{expand\_tderivatives\_noBCs}, which converts the PDE to a first-order in time system without imposing boundary conditions on the newly introduced state variable, $\mbf{x}_{2}(t)$.

\begin{boxEnv}{\textbf{WARNING:}}
For a \texttt{pde\_struct} object \texttt{PDE} involving higher-order temporal derivatives, running\\ \texttt{PIE1=convert(PDE)} will produce a different PIE representation from that obtained by running \texttt{PDE2=expand\_tderivatives(PDE)} and \texttt{PIE2=convert(PDE2)}. See Subsection~\ref{subsec:PIE:PDE2PIE:pde_2_pie:wave} for more information on how \texttt{convert} accounts for higher-order temporal derivatives.
\end{boxEnv}

\section{The \texttt{sys} Format for 1D ODE-PDEs}\label{sec:alt_PDE_input:sys}
In PIETOOLS 2022, a Command Line Parser was introduced for declaration of 1D ODE-PDE systems, using the \texttt{sys} and \texttt{state} structures. Although the \texttt{pde\_var} function has replaced this \texttt{sys}-based input format in PIETOOLS 2025, 1D ODE-PDEs can still be declared using this older Command Line Parser format as well. In this section, we give an overview of how the \texttt{state} and \texttt{sys} classes can be used for defining and manipulating 1D ODE-PDE systems. Furthermore, we will also specify valid modes of manipulating these objects in MATLAB and potential caveats while using these objects.

\subsection{\texttt{state} class objects}\label{ss-sec:state}
All symbols used to define a systems are either \texttt{polynomial} type (part of SOSTOOLS) or \texttt{state} type (part of PIETOOLS). Here, we will focus on \texttt{state} class objects and methods defined for such objects. First, any \texttt{state} class object has the following properties that can be freely accessed (but \textbf{should not} be modified directly).
\begin{Statebox}{\texttt{state}}
This class has the following properties:
\begin{enumerate}
\item \texttt{type}: Type of variable; It is a cell array of strings that can take values in $\{'ode'$ $,'pde'$ $,'out'$ $,'in'\}$
\item \texttt{veclength}: Positive integer 
\item \texttt{var}: Cell array of polynomial row vectors (Multipoly class object)
\item \texttt{diff\_order}: Cell array of non-negative integers (same size as var)
\end{enumerate}
\end{Statebox}

The first independent variable stored in each row of the \texttt{state.var} cell structure is always the time variable \texttt{t}. Spatial variables are stored in location 2 and on-wards. For example,
\begin{matlab}
>> X = state('pde'); x = state('ode');\\
>> X.var
\begin{verbatim}
ans =
  [ t, s]
\end{verbatim}
>> x.var
\begin{verbatim}
ans =
  [ t ]
\end{verbatim}
\end{matlab}

Differentiation information is stored as a cell array where the cell structure has the same size as \texttt{state.var} with non-negative integers specifying order of differentiation w.r.t. the independent variable based on the location.
For the above example, we have
\begin{matlab}
>> X.diff\_order
\begin{verbatim}
ans =
  [ 0, 0]
\end{verbatim}
>> y = diff(X,s,2);\\
>> y.diff\_order
\begin{verbatim}
ans =
  [ 0, 2]
\end{verbatim}
\end{matlab}

Note, user can indeed edit these properties directly by assignment. For example, the code
\begin{matlab}
>> x = state('pde');\\
>> x.diff\_order = [0,2];
\end{matlab}
defines the symbol \texttt{x} as a function $x(t,s)$, and converts it to the second derivative $\partial_s^2 x(t,s)$. This is same as the code
\begin{matlab}
>> x = state('pde'); \\
>> x = diff(x,s,2);
\end{matlab}
Since, this permanently changes \texttt{x} to its second spatial derivative in the workspace, such direct manipulation of the properties should be avoided at all costs. 

\paragraph{Declaring/initializing state variables} 
The initialization function \texttt{state()} takes two input arguments (both are optional):
\begin{itemize}
    \item type: The argument is reserved to specification of the type of the state object (defaults to 'ode', if not specified)
    \item veclength: The size of the vector-valued state (defaults to one, if not specified)
\end{itemize}
\begin{matlab}
d = state('pde',3);
\end{matlab}
Alternatively, multiple states can be defined collectively using the command shown below, however, all such states will default to the type 'ode' and length 1.
\begin{matlab}
state ~a ~b ~c;
\end{matlab}

\paragraph{Operations on state class objects} All of the following operations should give us a \texttt{terms} (an internal class that cannot be accessed or modified by users) class object which is defined by some PI operator times a vector of states. Operators/functions that are used to manipulate \texttt{state} objects are:
\begin{enumerate}
\item addition: \texttt{x+y} or  \texttt{x-y}  
\item multiplication: \texttt{K*x}
\item vertical concatenation: \texttt{[x;y]}
\item differentiation: \texttt{diff(x,s,3)}
\item integration: \texttt{int(x,s,[0,s])}
\item substitution: \texttt{subs(x,s,0)}
\end{enumerate}

\paragraph{Caveats in operations on state class objects} While manipulation of state class objects, the users must adhere the following rules stated in the table \ref{tab:invalid_state_operations}. All the operations listed in the table are invalid.

Addition of time derivatives is not allowed, since that usually leads to a descriptor dynamical PDE system which is not supported by PIETOOLS.
For example, consider the following PDE.
\begin{align*}
    \dot{\mbf x}(t) + \dot{\mbf y}(t) &=\partial_s^2\mbf x(t,s)\\
    2\dot{\mbf y}(t) &=5\partial_s^2\mbf y(t,s).
\end{align*}
\textbf{This PDE cannot be implemented directly using the command line parser.}
Since, the left hand side of the equation has a coefficient different from identity, the user needs to first separate it as
\begin{align*}
    \dot{\mbf x}(t) &=\partial_s^2\mbf x(t,s)-2.5\partial_s^2\mbf y(t,s)\\
    \dot{\mbf y}(t) &=2.5\partial_s^2\mbf y(t,s).
\end{align*}
Now, we can define this PDE using the following code:
\begin{matlab}
\begin{verbatim}
>> pvar t s; 
>> x = state('pde');  y = state('pde');\
>> odepde= sys(); 
>> odepde = addequation(odepde, diff(x,t)-diff(x,s,2)-2.5*diff(y,s,2));
>> odepde = addequation(odepde, diff(y,t)-2.5*diff(y,s,2));
\end{verbatim}
\end{matlab}

Likewise, we do not permit adding outputs with outputs, outputs with time derivatives, or right multiplication which also lead to descriptor type systems. Coupling on left hand side of these equations must be manually resolved before defining the PDE in PIETOOLS.

Other limitations to note are, PIETOOLS does not support temporal-spatial mixed derivatives, integration in time, and evaluation of functions at specific time or inside the spatial domain. For example, for a state $x(t,s)$ with $s\in[0,1]$ we cannot find $x(t=2,s)$ or $x(t,s=0.5)$. $x$ can only be evaluated at the boundary $s=0$ or $s=1$.

\begin{table}
\renewcommand{\arraystretch}{0.3}
\begin{tabular}{ | m{8em} | m{12cm} |} 
\hline 
\vspace*{0.15cm} Operation type & 
\vspace*{0.15cm} Incorrect or `not-permitted' operations\\[0.6em]
  \hline
 Addition & \begin{enumerate}[label=\textcolor{red}{\ding{54}}]
\item Adding two time derivatives: \texttt{ diff(x,t)+diff(x,t)}
\item Adding two outputs: \texttt{z1+z2}
\item Adding time derivative and outptu: \texttt{ diff(x,t)+z}
\end{enumerate}\\
  \hline
 Multiplication&\begin{enumerate}[label=\textcolor{red}{\ding{54}}]
\item Multiplying two states: \texttt{x*x}
\item Multiplying non-identity with time derivative/output: \texttt{ 2*diff(x,t)} or \texttt{ -1*z}
\item Right multiplication: \texttt{x*3} 
\end{enumerate}\\
  \hline
  Differentiation&\begin{enumerate}[label=\textcolor{red}{\ding{54}}]
\item Higher order time derivatives:\texttt{ diff(x,t,2)}
\item Mixed derivatives of space and time: \texttt{ diff(diff(x,t),s,2)}
\end{enumerate}\\
  \hline
  Substitution&\begin{enumerate}[label=\textcolor{red}{\ding{54}}]
\item Substituting a double for time variable: \texttt{ subs(x,t,2)}
\item Substituting positive time delay: \texttt{ subs(x,t,t+5)}
\item Substitution values other than \texttt{pvar} variable or boundary values
\end{enumerate}\\
  \hline
  Integration&\begin{enumerate}[label=\textcolor{red}{\ding{54}}]
\item Integration of time variable: \texttt{ int(x,t,[0,5])}
\item both limits being non-numeric: \texttt{ int(x,s,[theta,eta])}
\item limit same as variable of integration: \texttt{ int(x,s,[s,1])}
\end{enumerate}\\
  \hline
  Concatenation&
  \begin{enumerate}[label=\textcolor{red}{\ding{54}}]
    \item Horizontal concatenation: \texttt{ [x,x]}
    \item Blank spaces in vertical concatentation: \texttt{[x + y; z]}
\end{enumerate}\\
\hline
\end{tabular}
\caption{This table lists all the invalid forms of operations on \texttt{state} class objects. The left column specifies the type of operation whereas the right column lists the operations that are \textbf{INVALID} for that `type' of operation.}
\label{tab:invalid_state_operations}
\end{table}


\subsection{\texttt{sys} class objects}\label{subsec:sys}

\begin{Statebox}{\texttt{sys}}
This class has following accessible properties:
\begin{itemize}
    \item \texttt{equation}: stores all the equations added to the system object in a column vector where every row is an equation with zero on the right hand side (i.e., \texttt{row(i)=0} for every \texttt{i})
    \item \texttt{type}: type of the system (currently supports `pde' and `pie')
    \item \texttt{params}: either a \texttt{pde\_struct} or \texttt{pie\_struct} object
    \item \texttt{dom}: a $1\times 2$ vector double (value of first element must be strictly smaller than that of second element)
    \item Other hidden properties:
    \begin{enumerate}
        \item \texttt{states}: a vector of all states, inputs, outputs appearing in the equation property    
        \item \texttt{ControlledInputs}: A vector with length same as the states property with 0 or 1 value. This vector specifies whether a state is a controlled input or not.
        \item \texttt{ObservedOutputs}: A vector with length same as the states property with 0 or 1 value. This specifies whether a state is an observed output or not.
    \end{enumerate}
\end{itemize}
\end{Statebox}

\paragraph{\texttt{sys} class methods} Methods used to modify a \texttt{sys()} object are listed below.
\begin{itemize}
\item \texttt{addequation}: adds an equation to the \texttt{obj.equation} property; syntax \texttt{addequation(obj, eqn)}
\item \texttt{removeequation}: removes equation in row \texttt{i} from the \texttt{obj.equation} property; syntax \texttt{removeequation(obj,i)}
\item \texttt{setControl}: sets a chosen state \texttt{x} as a control input; syntax \texttt{setControl(obj,x)}
\item \texttt{setObserve}: sets a chosen state \texttt{x} as an observed output; syntax \texttt{setObserve(obj,x)}
\item \texttt{removeControl}: removes a chosen state \texttt{x} from the set of control inputs; syntax

\texttt{removeControl(obj,x)}
\item removeObserve: removes a chosen state \texttt{x} from the set of observed outputs; syntax \texttt{removeObserve(obj,x)}
\item getParams: parses symbolic equations from \texttt{obj.equation} property to get \texttt{pde\_struct} object which is stored in \texttt{obj.params}; syntax \texttt{getParams(obj)}
\item convert: converts \texttt{obj.params} from \texttt{pde\_struct} to \texttt{pie\_struct} object; syntax 

\texttt{convert(obj,'pie')}
\end{itemize}

\begin{boxEnv}{\textbf{WARNING:}}
\texttt{sys} class object properties should not be modified directly (unless you know what you are doing); Use the methods provided above.
\end{boxEnv}

\chapter{Batch Input Formats for Time-Delay Systems}\label{ch:alt_DDE_input}

In Chapter~\ref{ch:PDE_DDE_representation}, we showed how time-delay systems (TDSs) can be implemented as delay differential equations (DDEs) in PIETOOLS. In that chapter, we further hinted at the fact that PIETOOLS also allows TDSs to be declared in two alternative representations: as neutral type systems (NDSs) and as differential difference equations (DDFs). In this chapter, we will provide more details on how to work with such NDS and DDF systems in PIETOOLS. In particular, in Section~\ref{sec:alt_DDE_input:TDS_formats}, we recall the DDE representation, and show what NDS and DDF systems look like, and how systems of each type can be declared in PIETOOLS. In Section~\ref{ch:alt_DDE_inputs:convert}, we then show how NDS and DDE systems can be converted to the DDF representation in PIETOOLS, and how each type of TDS can be converted to a PIE.

\section{Representing Systems with Delay}\label{sec:alt_DDE_input:TDS_formats}
In this section, we show how time-delay systems in DDE, NDS and DDF representation can be declared in PIETOOLS, focusing on DDE systems in Subsection~\ref{subsec:DDE_format}, NDS systems in Subsection~\ref{subsec:NDS_format}, and DDF systems in Subsection~\ref{subsec:DDF_format}. For more information on how to declare systems in DDE representation in PIETOOLS, we refer to Section~\ref{sec:PDE_DDE_representation:DDEs}

\subsection{Input of Delay Differential Equations}\label{subsec:DDE_format}

The DDE data structure allows the user to declare any of the matrices in the following general form of Delay-Differential equation.
\begin{align}
	&\bmat{\dot{x}(t)\\z(t) \\ y(t)}=\bmat{A_0 & B_{1} & B_{2}\\ C_{1} & D_{11} &D_{12}\\ C_{2} & D_{21} &D_{22}}\bmat{x(t)\\w(t)\\u(t)}+\sum_{i=1}^K \bmat{A_i & B_{1i} & B_{2i}\\C_{1i} & D_{11i} & D_{12i}\\C_{2i} & D_{21i} & D_{22i}} \bmat{x(t-\tau_i)\\w(t-\tau_i)\\u(t-\tau_i)}\notag \\
& \hspace{2cm}+\sum_{i=1}^K \int_{-\tau_i}^0\bmat{A_{di}(s) & B_{1di}(s) &B_{2di}(s)\\C_{1di}(s) & D_{11di}(s) & D_{12di}(s)\\C_{2di}(s) & D_{21di}(s) & D_{22di}(s)} \bmat{x(t+s)\\w(t+s)\\u(t+s)}ds \label{eqn:DDE}
\end{align}
In this representation, it is understood that
\begin{itemize}
\item The present state is $x(t)$.\vspace{-2mm}
\item The disturbance or exogenous input is $w(t)$. These signals are not typically known or alterable. They can account for things like unmodelled dynamics, changes in reference, forcing functions, noise, or perturbations.\vspace{-2mm}
\item The controlled input is $u(t)$. This is typically the signal which is influenced by an actuator and hence can be accessed for feedback control. \vspace{-2mm}
\item The regulated output is $z(t)$. This signal typically includes the parts of the system to be minimized, including actuator effort and states. These signals need not be measured using senors.\vspace{-2mm}
\item The observed or sensed output is $y(t)$. These are the signals which can be measured using sensors and fed back to an estimator or controller.\vspace{-2mm}
\end{itemize}
To add any term to the DDE structure, simply declare is value. For example, to represent 
\[
\dot x(t)=-x(t-1),\qquad z(t)=x(t-2)
\]
we use
	\begin{matlab}
		>> DDE.tau = [1 2];\\
		>> DDE.Ai\{1\} = -1;\\
		>> DDE.C1i\{2\} = 1;
	\end{matlab}
All terms not declared are assumed to be zero. The exception is that we require the user to specify the values of the delay in \texttt{DDE.tau}. When you are done adding terms to the DDE structure, use the function \texttt{DDE=PIETOOLS\_initialize\_DDE(DDE)}, which will check for undeclared terms and set them all to zero. It also checks to make sure there are no incompatible dimensions in the matrices you declared and will return a warning if it detects such malfeasance. The complete list of terms and DDE structural elements is listed in Table~\ref{tab:DDE_parameters}.

\begin{table}[ht!]\vspace{-2mm}
\begin{center}{
\begin{tabular}{c|c||c|c||c|c}
  \multicolumn{6}{c}{\textbf{ODE Terms:}}\\
Eqn.~\eqref{eqn:DDE}  & \texttt{DDE.}  &   Eqn.~\eqref{eqn:DDE}  & \texttt{DDE.} &   Eqn.~\eqref{eqn:DDE}  & \texttt{DDE.}\\
\hline
$A_0$    & \texttt{A0} & $B_{1}$ & \texttt{B1} &$B_{2}$&\texttt{B2}\\ 
$C_{1}$ & \texttt{C1} & $D_{11}$ &\texttt{D11}&$D_{12}$&\texttt{D12}\\ 
$C_{2}$ & \texttt{C2} & $D_{21}$ &\texttt{D21}&$D_{22}$&\texttt{D22} \\ \hline \\
 \multicolumn{6}{c}{ \textbf{Discrete Delay Terms:}}\\
Eqn.~\eqref{eqn:DDE}  & \texttt{DDE.}  &   Eqn.~\eqref{eqn:DDE}  & \texttt{DDE.} &   Eqn.~\eqref{eqn:DDE}  & \texttt{DDE.}\\
\hline
$A_i$    & \texttt{Ai\{i\}} & $B_{1i}$ & \texttt{B1i\{i\}} &$B_{2i}$&\texttt{B2i\{i\}}\\ 
$C_{1i}$ & \texttt{C1i\{i\}} & $D_{11i}$ &\texttt{D11i\{i\}}&$D_{12i}$&\texttt{D12i\{i\}}\\ 
$C_{2i}$ & \texttt{C2i\{i\}} & $D_{21i}$ &\texttt{D21i\{i\}}&$D_{22i}$&\texttt{D22i\{i\}}\\
\hline \\  \multicolumn{6}{c}{\textbf{Distributed Delay Terms: May be functions of \texttt{pvar s}}}\\
Eqn.~\eqref{eqn:DDE}  & \texttt{DDE.}  &   Eqn.~\eqref{eqn:DDE}  & \texttt{DDE.} &   Eqn.~\eqref{eqn:DDE}  & \texttt{DDE.}\\
\hline
$A_{di} $   & \texttt{Adi\{i\}} & $B_{1di}$ & \texttt{B1di\{i\}} &$B_{2di}$&\texttt{B2di\{i\}}\\ 
$C_{1di} $& \texttt{C1di\{i\}} &$ D_{11di} $&\texttt{D11di\{i\}}&$D_{12di}$&\texttt{D12di\{i\}}\\ 
$C_{2di}$ & \texttt{C2di\{i\}} & $D_{21di} $&\texttt{D21di\{i\}}&$D_{22di}$&\texttt{D22di\{i\}}\\
\end{tabular}
}
\end{center}\vspace{-2mm}
\caption{ Equivalent names of Matlab elements of the \texttt{DDE} structure terms for terms in Eqn.~\eqref{eqn:DDE}. For example, to set term \texttt{XX} to \texttt{YY}, we use \texttt{DDE.XX=YY}.  In addition, the delay $\tau_i$ is specified using the vector element \texttt{DDE.tau(i)} so that if $\tau_1=1, \tau_2=2, \tau_3=3$, then \texttt{DDE.tau=[1 2 3]}. }\label{tab:DDE_parameters}\end{table}

\subsubsection{Initializing a DDE Data structure} The user need only add non-zero terms to the DDE structure. All terms which are not added to the data structure are assumed to be zero. Before conversion to another representation or data structure, the data structure will be initialized using the command
	\begin{flalign*}
		&\texttt{DDE = initialize\_PIETOOLS\_DDE(DDE)}&
	\end{flalign*}
This will check for dimension errors in the formulation and set all non-zero parts of the \texttt{DDE} data structure to zero. Not that, to make the code robust, all PIETOOLS conversion utilities perform this step internally.

\subsection{Input of Neutral Type Systems}\label{subsec:NDS_format}
The input format for a Neutral Type System (NDS) is identical to that of a DDE except for 6 additional terms: 
\begin{align}
	\bmat{\dot{x}(t)\\z(t) \\ y(t)}&=\bmat{A_0 & B_{1} & B_{2}\\ C_{1} & D_{11} &D_{12}\\ C_{2} & D_{21} &D_{22}}\bmat{x(t)\\w(t)\\u(t)}+\sum_{i=1}^K \bmat{A_i  & B_{1i} & B_{2i}& E_i\\C_{1i}& D_{11i} & D_{12i} & E_{1i}\\C_{2i} & D_{21i} & D_{22i}&E_{2i}} \bmat{x(t-\tau_i)\\w(t-\tau_i)\\u(t-\tau_i)\\ \dot x(t-\tau_i)}\notag\\[-3mm]
& +\hspace{-1mm}\sum_{i=1}^K \hspace{0mm}\int_{-\tau_i}^0\hspace{-1mm}\bmat{A_{di}(s) & \hspace{-1mm}B_{1di}(s) &\hspace{-1mm}B_{2di}(s)& \hspace{-1mm}E_{di}(s)\\C_{1di}(s) & \hspace{-1mm}D_{11di}(s) & \hspace{-1mm}D_{12di}(s)& \hspace{-1mm}E_{1di}(s)\\C_{2di}(s) &\hspace{-1mm}D_{21di}(s) & \hspace{-1mm}D_{22di}(s)& \hspace{-1mm}E_{2di}(s)} \hspace{-2mm}\bmat{x(t+s)\\w(t+s)\\u(t+s)\\ \dot x(t+s)}\hspace{-1mm}ds\label{eqn:NDS}
\end{align}
These new terms are parameterized by $E_i,E_{1i}$, and $E_{2i}$ for the discrete delays and by $E_{di},E_{1di}$, and $E_{2di}$ for the distributed delays. As for the DDE case, these terms should be included in a NDS object as, e.g. \texttt{NDS.E\{1\}=1}. 

\begin{table}[hb]\vspace{-2mm}
\begin{center}{\small
\begin{tabular}{c|c||c|c||c|c||c|c}
  \multicolumn{8}{c}{\textbf{ODE Terms:}}\\
Eqn.~\eqref{eqn:NDS}  & \texttt{NDS.}  &   Eqn.~\eqref{eqn:NDS}  & \texttt{NDS.} &   Eqn.~\eqref{eqn:NDS}  & \texttt{NDS.}\\
\hline
$A_0$    & \texttt{A0} & $B_{1}$ & \texttt{B1} &$B_{2}$&\texttt{B2}&&\\ 
$C_{1}$ & \texttt{C1} & $D_{11}$ &\texttt{D11}&$D_{12}$&\texttt{D12}&&\\ 
$C_{2}$ & \texttt{C2} & $D_{21}$ &\texttt{D21}&$D_{22}$&\texttt{D22}&& \\ \hline \\
 \multicolumn{8}{c}{ \textbf{Discrete Delay Terms:}}\\
Eqn.~\eqref{eqn:NDS}  & \texttt{NDS.}  &   Eqn.~\eqref{eqn:NDS}  & \texttt{NDS.} &   Eqn.~\eqref{eqn:NDS}  & \texttt{NDS.}&   Eqn.~\eqref{eqn:NDS}  & \texttt{NDS.}\\
\hline
$A_i$    & \texttt{Ai\{i\}} & $B_{1i}$ & \texttt{B1i\{i\}} &$B_{2i}$&\texttt{B2i\{i\}}&$E_i$& \texttt{Ei\{i\}}\\ 
$C_{1i}$ & \texttt{C1i\{i\}} & $D_{11i}$ &\texttt{D11i\{i\}}&$D_{12i}$&\texttt{D12i\{i\}}&$E_{1i}$& \texttt{E1i\{i\}}\\ 
$C_{2i}$ & \texttt{C2i\{i\}} & $D_{21i}$ &\texttt{D21i\{i\}}&$D_{22i}$&\texttt{D22i\{i\}}&$E_{2i}$& \texttt{E2i\{i\}}\\
\hline \\  \multicolumn{8}{c}{\textbf{Distributed Delay Terms: May be functions of \texttt{pvar s}}}\\
Eqn.~\eqref{eqn:NDS}  & \texttt{NDS.}  &   Eqn.~\eqref{eqn:NDS}  & \texttt{NDS.} &   Eqn.~\eqref{eqn:NDS}  & \texttt{NDS.}&   Eqn.~\eqref{eqn:NDS}  & \texttt{NDS.}\\
\hline
$A_{di} $   & \texttt{Adi\{i\}} & $B_{1di}$ & \texttt{B1di\{i\}} &$B_{2di}$&\texttt{B2di\{i\}}&$E_{di}$& \texttt{Edi\{i\}}\\ 
$C_{1di} $& \texttt{C1di\{i\}} &$ D_{11di} $&\texttt{D11di\{i\}}&$D_{12di}$&\texttt{D12di\{i\}}&$E_{1di}$& \texttt{E1di\{i\}}\\ 
$C_{2di}$ & \texttt{C2di\{i\}} & $D_{21di} $&\texttt{D21di\{i\}}&$D_{22di}$&\texttt{D22di\{i\}}&$E_{2di}$& \texttt{E2di\{i\}}\\
\end{tabular}
}
\end{center}\vspace{-2mm}
\caption{ Equivalent names of Matlab elements of the \texttt{NDS} structure terms for terms in Eqn.~\eqref{eqn:NDS}. For example, to set term \texttt{XX} to \texttt{YY}, we use \texttt{NDS.XX=YY}. In addition, the delay $\tau_i$ is specified using the vector element \texttt{NDS.tau(i)} so that if $\tau_1=1, \tau_2=2, \tau_3=3$, then \texttt{NDS.tau=[1 2 3]}. }\label{tab:NDS_parameters}\end{table}

\subsubsection{Initializing a NDS Data structure} The user need only add non-zero terms to the NDS structure. All terms which are not added to the data structure are assumed to be zero. Before conversion to another representation or data structure, the data structure will be initialized using the command
\begin{matlab}
\begin{verbatim}
 >> NDS = initialize_PIETOOLS_NDS(NDS);
\end{verbatim}
\end{matlab}
This will check for dimension errors in the formulation and set all non-zero parts of the \texttt{NDS} data structure to zero. Not that, to make the code robust, all PIETOOLS conversion utilities perform this step internally.


\subsection{The Differential Difference Equation (DDF) Format}\label{subsec:DDF_format}
A Differential Difference Equation (DDF) is a more general representation than either the DDE or NDS. Most importantly, unlike the DDE or NDS, it allows one to represent the structure of the delayed channels. As such, it is the only representation for which the minimal realization features of PIETOOLS are defined. Nevertheless, the general form of DDF is more compact that of the DDE or NDS. The distinguishing feature of the DDF is decomposition of the output signal from the ODE part into sub-components, $r_i$, each of which is delayed by amount $\tau_i$. Identification of these $r_i$ is often challenging and hence most users will input the system as an ODE or NDS and then convert to a minimal DDF representation. The form of a DDF is given as follows.
	
\begin{align}
	\bmat{\dot{x}(t)\\ z(t)\\y(t)\\r_i(t)}&=\bmat{A_0 & B_1& B_2\\C_1 &D_{11}&D_{12}\\C_2&D_{21}&D_{22}\\C_{ri}&B_{r1i}&B_{r2i}}\bmat{x(t)\\w(t)\\u(t)}+\bmat{B_v\\D_{1v}\\D_{2v}\\D_{rvi}} v(t) \notag\\
v(t)&=\sum_{i=1}^K C_{vi} r_i(t-\tau_i)+\sum_{i=1}^K \int_{-\tau_i}^0C_{vdi}(s) r_i(t+s)ds.\label{eqn:DDF}
\end{align}

As for a DDE or NDS, each of the non-zero parameters in Eqn.~\eqref{eqn:DDF} should be added to the \texttt{DDF} structure, along with the vector of values of the delays \texttt{DDF.tau}. The elements of the \texttt{DDF} structure which can be defined by the user are included in Table~\ref{tab:DDF_parameters}.

\begin{table}[hb]\vspace{-2mm}
\begin{center}{
\begin{tabular}{c|c||c|c||c|c||c|c}
  \multicolumn{6}{c}{\textbf{ODE Terms:}}\\
Eqn.~\eqref{eqn:DDF}  & \texttt{DDF.}  &   Eqn.~\eqref{eqn:DDF}  & \texttt{DDF.} &   Eqn.~\eqref{eqn:DDF}  & \texttt{DDF.}&   Eqn.~\eqref{eqn:DDF}  & \texttt{DDF.}\\
\hline
$A_0$    & \texttt{A0} & $B_{1}$ & \texttt{B1} &$B_{2}$&\texttt{B2}&$B_v$&\texttt{Bv}\\ 
$C_{1}$ & \texttt{C1} & $D_{11}$ &\texttt{D11}&$D_{12}$&\texttt{D12}&$D_{1v}$&\texttt{D1v}\\ 
$C_{2}$ & \texttt{C2} & $D_{21}$ &\texttt{D21}&$D_{22}$&\texttt{D22}&$D_{2v}$&\texttt{D2v} \\
$C_{ri}$ & \texttt{Cri\{i\}} & $B_{r1i}$ &\texttt{Br1i\{i\}}&$B_{r2i}$&\texttt{Br2i\{i\}}&$D_{rvi}$&\texttt{Drvi\{i\}} \\
\hline \\
 \multicolumn{6}{c}{ \textbf{Discrete Delay Terms:}}\\
Eqn.~\eqref{eqn:DDF}  & \texttt{DDF.}  &   &  &    & \\
\hline
$C_{vi}$    & \texttt{Cvi\{i\}} & & &&\\ 
\hline \\  \multicolumn{6}{c}{\textbf{Distributed Delay Terms: May be functions of \texttt{pvar s}}}\\
Eqn.~\eqref{eqn:DDF}  & \texttt{DDF.}  &    &  &    & \\
\hline
$C_{vdi}(s) $   & \texttt{Cvdi\{i\}} & &  &&\\ 
\end{tabular}
}
\end{center}\vspace{-2mm}
\caption{ Equivalent names of Matlab elements of the \texttt{DDF} structure terms for terms in Eqn.~\eqref{eqn:DDF}. For example, to set term \texttt{XX} to \texttt{YY}, we use \texttt{DDF.XX=YY}.  In addition, the delay $\tau_i$ is specified using the vector element \texttt{DDF.tau(i)} so that if $\tau_1=1, \tau_2=2, \tau_3=3$, then \texttt{DDF.tau=[1 2 3]}. }\label{tab:DDF_parameters}\end{table}

\subsubsection{Initializing a DDF Data structure} The user need only add non-zero terms to the DDF structure. All terms which are not added to the data structure are assumed to be zero. Before conversion to another representation or data structure, the data structure will be initialized using the command
\begin{matlab}
\begin{verbatim}
 >> DDF = initialize_PIETOOLS_DDF(DDF);
\end{verbatim}
\end{matlab}
This will check for dimension errors in the formulation and set all non-zero parts of the \texttt{DDF} data structure to zero. Not that, to make the code robust, all PIETOOLS conversion utilities perform this step internally.


\section{Converting between DDEs, NDSs, DDFs, and PIEs}\label{ch:alt_DDE_inputs:convert}
For a given delay system, there are several alternative representations of that system. For example, a DDE can be represented in the DDE, DDF, or PIE format. However, only the DDF and PIE formats allow one to specify structure in the delayed channels, which are infinite-dimensional. For that reason, it is almost always preferable to efficiently convert the DDE or NDS to either a DDF or PIE - as this will dramatically reduce computational complexity of the analysis, control, and simulation problems (assuming you have tools for analysis, control and simulation of DDFs and PIEs - which we do!). However, identifying an efficient DDF or PIE representation of a given DDF/NDS is laborious for large systems and requires detailed understanding of the DDF format. For this reason, we introduce a set of functions for automating this conversion process.

\subsection{DDF to PIE}\label{sec:DDF2pie2} To convert a DDF data structure to an equivalent PIE representation, we have two utilities which are typically called sequentially. The first uses the SVD to identify and eliminate unused delay channels. The second na\"ively converts a DDF to an equivalent PIE. 

\subsubsection{Minimal DDF Realization of a DDF} The typical first step in analysis, simulation and control of a DDF is elimination of unused delay channels. This is accomplished using the SVD to identify such channels in a DDF structure and output a smaller, equivalent DDF structure. To use this utility, simply declare your \texttt{DDF} and enter the command
\begin{matlab}
\begin{verbatim}
 >> DDF = minimize_PIETOOLS_DDF(DDF);
\end{verbatim}
\end{matlab}

\subsubsection{Converting a DDF to a PIE} Having constructed a minimal (or not) DDF representation of a DDE, NDS or DDF, the next step is conversion to an equivalent PIE. To use this utility, simply declare your \texttt{DDF} structure and enter the command
\begin{matlab}
\begin{verbatim}
 >> PIE = convert_PIETOOLS_DDF(DDF,'pie');
\end{verbatim}
\end{matlab}

\subsection{DDE to DDF or PIE} We next address the problem of converting a DDE data structure to a DDF or PIE data structure. Because the DDE representation does not allow one to represent structure, this conversion is almost always performed by first computing a minimized DDF representation using \texttt{minimize\_PIETOOLS\_DDE2DDF}, followed possibly by converting this DDF representation to a PIE. Both steps are included in the function \texttt{convert\_PIETOOLS\_DDE}, allowing the minimal DDF representation and PIE representation of a DDE structure \texttt{DDE} to be computed as
\begin{matlab}
\begin{verbatim}
 >> [DDF, PIE] = convert_PIETOOLS_DDE(DDE);
\end{verbatim}
\end{matlab}
Here, the function \texttt{convert\_PIETOOLS\_DDE} computes the minimal DDF representation by calling \texttt{minimal\_PIETOOLS\_DDE2DDF}, which uses the SVD to eliminate unused delay channels in the DDF - resulting in a much more compact representation of the same system. As such, the minimal DDF representation can be computed by calling \texttt{minimal\_PIETOOLS\_DDE2DDF} directly as
\begin{matlab}
\begin{verbatim}
 >> DDF = minimal_PIETOOLS_DDE2DDF(DDF);
\end{verbatim}
\end{matlab}
or by calling \texttt{convert\_PIETOOLS\_DDE} with a second argument \texttt{'ddf'}
\begin{matlab}
\begin{verbatim}
 >> DDF = convert_PIETOOLS_DDE(DDE,'dde');
\end{verbatim}
\end{matlab}
Similarly, if only the PIE representation is desired, the user can also call
\begin{matlab}
\begin{verbatim}
 >> PIE = convert_PIETOOLS_DDE(DDE,'pie');
\end{verbatim}
\end{matlab}
though the procedure for computing the PIE will still involve computing the DDF representation first.

\subsubsection{DDE direct to PIE [NOT RECOMMENDED!]} 
Although it should never be used in practice, we also include a utility to construct the equivalent na\"ive PIE representation of a DDE. This is occasionally useful for purposes of comparison. To use this utility, simply declare your \texttt{DDE} and enter the command
\begin{matlab}
\begin{verbatim}
 >> DDF = convert_PIETOOLS_DDE2PIE_legacy(DDE);
\end{verbatim}
\end{matlab}
Because of the limited utility of the unstructured representation, we have not included a na\"ive DDE to DDF utility.

\subsection{NDS to DDF or PIE} 
Finally, we next address the problem of converting a NDS data structure to a DDF or PIE data structure. Like the DDE, the NDS representation does not allow one to represent structure and so the typical process is involves 3 steps: direct conversion of the NDS to a DDF; constructing a minimal representation of the resulting DDF using \texttt{minimize\_PIETOOLS\_DDF}; and conversion of the reduced DDF to a PIE. These three steps have been combined into a single function \texttt{convert\_PIETOOLS\_NDS}, computing the DDF represenation, minimizing this representation, and converting this representation to a PIE. All three resulting structures can be returned by calling
\begin{matlab}
\begin{verbatim}
 >> [DDF\_max, DDF, PIE] = convert\_PIETOOLS\_NDS(NDS);
\end{verbatim}
\end{matlab}
where now \texttt{DDF\_max} corresponds to the non-minimized DDF representation of \texttt{NDS}, and \texttt{DDF} corresponds to the minimized representation. If the user only want to compute this DDF representation, it is computationally cheaper to call the function with only two outputs,
\begin{matlab}
\begin{verbatim}
 >> [DDF\_max, DDF] = convert\_PIETOOLS\_NDS(NDS);
\end{verbatim}
\end{matlab}
or to call the function with a second argument \texttt{'ddf'},
\begin{matlab}
\begin{verbatim}
 >> [DDF] = convert\_PIETOOLS\_NDS(NDS,'ddf');
\end{verbatim}
\end{matlab}
Similarly, if only the non-minimized DDF representation is desired, the function should called with a single output,
\begin{matlab}
\begin{verbatim}
 >> DDF\_max = convert\_PIETOOLS\_NDS(NDS,'ddf_max');
\end{verbatim}
\end{matlab}
or with a second argument \texttt{'ddf\_max'},
\begin{matlab}
\begin{verbatim}
 >> DDF\_max = convert\_PIETOOLS\_NDS(NDS,'ddf_max');
\end{verbatim}
\end{matlab}
It is also possible to pass an argument \texttt{'pie'}, calling
\begin{matlab}
\begin{verbatim}
 >> PIE = convert\_PIETOOLS\_NDS(NDS,'pie');
\end{verbatim}
\end{matlab}
returning only the PIE representation, even though the DDF representations will also be computed.


\chapter{Operations on PI Operators in PIETOOLS: opvar and dopvar}\label{ch:opvar}

In Chapter~\ref{ch:PIE}, we showed how PI operators could be declared as \texttt{opvar} and \texttt{opvar2d} objects in PIETOOLS. In Chapter~\ref{ch:LPIs}, we showed how the similar class of \texttt{dopvar} (and \texttt{dopvar2d}) objects can be used to represent PI operator decision variables in convex optimization programs. In this Chapter, we detail some features of these \texttt{opvar} and \texttt{dopvar} classes, showing how standard operations on PI operators can be easily performed using the \texttt{opvar} classes in PIETOOLS.
In particular, we first recall the structure of \texttt{opvar} and \texttt{opvar2d} objects in Setion~\ref{sec:opvar_declare}, also showing how such objects can be declared in PIETOOLS.. In Section~\ref{sec:opvar_binary_ops}, we then show algebraic operations such as addition of \texttt{opvar} objects can be performed in PIETOOLS, after which we show how matrix operations such as concatenation can be performed on \texttt{opvar} objects in Section~\ref{sec:opvar_matrix_ops}. Finally, in Section~\ref{sec:additional_methods}, we outline a few additional operations for \texttt{opvar} objects. 
For more information on the theory behind these operations, we refer to Appendix~\ref{appx:PI_theory}, as well as papers such as~\cite{shivakumar2024extension}.

\begin{boxEnv}{\textbf{Note}}
Unless stated otherwise, the operations on \texttt{opvar} objects presented in the following sections can also be performed on \texttt{opvar2d}, \texttt{dopvar} and \texttt{dopvar2d} class objects. To reduce notation, these operations will be illustrated only for \texttt{opvar} class objects.
\end{boxEnv}

\section{Declaring \texttt{opvar} and \texttt{dopvar} Objects}\label{sec:opvar_declare}

In this section, we briefly recall how \texttt{opvar} objects are structured, and how they represent PI operators in 1D. We also briefly introduce the \texttt{dopvar} class, showing how such objects can be declared in a similar manner to \texttt{opvar} objects. For more information on the \texttt{opvar2d} structure for PI operators in 2D, we refer to Chapter~\ref{ch:PIs}.

\subsection{The \texttt{opvar} Class}

In PIETOOLS, 4-PI operators are represented by \texttt{opvar} objects, which are structures with 8 accessible fields. In particular, letting $\mcl{T}:\bmat{\R^{n_0}\\L_2^{n_1}[a,b]}\rightarrow \bmat{\R^{m_0}\\L_2^{m_1}[a,b]}$ be a 4-PI operator of the form
\begin{align}\label{eq:4PI_standard_form_2}
    \bl(\mcl{T}\mbf{x}\br)(s)=
    \left[\begin{array}{ll}
        Px_0        \hspace*{-0.1cm}~& +\ \int_{a}^{b}Q_1(s)\mbf{x}_1(s)ds  \\
        Q_2(s)x_0   \hspace*{-0.1cm}& +\ R_{0}(s)\mbf{x}_1(s) + \int_{a}^{s}R_{1}(s,\theta)\mbf{x}_1(\theta)d\theta + \int_{s}^{b}R_{2}(s,\theta)\mbf{x}_1(\theta)d\theta
    \end{array}\right]
\end{align}
for $\mbf{x}=\bmat{x_0\\\mbf{x}_1}\in \bmat{\R^{n_0}\\L_2^{n_1}[a,b]}$, we can declare $\mcl{T}$ as an \texttt{opvar} object \texttt{T} with the following fields

\begin{Statebox}{\texttt{opvar} fields}

{\small
\centering
\hspace*{-0.425cm}
\begin{tabular}{p{0.75cm}p{1.85cm}p{12.75cm}}
    \texttt{dim}    & \texttt{= [m0,n0; \hspace*{0.4cm} m1,n1]} 
    &  $2\times 2$ array of type \texttt{double} specifying the dimensions of the function spaces $\sbmat{\R^{m_0}\\L_2^{m_1}[a,b]}$ and $\sbmat{\R^{n_0}\\L_2^{n_1}[a,b]}$ the operator maps to and from;\\
    \texttt{var1} & \texttt{= s}    &  $1\times 1$ \texttt{pvar} (\texttt{polynomial} class) object specifying the spatial variable $s$; \\
    \texttt{var2} & \texttt{= theta}    &  $1\times 1$ \texttt{pvar} (\texttt{polynomial} class) object specifying the dummy variable $\theta$;   \\
    \texttt{I} & \texttt{= [a,b]}       &  $1\times 2$ array of type \texttt{double}, specifying the interval $[a,b]$ on which the spatial variables $s$ and $\theta$ exist; \\
    \texttt{P} & \texttt{= P} & $m_0\times n_0$ array of type \texttt{double} or \texttt{polynomial} defining the matrix $P$; \\
    \texttt{Q1} & \texttt{= Q1} & $m_0\times n_1$ array of type \texttt{double} or \texttt{polynomial} defining the function $Q_1(s)$; \\
    \texttt{Q2} & \texttt{= Q2} & $m_1\times n_0$ array of type \texttt{double} or \texttt{polynomial} defining the function $Q_2(s)$; \\
    \texttt{R.R0} & \texttt{= R0} & $m_1\times n_1$ array of type \texttt{double} or \texttt{polynomial} defining the function $R_0(s)$; \\
    \texttt{R.R1} & \texttt{= R1} & $m_1\times n_1$ array of type \texttt{double} or \texttt{polynomial} defining the function $R_1(s,\theta)$; \\
    \texttt{R.R2} & \texttt{= R2} & $m_1\times n_1$ array of type \texttt{double} or \texttt{polynomial} defining the function $R_2(s,\theta)$; \\
 \end{tabular}
 }

\end{Statebox}

As an example, suppose we want to declare the 4-PI operator $\mcl{T}:\sbmat{\R^2\\ L_2^{3}[2,3]}\rightarrow \sbmat{\R^{3}\\L_2[2,3]}$ defined as\\[-2.5em]
\begin{align*}
\bl(\mcl{T}\mbf{x}\br)(r)&=
\bbbbl[\!\begin{array}{ll}
  \overbrace{\sbmat{1&0\\0&2\\3&4}}^{P}x_0 & \!+\ \int_{2}^{3}\overbrace{\sbmat{r^2& 0& 0\\ 3& r^3& 0\\ 0& r+2*r^2 & 0}}^{Q_1(r)}\mbf{x}_1(r)dr \\
  \underbrace{\sbmat{-5r & 6}}_{Q_2(s)}x_0       &\!+\ \int_{2}^{r}\underbrace{\sbmat{r&2\nu & 3(r-\nu)}}_{R_1(r,\nu)}\mbf{x}_1(\nu)d\nu
\end{array}\!\bbbbr],   &   r&\in[2,3], \\[-2.0em]
\end{align*}
for any $\mbf{x}=\sbmat{x_0\\\mbf{x}_1}\in\sbmat{\R^2\\ L_2^{3}[2,3]}$. To declare this operator, we first initialize an opvar object \texttt{T}, using the syntax
\begin{matlab}
\begin{verbatim}
 >> opvar T
 ans =
       [] | [] 
       -----------
       [] | ans.R 

 ans.R =
     [] | [] | [] 
\end{verbatim}     
\end{matlab}
This command creates an empty \texttt{opvar} object \texttt{T} with all dimensions 0. Consequently, the parameters \texttt{P,Qi,Ri} are initialized to $0\times 0$ matrices. Since we know the operator $\mcl{T}$ to map $\sbmat{\R^2\\ L_2^{3}[2,3]}\rightarrow \sbmat{\R^{3}\\L_2[2,3]}$, we can specify the desired dimension of the \texttt{opvar} object \texttt{T} using the command
\begin{matlab}
\begin{verbatim}
 >> T.dim = [3 2; 1 3]
 T =
       [0,0] | [0,0,0] 
       [0,0] | [0,0,0] 
       [0,0] | [0,0,0] 
       ----------------
       [0,0] | T.R 

 T.R =
     [0,0,0] | [0,0,0] | [0,0,0] 
\end{verbatim}    
\end{matlab}
Here, by assigning a value to \texttt{dim}, the parameters are adjusted to zero matrices of appropriate dimensions. We note that, this command is not strictly necessary, as the dimensions of \texttt{T} will also be automatically adjusted once we specify the values of the parameters.

Next, we assign the interval $[2,3]$ on which the functions $\mbf{x}_1\in L_2^3[0,3]$ are defined, by setting the field \texttt{I} as
\begin{matlab}
 >> T.I = [2,3];
\end{matlab}
Since the parameters defining $\mcl{T}$ also depend on $r,\nu\in[2,3]$, we have to assign these variables as well. For this, we represent them by \texttt{pvar} objects \texttt{r} and \texttt{nu}, and set the values of \texttt{T.var1} and \texttt{T.var2} as
\begin{matlab}
\begin{verbatim}
 >> pvar r nu
 >> T.var1 = r;     T.var2 = nu;
\end{verbatim}
\end{matlab}
Note that, if the parameters \texttt{Qi} and \texttt{R.Ri} are constant, there is no need to declare the variables \texttt{var1} or \texttt{var2}, in which case these fields will default to \texttt{var1=s} and \texttt{var2=theta}. Having declared the variables, we finally set the values of the parameters, by assigning them to the appropriate fields of \texttt{T}
\begin{matlab}
\begin{verbatim}
 >> T.P = [1,0; 0,2; 3,4];
 >> T.Q1 = [r^2, 0, 0; 3, r^3, 0; 0, r+2*r^2, 0];
 >> T.Q2 = [-5*r, 6];
 >> T.R.R1 = [r, 2*nu, 3*r-nu]
 T =
          [1,0] | [r^2,0,0] 
          [0,2] | [3,r^3,0] 
          [3,4] | [0,2*r^2+r,0] 
       ----------------------
       [-5*r,6] | T.R 

 T.R =
     [0,0,0] | [r,2*nu,-nu+3*r] | [0,0,0] 
\end{verbatim}
\end{matlab}


\begin{boxEnv}{\textbf{Note}}
	\texttt{dim} is dependent on size of the 6 parameters \texttt{P, Qi} and \texttt{R.Ri}. Modifying those parameters automatically changes the value stored in \texttt{dim} property. If the dimensions of the parameters are incompatible, \texttt{dim} will store \texttt{Nan} as its value to alert the user about the discrepancy.
\end{boxEnv}

\subsection{\texttt{dpvar} objects and the \texttt{dopvar} class}

In addition to polynomial functions, polynomial decision variables can also be declared in PIETOOLS. Such decision variables are used in polynomial optimization programs, optimizing over the values of the coefficients defining these polynomials. To declare such polynomial decision variables, we use the \texttt{dpvar} class, extending the \texttt{polynomial} class to represent polynomials with decision variables. For example, to represent a variable quadratic polynomial
\begin{align*}
    p(c_0,c_1,c_2;x)=c_0+c_1x+c_2x^2,
\end{align*}
with unknown coefficient $\{c_0,c_1,c_2\}$, we use
\begin{matlab}
\begin{verbatim}
 >> pvar x
 >> dpvar c0 c1 c2
 >> p = c0 + c1*x + c2*x^2
 p = 
   c0 + c1*x + c2*x^2
\end{verbatim}
\end{matlab}
Here, the second line decision variables $c_0$, $c_1$ and $c_2$, and the third line uses these decision variables to declare the decision variable polynomial $p(c_0,c_1,c_2;x)$ as a \texttt{dpvar} object \texttt{p}. Crucially, this variable is affine in the coefficients $c_0$, $c_1$ and $c_2$, as \texttt{dpvar} objects can only be used to represent decision variable polynomials that are affine with respect to their coefficients. This is because PIETOOLS cannot tackle polynomial optimization programs that are nonlinear in the decision variables, and accordingly, the \texttt{dpvar} structure has been built to exploit the linearity of the decision variables to minimize computational effort.

Using polynomial decision variables, we can also define PI operator decision variables, defining the parameters $Q_1$ through $R_2$ by polynomial decision variables rather than polynomials. For example, suppose we have an operator $\mcl{D}:L_2^{2}[0,1]\rightarrow L_2^{2}[0,1]$ defined as
\begin{align*}
 \bl(\mcl{D}\mbf{x}\br)(s)&=\underbrace{\bmat{c_1&c_2s\\c_2s&c_3s^2}}_{R_0(s)}\mbf{x}(s)+\int_{s}^{1}\underbrace{\bmat{c_4s^2&c_5s\theta\\ -c_6s\theta&c_7\theta^2}}_{R_2(s,\theta)}\mbf{x}(\theta)d\theta, &
 s&\in[0,1]
\end{align*}
for $\mbf{x}\in L_2^2[0,1]$, where $c_1$ through $c_7$ are unknown coefficients. Then, we can declare the parameters $R_1$ and $R_2$ as \texttt{dpvar} class objects by calling
\begin{matlab}
\begin{verbatim}
 >> pvar s th
 >> dpvar c1 c2 c3 c4 c5 c6 c7
 >> R0 = [c1, c2*s; c2*s, c3*s^2];
 >> R2 = [c4, c5*s*th; c6*s*th, c7*th^2];
\end{verbatim}
\end{matlab}
Then, we can declare $\mcl{D}$ as a \texttt{dopvar} object \texttt{D}. Such \texttt{dopvar} objects have a structure identical to that of \texttt{opvar} objects, with the only difference being that the parameters \texttt{P} through \texttt{R.R2} can be declared as \texttt{dpvar} objects. Accordingly, we declare the \texttt{dopvar} object \texttt{D} in a similar manner as we would and \texttt{opvar} objects:
\begin{matlab}
\begin{verbatim}
 >> dopvar D;
 >> D.I = [0,1];
 >> D.var1 = s;     D.var2 = th;
 >> D.R.R0 = R0;    D.R.R1 = R2
 D =
       [] | [] 
       ---------
       [] | D.R 

 D.R =
         [c1,c2*s] | [0,0] |      [c4,c5*s*th] 
     [c2*s,c3*s^2] | [0,0] | [c6*s*th,c7*th^2] 
\end{verbatim}
\end{matlab}

\section{Algebraic Operations on \texttt{opvar} Objects}\label{sec:opvar_binary_ops}
	
In this section, we go over various methods that help in manipulating and handling of \texttt{opvar} objects in PIETOOLS. In particular, in Subsection~\ref{sec:add_pi}, we show how the sum of two PI operators can be computed in PIETOOLS, followed by the composition of PI operators in Subsection~\ref{sec:comp_pi}. In Subsection~\ref{sec:adj_pi}, we then show how to take the adjoint of a PI operator, and finally, in Subsection~\ref{sec:inv_pi}, we show how a numerical inverse of a PI operator can be computed. 
To illustrate each of these operations, we use the PI operators $\mcl{A},\mcl{B}:\sbmat{\R^2\\L_2[-1,1]}\rightarrow\sbmat{\R^2\\L_2[-1,1]}$ defined as
\begin{align}\label{eq:opvar_ops_ex}
    \bl(\mcl{A}\mbf{x}\br)(s)&=
    \bbbbl[\begin{array}{ll}
      \sbmat{1&0\\2&-1}x_0   &\!+\ \int_{-1}^{1}\sbmat{1-s\\s+1}\mbf{x}_1(s)ds \\
      \sbmat{10s&-1}x_0   &\!+\ 2\mbf{x}_1(s) + \int_{-1}^{s}(s-\theta)\mbf{x}_1(\theta)d\theta + \int_{s}^{1}(s-\theta)\mbf{x}_1(\theta) d\theta
    \end{array}\bbbbr],  \\
    \bl(\mcl{B}\mbf{x}\br)(s)&=
    \bbbbl[\begin{array}{ll}
      \sbmat{1&0\\0&3}x_0   &  \nonumber\\
      \sbmat{5s&-s}x_0   &\!+\ s^2\mbf{x}_1(s) + \int_{s}^{1}\theta\mbf{x}_1(\theta) d\theta
    \end{array}\bbbbr], &   s&\in[-1,1],
\end{align}
for $\mbf{x}=\sbmat{x_0\\\mbf{x}_1}\in \sbmat{\R^2\\L_2[-1,1]}$. We declare these operators as
\begin{matlab}
\begin{verbatim}
 >> pvar s th
 >> opvar A B;
 >> A.I = [-1,1];         B.I = [-1,1];
 >> A.var1 = s;           B.var1 = s;
 >> A.var2 = th;          B.var2 = th;
 >> A.P = [1,0;2,-1];     B.P = [1,0;0,3];
 >> A.Q1 = [1-s;s+1];
 >> A.Q2 = [10*s,-1];     B.Q2 = [5*s,-s];
 >> A.R.R0 = 2;           B.R.R0 = s^2;
 >> A.R.R1 = (s-th);
 >> A.R.R2 = (s-th);      B.R.R2 = th;
\end{verbatim}
\end{matlab}

\subsection{Addition (\texttt{A+B})}\label{sec:add_pi}
\texttt{opvar} objects, \texttt{A} and \texttt{B}, can be added simply by using the command
\begin{matlab}
 >> A+B
\end{matlab}
For two \texttt{opvar} objects to be added, they \textbf{must} have same dimensions (\texttt{A.dim=B.dim}), domains (\texttt{A.I=B.I}), and variables (\texttt{A.var1=B.var1}). Furthermore, if \texttt{A} (or \texttt{B}) is a scalar then PIETOOLS considers that as adding \texttt{A*I} (or \texttt{B*I}) where \texttt{I} is an identity matrix. Again, this operation is appropriate if and only if dimensions match. Similarly, if \texttt{A} (or \texttt{B}) is a matrix with matching dimension, it can be added to \texttt{opvar} \texttt{B} (or  \texttt{A}) using the same command.

\paragraph*{Example}
Adding the \texttt{opvar} objects \texttt{A} and \texttt{B} corresponding to operators $\mcl{A},\mcl{B}$ defined as in Equation~\eqref{eq:opvar_ops_ex}, we find
\begin{matlab}
\begin{verbatim}
 >> C = A+B
 C =
            [2,0] | [-s+1] 
            [2,2] | [s+1] 
        ------------------
      [15*s,-s-1] | C.R 

 C.R =
     [s^2+2] | [s-th] | [s]
\end{verbatim}
\end{matlab}
suggesting that, for $\mbf{x}=\sbmat{x_0\\\mbf{x}_1}\in \sbmat{\R^2\\L_2[-1,1]}$,
\begin{align*}
    \bl(\mcl{A}\mbf{x}\br)(s)+\bl(\mcl{B}\mbf{x}\br)(s)&=
    \bbbbl[\begin{array}{ll}
      \sbmat{2&0\\2&2}x_0   &\!+\ \int_{-1}^{1}\sbmat{1-s\\s+1}\mbf{x}_1(s)ds \\
      \sbmat{15s&-s-1}x_0   &\!+\ (s^2+2)\mbf{x}_1(s) + \int_{-1}^{s}(s-\theta)\mbf{x}_1(\theta)d\theta + \int_{s}^{1}s\mbf{x}_1(\theta) d\theta
    \end{array}\bbbbr].
\end{align*}

\subsection{Multiplication (\texttt{A*B})}\label{sec:comp_pi}
\texttt{opvar} objects, \texttt{A} and \texttt{B}, can be composed simply by using the command
\begin{matlab}
>> A*B
\end{matlab}
For two opvar objects to be composed, they \textbf{must} have the same domains (\texttt{A.I=A.B}), the same variables (\texttt{A.var1=B.var1} and \texttt{A.var2=B.var2}), and the output dimension of \texttt{B} must match the input dimension of \texttt{A} (\texttt{A.dim(:,2)=B.dim(:,1)}). Furthermore, if \texttt{A} (or \texttt{B}) is a scalar then PIETOOLS considers that as a scalar multiplication operation, thus multiplying all parameters of \texttt{B} (or \texttt{A}) by that value. 

\paragraph*{Example}
Composing the \texttt{opvar} objects \texttt{A} and \texttt{B} corresponding to operators $\mcl{A},\mcl{B}$ defined as in Equation~\eqref{eq:opvar_ops_ex}, we find
\begin{matlab}
\begin{verbatim}
 >> C = A*B
 C = 
              [-2.3333,0.66667] | [-1.5*s^3+2*s^2+1.5*s] 
               [5.3333,-3.6667] | [1.5*s^3+2*s^2+0.5*s] 
      -------------------------------------------
      [20*s-3.3333,-2*s-2.3333] | C.R 

 C.R =
     [2*s^2] | [2*s*th^2-1.5*th^3+s*th+0.5*th] | [2*s*th^2-1.5*th^3+s*th+2.5*th] 
\end{verbatim}
\end{matlab}
suggesting that, for $\mbf{x}=\sbmat{x_0\\\mbf{x}_1}\in \sbmat{\R^2\\L_2[-1,1]}$,
\begin{align*}
    \bbl(\mcl{A}\bl(\mcl{B}\mbf{x}\br)\bbr)(s)&=
    \left[\begin{array}{ll}
      \sbmat{-2\frac{1}{3}&\frac{2}{3}\\5\frac{1}{3}&-3\frac{2}{3}}x_0   &\!+\ \int_{-1}^{1}\sbmat{-1\frac{1}{2}s^3+2s^2+1\frac{1}{2}s\\1\frac{1}{2}s^3+2s^2+\frac{1}{2}s}\mbf{x}_1(s)ds \\
      \sbmat{20s-3\frac{1}{3}&-2s-2\frac{1}{3}}x_0   &\!+\ 2s^2\mbf{x}_1(s) + \int_{-1}^{s}(2s\theta^2-1\frac{1}{2}\theta^3+s\theta+\frac{1}{2}\theta)\mbf{x}_1(\theta)d\theta \\
      &\quad + \int_{s}^{1}(2s\theta^2-1\frac{1}{2}\theta^3+s\theta+2\frac{1}{2}\theta)\mbf{x}_1(\theta) d\theta
    \end{array}\right].
\end{align*}

\begin{boxEnv}{\textbf{Note}}
 Although \texttt{dopvar} objects can be multiplied with \texttt{opvar} objects and vice versa, producing a \texttt{dopvar} object in both cases, it is not possible to compute the composition of two \texttt{dopvar} objects. This is because \texttt{dopvar} objects depend linearly (affinely) on the decision variables, and the composition of two \texttt{dopvar} objects would require taking the product of decision variables. Similarly for \texttt{dopvar2d} objects.
\end{boxEnv}

\subsection{Adjoint (\texttt{A'})}\label{sec:adj_pi}
The adjoint of an \texttt{opvar} object \texttt{A} can be calculated using the command
\begin{matlab}
 >> A'
\end{matlab}
For an operator $\mcl{A}:RL^{n_0,n_1}[a,b]\rightarrow RL^{m_0,m_1}[a,b]$, the adjoint $\mcl{A}^*:RL^{m_0,m_1}[a,b]\rightarrow RL^{n_0,n_1}[a,b]$ will be such that, for any $\mbf{x}\in RL^{n_0,n_1}[a,b]$ and $\mbf{y}\in RL^{m_0,m_1}[a,b]$,
\begin{align*}
 \ip{\mcl{A}\mbf{x}}{\mbf{y}}_{RL} = \ip{\mbf{x}}{\mcl{A}^*\mbf{y}}_{RL},
\end{align*}
where for $\mbf{x}=\sbmat{x_0\\\mbf{x}_1}\in\sbmat{\R^{n_0}\\L_2^{n_1}[a,b]}=RL^{n_0,n_1}[a,b]$ and $\mbf{y}=\sbmat{y_0\\\mbf{y}_1}\in\sbmat{\R^{n_0}\\L_2^{n_1}[a,b]}=RL^{n_0,n_1}[a,b]$,
\begin{align*}
    \ip{\mbf{x}}{\mbf{y}}_{RL}:=\ip{x_0}{y_0} + \ip{\mbf{x}_1}{\mbf{y}_1}_{L_2} = x_0^T y_0 + \int_{a}^{b}[\mbf{x}_1(s)]^T\mbf{y}_1(s)ds
\end{align*}

\paragraph*{Example}
Computing the adjoint of the \texttt{opvar} object \texttt{A} corresponding to operator $\mcl{A}$ defined as in Equation~\eqref{eq:opvar_ops_ex}, we find
\begin{matlab}
\begin{verbatim}
 >> AT = A'
 AT =
            [1,2] | [10*s] 
           [0,-1] | [-1] 
       ------------------
       [-s+1,s+1] | AT.R 

 AT.R =
     [2] | [-s+th] | [-s+th] 
\end{verbatim}
\end{matlab}
suggesting that, for $\mbf{x}=\sbmat{x_0\\\mbf{x}_1}\in \sbmat{\R^2\\L_2[-1,1]}$,
\begin{align*}
    \bl(\mcl{A}^*\mbf{x}\br)(s)&=
    \left[\begin{array}{ll}
      \sbmat{1&2\\0&-1}x_0   &\!+\ \int_{-1}^{1}\sbmat{10s\\-1}\mbf{x}_1(s)ds \\
      \sbmat{1-s&s+1}x_0   &\!+\ 2\mbf{x}_1(s) + \int_{-1}^{1}(\theta-s)\mbf{x}_1(\theta)d\theta
      + \int_{s}^{1}(\theta-s)\mbf{x}_1(\theta) d\theta
    \end{array}\right].
\end{align*}

\subsection{Inverse (\texttt{inv\_opvar(A)})}\label{sec:inv_pi}
The inverse of an opvar object \texttt{A} can be numerically calculated, using the function 
\begin{matlab}
 >> inv\_opvar(A)
\end{matlab}
See Lemma \ref{lem:inverse} for details on Inversion formulae.

\paragraph*{Example}
Computing the inverse of the \texttt{opvar} object \texttt{A} corresponding to operator $\mcl{A}$ defined as in Equation~\eqref{eq:opvar_ops_ex}, we find
\begin{matlab}
\begin{verbatim}
 >> Ainv = inv_opvar(A)
 Ainv =
                    [-0.2,-0.4] | [  0.3*s + 0.2]
                      [1.2,0.4] | [ -0.3*s - 0.7]
       ------------------------------------------
       [ 0.3*s+0.7,1.35*s+0.65] | AT.R 

 Ainv.R =
     [0.5] | [-1.2*s*th-0.675*s-0.3*th-0.575] | [-1.2*s*th-0.675*s-0.3*th-0.575]
\end{verbatim}
\end{matlab}
suggesting that, for $\mbf{x}=\sbmat{x_0\\\mbf{x}_1}\in \sbmat{\R^2\\L_2[-1,1]}$
\begin{align*}
    \bl(\mcl{A}^{-1}\mbf{x}\br)(s)&=
    \left[\begin{array}{ll}
      \sbmat{-0.2&-0.4\\1.2&0.4}x_0   &\!+\ \int_{-1}^{1}\sbmat{0.3s+0.2\\-0.3s-0.7}\mbf{x}_1(s)ds \\
      \sbmat{0.3s+0.7&1.35s+0.65}x_0   &\!+\ 0.5\mbf{x}_1(s) + \int_{-1}^{1}(-1.2s\theta-0.675s-0.3\theta-0.575)\mbf{x}_1(\theta)d\theta \\
      & \quad + \int_{s}^{1}(-1.2s\theta-0.675s-0.3\theta-0.575)\mbf{x}_1(\theta) d\theta
    \end{array}\right].
\end{align*}


\begin{boxEnv}{\textbf{Note}}
	An inverse function has not been defined for \texttt{opvar2d}, \texttt{dopvar}, or \texttt{dopvar2d} objects.
\end{boxEnv}


\section{Matrix Operations on \texttt{opvar} Objects}\label{sec:opvar_matrix_ops}

In this section, we show how matrix operations on \texttt{opvar} objects can be performed. In particular, in Subsection~\ref{sec:opvar_subsref}, we show how desired rows and columns of \texttt{opvar} objects can be extracted, and in Subsection~\ref{sec:concat_pi} we show how \texttt{opvar} objects can be concatenated. Although we explain these operations only for \texttt{opvar} objects, they can also be applied to \texttt{dopvar}, \texttt{opvar2d}, and \texttt{dopvar2d} objects. For the purpose of illustration, we once more use the 4-PI operator $\mcl{A}:\sbmat{\R^2\\L_2[-1,1]}\rightarrow \sbmat{\R^2\\L_2[-1,1]}$ defined in Equation~\eqref{eq:opvar_ops_ex}, represented in PIETOOLS by the \texttt{opvar} object \texttt{A},
\begin{matlab}
\begin{verbatim}
 >> A
 A =
          [1,0] | [-s+1] 
         [2,-1] | [s+1] 
      ----------------
      [10*s,-1] | A.R 

 A.R =
    [2] | [s-th] | [s-th] 
\end{verbatim}
\end{matlab}

\subsection{Subs-indexing (\texttt{A(i,j)})}\label{sec:opvar_subsref}

Just like matrices, 4-PI operators can be sliced to extract desired rows or columns, returning new 4-PI operators in lower dimensions. This index slicing is performed in the same manner as for matrices, extracting rows \texttt{row\_ind} and columns \texttt{col\_ind} of an \texttt{opvar} object \texttt{T} as 
\begin{matlab}
>> T(row\_ind, col\_ind)
\end{matlab}
However, indexing 4-PI operators is slightly different from matrix indexing due to presence of multiple components. These components can be visualized as being stacked as in a matrix:
\begin{align*}
\texttt{B}=\left[\begin{array}{c|c}
\texttt{P}&\texttt{Q1}\\\hline
\texttt{Q2}&\texttt{Ri}
\end{array}\right]
\end{align*}
Then, row indices specified in \texttt{row\_ind} correspond to the rows in this big matrix. Column indices, \texttt{col\_ind}, are associated with the columns of this big matrix in similar manner. The retrieved slices are put in appropriate components and a 4-PI operator is returned.
Note that the bottom-right block of the big matrix  \texttt{B} has 3 components in \texttt{Ri}. If indices in the slice correspond to rows and columns in this block, then the slice is extracted from all three components and stored in a \texttt{Ri} part of the new sliced PI operator.

\paragraph*{Example}
Consider the \texttt{opvar} object \texttt{A} corresponding to the operator $\mcl{A}:\sbmat{\R^2\\L_2[-1,1]}\rightarrow \sbmat{\R^2\\L_2[-1,1]}$ defined as in Equation~\eqref{eq:opvar_ops_ex}. We can decompose this operator into subcomponents $\mcl{A}=\sbmat{\mcl{A}_{11}&\mcl{A}_{12}\\\mcl{A}_{21}&\mcl{A}_{22}}$, where $\mcl{A}_{11}:\R\to\R^2$, $\mcl{A}_{12}:\sbmat{\R\\L_2[-1,1]}\to\R^2$, $\mcl{A}_{21}:\R\to L_2[-1,1]$ and $\mcl{A}_{22}:L_2[-1,1]\to L_2[-1,1]$, by taking
\begin{matlab}
\begin{verbatim}
 >> A11 = A([1,2],1)
 A11 =
       [1] | [] 
       [2] |   
       -----------
        [] | A11.R 
 A11.R =
     [] | [] | []
     
 >> A12 = A([1,2],2:3)
 A12 =
        [0] | [-s+1] 
       [-1] | [s+1] 
       -------------
         [] | A12.R 
 A12.R =
     [] | [] | []
     
 >> A21 = A(3,1)
 A21 =
           [] | [] 
       ---------------
       [10*s] | A21.R 
 A21.R =
     [] | [] | [] 
     
 >> A22 = A(3,2:3)
 A22 =
         [] | [] 
       -------------
       [-1] | A12.R 
 A22.R =
     [2] | [s-th] | [s-th]
\end{verbatim}
\end{matlab}

\subsection{Concatenation (\texttt{[A,B]}, \texttt{[A;B]})}\label{sec:concat_pi}
Just like matrices, multiple \texttt{opvar} objects can be concatenated, provided the dimensions match.
In particular, two \texttt{opvar} objects \texttt{A} and \texttt{B} can be horizontally or vertically concatenated by respectively using the commands
\begin{matlab}
>> [A B]  \% for horizontal concatenation\\ 
>> [A; B] \% for vertical concatenation
\end{matlab}
Note that concatenation of \texttt{opvar} objects is allowed only if their spatial domain is the same (\texttt{A.I=B.I}), and the variables involved in each are identical (\texttt{A.var1=B.var1} and \texttt{A.var2=B.var2}). Moreover, \texttt{A} and \texttt{B} can be horizontally concatenated only if they have the same row dimensions (\texttt{A.dim(:,1)=B.dim(:,1)}, and they can be concatenated vertically only if they have the same column dimensions (\texttt{A.dim(:,2)=B.dim(:,2)}). Finally, note that \texttt{opvar} objects can only represent maps $\R^{n_{0}}\times L_2^{n_{1}}\to \R^{m_{0}}\times L_2^{m_{1}}$. Therefore, if e.g. \texttt{A}$:L_2\to\R$ and \texttt{B}$:L_2\to L_2$, then the concatenation \texttt{[B;A]}$:L_2\to \R\times L_2$ can be represented as an \texttt{opvar} object, but the concatenation \texttt{[B;A]}$:L_2\to L_2\times\R$ cannot. Therefore, this latter concatenation is currently prohibited. Similarly, if \texttt{A}$:\R\to L_2$ and \texttt{B}$:L_2\to L_2$, we can take the concatenation \texttt{[A,B]}$:\R\times L_2\to L_2$, but we cannot concatenate \texttt{[B,A]}$:L_2\times R\to L_2$.

\paragraph*{Example}
Consider the \texttt{opvar} objects \texttt{A11}, \texttt{A12}, \texttt{A21} and \texttt{A22}, corresponding to the operator $\mcl{A}=\sbmat{\mcl{A}_{11}&\mcl{A}_{12}\\\mcl{A}_{21}&\mcl{A}_{22}}$ defined in~\eqref{eq:opvar_ops_ex}, decomposed as in the previous subsection. Then, we can reconstruct an \texttt{opvar} object \texttt{A} representing the full operator $\mcl{A}$ as
\begin{matlab}
\begin{verbatim}
 >> A = [A11, A12; A21, A22]
 A =
          [1,0] | [-s+1] 
         [2,-1] | [s+1] 
      ----------------
      [10*s,-1] | A.R 
 A.R =
     [2] | [s-th] | [s-th]
\end{verbatim}
\end{matlab}



 \section{Additional Methods for \texttt{opvar} Objects}\label{sec:additional_methods}
There are some additional functions included in PIETOOLS that can be used in debugging or as the user sees fit. In this section, we compile the list of those functions, without going into details or explanation. However, users can find additional information by using \texttt{help} command in MATLAB.
	
\vspace{5mm}
\begin{center}
\begin{tabular}{|p{4.5cm}|p{10cm}|}
		\hline
		\textbf{Function Name} & \textbf{Description}\\\hline
        \texttt{A==B} & The function checks whether the variables, domain, dimensions and parameters of the operators \texttt{A} and \texttt{B} are equal, and returns a binary value 1 if this is the case, or 0 if it is not. The function can also be used to check if an operator \texttt{A} has all parameters equal to zero operator by calling \texttt{A==0}.
        \\\hline
    	\texttt{isvalid(P)}& The function returns a logical value. 0 is everything is in order, 1 if the object has incompatible dimensions, 2 if property \texttt{P} is not a matrix, 3 if properties \texttt{Q1, Q2} or \texttt{R0} are not polynomials in $s$, 4 if properties \texttt{R1} or \texttt{R2} are not polynomials in $s$ and $\theta$. For \texttt{opvar2d} objects, returns boolean value \texttt{true} or \texttt{false} to indicate if \texttt{P} is appropriate or not.
		\\\hline
		\texttt{degbalance(T)}& Estimates polynomial degrees needed to create an \texttt{opvar} object \texttt{Q} associated to a positive PI operator $\mcl{Q}\succeq 0$ in \texttt{poslpivar}, such that \texttt{T=Q} has at least one solution.
        \\\hline
		\texttt{getdeg(T)}& Returns highest and lowest degree of $s$ and $\theta$ in the components of the opvar object \texttt{T}.
        \\\hline
		\texttt{rand\_opvar(dim, deg)}& Creates a random opvar object of specified dimensions \texttt{dim} and polynomial degrees \texttt{deg}.
        \\\hline
		\texttt{show(T,opts)}& Alternative display format for opvar objects with optional argument to omit selected properties from display output. \textbf{Not defined for \texttt{opvar2d} objects.}
        \\\hline
		\texttt{opvar\_postest(T)}& Numerically test for sign definiteness of \texttt{T}. Returns -1 if negative definite, 0 if indefinite and 1 if positive definite. Use \texttt{opvar\_postest\_2d} for \texttt{opvar2d} objects.
        \\\hline
		\texttt{diff\_opvar(T)}& Returns composition of derivative operator with opvar \texttt{T} as described in Lem. \ref{lem:diff_PI}. Use \texttt{diff(T)} for \texttt{opvar2d} objects.
        \\\hline
\end{tabular}
\end{center}

\part{Examples and Applications}

\chapter{PIETOOLS Demonstrations}\label{ch:demos}

In this Chapter, we illustrate several applications of PIE simulation and LPI programming, and how each of these problems can be implemented in PIETOOLS. Each of these problems has also been implemented as a \texttt{DEMO} file in PIETOOLS, which can be found in the \texttt{PIETOOLS\_demos} directory. Note that, although running these demos will not produce the plots presented throughout this chapter, the code to reproduce each of these figures has been added in the script ``PIETOOLS\_Code\_Illustrations\_Ch11\_Demos''.

\section{DEMO 1: Stability Analysis and Simulation}\label{sec:demos:stability}

With this demo, we illustrate how an ODE-PDE model can be simulated with PIESIM, and how stability of the system can be tested using LPI programming. For this purpose, we consider the same damped 1D wave equation coupled to a stable ODE as in Chapter~\ref{ch:scope}, given by
\begin{align*}
\dot{x}(t) &= -x(t), \\
 \ddot{\mbf{x}}(t,s) &= c^2  \partial_s^2 \mbf x(t,s) -b \partial_s \mbf x(t,s) + s w(t), s\in (0,1), t\geq 0, \\
 r(t)&=\mbf{x}(t,1)-\mbf{x}(t,0).
\end{align*}
for some given velocity $c$ and viscous damping coefficient $b$, and with external disturbance $w(t)\in\R$ and regulated output $r(t)\in\R$. To implement this system in PIETOOLS, we introduce the PDE state $\phi = (\partial_s \mbf{x}, \dot{\mbf x} )$. Then, the system can be equivalently expressed as
\begin{align}\label{eq:demos:stability:odepde}
    \dot{x}(t) &= -x(t) \\
    \dot{\phi}(t,s) &= \bmat{0 & 1 \\ c & 0}  \partial_s \phi (t,s) +\bmat{0 & 0\\0 & -b} \phi (t,s) + \bmat{0\\s} w(t), s\in (0,1), t\geq 0,   \notag\\
    r(t)&=\int_{0}^{1}\phi_{1}(t,s)ds.      \notag
\end{align}
For e.g. $c=1$ and $b=0.01$, this system can be declared in PIETOOLS as
\begin{matlab}
\begin{verbatim}
 % Declare independent variables
 pvar t s;   
 % Declare state, input, and output variables
 x = pde_var();               phi = pde_var(2,s,[0,1]);
 w = pde_var('in');           r = pde_var('out');
 % Declare system equations
 c=1;    b=.01;
 eq_dyn = [diff(x,t,1)==-x
           diff(phi,t,1)==[0 1; c 0]*diff(phi,s,1)+[0;s]*w+[0 0;0 -b]*phi];
 eq_out= r==int([1 0]*phi,s,[0,1]);
 bc1 = [0 1]*subs(phi,s,0)==0;   % add the boundary conditions
 bc2 = [1 0]*subs(phi,s,1)==x;
 odepde = [eq_dyn;eq_out;bc1;bc2];
\end{verbatim}
\end{matlab}
To see whether this system is stable, we first simulate it using PIESIM. For this, we consider an initial value of the ODE state $x(0)=\frac{1}{2}$, and initial PDE state values $\phi_{1}(0,s)=0.5-s$ and $\phi_{2}(0,s)=\sin(\pi s)$. We further assume the disturbance to be given by a decaying but oscillating function $w(t)=e^{-t}\sin(5t)$. We declare these values as
\begin{matlab}
\begin{verbatim}
 % % Declare initial values and disturbance
 syms st sx;
 uinput.ic = [0.5,0.5-sx,sin(pi*sx)];   
 uinput.w = sin(5*st)*exp(-st); 
\end{verbatim}
\end{matlab}
Now, to perform simulation with these values, we will use PIESIM, expanding the PDE state $\phi(t,s)$ using 16 Chebyshev in $s$ polynomial, and simulating up to time $t=9$, using a time step of $0.03$.
\begin{matlab}
\begin{verbatim}
 opts.plot = 'yes';  % plot the solution
 opts.N = 16;        % expand using 16 Chebyshev polynomials
 opts.tf = 9;        % simulate up to t = 9;
 opts.dt = 0.03;     % use time step of 3*10^-2
\end{verbatim}
\end{matlab}
Having declared these settings, we call PIESIM to simulate as
\begin{matlab}
\begin{verbatim}
 [solution,grid] = PIESIM(odepde, opts, uinput);
 tval = solution.timedep.dtime;
 xval = solution.timedep.primary{1};
 phi1 = reshape(solution.timedep.primary{2}(:,1,:),opts.N+1,[]);
 phi2 = reshape(solution.timedep.primary{2}(:,2,:),opts.N+1,[]);
 zval = solution.timedep.regulated{1};
 wval = subs(uinput.w,st,tval);
\end{verbatim}
\end{matlab}
extracting the values of the ODE and PDE state at each time step, as well as those of the  regulated output and the disturbance. Note that the PDE state solutions are provided only at the $N+1=17$ grid points at each time step. Fig.~\ref{fig:demos:demo1_state} and Fig.~\ref{fig:demos:demo1_output} show the simulated evolution of the different ODE and PDE state variables, as well as the regulated output.

\begin{figure}[H]
	\centering
	\includegraphics[width=\textwidth]{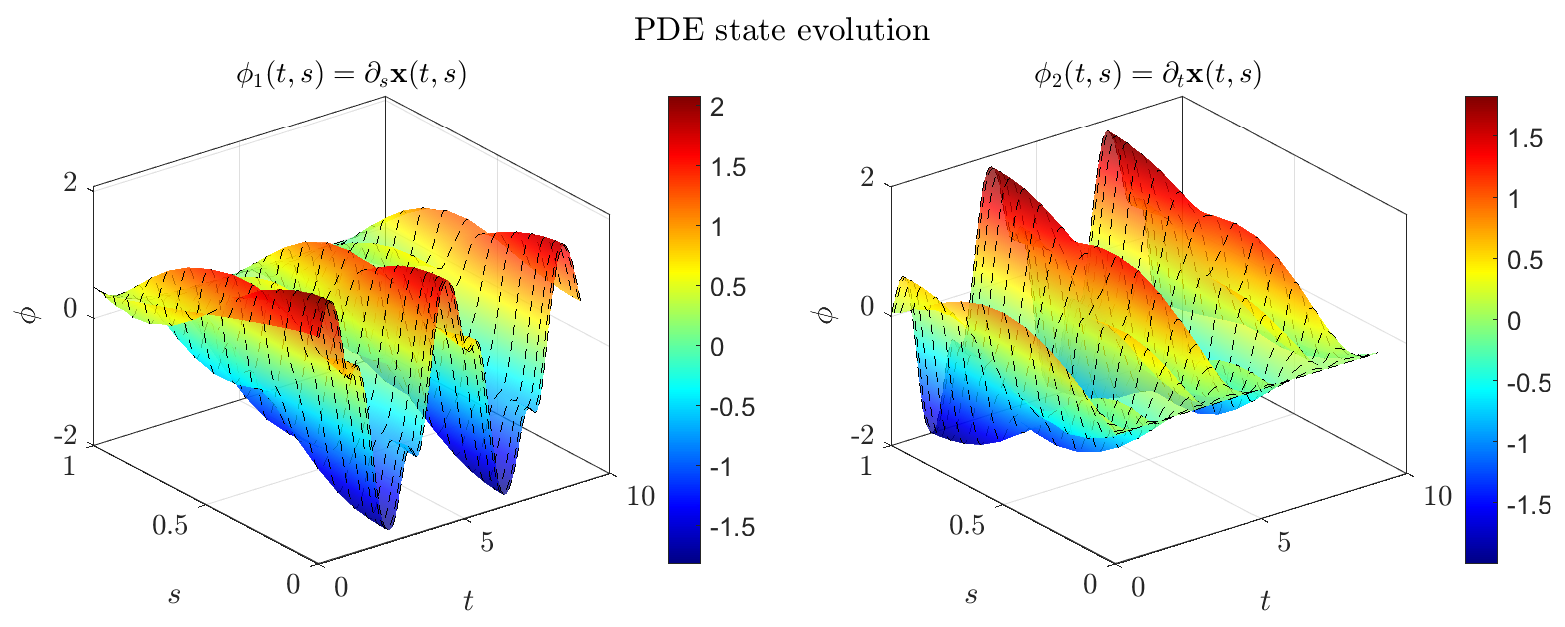}
	\caption{Simulated value of PDE state variables $\phi_{1}(t,s)=\partial_{s}\mbf{x}(t,s)$ and $\phi_{2}(t,s)=\partial_{t}\mbf{x}(t,s)$ associated to the ODE-PDE~\eqref{eq:demos:stability:odepde}.}\label{fig:demos:demo1_state}
\end{figure}

\begin{figure}[H]
	\centering
	\includegraphics[width=\textwidth]{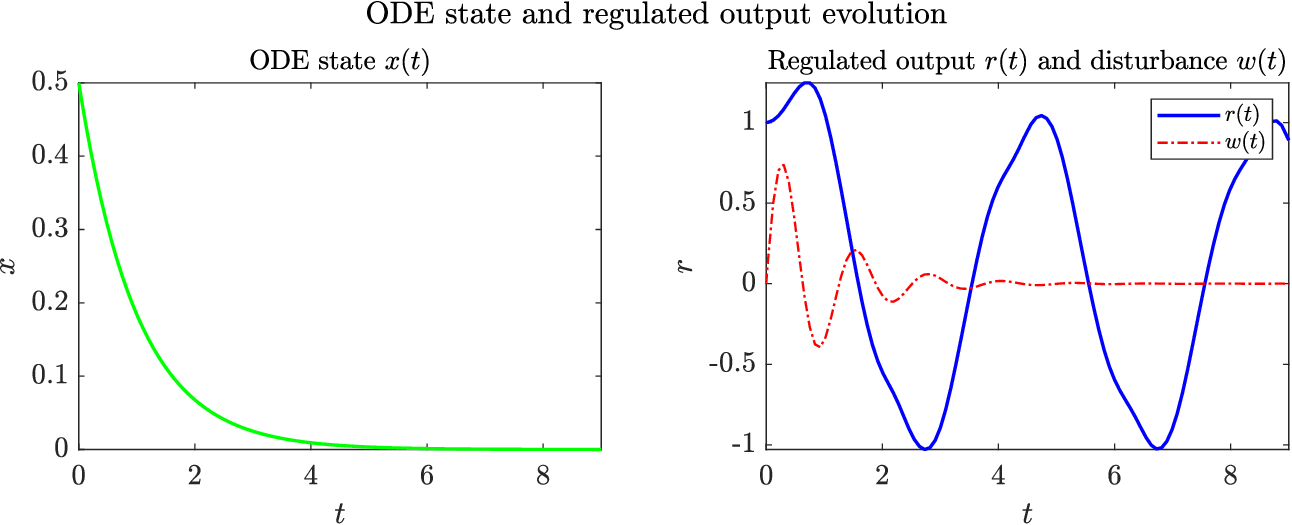}
	\caption{Simulated value of ODE state variable $x(t)$ and regulated output $r(t)=\int_{0}^{1}\phi_{1}(t,s)ds$ associated to the ODE-PDE~\eqref{eq:demos:stability:odepde}.}\label{fig:demos:demo1_output}
\end{figure}

From the figures it appears that, although oscillatory, solutions to the system are stable, not growing with time (despite the presence of disturbances). To verify that the system is indeed stable, we can test stability in the PIE representation. To this end, we first convert the system to a PIE as
\begin{matlab}
\begin{verbatim}
 pie = convert(odepde);
 T = pie.T;    
 A = pie.A;    B = pie.B1;
 C = pie.C1;   D = pie.D11
\end{verbatim}
\end{matlab}
For our ODE-PDE system~\eqref{eq:demos:stability:odepde}, the associated PIE representation takes the form
\begin{align*}
    \partial_{t}( \mcl{T}\mbf{x}_{\text{f}})(t)&=\mcl{A}\mbf{x}_{\text{f}}(t) +\mcl{B} w(t),    \\
    z(t)&=\mcl{C}\mbf{x}_{\text{f}}(t) +\mcl{D} w(t),
\end{align*}
where now $\mbf{x}_{\text{f}}(t,s)=(x(t),\partial_{s}\phi_{1}(t,s),\partial_{s}\phi_{2}(t,s))\in\R\times L_{2}^2[0,1]$. In the absence of disturbances, $w(t)=0$, stability of the autonomous system can be tested by solving the LPI
\begin{equation}
    \mcl{P}\succ 0,\qquad 
    \mcl{T}^*\mcl{P}\mcl{A} +\mcl{A}\mcl{P}\mcl{T}\preceq 0.
\end{equation}
This LPI can be declared and solved as
\begin{matlab}
\begin{verbatim}
 % % Initialize LPI program
 prog = lpiprogram(s,[0,1]);

 % % Declare decision variables:
 % %   P: R x L2^2 --> R x L2^2,    P>0
 [prog,P] = poslpivar(prog,[1;2]);
 P = P + 1e-4;                   % enforce P>=1e-4

 % % Set inequality constraints:
 % %   A'*P*T + T'*P*A <= 0
 Q = A'*P*T + T'*P*A;
 opts.psatz = 1;                 % allow Q>=0 outside domain
 prog = lpi_ineq(prog,-Q,opts);

 % % Solve and retrieve the solution
 solve_opts.solver = 'sedumi';   % use SeDuMi to solve
 solve_opts.simplify = true;     % simplify SDP before solving
 prog = lpisolve(prog,solve_opts);
\end{verbatim}
\end{matlab}
Running this code, we find that the optimization program defined by \texttt{prog} is feasible, indicating that the operators defined by \texttt{A} and \texttt{T} indeed represent a stable system. Thus, the ODE-PDE system~\eqref{eq:demos:stability:odepde} is stable.

\section{DEMO 2: Estimating the Volterra Operator Norm}\label{sec:demos:volterra}

The Volterra integral operator $\mcl{T}:L_2[0,1]\rightarrow L_2[0,1]$ is perhaps the simplest example of a 3-PI operator, defined as
\begin{align*}
    \bl(\mcl{T}\mbf{x}\br)(s)&=\int_{0}^{s}\mbf{x}(\theta)d\theta,   &   s&\in[0,1],
\end{align*}
for any $\mbf{x}\in L_2[0,1]$. In PIETOOLS, this operator can be easily declared as
\begin{matlab}
\begin{verbatim}
 a=0;    b=1;
 opvar Top;
 Top.R.R1 = 1;   Top.I = [a,b];
\end{verbatim}
\end{matlab}
Then, an upper bound on the norm of this operator can be computed by solving the LPI
\begin{align*}
    &\min_{\gamma\geq 0}\ \gamma,    \\
    \text{s.t.}\qquad &\mcl{T}^*\mcl{T}\leq\gamma,
\end{align*}
so that case $C=\sqrt{\gamma}$ for any feasible value $\gamma$ is an upper bound on the norm of $\mcl{T}$. This optimization problem is an LPI, that can be declared and solved in PIETOOLS as
\begin{matlab}
\begin{verbatim}
% % Initialize LPI program
prob = lpiprogram(Top.vars,Top.I);

% % Declare decision variables:
% %   gam \in \R
[prob,gam] = lpidecvar(prob,'gam');

% % Set inequality constraints
% %   Top'*Top-gam <= 0
opts.psatz = 1;                          % Allow gam-Top'*Top<0 outside of [a,b]
prob = lpi_ineq(prob,gam-Top'*Top,opts); % lpi_ineq(prob,Q) enforces Q>=0

% % Set objective function:
% %   min gam
prob = lpisetobj(prob, gam);

% % Solve LPI and retrieve solution
prob = lpisolve(prob);
operator_norm = sqrt(double(lpigetsol(prob,gam)));
\end{verbatim}
\end{matlab}
This code can also be run by calling ``volterra\_operator\_norm\_DEMO''. We obtain an upper bound $C=0.68698$ on the induced norm $\|\mcl{T}\|$ of the Volterra operator. The exact value of the induced norm of this operator is known to be equal to $\|\mcl{T}\|=\frac{2}{\pi}= 0.6366...$.

\section{DEMO 3: Solving the Poincar\'e Inequality}\label{sec:demos:poincare}

In a one dimensional domain $\Omega=[a,b]$, the Poincar\'e inequality imposes a bound on the norm of a function $x(s)$ in terms of the spatial derivative $\partial_s x(s)$ of this function,
\begin{align*}
    \|x\|_{L_2}&\leq C\|\partial_{s}x\|_{L_2},  \hspace*{2.0cm}   \forall x\in W_1[a,b],
    \intertext{where}
    W_1[a,b]&:=\{x\in L_2[a,b]\mid \partial_s x\in L_2[a,b], x(a)=x(b)=0\}.
\end{align*}
The challenge of finding a smallest value of $C$ for which this inequality holds can be posed as an optimization problem
\begin{align*}
    &\min_{\gamma\geq 0}\ \gamma, \\
    \text{s.t.}\qquad &\ip{x}{x}_{L_2}-\gamma\ip{\partial_{s}x}{\partial_{s}x}_{L_2}\leq 0, & \forall x&\in W_1[a,b]
\end{align*}
setting $C=\sqrt{\gamma}$.
To declare this problem, we first note that (assuming $\partial_{s}^2 x$ exists) we can represent $x$ and $\partial_{s}x$ in terms of a fundamental state $\partial_{s}^2 x\in L_2[a,b]$, which is free of the boundary conditions and continuity constraints imposed upon $x$ and $\partial_{s}x$. To represent this in PIETOOLS, we introduce an artificial PDE
\begin{align*}
    \dot{x}(t,s)&=\partial_{s}^2 x(t,s), &   s&\in[a,b], \\
    z(t,s)&=\partial_{s} x(t,s),    \\
    x(t,a)&=x(t,b)=0,
\end{align*}
which we declare as
\begin{matlab}
\begin{verbatim}
 a = 0;  b = 1;
 pvar t s
 x = pde_var(1,s,[a,b]);
 z = pde_var('output',1,s,[a,b]);
 PDE = [diff(x,t)==diff(x,s,2);
        z==diff(x,s,1);
        subs(x,s,a)==0;
        subs(x,s,b)==0];
\end{verbatim}
\end{matlab}
Here, we take a second-order derivative of $x$ in the PDE to explicitly indicate that $x$ must be second order differentiable, as we wish to derive a PIE representation in terms of the fundamental state $x_{\text{f}}:=\partial_{s}^2 x$. To obtain this PIE representation, we call \texttt{convert},
\begin{matlab}
\begin{verbatim}
 PIE = convert(PDE);
 H2op = PIE.T;       % (H2op*x_{ss}) = x;
 H1op = PIE.C1;      % (H1op*x_{ss}) = x_{s}
\end{verbatim}
\end{matlab}
arriving at an equivalent PIE representation of the system as
\begin{align*}
    \partial_{t} (\mcl{H}_2 x_{\text{f}})(t,s)&=x_{\text{f}}(t,s), &   s&\in[a,b],     \\
    z(t,s)&=(\mcl{H}_{1}x_{\text{f}})(t,s).
\end{align*}
In this representation, the fundamental state $x_{\text{f}}:=\partial_{s}^2 x$ is free of any boundary conditions and continuity constraints. Moreover, we note that
\begin{align*}
    \mcl{H}_{2}x_{\text{f}}(t,s)&=x(t,s),\qquad \text{and},\qquad
    \mcl{H}_1 x_{\text{f}}(t,s)=\partial_{s}x(t,s).
\end{align*}
As such, the Poincar\'e inequality optimization problem can be equivalently represented as
\begin{align*}
    &\min_{\gamma\geq 0}\ \gamma, \qquad
    \text{s.t.} \quad \ip{\mcl{H}_{2}v}{\mcl{H}_{2}v}-\gamma\ip{\mcl{H}_1 v}{\mcl{H}_1 v}\leq 0,   \qquad \forall v\in L_2[a,b]
\end{align*}
giving rise to an LPI
\begin{align*}
    &\min_{\gamma\geq 0}\ \gamma, \qquad
    \text{s.t.}\quad \gamma\mcl{H}_1^*\mcl{H}_1 - \mcl{H}_{2}^*\mcl{H}_{2}\succeq 0.
\end{align*}
We declare and solve this LPI in PIETOOLS as
\begin{matlab}
\begin{verbatim}
 % % Initialize LPI program
 prob = lpiprogram(s,[a,b]);

 % % Declare decision variables:
 % %   gam \in \R
 [prob,gam] = lpidecvar(prob,'gam');     % scalar decision variable

 % % Set inequality constraints:
 % %   gam*H1op'*H1op' - H2op'*H2op >= 0
 opts.psatz = 1;    % allow gam H1op'*H1op < H2op'*H2op outside of [a,b]
 prob = lpi_ineq(prob,gam*(H1op'*H1op)-H2op'*H2op,opts);

 % % Set objective function:
 % %   min gam
 prob = lpisetobj(prob, gam);

 % % Solve LPI and retrieve solution
 prob = lpisolve(prob);
 poincare_constant = sqrt(double(lpigetsol(prob,gam)));
\end{verbatim}
\end{matlab}
This code can also be run calling the function ``DEMO3\_poincare\_inequality'', arriving at a constant $C=0.4271$ that satisfies the Poincar\'e inequality on the domain $[a,b]=[0,1]$. On this domain, a minimal value of $C$ is known to be $C=\frac{1}{\pi}=0.3183...$.

\section{DEMO 4: Finding an Optimal Stability Parameter}\label{sec:demos:stability_bisection}

We consider a reaction-diffusion equation on an interval $[a,b]$,
\begin{align*}
    \dot{\mbf{x}}(t,s)&=\lambda \mbf{x}(t,s) + \partial_{s}^2\mbf{x}(t,s),  &   s&\in[a,b],
    &\mbf{x}(t,a)=\mbf{x}(t,b)=0,
\end{align*}
for some $\lambda\in\R$. On the interval $[a,b]=[0,1]$, this system is know to be stable whenever $\lambda\leq\pi^2=9.8696...$. In PIETOOLS, we can numerically estimate this limit using the function \texttt{stability\_PIETOOLS}. In particular, we may equivalent represent the PDE as a PIE of the form
\begin{align*}
    \partial_t\bl(\mcl{T}\mbf{x}_{\text{f}}\br)(t,s&=\bl(\mcl{A}(\lambda)\mbf{x}_{\text{f}}\br)(t,s),   &   s&\in[a,b],
\end{align*}
where $\mbf{x}_{\text{f}}(t,s):=\partial_s^2\mbf{x}(t,s)$. Then, a maximal value of $\lambda$ for which the PIE is stable can be estimated by solving the optimization problem
\begin{align*}
    &\max_{\lambda\in\R,\mcl{P}\in\Pi}\ \lambda,  \\
    \text{s.t.}\qquad &\mcl{P}\succ 0,   \\
    &\mcl{T}^*\mcl{P}\mcl{A}(\lambda) + \mcl{A}^*(\lambda)\mcl{P}\mcl{T}\preccurlyeq 0.
\end{align*}
Unfortunately, both $\lambda$ and $\mcl{P}$ are decision variables in this optimization program, and so the product $\mcl{P}\mcl{A}(\lambda)$ is not linear in the decision variables. As such, this problem cannot be directly implemented as a convex optimization program. However, for any fixed value of $\lambda$, stability of the PIE can be verified by testing feasibility of the LPI
\begin{align*}
    &\mcl{P}\succ 0,
    &\mcl{T}^*\mcl{P}\mcl{A}(\lambda) + \mcl{A}^*(\lambda)\mcl{P}\mcl{T}\preccurlyeq 0,
\end{align*}
an optimization program that has already been implemented in \texttt{stability\_PIETOOLS}. Therefore, to estimate an upper bound on the value of $\lambda$ for which the PDE is stable, we can test stability for given values of $\lambda$, and perform bisection over some domain $\lambda\in[\lambda_{\min},\lambda_{\max}]$ to find an optimal value. For our demonstration, we use $\lambda_{\min}=0$ and $\lambda_{\max}=20$, testing stability for $8$ values of $\lambda$ between these upper and lower bounds:
\begin{matlab}
\begin{verbatim}
 %%% Set bisection limits for lam.
 lam_min = 0;        lam_max = 20;
 lam = 0.5*(lam_min + lam_max);
 n_iters = 8;
\end{verbatim}
\end{matlab}


Using these settings, we perform bisection to find a largest value of $\lambda$ for which stability of the system can be verified. In particular, for a given value of $\lambda\in[\lambda_{\min},\lambda_{\max}]$, we build a PDE structure \texttt{PDE} representing the system, compute the associated PIE structure \texttt{PIE}, and test stability of this PIE by solving the LPI. Then, we check the value of \texttt{feasratio} in the output optimization program structure, which should be close to 1 if the LPI was successfully solved, and thus the system was found to be stable. 
\begin{itemize}
    \item If the system is stable, then it is stable for any value of $\lambda$ smaller than the value \texttt{lam} used in the test. As such, we update the value of $\lambda_{\min}\leftarrow\lambda$, and repeat the test with a greater value $\lambda=\frac{1}{2}(\lambda_{\min}+\lambda_{\max})$.
    \item If stability could not be verified, then stability can also not be verified for any value of $\lambda$ greater than the value \texttt{lam} used in the test. As such, we update the value of $\lambda_{\max}\leftarrow\lambda$, and repeat the test with a greater value $\lambda=\frac{1}{2}(\lambda_{\min}+\lambda_{\max})$.
\end{itemize}
This algorithm can be implemented as
\begin{matlab}
\begin{verbatim}
%%% Perform bisection on the value of lam
for iter = 1:n_iters
    % =============================================
    % === Declare the operators of interest

    % Declare system as PDE. 
    x = pde_var('state',1,s,[0,1]);
    PDE = [diff(x,t)==diff(x,s,2)+lam*x;
                         subs(x,s,0)==0;
                         subs(x,s,1)==0];

    % Convert to PIE.
    PIE = convert(PDE,'pie');
    T = PIE.T;      A = PIE.A;

    % =============================================
    % === Declare the LPI

    % % Initialize LPI program
    prog = lpiprogram(s,[0,1]);
    
    % % Declare decision variables:
    % %   P:L2-->L2,    P>0
    [prog,P] = poslpivar(prog,T.dim);
    P = P + 1e-4;                   % enforce P>=1e-4
    
    % % Set inequality constraints:
    % %   A'*P*T + T'*P*A <= 0
    Q = A'*P*T + T'*P*A;
    opts.psatz = 1;                 % allow Q>=0 outside domain
    prog = lpi_ineq(prog,-Q,opts);
    
    % % Solve and retrieve the solution
    solve_opts.solver = 'sedumi';   % use SeDuMi to solve
    solve_opts.params.fid = 0;      % suppress output in command window
    prog = lpisolve(prog,solve_opts);

    % % Alternatively, uncomment to run pre-defined stability executive
    % prog = lpiscript(PIE,'stability',settings);

    % Check if the system is stable
    is_pinf = prog.solinfo.info.pinf;       % is primal feasible?
    is_dinf = prog.solinfo.info.dinf;       % is dual feasible?
    feasrat = prog.solinfo.info.feasratio;  % ratio should be close to 1 
    if is_dinf || is_pinf || abs(feasrat-1)>0.1
        % Stability cannot be verified --> decrease value of lam...
        lam_max = lam;
        lam = 0.5*(lam_min + lam_max);
    else
        % The system is stable --> try larger value of lam...
        lam_min = lam;
        lam = 0.5*(lam_min + lam_max);
    end    
end
\end{verbatim}
\end{matlab}
This code can also be called using ``DEMO4\_stability\_parameter\_bisection''. Running this demo, we find that stability can be verified whenever $\lambda\leq 9.8438$, approaching the analytic bound of $\lambda=\pi^2\approx 9.8696$

\section{DEMO 5: Constructing and Simulating an Optimal Estimator}\label{sec:demos:estimator}

We consider a reaction-diffusion PDE, with an observed output $y$,
\begin{align}\label{eq:demos:estimator:PDE}
    && \dot{\mbf{x}}(t,s)&=\partial_{s}^2\mbf{x}(t,s) + 4\mbf{x}(t,s) + w(t), & &&    s&\in[0,1], &&\nonumber\\
    \text{with BCs}& & 0&=\mbf{x}(t,0)=\partial_{s}\mbf{x}(t,1),  \nonumber\\
    \text{and outputs}& & z(t)&=\int_{0}^{1}\mbf{x}(t,s)ds + w(t),   \nonumber\\
    &&y(t)&=\mbf{x}(t,1). 
\end{align}
This PDE can be easily declared in PIETOOLS as
\begin{matlab}
\begin{verbatim}
 % Declare independent variables (time and space)
 pvar t s
 % Declare state, input, and output variables
 x = pde_var(1,s,[0,1]);   w = pde_var('in');     
 z = pde_var('out');       y = pde_var('sense');
 % Declare the sytem equations
 lam = 4;
 PDE = [diff(x,t) == diff(x,s,2) + lam*x + w;  % PDE
        z == int(x,s,[0,1]) + w;               % regulated output
        y == subs(x,s,1);                      % observed output
        subs(x,s,0) == 0;                      % first boundary condition
        subs(diff(x,s),s,1) == 0];             % second boundary condition
 display_PDE(PDE);
\end{verbatim}  
\end{matlab}
at which point an equivalent PIE representation can be derived by calling \texttt{convert}:
\begin{matlab}
\begin{verbatim}
 PIE = convert(PDE,'pie');        PIE = PIE.params;
 T = PIE.T;
 A = PIE.A;      C1 = PIE.C1;    C2 = PIE.C2;
 B1 = PIE.B1;    D11 = PIE.D11;  D21 = PIE.D21;
\end{verbatim}
\end{matlab}
Then, the PDE~\eqref{eq:demos:estimator:PDE} can be equivalently represented by a PIE
\begin{align}\label{eq:demos:estimator:PIE}
    \partial_t\bl(\mcl{T}\mbf{x}_{\text{f}}\br)(t,s)&=\bl(\mcl{A}\mbf{x}_{\text{f}}\br)(t,s)+\bl(\mcl{B}_1 w\br)(t,s), &   s&\in[0,1]  && \nonumber\\
    z(t)&=\bl(\mcl{C}_1\mbf{x}_{\text{f}}\br)(t)+\bl(\mcl{D}_{11} w)(t), \nonumber\\
    y(t)&=\bl(\mcl{C}_2\mbf{x}_{\text{f}}\br)(t)+\bl(\mcl{D}_{12} w)(t),
\end{align}
where we define $\mbf{x}_{\text{f}}:=\partial_{s}^2\mbf{x}$, and where $\mbf{x}=\mcl{T}\mbf{x}_{\text{f}}$.

We consider the problem of designing an optimal estimator for the PIE~\eqref{eq:demos:estimator:PIE}. In particular, we construct an estimator of the form
\begin{align}\label{eq:demos:estimator:PIE_estimator}
	\partial_t( \mcl{T}\hat{\mbf{x}_{\text{f}}})(t) &=\mcl{A}{\mbf{\hat{\mbf{x}_{\text{f}}}}}(t)+\mathcal{L}\bl(y(t)-\hat{y}(t)\br), \nonumber\\
	\hat{z}(t) &= \mcl{C}_1\mbf{\hat{\mbf{x}_{\text{f}}}}(t),   \nonumber\\
	\hat{y}(t) &= \mcl{C}_2\mbf{\hat{\mbf{x}_{\text{f}}}}(t),   
\end{align}
so that the error $\mbf{e}(t,s):=\hat{\mbf{x}}_{\text{f}}(t,s)-\mbf{x}_{\text{f}}(t,s)$ in the state and $\tilde{z}(t):=\hat{z}(t)-z(t)$ in the output satisfy
\begin{align}\label{eq:demos:estimator:PIE_error}
	\partial_t(\mcl T \mbf{e})(t) &=(\mcl{A}-\mcl{L}\mcl{C}_2)\mbf{e}(t)-(\mcl{B}_1+\mcl{L}\mcl{D}_{21})w(t) \nonumber\\
	\tilde{z}(t) &= \mcl{C}_1\mbf{e}(t) - \mcl{D}_{11} w(t).
\end{align}
The goal of $H_{\infty}$-optimal estimation, then, is to determine a value for the observer operator $\mcl{L}$ that minimizes the gain $\gamma:=\frac{\|\tilde{z}\|_{L_2}}{\|w\|_{L_2}}$ from disturbances $w$ to errors $\tilde{z}$ in the output. To construct such an operator, we solve the LPI,
\begin{align}\label{eq:demos:estimator:LPI}
	&\min\limits_{\gamma,\mcl{P},\mcl{Z}} ~~\gamma&\notag\\
	&\mcl{P}\succ0, &
	&\hspace{-1ex}Q:=\bmat{-\gamma I& -\mcl D_{11}^{\top}&-(\mcl P\mcl B_1+\mcl Z\mcl D_{21})^*\mcl T\\(\cdot)^*&-\gamma I&\mcl C_1\\(\cdot)^*&(\cdot)^*&(\mcl P\mcl A+\mcl Z\mcl C_2)^*\mcl T+(\cdot)^*}\preccurlyeq 0&
\end{align}
so that, for any solution $(\gamma,\mcl{P},\mcl{Z})$ to this problem, letting $\mcl{L}:=\mcl{P}^{-1} \mcl{Z}$, the estimation error will satisfy $\norm{\hat{z}-z} \leq \gamma \norm{w}$. For more details, see Section~\eqref{sec:LPI_examples:estimation}. This LPI can be implemented in PIETOOLS as
\begin{matlab}
\begin{verbatim}
 % % Initialize LPI program
 prog = lpiprogram(s,[0,1]);

 % % Declare decision variables:
 % %   gam \in \R,     P:L2-->L2,    Z:\R-->L2
 % Scalar decision variable
 [prog,gam] = lpidecvar(prog,'gam');
 % Positive operator variable P>=0
 opts.sep = 1;                   % set P.R.R1=P.R.R2
 [prog,P] = poslpivar(prog,T.dim,4,opts);
 % Indefinite operator variable Z: \R-->L2
 [prog,Z] = lpivar(prog,[0,1;1,0],4);

 % % Set inequality constraints:
 % %   Q <= 0
 Q = [-gam,              -D11',        -(P*B1+Z*D21)'*T;
      -D11,              -gam,         C1;
      -T'*(P*B1+Z*D21),  C1',          (P*A+Z*C2)'*T+T'*(P*A+Z*C2)];
 prog = lpi_ineq(prog,-Q);

 % % Set objective function:
 % %   min gam
 prog = lpisetobj(prog, gam);

 % % Solve and retrieve the solution
 prog_sol = lpisolve(prog);
 % Extract solved value of decision variables
 gam_val = lpigetsol(prog_sol,gam);
 Pval = lpigetsol(prog_sol,P);
 Zval = lpigetsol(prog_sol,Z);
 % Build optimal observer operator L
 Lval = getObserver(Pval,Zval);
\end{verbatim}
\end{matlab}
returning a value \texttt{Lval} of the operator $\mcl{L}$ such that the $L_2$-gain $\frac{\|\tilde{z}\|_{L_2}}{\|w\|_{L_2}}$ is bounded from above by \texttt{gam\_val}.

Given the observer operator $\mcl{L}$, we can construct the Estimator~\eqref{eq:demos:estimator:PIE_estimator}, obtaining a PIE
\begin{align*}
    \partial_t\left(\bmat{\mcl{T}&0\\0&\mcl{T}}\bmat{\mbf{x}_{\text{f}}\\\hat{\mbf{x}}_{\text{f}}}\right)(t,s)&=\left(\bmat{\mcl{A}&0\\-\mcl{L}\mcl{C}_{2}&\mcl{A}+\mcl{L}\mcl{C}_{2}}\bmat{\mbf{x}_{\text{f}}\\\hat{\mbf{x}}_{\text{f}}}\right)(t,s)+\left(\bmat{\mcl{B}_{1}\\\mcl{L}\mcl{D}_{21}}w\right)(t) \\
    \bmat{z\\\hat{z}}(t)&=\left(\bmat{\mcl{C}_1&0\\0&\mcl{C}_1}\bmat{\mbf{x}_{\text{f}}\\\hat{\mbf{x}}_{\text{f}}}\right)(t)+\left(\bmat{\mcl{D}_{11}\\0}w\right)(t)
\end{align*}
We can construct this PIE in PIETOOLS by first declaring the estimator dynamics with input signal $y$ using \texttt{piess}, and then taking the linear fractional transformation of the original PIE with this estimator PIE, as
\begin{matlab}
\begin{verbatim}
 PIE_est = piess(T,A+Lval*C2,-Lval,C1(1,:));
 PIE_CL = pielft(PIE,PIE_est);
\end{verbatim}
\end{matlab}
Alternatively, this system could also be constructed using the function \texttt{closedLoopPIE} as
\begin{matlab}
\begin{verbatim}
 PIE_CL = closedLoopPIE(PIE,Lval,'observer');
\end{verbatim}    
\end{matlab}
returning the closed-loop system associated to the original PIE as in~\eqref{eq:demos:estimator:PIE} defined by \texttt{PIE}, and the observer as in~\eqref{eq:demos:estimator:PIE_estimator} with operator $\mcl{L}$ defined by \texttt{Lval}. 
Then, we can use PIESIM to simulate the evolution of the PDE state $\mbf{x}(t)=\bl(\mcl{T}\mbf{x}_{\text{f}}\br)(t)$ and its estimate $\hat{\mbf{x}}(t)=\bl(\mcl{T}\hat{\mbf{x}_{\text{f}}}\br)(t)$ associated to the system defined by \texttt{PIE\_CL}. For this, we first declare initial conditions $\sbmat{\mbf{x}_{\text{f}}(0,s)\\\hat{\mbf{x}}_{\text{f}}(0,s)}$ and values of the disturbance $w(t)$ as
\begin{matlab}
\begin{verbatim}
 % % Declare initial values and disturbance
 syms st sx real
 uinput.ic = [-10*sx;        % actual initial PIE state value
                  0];        % estimated initial PIE state value
 uinput.w = 2*sin(pi*st);
\end{verbatim}
\end{matlab}
Here we use symbolic objects \texttt{st} and \texttt{sx} to represent a temporal variable $t$ and spatial variable $s$ respectively. Note that, since we will be simulating the PIE directly, the initial conditions \texttt{uinput.ic} will also correspond to the initial values of the PIE state $\sbmat{\mbf{x}_{\text{f}}(0,s)\\\hat{\mbf{x}}_{\text{f}}(0,s)}$, not to those of the PDE state. Here, we let the initial PIE state $\mbf{x}_{\text{f}}(0)$ be linear, and start with an estimate of this state that is just $\hat{\mbf{x}}_{\text{f}}(0)=0$. Given these initial conditions, we then simulate the evolution of the PDE state $\sbmat{\mbf{x}\\\mbf{\hat{x}}}=\sbmat{\mcl{T}&0\\0&\mcl{T}}\sbmat{\mbf{x}_{\text{f}}\\\hat{\mbf{x}}_{\text{f}}}$ using PIESIM as
\begin{matlab}
\begin{verbatim}
 % % Set options for discretization and simulation
 opts.plot = 'yes';  % plot the solution
 opts.N = 8;         % expand using 8 Chebyshev polynomials
 opts.tf = 2;        % simulate up to t = 2
 opts.dt = 1e-3;     % use time step of 10^-3

 % % Simulate solution to the PIE with estimator.
 [solution,grid] = PIESIM(PIE_CL,opts,uinput);
 % % Extract actual and estimated state and output at each time step.
 tval = solution.timedep.dtime;
 x_act = reshape(solution.timedep.primary{2}(:,1,:),opts.N+1,[]);
 x_est = reshape(solution.timedep.primary{2}(:,2,:),opts.N+1,[]);
 z_act = solution.timedep.regulated{1}(1,:);
 z_est = solution.timedep.regulated{1}(2,:);
\end{verbatim}
\end{matlab}
Here, we expand the solution using 8 Chebyshev polynomials, and we simulate up to \texttt{opts.tf=2}, taking time steps of \texttt{opts.dt=1e-3}. Having obtained the solution, we then collect the actual values $\mbf{x}(t,s)$ of the PDE state at each time step and each grid point in \texttt{x\_act}, and the estimated values $\mbf{\hat{x}}(t,s)$ of the PDE state at each time step and each grid point in \texttt{x\_est} (both states depend on $s$, therefore they are stored in \texttt{solution.timedep.primary\{2\}}). The regulated outputs, however, are finite-dimensional; thus, they are stored in \texttt{solution.timedep.regulated\{1\}}. Plotting the obtained values for the PDE state and its estimate at several grid points, as well as the error \texttt{x\_eta-x\_act} in estimated state, we obtain a graph as in Figure~\ref{fig:demos:estimator}.

\begin{figure}[H]
	\centering
	\includegraphics[width=\textwidth]{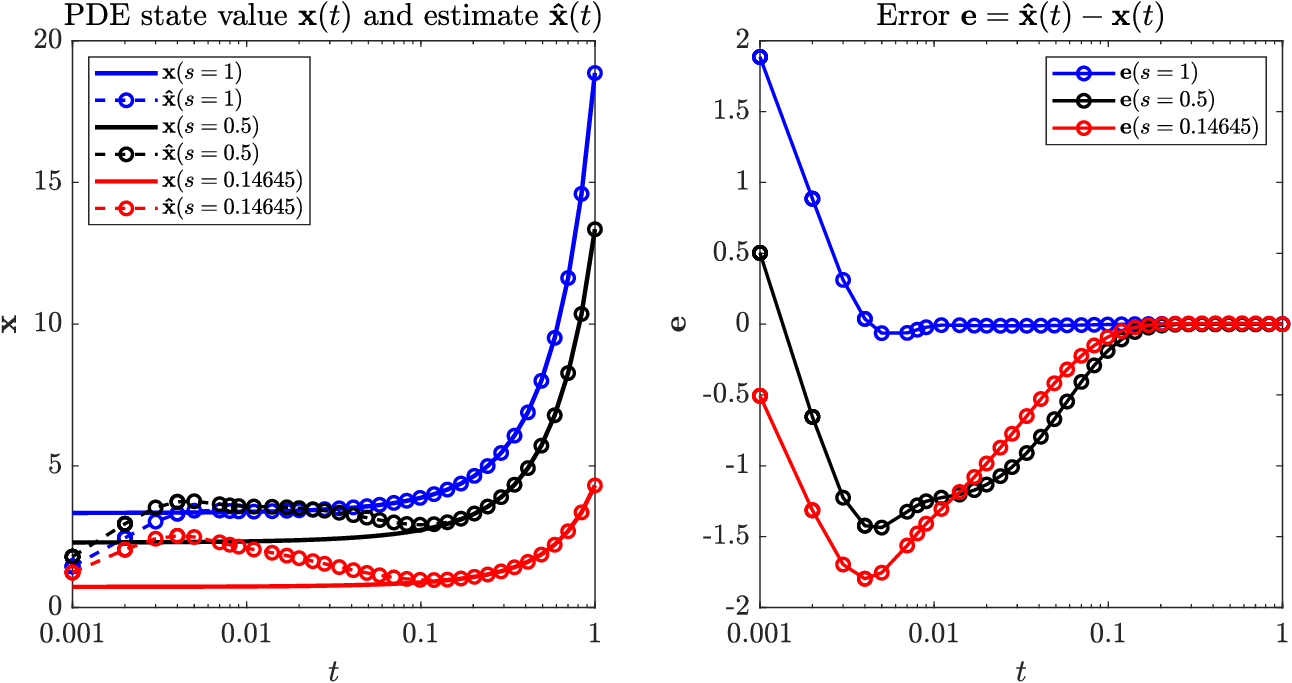}
	\caption{Simulated value of PDE state $\mbf{x}(t,s)$ and estimated state $\mbf{\hat{x}}(t,s)$, along with the error $\mbf{e}=\hat{\mbf{x}}(t,s)-\mbf{x}(t,s)$ associated to the PDE~\eqref{eq:demos:estimator:PDE} at several grid points $s\in[0,1]$, using the Estimator~\eqref{eq:demos:estimator:PIE_estimator} with operator $\mcl{L}$ computed by solving the $H_{\infty}$-optimal estimator LPI. See the file ``PIETOOLS\_Code\_Illustrations\_Ch11\_Demos'' for the code to achieve these plots.}\label{fig:demos:estimator}
\end{figure}

As Figure~\ref{fig:demos:estimator} shows, the value of the estimate state $\mbf{\hat{x}}$ at each of the grid points quickly converges to that of the actual state $\mbf{x}$, despite starting off with a rather poor estimate of $\mbf{\hat{x}}=0$. Within a time of 1, the values of the actual and estimated state become indistinguishable (at the displayed scale), with the error converging to zero. This is despite the fact that the value of the state itself continues to increase, as the PDE~\eqref{eq:demos:estimator:PDE} is unstable.

The full code constructing the optimal estimator, simulating the PDE state and its estimate, and plotting the results has been included in PIETOOLS as the demo file\\
``DEMO5\_Hinf\_Optimal\_Estimator''.

\section{DEMO 6: \texorpdfstring{$\hinf$}{Hinf}-optimal Controller synthesis for PDEs}\label{sec:demo:control}
We consider an unstable reaction-diffusion PDE, with output $z$, disturbance $w$ and control input $u$
\begin{align}\label{eq:demos:control:PDE}
    && \dot{\mbf{x}}(t,s)&=\partial_{s}^2\mbf{x}(t,s) + 4\mbf{x}(t,s) + sw(t)+su(t), & &&    s&\in[0,1], &&\nonumber\\
    \text{with BCs}& & 0&=\mbf{x}(t,0)=\partial_{s}\mbf{x}(t,1),  \nonumber\\
    \text{and outputs}& & z(t)&=\bmat{\int_{0}^{1}\mbf{x}(t,s)ds + w(t)\\u(t)}. 
\end{align}
This PDE can be easily declared in PIETOOLS as
\begin{matlab}
\begin{verbatim}
 % Declare independent variables (time and space)
 pvar t s
 % Declare state, input, and output variables
 x = pde_var('state',1,s,[0,1]);   
 z = pde_var('output',2);
 w = pde_var('input',1);          u = pde_var('control',1);
 % Declare the sytem equations
 lam = 5;
 PDE = [diff(x,t) == diff(x,s,2) + lam*x + s*w + s*u;
        z == [int(x,s,[0,1])+w; u];
        subs(x,s,0)==0;
        subs(diff(x,s),s,1)==0];
\end{verbatim}  
\end{matlab}
Next, we convert the PDE system to a PIE by using the following code and extract relevant PI operators for easier access:
\begin{matlab}
\begin{verbatim}
 PIE = convert(PDE,'pie');
 T = PIE.T;
 A = PIE.A;      C1 = PIE.C1;    B2 = PIE.B2;
 B1 = PIE.B1;    D11 = PIE.D11;  D12 = PIE.D12;
\end{verbatim}
\end{matlab}
Then, the PDE~\eqref{eq:demos:control:PDE} can be equivalently represented by a PIE
\begin{align}\label{eq:demos:control:PIE}
 \partial_t\bl(\mcl{T}\dot{\mbf{x}}_{\text{f}}\br)(t,s)&=\bl(\mcl{A}\mbf{x}_{\text{f}}\br)(t,s)+\bl(\mcl{B}_1 w\br)(t,s)+\bl(\mcl{B}_2 u\br)(t,s), &   s&\in[0,1]  && \nonumber\\
    z(t)&=\bl(\mcl{C}_1\mbf{x}_{\text{f}}\br)(t)+\bl(\mcl{D}_{11} w)(t)+\bl(\mcl{D}_{12} u)(t),
\end{align}
where we define $\mbf{x}_{\text{f}}:=\partial_{s}^2\mbf{x}$ and  $\mbf{x}=\mcl{T}\mbf{x}_{\text{f}}$.

We now attempt to find a state-feedback $\hinf$-optimal controller for the PIE~\eqref{eq:demos:control:PIE}. In particular, we use $u(t) = \mcl K \mbf x_{\text{f}}(t)$, where $\mcl K:L_2\to \R$ is a PI operator of the form
\[\mcl K \mbf x_{\text{f}} = \int_a^b K(s)\mbf x_{\text{f}}(s) ds.\] 
Then we get the closed-loop system
\begin{align}\label{eq:demos:control:PIE_CL}
	\partial_t(\mcl{T} \mbf{x}_{\text{f}})(t) &=(\mcl{A}+\mcl B_2\mcl K){\mbf{x}_{\text{f}}}(t)+\mcl B_1 w(t), \nonumber\\
	z(t) &= (\mcl{C}_1+\mcl D_{12}\mcl K)\mbf{x}_{\text{f}}(t)+\mcl D_{11} w(t).
\end{align}
Next, we solve the LPI for $\hinf$-optimal control to find a $\mcl K$ that minimizes the gain $\gamma:=\frac{\|z\|_{L_2}}{\|w\|_{L_2}}$ from disturbances $w$ to errors $\tilde{z}$ in the output. For more details, see Section~\eqref{sec:LPI_examples:control}. To find such an operator, we solve the LPI,
\begin{flalign}\label{eq:demos:control:LPI}
	\min\limits_{\gamma,\mcl{P},\mcl{Z}}& & \gamma& \notag\\
    s.t.&   &   &\mcl{P}\succ0,
    \qquad
	\bmat{-\gamma I & \mcl{D}_{11}&\mcl{T}(\mcl{P}\mcl{C}_1+\mcl{Z}\mcl{D}_{12})\\(\cdot)^*&-\gamma I&\mcl{B}_1^*\\(\cdot)^*&(\cdot)^*&\mcl{T}(\mcl{A}\mcl{P}+\mcl{B}_2\mcl{Z})^*+(\mcl{A}\mcl{P}+\mcl{B}_2\mcl{Z})\mcl{T}^*}\preccurlyeq 0,
\end{flalign}
so that, for any solution $(\gamma,\mcl{P},\mcl{Z})$ to this problem, letting $\mcl{K}:=\mcl Z\mcl{P}^{-1}$, we have $\norm{z} \leq \gamma \norm{w}$. We can declare and solve this LPI in PIETOOLS as follows:
\begin{matlab}
\begin{verbatim}
 % % Initialize LPI program
 prog = lpiprogram(s,[0,1]);

 % % Declare decision variables:
 % %   gam \in \R,     P:L2-->L2,    Z:\R-->L2
 % Scalar decision variable
 [prog,gam] = lpidecvar(prog,'gam');
 % Positive operator variable P>=0
 [prog,P] = poslpivar(prog,[0,0;1,1],4);
 % Enforce strict positivity P >= 1e-3
 P = P + 1e-3;
 % Indefinite operator variable Z
 [prog,Z] = lpivar(prog,[1,0;0,1],2);

 % % Set inequality constraints:
 % %   Q <= 0
 Q = [-gam*eye(2),       D11,    (C1*P+D12*Z)*(T');
      D11',              -gam,   B1';
      T*(C1*P+D12*Z)',   B1,     (A*P+B2*Z)*(T')+T*(A*P+B2*Z)'];
 prog = lpi_ineq(prog,-Q);

 % % Set objective function:
 % %   min gam
 prog = lpisetobj(prog, gam);

 % % Solve and retrieve the solution
 prog_sol = lpisolve(prog);
 % Extract solved value of decision variables
 gam_val = lpigetsol(prog_sol,gam);
 Pval = lpigetsol(prog_sol,P);
 Zval = lpigetsol(prog_sol,Z);
 % Build the optimal control operator K.
 Kval = getController(Pval,Zval,1e-3);
\end{verbatim}
\end{matlab}
Alternatively, this LPI could also be solved using the executive \texttt{PIETOOLS\_Hinf\_control}, as e.g.
\begin{matlab}
\begin{verbatim}
 settings = lpisettings('heavy');
 [prog, Kval, gam_val] = PIETOOLS_Hinf_control(PIE, settings); 
\end{verbatim}
\end{matlab}
Using either option, we obtain an \texttt{opvar} class object \texttt{Kval} representing a value of the operator $\mcl{K}$ such that the $L_2$-gain $\frac{\|z\|_{L_2}}{\|w\|_{L_2}}$ is bounded from above by \texttt{gam\_val}. Next, we construct the closed-loop PIE using the function \texttt{piess} as
\begin{matlab}
\begin{verbatim}
 PIE_CL = piess(T,A+B2*Kval,B1,C1+D12*Kval,D11);
\end{verbatim}
\end{matlab}
or, equivalently, using the function \texttt{closedLoop\_PIE} as
\begin{matlab}
\begin{verbatim}
    PIE_CL = closedLoopPIE(PIE,Kval);
\end{verbatim}
\end{matlab}

Having constructed the closed-loop PIE representation, we can use PIESIM to simulate the evolution of the PDE state $\mbf{x}(t)=\bl(\mcl{T}\mbf{x}_{\text{f}}\br)(t)$ associated to the system defined by \texttt{PIE\_CL}. In this case, we simulate with an initial condition $\mbf{x}(0,s)=\frac{4}{\pi^2}\sin(\frac{\pi}{2}s)$, and disturbance $w(t)=\frac{\sin(\pi t)}{t+\epsilon}$, where we introduce $\epsilon=2.2204\cdot 10^{-16}$ to avoid a singularity at $t=0$. Note that, since we are simulating a PIE, we must also specify the initial value of the PIE state $\mbf{x}_{\text{f}}(t)=\partial_{s}^2\mbf{x}(t)$, yielding $\mbf{x}_{\text{f}}(0,s)=\sin(\frac{\pi}{2}s)$. We declare these values as shown below
\begin{matlab}
\begin{verbatim}
 % % Declare initial values and disturbance
 syms st sx real
 uinput.ic = sin(sx*pi/2);
 uinput.w = sin(pi*st)./(st+eps); 
\end{verbatim}
\end{matlab}
Next, we set the parameters related to the numerical scheme and simulation as shown below
\begin{matlab}
\begin{verbatim}
 % % Set options for discretization and simulation
 opts.plot = 'yes';  % plot the solution
 opts.N = 16;        % expand using 16 Chebyshev polynomials
 opts.tf = 2;        % simulate up to t = 2
 opts.dt = 1e-2;     % use time step of 10^-2
\end{verbatim}
\end{matlab}
In this case, the state involves only a single 2nd-order differentiable variable $\mbf{x}(t,s)$, and we expand this state using 16 Chebyshev polynomials in $s$. We use these settings to simulate the open- (without controller) and closed-loop (with controller) PIEs as
\begin{matlab}
\begin{verbatim}
 % Simulate uncontrolled PIE and extract solution
 [solution_OL,grid] = PIESIM(PIE,opts,uinput);
 tval = solution_OL.timedep.dtime;
 x_OL = reshape(solution_OL.timedep.primary{2}(:,1,:),opts.N+1,[]);
 z_OL = solution_OL.timedep.regulated{1}(1,:);
 % Simulate controlled PIE and extract solution
 [solution_CL,~] = PIESIM(PIE_CL,opts,uinput);
 x_CL = reshape(solution_CL.timedep.primary{2}(:,1,:),opts.N+1,[]);
 z_CL = solution_CL.timedep.regulated{1}(1,:);
 u_CL = solution_CL.timedep.regulated{1}(2,:);
 w = double(subs(uinput.w,st,tval));
\end{verbatim}
\end{matlab}
The resulting values of the PDE state in the open- and closed-loop setting are plotted in Fig.~\ref{fig:demos:control:state}, at several time instances. The corresponding evolution of the regulated output is displayed in Fig.~\ref{fig:demos:control:output}. The figures show that, as expected, the state and regulated output blow up if no control is exerted, as the PDE~\eqref{eq:demos:control:PDE} is unstable. However, using the control $u(t)=\mcl{K}\mbf{x}_{\text{f}}(t)$ with gain $\mcl{K}$ computed by solving the LPI~\eqref{eq:demos:control:LPI}, the PDE state converges to zero within 1 second -- despite the presence of a disturbance $w(t)$ -- and the regulated output $z(t)=\int_{0}^{1}\mbf{x}(t,s)ds+w(t)$ becomes driven entirely by the disturbance $w(t)$.

The full code for designing the optimal controller and simulating the PDE state has been included in PIETOOLS as the demo file 
``DEMO6\_Hinf\_optimal\_control''.

\begin{figure}[H]
	\centering
	\includegraphics[width=\textwidth]{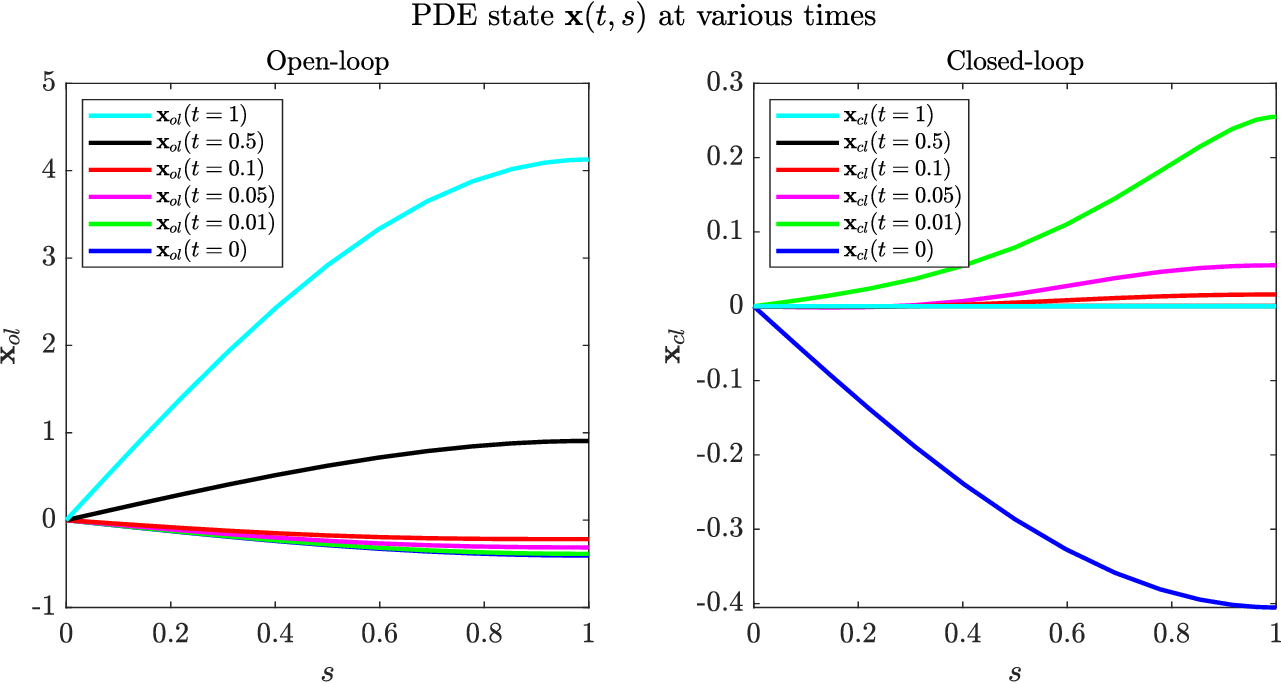}
 \vspace*{-0.6cm}
	\caption{Simulated value of PDE state $\mbf{x}(t,s)$ solution to~\eqref{eq:demos:control:PDE} at several times $t\in[0,1]$, both without (left) and with (right) feedback $u(t)=(\mcl{K}\mbf{x}_{\text{f}})(t)$, starting with $\mbf{x}(0,s)=-\frac{4}{\pi^2}\sin(\frac{\pi}{2}s)$ and with disturbance $w(t)= \frac{\sin(\pi t)}{t+\epsilon}$.}\label{fig:demos:control:state}
\end{figure}

\begin{figure}[H]
	\centering
	\includegraphics[width=\textwidth]{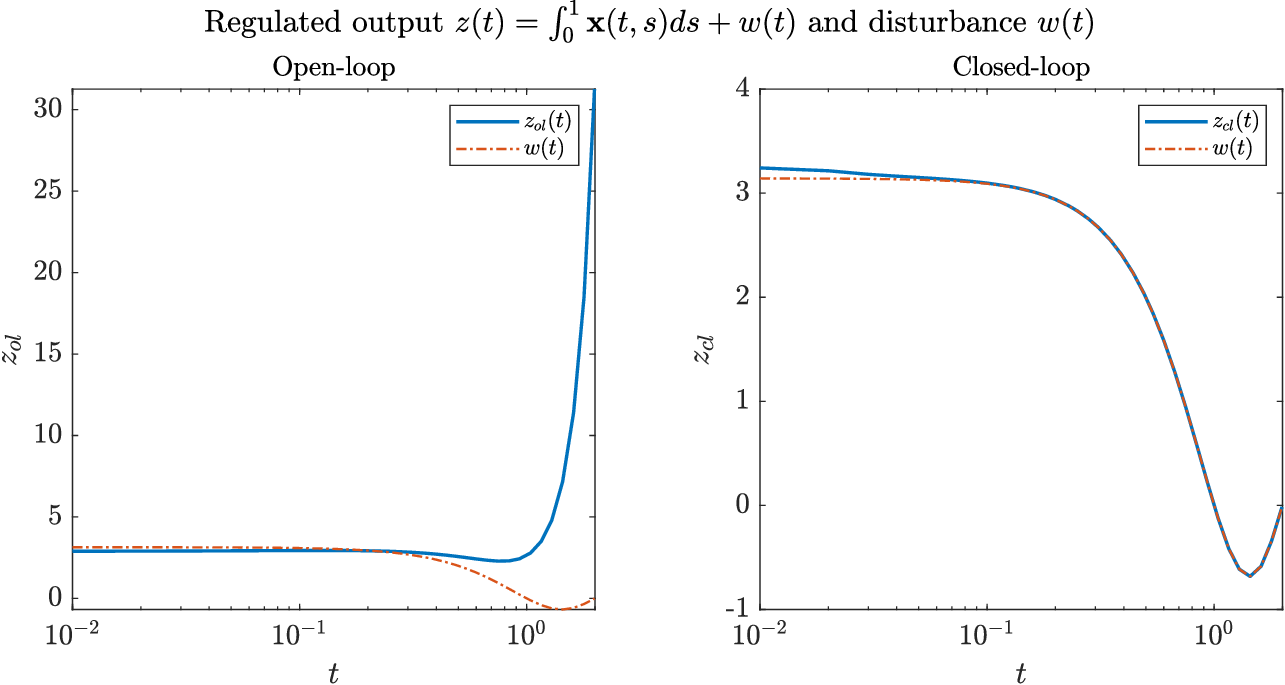}
 \vspace*{-0.6cm}
	\caption{Simulated evolution of regulated output $z(t)$ to~\eqref{eq:demos:control:PDE}, both without (left) and with (right) feedback $u(t)=(\mcl{K}\mbf{x}_{\text{f}})(t)$, starting with $\mbf{x}(0,s)=-\frac{4}{\pi^2}\sin(\frac{\pi}{2}s)$ and with disturbance $w(t)= \frac{\sin(\pi t)}{t+\epsilon}$.}\label{fig:demos:control:output}
\end{figure}

\section{DEMO 7: Observer-based Controller design and simulation for PDEs}\label{sec:demo:observer-control}
In Sections \ref{sec:demos:estimator} and \ref{sec:demo:control}, we showed the procedure to find optimal estimator and controller for a PDE. The controller was designed based on the information of the complete PDE state, which, in practice, is unrealistic. So, in this section, we use the observer and controller obtained in the previous sections and couple them by changing the control law to be dependent on the estimated state instead of PDE state, i.e., we set $u = \mcl K \hat{\mbf x}_{\text{f}}$, where $\mcl K$ is the control gains obtained in~\ref{sec:demo:control} and $\hat{\mbf x}_{\text{f}}$ is the PIE state estimate from the observer found in~\ref{sec:demos:estimator}.

For demonstration, we reuse the unstable reaction-diffusion PDE, with output $z$, disturbance $w$ and control input $u$
\begin{align}\label{eq:demos:obs-control:PDE}
    && \dot{\mbf{x}}(t,s)&=\partial_{s}^2\mbf{x}(t,s) + 4\mbf{x}(t,s) + sw(t)+su(t), & &&    s&\in[0,1], &&\nonumber\\
    \text{with BCs}& & 0&=\mbf{x}(t,0)=\partial_{s}\mbf{x}(t,1),  \nonumber\\
    \text{and outputs}& & z(t)&=\bmat{\int_{0}^{1}\mbf{x}(t,s)ds\\u(t)}\notag\\
    && y(t) &= \mbf x(t,1). 
\end{align}
Recall that this PDE can be defined using the code
\begin{matlab}
\begin{verbatim}
 % Declare independent variables (time and space)
 pvar t s
 % Declare state, input, and output variables
 x = pde_var('state',1,s,[0,1]);   
 w = pde_var('input',1);      u = pde_var('control',1);
 z = pde_var('output',2);     y = pde_var('sense',1);
 % Declare the sytem equations
 lam = 5;
 PDE = [diff(x,t) == diff(x,s,2) + lam*x + s*w + s*u;
        z == [int(x,s,[0,1]); u];
        y == subs(x,s,1);
        subs(x,s,0)==0;
        subs(diff(x,s),s,1)==0];
\end{verbatim}  
\end{matlab}

Next, we convert the PDE system to a PIE by using the following code and extract relevant PI operators for easier access:
\begin{matlab}
\begin{verbatim}
 PIE = convert(PDE,'pie');
 T = PIE.T;      
 A = PIE.A;      B1 = PIE.B1;    B2 = PIE.B2;
 C1 = PIE.C1;    D11 = PIE.D11;  D12 = PIE.D12;
 C2 = PIE.C2;    D21 = PIE.D21;  D22 = PIE.D22;
\end{verbatim}
\end{matlab}

A detailed presentation on how to manually perform optimal observer and controller synthesis for PIE systems is already given in Section~\ref{sec:demos:estimator} and Section~\ref{sec:demo:control} respectively, and therefore, will not be repeated here. Instead, we will simply compute the optimal observer gain $\mcl{L}$ and controller gain $\mcl{K}$ using the respective executive functions, as
\begin{matlab}
\begin{verbatim}
 settings = lpisettings('heavy');
 [prog_k, Kval, gam_co_val] = PIETOOLS_Hinf_control(PIE, settings);
 [prog_l, Lval, gam_ob_val] = PIETOOLS_Hinf_estimator(PIE, settings);
\end{verbatim}
\end{matlab}

Now, recall from Equations~\eqref{eq:demos:estimator:PIE_estimator} and ~\eqref{eq:demos:control:PIE} that the combined system is given by
\begin{align}\label{eq:demos:obscon:PIE}
\partial_t\bl(\mcl{T}\mbf{x}_{\text{f}}\br)(t)&=\bl(\mcl{A}\mbf{x}_{\text{f}}\br)(t)+\bl(\mcl{B}_1 w\br)(t)+\bl(\mcl{B}_2 u\br)(t), \nonumber\\
    z(t)&=\bl(\mcl{C}_1\mbf{x}_{\text{f}}\br)(t)+\bl(\mcl{D}_{11} w)(t)+\bl(\mcl{D}_{12} u)(t),\notag\\
    \partial_t(\mcl{T} \hat{\mbf{x}_{\text{f}}})(t) &=\mcl{A}\hat{\mbf{x}_{\text{f}}}(t)+\mathcal{L}\bl(\mcl{C}_2\mbf{x}_{\text{f}}(t)-\mcl{C}_2\mbf{\hat{\mbf{x}_{\text{f}}}}(t)\br), \nonumber\\
	\hat{z}(t) &= \mcl{C}_1\hat{\mbf{x}_{\text{f}}}(t).
\end{align}
Using the observer-controller coupling, $u = \mcl K \hat{\mbf x}_{\text{f}}$, we get the closed-loop PIE
\begin{align}\label{eq:demos:obcon:PIE_CL}
\partial_t\left(\bmat{\mcl{T}&0\\0&\mcl{T}}\bmat{\mbf{x}_{\text{f}}\\\mbf{x}_{\text{f}}}\right)(t,s)&=\left(\bmat{\mcl{A}&\mcl B_2\mcl K\\-\mcl{L}\mcl{C}_{2}&\mcl{A}+\mcl{L}\mcl{C}_{2}}\bmat{\mbf{x}_{\text{f}}\\\mbf{x}_{\text{f}}}\right)(t,s)+\left(\bmat{\mcl{B}_{1}\\\mcl{L}\mcl{D}_{21}}w\right)(t) \notag\\
    \bmat{z\\\hat{z}}(t)&=\left(\bmat{\mcl{C}_1&\mcl D_{12}\mcl K\\0&\mcl{C}_1}\bmat{\mbf{x}_{\text{f}}\\\hat{\mbf{x}_{\text{f}}}}\right)(t)+\left(\bmat{\mcl{D}_{11}\\0}w\right)(t).
\end{align}

We can construct the above closed-loop system using the code
\begin{matlab}
\begin{verbatim}
 PIE_est = piess(T,A+Lval*C2,-Lval,{C1(1,:);Kval});
 PIE_CL = pielft(PIE,PIE_est);
\end{verbatim}
\end{matlab}
first constructing a PIE for the estimated state $\hat{\mbf x}_{\text{f}}(t,s)$, with input $y(t)$ and output $u(t)=(\mcl{K}\hat{\mbf x}_{\text{f}})(t)$, and then taking the linear fractional transformation with the original PIE to obtain the closed-loop system.

Once, we have the closed-loop PIE object, we can use \texttt{PIESIM} to simulate the system for some initial conditions and disturbances. In this case, we consider the disturbance $w(t)=10e^{-t}$ and initial state $\mbf{x}(0,s)=-\frac{5}{3}s^3+5s$, so that the initial fundamental state is $\mbf{x}_{\text{f}}(0,s)=-10s$. For the state estimate $\hat{\mbf{x}}_{\text{f}}(t,s)$, we assume to start with no information, setting $\hat{\mbf{x}}_{\text{f}}(0,s)=0$. We simulate the open-loop (without control or observer) and closed-loop systems for these values of the initial state and disturbance using the code
\begin{matlab}
\begin{verbatim}
 % % Declare initial values and disturbance
 syms st sx real
 uinput.ic = [-10*sx; 0];
 uinput.w = 10*exp(-st);

 % % Set options for discretization and simulation
 opts.plot = 'yes';  % plot the solution
 opts.N = 8;         % Expand using 8 Chebyshev polynomials
 opts.tf = 2;        % Simulate up to t = 2
 opts.dt = 1e-2;     % Use time step of 10^-2

 % % Perform the actual simulation
 % Simulate uncontrolled PIE and extract solution
 [solution_OL,grid] = PIESIM(PIE,opts,uinput);
 tval = solution_OL.timedep.dtime;
 x_OL = reshape(solution_OL.timedep.primary{2}(:,1,:),opts.N+1,[]);
 z_OL = solution_OL.timedep.regulated{1}(1,:);
 % Simulate controlled PIE and extract solution
 [solution_CL,~] = PIESIM(PIE_CL,opts,uinput);
 x_CL = reshape(solution_CL.timedep.primary{2}(:,1,:),opts.N+1,[]);
 xhat_CL = reshape(solution_CL.timedep.primary{2}(:,2,:),opts.N+1,[]);
 z_CL = solution_CL.timedep.regulated{1}(1,:);
 zhat_CL = solution_CL.timedep.regulated{1}(3,:);
 u_CL = solution_CL.timedep.regulated{1}(2,:);
 w = double(subs(uinput.w,st,tval));
\end{verbatim}
\end{matlab}

The simulated evolutions of the PDE state and regulated output are plotted in Fig.~\ref{fig:demos:observer-control:state} and Fig.~\ref{fig:demos:observer-control:output}, respectively, both for the open- and closed-loop systems. The figures show that the observer-based controller indeed stabilizes the system, with the PDE state and regulated output of the closed-loop system both converging to zero. The control effort $u(t)=\mcl{K}\hat{\mbf{x}}_{\text{f}}(t)$ used to achieve this stabilization is displayed in Fig.~\ref{fig:demos:observer-control:control}. Although this control effort is initially quite large, the effort quickly decays as the PDE state converges to the equilibrium. 

The full code for designing the optimal controller and simulating the PDE state has been included in PIETOOLS as the demo file 
``DEMO7\_observer\_based\_control''.

\begin{figure}[H]
	\centering
	\includegraphics[width=1\textwidth]{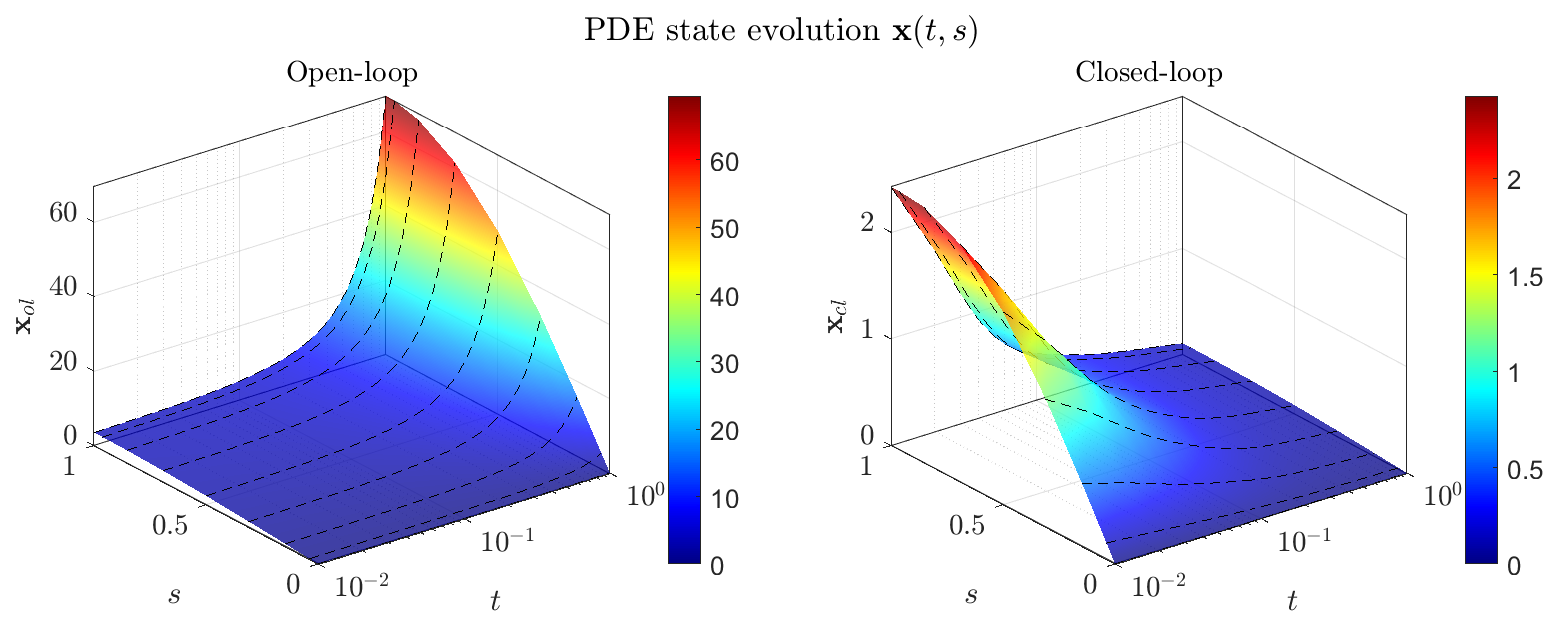}
 \vspace*{-0.6cm}
	\caption{Simulated value of PDE state $\mbf{x}(t,s)$ solution to~\eqref{eq:demos:control:PDE} at several times $t\in[0,1]$, without (left) and with (right) the observer-based feedback $u(t)=(\mcl{K}\hat{\mbf{x}}_{\text{f}})(t)$, starting with state $\mbf{x}(0,s)=-\frac{5}{3}s^3+5s$ and estimate $\hat{\mbf{x}}(0,s)=0$, and with disturbance $w(t)= 10 e^{-t}$.}\label{fig:demos:observer-control:state}
\end{figure}

\begin{figure}[H]
	\centering
	\includegraphics[width=\textwidth]{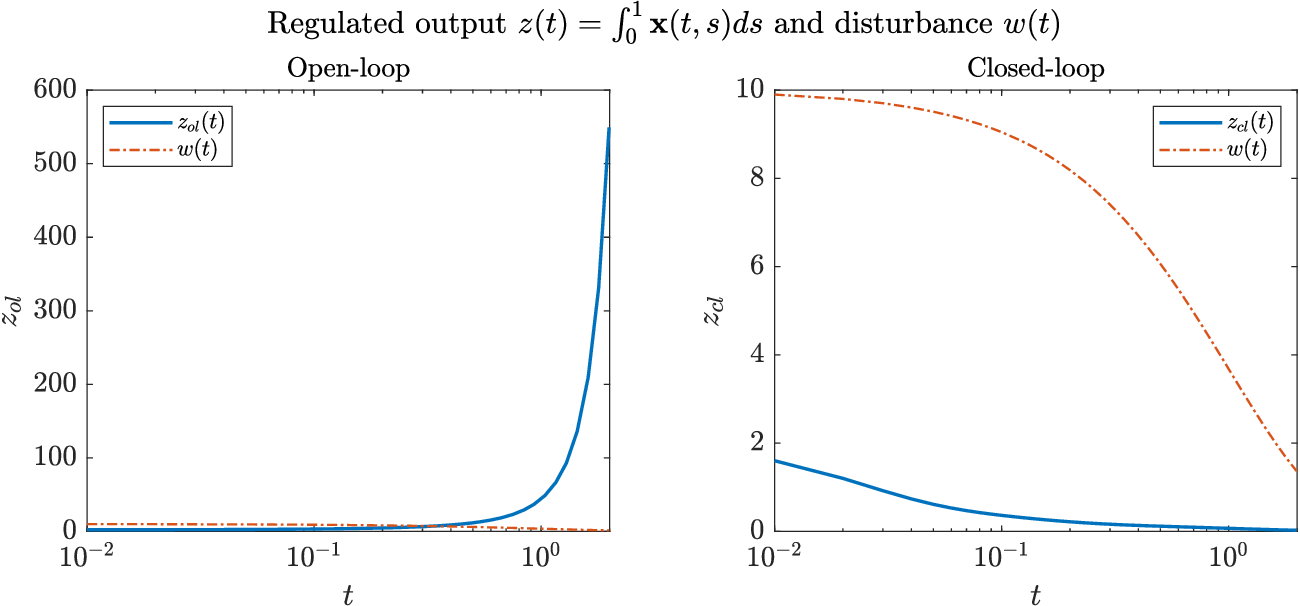}
 \vspace*{-0.6cm}
	\caption{Simulated evolution of regulated output $z(t)$ to~\eqref{eq:demos:control:PDE}, both without (left) and with (right) the observer-based feedback $u(t)=(\mcl{K}\hat{\mbf{x}}_{\text{f}})(t)$, starting with state $\mbf{x}(0,s)=-\frac{5}{3}s^3+5s$ and estimate $\hat{\mbf{x}}(0,s)=0$, and with disturbance $w(t)= 10 e^{-t}$.}\label{fig:demos:observer-control:output}
\end{figure}

\begin{figure}[H]
	\centering
	\includegraphics[width=\textwidth]{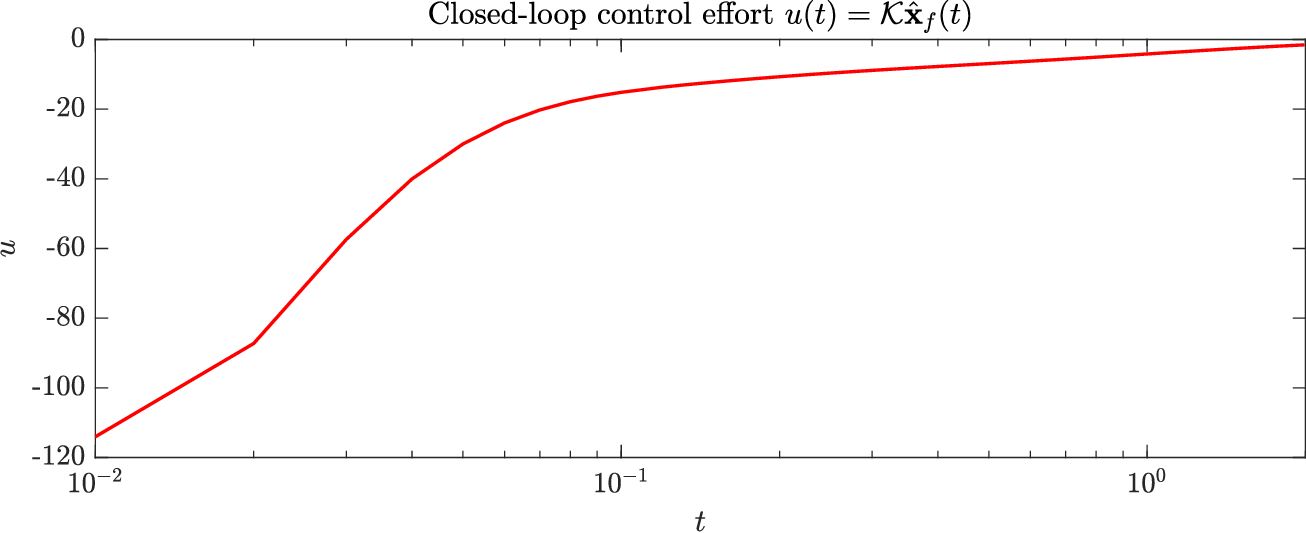}
 \vspace*{-0.6cm}
	\caption{Control effort $u(t)=(\mcl{K}\hat{\mbf{x}}_{\text{f}})(t)$ for~\eqref{eq:demos:control:PDE}, computed by solving the LPIE~\eqref{eq:demos:control:LPI}, and simulated with initial state $\mbf{x}(0,s)=-\frac{5}{3}s^3+5s$, initial estimate $\hat{\mbf{x}}(0,s)=0$, and disturbance $w(t)= 10 e^{-t}$.}\label{fig:demos:observer-control:control}
\end{figure}

\section{DEMO 8: $H_{2}$-Norm Analysis of PDEs}\label{sec:demo:h2-norm}

Consider the following 1D pure-convection equation with homogeneous Dirichlet boundary condition:
\begin{align}\label{eq:demos:h2_gain:1D_PDE}
    \partial_t\mbf x(t,s)&=\partial_s\mbf x(t,s)+(s-s^2)w(t)\quad s\in [0,1],t\in \R_+ \\
    z(t)&=\int_0^1 \mbf x(t,s) ds,\notag\\
        \mbf x(t,1)&=0.
\end{align}

We want to compute the $H_2$-norm of this system, as defined in Sec.~\ref{subsec:H2norm}. To compute the $H_2$ norm, the following LPI may be used. First, we declare the PDE in PIETOOLS as
\begin{matlab}
\begin{verbatim}
% % Declare system as PDE
% Declare independent variables (time and space)
pvar s t
% Declare state, input, and output variables
x = pde_var('state',1,s,[0,1]);
w = pde_var('in',1);
z = pde_var('out',1);
% Declare the sytem equations
pde = [diff(x,t,1)==diff(x,s,1)+(s-s^2)*w;    % dynamics
                z==int(x,s,[0,1]);                     % output equation
                subs(x,s,1)==0];;                            % boundary condition
pde=initialize(pde);
\end{verbatim}
\end{matlab}
Then, we convert the system to the equivalent PIE representation as
\begin{matlab}
\begin{verbatim}
PIE = convert(pde,'pie');
T = PIE.T;      A = PIE.A;
B1 = PIE.B1;    C1 = PIE.C1;
\end{verbatim}
\end{matlab}

\begin{align}\label{cgramian_PIE}
&\min\limits_{\mu,\mcl{P}} ~~\gamma&\notag\\
	&\mcl{P}\succ0&\notag\\
	&trace(\mcl C_1 \mcl P \mcl C_1^*)\leq \mu& \notag\\
    &\mcl A \mcl P \mcl T^* + \mcl T \mcl P \mcl A^* + \mcl B_1 \mcl B_1^* \preccurlyeq 0&
\end{align}

If feasible for some $\mu \geq 0$ and PI operator $\mcl P$, then $\norm{\Sigma}_{H_2}\leq \gamma = \sqrt{\mu}$. Note that this is reduced version of~\eqref{cgramian_LPI_NC} in the particular case where $\mcl Q =\mcl{PT^*}$, with a coercive operator $\mcl P \succ 0$.

To declare and solve the LPI, the following command-lines are used.
\begin{matlab}
\begin{verbatim}
% % Initialize LPI program
prog = lpiprogram(s,[0,1]); 

% % Declare decision variables:
% %   gam \in \R,     W:L2-->L2,    Z:\R-->L2
% Scalar decision variable
[prog,gam] = lpidecvar(prog,'gam');
% Positive operator variable W>=0 with default polynomial degrees up to 3.
[prog,W] = poslpivar(prog,[0;1]);

% % Set inequality constraints:
% %   A W T* + T W A* + B1 B1* <= 0
% %   gam >= trace(C1 W C1*)
% Operator inequality Dop<=0
Dop =  A*W*T'+T*W*A'+B1*B1';
prog = lpi_ineq(prog,-Dop);
% Scalar inequality gam >= trace(C1 W C1*)
Aux = C1*W*C1';
traceVal = trace(Aux.P);
prog = lpi_ineq(prog, gam-traceVal);

% % Set objective function:
% %   min gam
prog = lpisetobj(prog, gam);

% % Solve and retrieve the solution
opts.solver = 'sedumi';         % Use SeDuMi to solve the SDP
prog_sol = lpisolve(prog,opts);
% Extract solved value of decision variables
gamd = sqrt(double(lpigetsol(prog_sol,gam)));
Wc = lpigetsol(prog_sol,W);
\end{verbatim}
\end{matlab}

PIETOOLS gives $\gamma = 0.1016$ as output with the default settings used, an upper-bound on the $H_2$ norm of the system. The same value of $0.1016$ can also be obtained by numerical integration of the output $z(t)$ to the non-zero initial condition as defined in Eq.~\eqref{eqn:PIEaux}.

\section{DEMO 9: $L_{2}$-Gain Analysis of (2D) PDEs}\label{sec:demo:l2gain_2D}

Consider the following 2D reaction-diffusion equation with Dirichlet-Neumann boundary conditions:
\begin{align}\label{eq:demos:l2_gain:2D_PDE}
    \partial_{t}\mbf{x}(t,s_{1},s_{2})
    &=\partial_{s_{1}}^{2}\mbf{x}(t,s_{1},s_{2}) +\partial_{s_{2}}^{2}\mbf{x}(t,s_{1},s_{2}) +\lambda\mbf{x}(t,s_{1},s_{2}) +\mbf{w}(t),    & &s_{1},s_{2}\in[a,b]\times[c,d],  \notag\\
    z(t)&=\int_{a}^{b}\int_{c}^{d}\mbf{x}(t,s_{1},s_{2})\, ds_{2} ds_{1}, &
     &t\in \R_+,\notag\\
    \mbf{x}(t,a,s_{2})&=\mbf{x}(t,b,s_{2})=0,\qquad \mbf{x}(t,s_{1},c)=\partial_{s_{2}}\mbf{x}(t,s_{2},d)=0.    
\end{align}
We simulate the output response of the system to a bounded disturbance, for parameter values $a=c=0$, $b=d=1$, and $\lambda=5$. To this end, we first declare the PDE in PIETOOLS as
\begin{matlab}
\begin{verbatim}
 % Declare independent variables (time and space)
 pvar s1 s2 t
 % Declare state, input, and output variables
 a = 0;      b = 1;
 c = 0;      d = 1;
 x = pde_var('state',1,[s1;s2],[a,b;c,d]);
 w = pde_var('in',1);
 z = pde_var('out',1);
 % Declare the sytem equations
 lam = 5;
 PDE = [diff(x,t) == diff(x,s1,2) +diff(x,s2,2) + lam*x + w;
        z == int(x,[s1;s2],[a,b;c,d]);
        subs(x,s1,a)==0;
        subs(x,s1,b)==0;
        subs(x,s2,c)==0;
        subs(diff(x,s2),s2,d)==0];
 PDE = initialize(PDE);
\end{verbatim}
\end{matlab}
Next, we declare the disturbance and initial values for simulation. For our purposes, we consider a zero initial state $\mbf{x}(0,s_{1},s_{2})=0$, and a decaying but oscillating disturbance $w(t)=20\sin(\pi t)e^{-t/2}$, which we declare using symbolic variables as
\begin{matlab}
\begin{verbatim}
 % % Declare initial values and disturbance
 syms st sx sy real
 uinput.ic = 0;
 uinput.w = 20*sin(pi*st)*exp(-st/2);
\end{verbatim}
\end{matlab}
We simulate the solution up to a time $T=10$, performing temporal integration using a time step of $\Delta t=10^{-2}$, and expanding in space using $8\times 8$ Chebyshev polynomials. To perform this simulation, we call PIESIM as
\begin{matlab}
\begin{verbatim}
 opts.plot = 'yes';  % don't plot the final solution
 opts.N = 8;         % Expand using 8x8 Chebyshev polynomials
 opts.tf = 10;       % Simulate up to t = 10
 opts.dt = 1e-2;     % Use time step of 10^-2
[solution,grid] = PIESIM(PDE,opts,uinput);
\end{verbatim}
\end{matlab}
Note that, since the PDE state in this example is two-dimensional, it would be contained in \\\texttt{solution.timedep.primary\{4\}}, while the finite-dimensional regulated output is, as before, stored in \texttt{solution.timedep.regulated\{1\}}. 
A plot of the simulated evolution of the regulated state is shown in Fig.~\ref{fig:demos:l2_gain:output}. 

\begin{figure}[H]
	\centering
	\includegraphics[width=\textwidth]{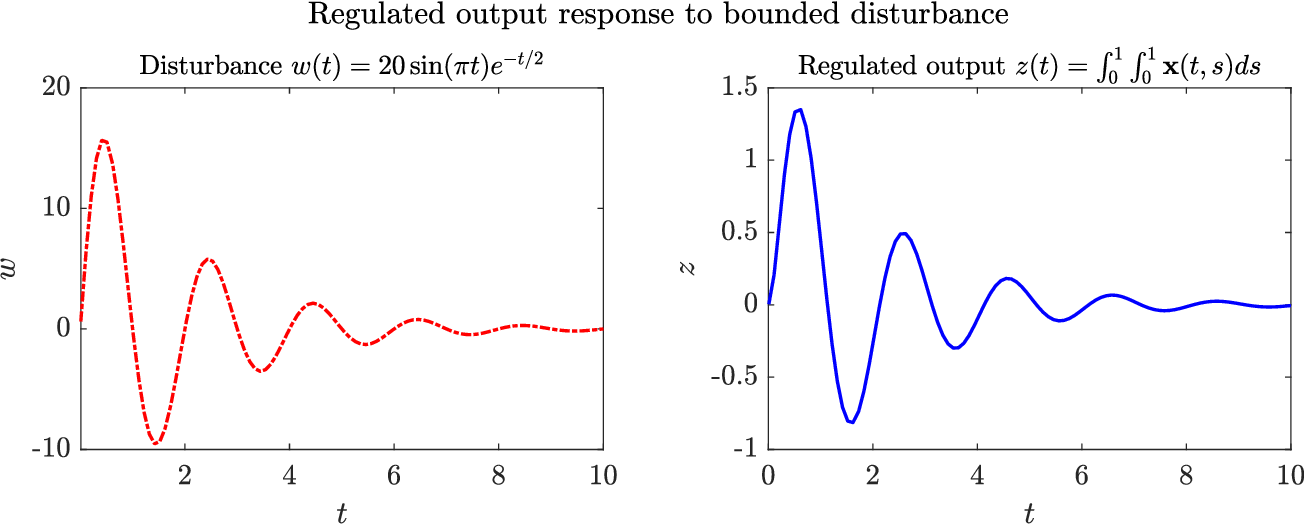}
 \vspace*{-0.6cm}
	\caption{Simulated evolution of regulated output $z(t)$to~\eqref{eq:demos:l2_gain:2D_PDE} in response to a disturbance $w(t)= 20\sin(\pi t)e^{-t/2}$, starting with a zero initial state.}\label{fig:demos:l2_gain:output}
\end{figure}

From the figure, it appears that the system is stable, with the bounded (decaying) disturbance $w\in L_{2}[0,\infty)$ also resulting in a bounded (decaying) output $z\in L_{2}[0,\infty)$. In fact, the value of the regulated output appears to be roughly one tenth that of the disturbance at any time $t\geq 0$, suggesting that the $L_{2}$-gain $\sup_{w,z\in L_{2}[0,\infty),w\neq0}\frac{\|z\|_{L_{2}}}{\|w\|_{L_{2}}}$ of the system may also be bounded by $0.1$. To compute a more accurate upper bound $\gamma$ on the value of this gain, we solve the $L_{2}$-gain LPI presented in Sec.~\ref{sec:LPI_examples:analysis:L2-gain}. To this end we first, compute the equivalent PIE representation of the PDE as
\begin{matlab}
\begin{verbatim}
 PIE = convert(PDE);
 T = PIE.T;
 A = PIE.A;     C = PIE.C1;
 B = PIE.B1;    D = PIE.D11;
\end{verbatim}
\end{matlab}
Given these operators, the PIE representation of the PDE~\eqref{eq:demos:l2_gain:2D_PDE} takes the form
\begin{align*}
    \partial_{t}(\mcl{T}\mbf{x}_{\text{f}})(t)&=\mcl{A}\mbf{x}_{\text{f}}(t) +\mcl{B}w(t),   \\
    z(t)&=\mcl{C}\mbf{x}_{\text{f}}(t) +\mcl{D}w(t),
\end{align*}
where the fundamental state $\mbf{x}_{\text{f}}(t)\in L_{2}[[a,b]\times[c,d]]$ corresponds to the highest-order mixed derivative of the PDE state, $\mbf{x}_{\text{f}}(t)=\partial_{s_{1}}^2\partial_{s_{2}}^2\mbf{x}(t)$. Given this PIE representation, a smallest upper bound on the $L_{2}$-gain of the PDE can be computed by solving the LPI:
\begin{align*}
    \min_{\gamma\in\R,~\mcl{P}}&     &       &\gamma,    \\
    \text{s.t.}&       &     &\mcl{P}\succ 0,\qquad 
    \bmat{-\gamma & \mcl{D} & \mcl{B}^*\mcl{P}\mcl{T}\\
        \mcl{D}& -\gamma & \mcl{C}\\
        \mcl{T}^*\mcl{P}\mcl{B} & \mcl{C}^* & \mcl{A}^*\mcl{P}\mcl{T}+\mcl{T}^*\mcl{P}\mcl{A}}
\end{align*}

Note that this is reduced version of~\eqref{eq:hinf_lpi} in the particular case where $\mcl Q =\mcl{PT^*}$, with a coercive operator $\mcl P \succ 0$. We declare and solve this LPI as
\begin{matlab}
\begin{verbatim}
 % % Initialize LPI program
 prog = lpiprogram([s1;s2],[a,b;c,d]);

 % % Declare decision variables:
 % %   gam \in \R,     P:L2-->L2,    Z:\R-->L2
 % Scalar decision variable
 [prog,gam] = lpidecvar(prog,'gam');
 % Positive operator variable P>=0
 [prog,P] = poslpivar(prog,T.dim);

 % % Set inequality constraints:
 % %   Q <= 0
 Q = [-gam,        D',      (P*B)'*T;
      D,           -gam,    C;
      T'*(P*B),    C',     (P*A)'*T+T'*(P*A)];
 opts_Q.psatz = 2;
 prog = lpi_ineq(prog,-Q,opts_Q);

 % % Set objective function:
 % %   min gam
 prog = lpisetobj(prog, gam);

 % % Solve and retrieve the solution
 prog_sol = lpisolve(prog,opts);
 % Extract solved value of decision variable
 gam_val = double(lpigetsol(prog_sol,gam));
\end{verbatim}
\end{matlab}
Solving this LPI, we find an upper bound on the $L_{2}$-gain of the system as $\gamma=0.09431$, approaching the true value of the $L_{2}$-gain of approximately $0.09397$.

\chapter{Libraries of PDE and TDS Examples in PIETOOLS}\label{ch:examples}

In Chapters~\ref{ch:PDE_DDE_representation},~\ref{ch:alt_PDE_input} and~\ref{ch:alt_DDE_input}, we have shown how partial differential equations and time-delay systems can be declared in PIETOOLS through different input formats. To help get started with each of these input formats, PIETOOLS includes a variety of example pre-defined PDE and TDS systems, including common examples and particular models from the literature. These examples are collected in the \texttt{PIETOOLS\_examples} folder, and can be accessed calling the function \texttt{examples\_PDE\_library\_PIETOOLS} and the scripts \texttt{examples\_DDE\_library\_PIETOOLS} and \texttt{examples\_NDSDDF\_library\_PIETOOLS}. In this Chapter, we illustrate how this works, focusing on PDE examples in Section~\ref{sec:examples:PDE}, and TDS examples in Section~\ref{sec:examples:TDS}.

\section{A Library of PDE Example Problems}\label{sec:examples:PDE}

To help get started with analysing and simulating PDEs in PIETOOLS, a variety of PDE models have been included in the directory \texttt{PIETOOLS\_examples/Examples\_Library}. These systems include common PDE models, as well as examples from the literature, and are defined through separate MATLAB functions. Each of these functions takes two arguments
\begin{enumerate}
    \item \texttt{GUI}: A binary index (0 or 1) indicating whether the example should be opened in the graphical user interface;
    \item \texttt{params}: A string for specifying allowed parameters in the system.
\end{enumerate}
For example, calling \texttt{help PIETOOLS\_PDE\_Ex\_Reaction\_Diffusion\_Eq}, PIETOOLS indicates that this function declares a reaction-diffusion PDE
\begin{align*}
    \dot{\mbf{x}}(t,s)&=\lambda \mbf{x}(t,s) + \partial_{s}^2\mbf{x}(t,s),  &   s&\in[0,1], \\
    \mbf{x}(t,0)&=\mbf{x}(t,1)=0,
\end{align*}
where the value of the parameter $\lambda$ can be set. Then, calling
\begin{matlab}
\begin{verbatim}
 >> PDE = PIETOOLS_PDE_Ex_Reaction_Diffusion_Eq(0,{'lam=10;'})    
\end{verbatim}
\end{matlab}
we obtain a \texttt{pde\_struct} object \texttt{PDE} representing the reaction-diffusion equation with $\lambda=10$. Calling
\begin{matlab}
\begin{verbatim}
 >> PDE = PIETOOLS_PDE_Ex_Reaction_Diffusion_Eq(1,{'lam=10;'})    
\end{verbatim}
\end{matlab}
The PDE will also be loaded in the GUI, though a default value of $\lambda=9.86$ will be used, as the GUI will always load a pre-defined file, which cannot be adjusted from the command line.

To simplify the process of extracting PDE examples, PIETOOLS includes a function\\ \texttt{examples\_PDE\_library\_PIETOOLS}. In this function, each of the pre-defined PDEs is assigned an index, allowing desired PDEs to be extracted by calling \texttt{examples\_PDE\_library\_PIETOOLS} with the associated index. For example, scrolling through this function we find that the reaction-diffusiong equation is the fifth system in the list, and therefore, we can obtain a \texttt{pde\_struct} object defining this system by calling the library function with argument ``5'', returning
\begin{matlab}
\begin{verbatim}
 >> PDE = examples_PDE_library_PIETOOLS(5);  
  --- Extracting ODE-PDE example 5 ---
  
  (d/dt) x(t,s) = 9.86 * x(t,s) + (d^2/ds^2) x(t,s);
  0 = x(t,0);
  0 = x(t,1);
 
 Would you like to run the executive associated to this problem? (y/n) 
 -->
\end{verbatim}
\end{matlab}
We note that the function asks whether an executive should be run for the considered PDE. This is because, for each of the PDE examples, an associated LPI problem has also been declared, matching one of the \texttt{executive} files (see also Chapter~\ref{ch:LPI_examples}). For the reaction-diffusion equation, the proposed executive is the \texttt{PIETOOLS\_stability} function, testing stability of the PDE. Entering \texttt{yes} in the command line window, this executive will be automatically run, whilst entering \texttt{no} will stop the function, and just return the \texttt{pde\_struct} object \texttt{PDE}.

Using the \texttt{examples\_PDE\_library\_PIETOOLS} function, parameters in the PDE can also be adjusted, calling e.g.
\begin{matlab}
\begin{verbatim}
 >> PDE = examples_PDE_library_PIETOOLS(5,'lam=10;');  
  --- Extracting ODE-PDE example 5 ---
  
  (d/dt) x(t,s) = 10 * x(t,s) + (d^2/ds^2) x(t,s);
  0 = x(t,0);
  0 = x(t,1);
\end{verbatim}
\end{matlab}
Similarly, if multiple parameters can be specified, we specify each of these parameters separately. For example, we note that Example 7 corresponds to a PDE
\begin{align*}
    \dot{\mbf{x}}(t,s)&= c(s) \mbf{x}(t,s) + b(s)\partial_{s}\mbf{x}(t,s) + a(s)\partial_{s}^2\mbf{x}(t,s),    &
    \mbf{x}(t,0)&=\partial_{s}\mbf{x}(t,1)=0,
\end{align*}
where the values of the functions $a(s)$, $b(s)$ and $c(s)$ for $s\in[0,1]$ can be specified. As such, we can declare this PDE for $a=1$, $b=2$ and $c=3$ by calling
\begin{matlab}
\begin{verbatim}
 >> PDE = examples_PDE_library_PIETOOLS(7,'a=1;','b=2;','c=3;');
  --- Extracting ODE-PDE example 7 ---
  
  (d/dt) x(t,s) = 3 * x(t,s) + 2 * (d/ds) x(t,s) + (d^2/ds^2) x(t,s);
  0 = x(t,0);
  0 = (d/ds) x(t,1);
\end{verbatim}
\end{matlab}
Finally, we can also open the PDE in the GUI by calling
\begin{matlab}
\begin{verbatim}
 >> examples_PDE_library_PIETOOLS(7,'GUI');
  --- Extracting ODE-PDE example 7 ---
\end{verbatim}
\end{matlab}
or extract the PDE as a \texttt{pde\_struct} (terms-format) and open it in the GUI by calling
\begin{matlab}
\begin{verbatim}
 >> PDE = examples_PDE_library_PIETOOLS(7,'TERM','GUI');
  --- Extracting ODE-PDE example 7 ---
\end{verbatim}
\end{matlab}
In each case, the function will still ask whether the executive associated with the PDE should be run as well. Of course, you can also convert the PDE to a PIE yourself using \texttt{convert}, and then run any desired executive manually, assuming this executive makes sense (e.g. there's no sense in computing an $H_{\infty}$-gain if your system has no outputs).


\section{Libraries of DDE, NDS, and DDF Examples}\label{sec:examples:TDS}

Aside from the PDE examples, a list of TDS examples is also included in PIETOOLS, in DDE, NDS, and DDF format. Unlike the PDE problems, however, the TDS examples are not declared in distinct functions, but are divided over two scripts: \texttt{examples\_DDE\_library\_PIETOOLS} and \texttt{examples\_NDSDDF\_library\_PIETOOLS}. In each of these scripts, most examples are commented, and only one example should be uncommented at any time. This example can then be extracted by calling the script, adding a structure \texttt{DDE}, \texttt{NDS} or \texttt{DDF} to the MATLAB workspace. To extract a different example, the desired example must be uncommented, and all other examples must be commented, at which point the script can be called again to obtain a structure representing the desired system.
We expect to update the DDE and NDS/DDF example files in a future release to match the format used for the PDE example library.

\subsection{DDE Examples} 
We have compiled a list of 23 DDE numerical examples, grouped into: stability analysis problems; input-output systems; estimator design problems; and feedback control problems. These examples are drawn from the literature and citations are used to indicate the source of each example. For each group, the relevant flags have been included to indicate which executive mode should be called after the example has been loaded.

\subsection{NDS and DDF Examples} 
There are relatively few DDFs which do not arise from a DDE or NDS. Hence, we have combined the DDF and NDS example libraries into the script \texttt{examples\_NDSDDF\_library\_PIETOOLS}. The Neutral Type systems are listed first, and currently consist only of stability analysis problems - of which we include 13. As for the DDE case, the library is a script, so the user must uncomment the desired example and call the script from the root file or command window. For the NDS problems, after calling the example library, in order to convert the NDS to a DDF or PIE, the user can use the following commands:
\begin{matlab}
\begin{verbatim}
 >> NDS = initialize_PIETOOLS_NDS(NDS);
 >> DDF = convert_PIETOOLS_NDS(NDS,'ddf');
 >> DDF = minimize_PIETOOLS_DDF(DDF);
 >> PIE = convert_PIETOOLS_DDF(DDF,'pie');
\end{verbatim}
\end{matlab}
In contrast to the NDS case, we only include 3 DDF examples. The first two are difference equations which cannot be represented in either the NDS or DDE format. The third is a network control problem, which is also included in the DDE library in DDE format.

\chapter{Standard Applications of LPI Programming}\label{ch:LPI_examples}
In Chapter~\ref{ch:LPIs}, we showed how general LPI optimization programs can be declared and solved in PIETOOLS. In this chapter, we provide several applications of LPI programming for analysis, estimation, and control of PIEs. Recall that such PIEs take the form
\begin{align}\label{eq:standardizedPIE_LPIforAnalysisofPIEs}
	\partial_t(\mcl T \mbf{x}_{\text{f}})(t)+\mcl{T}_{w} \dot{w}(t)+\mcl{T}_{u}\dot{u}(t)&=\mcl A\mbf{x}_{\text{f}}(t)+\mcl{B}_1w(t)+\mcl B_2u(t), \quad \mbf{x}_{\text{f}}(0)=\mbf \mbf{x}_{\text{I}}\notag\\
	z(t) &= \mcl{C}_1\mbf{x}_{\text{f}}(t) + \mcl{D}_{11}w(t) + \mcl{D}_{12}u(t),\notag\\
	y(t) &= \mcl{C}_2\mbf{x}_{\text{f}}(t) + \mcl{D}_{21}w(t) + \mcl{D}_{22}u(t),
\end{align}
where $\mbf{x}_{\text{f}}=\sbmat{x_0\\\mbf{x}_1\\\mbf{x}_2\\\mbf{x}_3}\in\sbmat{\R^{n_0}\\L_2^{n_1}[a,b]\\L_2^{n_2}[c,d]\\L_2^{n_3}[[a,b]\times[c,d]]}$, and where $\mcl{T}$ through $\mcl{D}_{22}$ are all PI operators. In Section~\ref{sec:LPI_examples:analysis}, we provide several LPIs for stability analysis and $H_{\infty}$-gain estimation of such PIEs, setting $u=0$. Then, in Section~\ref{sec:LPI_examples:estimation}, we present an LPI for $H_{\infty}$-optimal estimation of PIEs of the form~\eqref{eq:standardizedPIE_LPIforAnalysisofPIEs}, followed by an LPI for $H_{\infty}$-optimal full-state feedback control in Section~\ref{sec:LPI_examples:control}. We note that, almost all of these LPIs have already been implemented as \texttt{executive} functions in PIETOOLS, and we will refer to these executives when applicable.

\section{LPIs for Analysis of PIEs}\label{sec:LPI_examples:analysis}
Using LPIs, several properties of a PIE as in Equation~\eqref{eq:standardizedPIE_LPIforAnalysisofPIEs} may be tested, as listed in this section. In particular, the LPIs listed below are extensions of classical results used in analysis of ODEs using LMIs. For most of these LPIs, PIETOOLS includes an executive function that may be run to solve it for a given PIE.

\subsection{Operator Norm}
For a PI operator $\mcl{T}$, an upper bound $\sqrt{\gamma}$ on the operator norm $\norm{\mcl{T}}$ can be computed by solving the following LPI.

\begin{align}\label{eq:PI_op_norm}
    & \min\limits_{\gamma,\mcl{P}} ~~\gamma&\notag\\
    &\mcl{T}^*\mcl{T}-\gamma\preccurlyeq 0&
\end{align}

This LPI has not been implemented as an executive in PIETOOLS, but has been implemented in the demo file \texttt{volterra\_operator\_norm\_DEMO} (see also Section~\ref{sec:demos:volterra}).

\subsection{PIE to PDE Stability} This tests For the existence of a $C>0$ such that $\norm{(\mcl T \mbf x_f)(t)}\le C \norm{(\mbf x_f)(t)}$ for any $\mbf x_f(t)$ which satisfies the PIE $\mcl T \dot{\mbf x}_f(t)=\mcl A \mbf x_f(t)$. This is a weaker notion of stability than PDE stability and is typically required for naive formulations of the wave and beam equations.\\

\noindent \textbf{LPI Formulation: } For given PI operators $\mcl{T}$ and $\mcl{A}$, stability of the PIE
\begin{align}\label{PIE2PDE_stab}
    \partial_t(\mcl{T}\mbf{x}_{\text{f}})(t)=\mcl{A}\mbf{x}_{\text{f}}(t)
\end{align}
can be tested by solving the following LPI for $\alpha>0$, $\beta \ge 0$.

\begin{align}\label{eq:stab_lpi_PIE2PDE}
	&\mcl T^* \mcl Q= \mcl Q^* \mcl T=\mcl{P}\succeq \alpha \mcl T^* \mcl T&\notag\\
	&\mcl{Q}^*\mcl{A}+\mcl{A}^*\mcl{Q}\preccurlyeq -\beta \mcl P&.
\end{align}

If this LPI is feasible, then the PIE is PIE2PDE stable. 
Given a structure \texttt{PIE}, this LPI may be solved for the associated PIE by calling
\begin{matlab}
\begin{verbatim}
 >> [prog, Pop] = PIETOOLS_PIE2PDEstability(PIE, settings);
\end{verbatim}
\end{matlab}
or
\begin{matlab}
\begin{verbatim}
 >> [prog, Pop] = lpiscript(PIE, 'stability', settings);
\end{verbatim}
\end{matlab}
Here \texttt{prog} will be an LPI program structure describing the solved problem and \texttt{Pop} will be a \texttt{dopvar} object describing the (unsolved) decision operator $\mcl{P}$ from which the solved operator can be derived using
\begin{matlab}
\begin{verbatim}
 >> Pop = getsol_lpivar(prog,Pop);
\end{verbatim}
\end{matlab}
See Chapter~\ref{ch:LPIs} for more information on the operation of this function and the \texttt{settings} input.

\subsection{PDE Stability} This tests For the existence of a $C>0$ such that $\norm{(\mcl T \mbf x_f)(t)}\le C \norm{(\mcl T \mbf x_f)(t)}$. This is a strong notion of stability and is not typically satisfied by, e.g. the wave equation.\\

\noindent \textbf{LPI Formulation: } For given PI operators $\mcl{T}$ and $\mcl{A}$, stability of the PIE
\begin{align}\label{PIE_stab}
    \partial_t(\mcl{T}\mbf{x}_{\text{f}})(t)=\mcl{A}\mbf{x}_{\text{f}}(t)
\end{align}
can be tested by solving the following LPI.

\begin{align}\label{eq:stab_lpi}
	&\mcl{P}\succ0&\notag\\
	&\mcl{T}^*\mcl{P}\mcl{A}+\mcl{A}^*\mcl{P}\mcl{T}\preccurlyeq 0&.
\end{align}

If there exists a PI operator $\mcl{P}$ such that this LPI is feasible, then the PIE is stable. 
Given a structure \texttt{PIE}, this LPI may be solved for the associated PIE by calling
\begin{matlab}
\begin{verbatim}
 >> [prog, Pop] = PIETOOLS_PDEstability(PIE, settings);
\end{verbatim}
\end{matlab}

Here \texttt{prog} will be an LPI program structure describing the solved problem and \texttt{Pop} will be a \texttt{dopvar} object describing the (unsolved) decision operator $\mcl{P}$ from which the solved operator can be derived using
\begin{matlab}
\begin{verbatim}
 >> Pop = getsol_lpivar(prog,Pop);
\end{verbatim}
\end{matlab}
See Chapter~\ref{ch:LPIs} for more information on the operation of this function and the \texttt{settings} input.

\subsection{Dual PIE to PDE Stability}
This tests For the existence of a $C>0$ such that $\norm{(\mcl T^* \mbf x_f)(t)}\le C \norm{(\mbf x_f)(0)}$ for $\mbf x_f(t)$ which satisfies the dual PIE $\mcl T^* \dot{\mbf x}_f(t)=\mcl A^* \mbf x_f(t)$. This then implies PIE2PDE stability of the primal PIE $\mcl T \dot{\mbf x}_f(t)=\mcl A \mbf x_f(t)$.

For given PI operators $\mcl{T}$ and $\mcl{A}$, PIE2PDE stability of the dual PIE
can be tested by solving the following LPI for $\alpha>0$, $\beta \ge 0$.

\begin{align}\label{eq:stab_lpi_PIE2PDE_dual}
	&\mcl T \mcl Q= \mcl Q^* \mcl T^*=\mcl{P}\succeq \alpha \mcl T \mcl T^*&\notag\\
	&\mcl{Q}^*\mcl{A}^*+\mcl{A}\mcl{Q}\preccurlyeq -\beta \mcl P&.
\end{align}

If this LPI is feasible, then the dual PIE is PDE stable. NOTE, however, that it is not currently known if PDE stability of the dual system implies PDE stability of the primal system.\\

Given a structure \texttt{PIE}, this LPI may be solved for the associated PIE by calling
\begin{matlab}
\begin{verbatim}
 >> [prog, Pop] = PIETOOLS_stability_dual(PIE, settings);
\end{verbatim}
\end{matlab}
or
\begin{matlab}
\begin{verbatim}
 >> [prog, Pop] = lpiscript(PIE, 'stability-dual', settings);
\end{verbatim}
\end{matlab}
Here \texttt{prog} will be an LPI program structure describing the solved problem and \texttt{Pop} will be a \texttt{dopvar} object describing the (unsolved) decision operator $\mcl{P}$.

\subsection{Dual PDE Stability}
For given PI operators $\mcl{T}$ and $\mcl{A}$, stability of the PIE~\eqref{PIE_stab}
can also be tested by solving the following LPI.

\begin{align}\label{eq:dual-stab_lpi}
	&\mcl{P}\succ0&\notag\\
	&\mcl{T}\mcl{P}\mcl{A}^*+\mcl{A}\mcl{P}\mcl{T}^*\preccurlyeq 0&
\end{align}

If this LPI is feasible, then the dual PIE is PDE stable. NOTE, however, that it is not currently known if PDE stability of the dual system implies PDE stability of the primal system.\\

Given a structure \texttt{PIE}, this LPI may be solved for the associated PIE by calling
\begin{matlab}
\begin{verbatim}
 >> [prog, Pop] = PIETOOLS_stability_dual(PIE, settings);
\end{verbatim}
\end{matlab}

Here \texttt{prog} will be an LPI program structure describing the solved problem and \texttt{Pop} will be a \texttt{dopvar} object describing the (unsolved) decision operator $\mcl{P}$.

\subsection{Well-Posedness}
For given PI operators $\mcl{T}$ and $\mcl{A}$, well-posedness of the PIE~\eqref{PIE_stab}
can be tested by solving the following LPI.

\begin{align}\label{eq:wellposed_lpi}
	&\mcl{P}\succeq \epsilon I,     &
	&\mcl{T}^*\mcl{P}\mcl{A}+\mcl{A}^*\mcl{P}\mcl{T}\preccurlyeq -2\epsilon_{2}\mcl{T}^*\mcl{P}\mcl{T}, \notag\\
    &\mcl{R}\succeq I,  &
    &(\mcl{T}-\mcl{A})\mcl{R}(\mcl{T}-\mcl{A})^*\succcurlyeq \epsilon I.
\end{align}
If there exists $\omega\in\R$ and PI operators $\mcl{P}$ and $\mcl{R}$ such that this LPI is feasible, then the operator $A:=\mcl{A}\mcl{T}^{-1}:D(A)\to L_{2}$ for $D(A):=\text{Range}(\mcl{T})$ generates a strongly continuous semigroup, $\{S(t)\in\mcl{B}(L_{2},L_{2})\mid t\geq 0\}$, which satisfies $\norm{S(t)}_{\text{op}}\leq Me^{\omega t}$ for some $M\geq 1$ and all $t\geq 0$. Note here that, in practice, $A$ will usually be a differential operator defining a PDE, $\dot{\mbf{x}}(t)=A\mbf{x}(t)$, in which case solutions to this PDE are given by $\mbf{x}(t)=S(t)\mbf{x}(0)$ for any $\mbf{x}(0)\in D(A)$.

Given a structure \texttt{PIE}, the LPI~\eqref{eq:wellposed_lpi} may be solved for the associated PIE by calling
\begin{matlab}
\begin{verbatim}
 >> [prog, Pop, Rop, omega] = PIETOOLS_well_posed(PIE, settings);
\end{verbatim}
\end{matlab}
or
\begin{matlab}
\begin{verbatim}
 >> [prog, Pop] = lpiscript(PIE, 'well-posed', settings);
\end{verbatim}
\end{matlab}
Here \texttt{prog} will be an LPI program structure describing the solved problem, and \texttt{Pop} and \texttt{Rop} will be \texttt{opvar} objects representing the operators $\mcl{P}$ and $\mcl{R}$ for which the LPI holds (if feasible). The output \texttt{omega} is a scalar value representing $\omega=-\epsilon_{2}$ for $\epsilon_{2}$ as in the LPI~\eqref{eq:wellposed_lpi}. The value of $\epsilon_{2}$ must be specified in the \texttt{settings} structure under \texttt{settings.epneg}, and this value may be both negative and positive.

\subsection{Input-Output Gain}\label{sec:LPI_examples:analysis:L2-gain}
Consider a system of the form
\begin{align}\label{eq:standardizedPIE_noControl}
	\partial_t(\mcl T \mbf{x}_{\text{f}})(t)&=\mcl A\mbf{x}_{\text{f}}(t)+\mcl{B}_1w(t), \quad \mbf \mbf{x}_{\text{f}}(0)=\mbf{0}\notag\\
	z(t) &= \mcl{C}_1\mbf{x}_{\text{f}}(t) + \mcl{D}_{11}w(t), 
\end{align}
where $z(t,\cdot)$ and $w(t,\cdot)$ are $\R^{n_{z_1}}\times L_2^{n_{z_2}}[a,b]$ and $\R^{n_{w_1}}\times L_2^{n_{w_2}}[a,b]$, respectively. Moreover, if $w$ is $L_2$-bounded in time, i.e.,
\[
\norm{w}_{L_2}^2:=\int_0^\infty \norm{w(t,\cdot)}_{\R\times L_2[a,b]}^2 dt < \infty,
\]
then, $\norm{z}_{L_2}\leq {\gamma}\norm{w}_{L_2}$, if the following LPI is feasible.

\begin{align}\label{eq:hinf_lpi}
	&\min\limits_{\gamma,\mcl{Q},\mcl R} ~~\gamma&\notag\\
	&\mcl{T}^*\mcl Q=\mcl Q^*\mcl T=\mcl R\succeq 0&\notag\\
	&\bmat{-\gamma I & \mcl{D}_{11}^*&\mcl{B}_1^*\mcl{Q}\\(\cdot)^*&-\gamma I&\mcl{C}_1\\(\cdot)^*&(\cdot)^*&\mcl{Q}^*\mcl{A}+\mcl{A}^*\mcl{Q}}\preccurlyeq 0&
\end{align}
Given a structure \texttt{PIE}, this LPI may be solved for the associated PIE by calling
\begin{matlab}
\begin{verbatim}
 >> [prog, Qop, gam] = PIETOOLS_Hinf_gain(PIE, settings);
\end{verbatim}
\end{matlab}
or
\begin{matlab}
\begin{verbatim}
 >> [prog, Qop, gam] = lpiscript(PIE, 'l2gain', settings);
\end{verbatim}
\end{matlab}
Here \texttt{prog} will be an LPI program structure describing the solved problem, and \texttt{gam} will be the smallest value of $\gamma$ for which the LPI was found to be feasible, offering a bound on the $L_2$-gain from $w$ to $z$ of the system. The output \texttt{Pop} will be a \texttt{dopvar} object describing the (unsolved) decision operator $\mcl{P}$. 
				
\subsection{Dual Input-Output Gain}
For a System~\eqref{eq:standardizedPIE_noControl} with distributed input $w(t,\cdot)\in\R^{n_{w_1}}\times L_2^{n_{w_2}}[a,b]$
such that $\norm{w}_{L_2}< \infty$, an upper bound $\gamma$ on the $L_2$-gain from $w$ to $z$ can also be obtained by solving the LPI	
\begin{align}\label{eq:dual_hinf_lpi}
	&\min\limits_{\gamma,\mcl{Q},\mcl R} ~~\gamma&\notag\\
	&\mcl{TQ}=\mcl Q^*\mcl T^*=\mcl R\succeq 0&\notag\\
	&\bmat{-\gamma I & \mcl{D}_{11}&\mcl{C}_1\mcl Q\\(\cdot)^*&-\gamma I&\mcl{B}_1^*\\(\cdot)^*&(\cdot)^*&\mcl{Q}^*\mcl{A}^*+\mcl{A}\mcl{Q}}\preccurlyeq 0&
\end{align}
Given a structure \texttt{PIE}, this LPI may be solved for the associated PIE by calling
\begin{matlab}
\begin{verbatim}
 >> [prog, Qop, gam] = PIETOOLS_Hinf_gain_dual(PIE, settings);
\end{verbatim}
\end{matlab}
or
\begin{matlab}
\begin{verbatim}
 >> [prog, Qop, gam] = lpiscript(PIE, 'l2gain-dual', settings);
\end{verbatim}
\end{matlab}
Here \texttt{prog} will be an LPI program structure describing the solved problem, and \texttt{gam} will be the found optimal value for $\gamma$. The output \texttt{Pop} will be a \texttt{dopvar} object describing the (unsolved) decision operator $\mcl{P}$.

\subsection{Positive Real Lemma}
For a PIE of the form of Eq.~\eqref{eq:standardizedPIE_noControl},
we can test whether the system is passive by solving the LPI
{
\begin{align}\label{eq:positive_real_lpi}
	&\mcl{P}\succ0&    \notag\\
	&\bmat{-\mcl{D}_{11}^*-\mcl{D}_{11}&\mcl{B}_1^*\mcl{PT}-\mcl{C}_1\\(\cdot)^*&\mcl{T}^*\mcl{P}\mcl{A}+\mcl{A}^*\mcl{P}\mcl{T}}\preccurlyeq 0&
\end{align}
}
If there exists a PI operator $\mcl{P}$ such that this LPI is feasible, then the system is passive. Note that this LPI has not been implemented as an executive in PIETOOLS.

\subsection{$H_2$-Norm}\label{subsec:H2norm}

For a PIE of the form
\begin{align}\label{eq:Sigma}
	\partial_t(\mcl T\mbf{x}_{\text{f}})(t)&=\mcl A\mbf{x}_{\text{f}}(t)+\mcl{B}_1w(t), \quad \mbf \mbf{x}_{\text{f}}(0)=\mbf{0}\notag\\
	z(t) &= \mcl{C}_1\mbf{x}_{\text{f}}(t), 
\end{align}
we can compute the $H_2$-norm by extending its usual definition to PIEs as follows: consider solutions of the auxiliary PIE
\begin{align}
\partial_t (\mcl T \mbf{x}_{\text{f}})(t)&= \mcl A \mbf{x}_{\text{f}}(t),\notag \\
z(t)&=\mcl C_1 \mbf{x}_{\text{f}}(t), \qquad \mcl T \mbf{x}_{\text{f}}(0) =\mcl{B}_1 x_0.\label{eqn:PIEaux}
\end{align}

We define the $H_2$ norm of System~\eqref{eq:Sigma}, denoted $\Sigma$, as
\[
 \norm{\Sigma}_{H_2}:=\sup_{\substack{z,\mbf x\, \text{satisfy~\eqref{eqn:PIEaux}}\\ \norm{x_0}=1}} \norm{z}_{L_2}.
\]

Then, we can compute an optimal upper-bound on this $H_2$-norm by solving the following LPI:
\begin{align}\label{ogramian_LPI_NC}
&\min\limits_{\gamma,\mcl{R,Q,W}} ~~\gamma&\notag\\
	&\mcl{R}\succcurlyeq0, \gamma >0&\notag\\
    &\mcl {Q^* T} = \mcl{T^* Q} = \mcl R&\notag\\
	&\bmat{-\gamma I & \mcl C_1  \\ \mcl C_1^* & \mcl{ A^* Q+ Q^* \mcl A}}\preccurlyeq 0&\notag\\
    &\bmat{\mcl W&\mcl B_1^*\mcl Q\\\mcl Q^*\mcl B_1 &\mcl R}\succcurlyeq 0& \notag\\
    &\trace(\mcl W) \leq \gamma&
\end{align}
        
If~\eqref{ogramian_LPI_NC} is feasible for some $\gamma > 0$, PI operator $\mcl {R, W} \succcurlyeq 0$, and $\mcl Q$, then $\norm{\Sigma}_{H_2} \leq \gamma$.

Given a structure \texttt{PIE}, this LPI may be solved for the associated PIE by calling
\begin{matlab}
\begin{verbatim}
 >> [prog, Wm, gam, Rop, Qop] = PIETOOLS_H2_norm_o(PIE, settings);
\end{verbatim}
\end{matlab}
or
\begin{matlab}
\begin{verbatim}
 >> [prog, Wm, gam, Rop, Qop] = lpiscript(PIE, 'h2norm', settings);
\end{verbatim}
\end{matlab}
Here \texttt{prog} will be an LPI program structure describing the solved problem, and \texttt{gam} will be the smallest value of $\gamma$ for which the LPI was found to be feasible, offering a bound on the $H_2$-norm of the system. The outputs \texttt{Rop, Qop, Wm} will be \texttt{dopvar} objects describing the (unsolved) decision operators. Note that, in the most common case when the input $w(t)$ is finite-dimensional, the operator $\mcl W$ is a matrix. 

\subsection{Dual $H_2$-Norm}

For a system~\eqref{eq:Sigma}, an optimal upper bound $\gamma$ on the $H_2$ norm can also be found by solving the LPIs
\begin{align}\label{cgramian_LPI_NC}
&\min\limits_{\gamma,\mcl{R,Q},W} ~~\gamma&\notag\\
	&\mcl{R}\succcurlyeq0, \gamma >0&\notag\\
    &\mcl {Q^* T^*} = \mcl{T Q} = \mcl R&\notag\\
	&\bmat{-\gamma I & \mcl B_1^*  \\ \mcl B_1 & \mcl{ A Q+ Q^* \mcl A^*}}	\preccurlyeq 0&\notag\\
    &\bmat{\mcl W&\mcl C_1\mcl Q\\\mcl Q^*\mcl C_1^* &\mcl R}\succcurlyeq 0& \notag\\
    &\trace(\mcl W) \leq \gamma&
\end{align}

Given a structure \texttt{PIE}, this LPI may be solved for the associated PIE by calling
\begin{matlab}
\begin{verbatim}
 >> [prog, Wm, gam, Rop, Qop] = PIETOOLS_H2_norm_c(PIE, settings);
\end{verbatim}
\end{matlab}
or
\begin{matlab}
\begin{verbatim}
 >> [prog, Wm, gam, Rop, Qop] = lpiscript(PIE, 'h2norm-dual', settings);
\end{verbatim}
\end{matlab}
Here \texttt{prog} will be an LPI program structure describing the solved problem, and \texttt{gam} will be the smallest value of $\gamma$ for which the LPI was found to be feasible, offering a bound on the $H_2$-norm of the system. The outputs \texttt{Rop, Qop, Wm} will be \texttt{dopvar} objects describing the (unsolved) decision operators. Note that, in the most common case when the output $z(t)$ is finite-dimensional, the operator $\mcl W$ is a matrix.

\section{LPIs for Optimal Estimation of PIEs}\label{sec:LPI_examples:estimation}
\subsection{$H_{\infty}$ Estimator}\label{subsec:Hinfest}
For the following PIE
\begin{align}
	\partial_t(\mcl T \mbf{x}_{\text{f}})(t)+\mcl{T}_{w} \dot{w}(t)+\mcl{T}_{u} \dot{u}(t)&=\mcl A\mbf{x}_{\text{f}}(t)+\mcl{B}_1w(t)+\mcl{B}_2u(t),\notag\\
	z(t) &= \mcl{C}_1\mbf{x}_{\text{f}}(t) + \mcl{D}_{11}w(t) + \mcl{D}_{12}u(t),\notag\\
	y(t) &= \mcl{C}_2\mbf{x}_{\text{f}}(t) + \mcl{D}_{21}w(t) + \mcl{D}_{22}u(t),
\end{align}
a state estimator has the following structure:
\begin{align} 
	\partial_t(\mcl T \hat{\mbf{x}}_{\text{f}})(t) +\mcl{T}_{u} \dot{u}(t)&=\mcl A\hat{\mbf{x}}_{\text{f}}(t)+\mathcal{L}(\hat{y}(t)-y(t))+\mcl{B}_2u(t),\notag\\
	\hat{z}(t) &= \mcl{C}_1\hat{\mbf{x}}_{\text{f}}(t)  + \mcl{D}_{12}u(t)\notag\\
	\hat{y}(t) &= \mcl{C}_2\hat{\mbf{x}}_{\text{f}}(t)+ \mcl{D}_{22}u(t),
\end{align}
so that the errors $\mbf{e}:=\hat{\mbf{x}}_{\text{f}}-\mbf{x}_{\text{f}}$ and $\tilde{z}:=\hat{z}-z$ in respectively the state and regulated output estimates satisfy
\begin{align*}
    \partial_t(\mcl{T}\mbf{e})(t)-\mcl{T}_w \dot w(t)&=(\mcl{A}+\mcl{L}\mcl{C}_2)\mbf{e}(t) - (\mcl{B}_1+\mcl{L}\mcl{D}_{21})w(t), \\
    \tilde{z}(t)&=\mcl{C}_1\mbf{e}(t) - \mcl D_{11}w(t)
\end{align*}
The $H_{\infty}$-optimal estimation problem amounts to synthesizing $\mathcal{L}$ such that the estimation error $\tilde{z}:={\hat{z}-z}$ admits $\norm{\tilde{z}} \leq \gamma \norm{w}$ for a particular $\gamma >0$. To establish such an estimator, we can solve the following LPI.

\begin{align}\label{eq:esti_lpi}
	&\min\limits_{\gamma,\mcl{P},\mcl{Z}} ~~\gamma&\notag\\
	&\mcl{P}\succ0&\notag\\
	&\bmat{\mcl T_{w}^*(\mcl P\mcl B_1+\mcl Z\mcl D_{21})+(\cdot)^*& 0 &(\cdot)^*\\ 0  &0 & 0 \\ -(\mcl P\mcl A+\mcl Z\mcl C_2)^*\mcl T_{w}& 0 &0}\hspace{-1ex}+\hspace{-1ex}\bmat{-\gamma I& -\mcl D_{11}^{\top}&-(\mcl P\mcl B_1+\mcl Z\mcl D_{21})^*\mcl T\\(\cdot)^*&-\gamma I&\mcl C_1\\(\cdot)^*&(\cdot)^*&(\mcl P\mcl A+\mcl Z\mcl C_2)^*\mcl T+(\cdot)^*}\preccurlyeq 0&
\end{align}

Then, if this LPI is feasible for some $\gamma>0$ and PI operators $\mcl{P}$ and $\mcl{Z}$, 
then, letting $\mcl{L}:=\mcl{P}^{-1} \mcl{Z}$, the estimation error will satisfy $\norm{\tilde{z}} \leq \gamma \norm{w}$. 
Given a structure \texttt{PIE}, this LPI may be solved for the associated PIE by calling
\begin{matlab}
\begin{verbatim}
 >> [prog, Lop, gam, Pop, Zop] = PIETOOLS_Hinf_estimator(PIE, settings);
\end{verbatim}
\end{matlab}
or
\begin{matlab}
\begin{verbatim}
 >> [prog, Lop, gam, Pop, Zop] = lpiscript(PIE, 'hinf-observer', settings);
\end{verbatim}
\end{matlab}
Here \texttt{prog} will be an LPI program structure describing the solved problem, \texttt{gam} will be the found optimal value for $\gamma$, and \texttt{Lop} will be an \texttt{opvar} object describing the optimal estimator $\mcl{L}$. Outputs \texttt{Pop} and \texttt{Zop} will be \texttt{opvar} objects describing the solved operators $\mcl{P}$ and $\mcl{Z}$. See Chapter~\ref{ch:LPIs} for more information on how LPIs are solved and on the \texttt{settings} input.

\subsection{$H_2$ Estimator}\label{subsec:H2est}
For the following PIE
\begin{align}
	\partial_t(\mcl T \mbf{x}_{\text{f}})(t)+\mcl T_u \dot{u}(t)&=\mcl A\mbf{x}_{\text{f}}(t)+\mcl{B}_1w(t)+\mcl{B}_2u(t),\notag\\
	z(t) &= \mcl{C}_1\mbf{x}_{\text{f}}(t)+\mcl D_{12}u(t),\notag\\
	y(t) &= \mcl{C}_2\mbf{x}_{\text{f}}(t)+\mcl D_{21}w(t)+\mcl D_{22}u(t),
\end{align}
a state estimator with the following structure:
\begin{align} 
	\partial_t(\mcl T \hat{\mbf{x}}_{\text{f}})(t) +\mcl T_u \dot{u}(t)&=\mcl A \hat{\mbf{x}}_{\text{f}}(t)+\mathcal{L}(\hat{y}(t)-y(t))+\mcl{B}_2u(t),\notag\\
	\hat{z}(t) &= \mcl{C}_1\hat{\mbf{x}}_{\text{f}}(t)+\mcl D_{12}u(t) \notag\\
	\hat{y}(t) &= \mcl{C}_2\hat{\mbf{x}}_{\text{f}}(t)+\mcl D_{22}u(t).
\end{align}
so that the errors $\mbf{e}:=\hat{\mbf{x}_{\text{f}}}-\mbf{x}_{\text{f}}$ and $\tilde{z}:=\hat{z}-z$ in respectively the state and regulated output estimates satisfy
\begin{align*}
    \partial_t(\mcl T \mbf{e})(t)&=(\mcl{A}+\mcl{L}\mcl{C}_2)\mbf{e}(t) - (\mcl{B}_1+\mcl{L}\mcl{D}_{21})w(t), \\
    \tilde{z}(t)&=\mcl{C}_1\mbf{e}(t)
\end{align*}
The estimation problem is to find $\mcl L$ such that, for some $\gamma > 0$ the above system, called $\Sigma_e$, has $H_2$-norm $\norm{\Sigma_e}_{H_2} \leq \gamma$. The optimal estimator can be found by solving the following LPI.
\begin{align}\label{H2_estimator_LPI}
&\min\limits_{\gamma,\mcl{Z,P,W}} ~~\gamma&\notag\\
    &\gamma>0&\notag\\
     &\trace(\mcl W) \leq \gamma&\notag\\
        &\mcl P \succ 0&\notag\\
	&\bmat{-\gamma I & \mcl C_1  \\ \mcl C_1^* & \mcl{A^*P T+ T^*P A}+ \mcl{T^* Z C}_2+\mcl C_2^* \mcl Z^*T}	\preccurlyeq 0&\notag\\
    &\bmat{\mcl W&-(\mcl B_1^*\mcl P+\mcl D_{21}^*\mcl Z^*)\\-(\mcl P\mcl B_1 +\mcl {ZD}_{21}) &\mcl P }\succcurlyeq 0&
\end{align}
If this LPI is feasible for some $\gamma>0$, PI operators $\mcl{P}$, $\mcl Z$, and $\mcl W$, then, letting $\mcl{L}:=\mcl{P}^{-1} \mcl{Z}$, the estimation error will satisfy $\norm{\tilde{z}} \leq \gamma \norm{w}$. 
Given a structure \texttt{PIE}, this LPI may be solved for the associated PIE by calling
\begin{matlab}
\begin{verbatim}
 >> [prog, Lop, gam, Pop, Zop, Wop] = PIETOOLS_H2_estimator(PIE,settings);
\end{verbatim}
\end{matlab}
or
\begin{matlab}
\begin{verbatim}
 >> [prog, Lop, gam, Pop, Zop, Wop] = lpiscript(PIE, 'h2-observer', settings);
\end{verbatim}
\end{matlab}
Here \texttt{prog} will be an LPI program structure describing the solved problem, \texttt{gam} will be the found optimal value for $\gamma$, and \texttt{Lop} will be an \texttt{opvar} object describing the optimal estimator $\mcl{L}$. Outputs \texttt{Pop}, \texttt{Zop}, and \texttt{Wop} will be \texttt{opvar} objects describing the solved operators $\mcl{P,  Z}$ and $\mcl{W}$. See Chapter~\ref{ch:LPIs} for more information on how LPIs are solved and on the \texttt{settings} input.

\section{LPIs for Optimal Control of PIEs}\label{sec:LPI_examples:control}

\subsection{$H_{\infty}$ Control}\label{Hinfcontrol}
In this section, we discuss the synthesis of $H_{\infty}$ optimal control of a PIE of the form
\begin{align}\label{eq:LPI_examples:control_PIE}
	\partial_t(\mcl T \mbf{x}_{\text{f}})(t)&=\mcl A\mbf{x}_{\text{f}}(t)+\mcl{B}_1w(t)+\mcl B_2u(t), \qquad  \mbf{x}_{\text{f}}(0)=\mbf 0\notag\\
	z(t) &= \mcl{C}_1\mbf{x}_{\text{f}}(t) + \mcl{D}_{11}w(t) + \mcl{D}_{12}u(t),
\end{align}
where $w,z \in L_2[0,\infty)$. The problem of synthesizing an $H_{\infty}$-optimal controller amounts to determining a PIE operator $\mcl{K}$ such that, using the full-state feedback law $u(t) = \mathcal{K}\mbf{v}(t)$, the regulated output ${z}$ admits $\norm{z}_{L_2} \leq \gamma \norm{w}_{L_2}$ for a particular $\gamma >0$. To establish such a controller, we can solve the LPI
\begin{align}\label{eq:cont_lpi}
	&\min\limits_{\gamma,\mcl{P},\mcl{Z}} ~~\gamma&\notag\\
	&\mcl{P}\succ0&\notag\\
	&\bmat{-\gamma I& \mcl D_{11}& (\mcl{C}_1\mcl P+\mcl D_{12}\mcl{Z})\mcl T^*\\
		\mcl D_{11}^* & -\gamma I & \mcl B_1^*\\
		()^* & \mcl B_1& ()^*+\left(\mcl{AP}+\mcl{B}_2\mcl{Z}\right)\mcl{T}^*}\preccurlyeq 0&
\end{align}

If this LPI is feasible for some $\gamma>0$ and PI operators $\mcl{P}$ and $\mcl{Z}$, then, letting $\mcl{K}:=\mcl{Z}\mcl{P}^{-1}$, the $L_2$-gain for the controlled system with $u=\mcl{K}\mbf{x}_{\text{f}}$ will be such that $\norm{z} \leq \gamma \norm{w}$.  
Given a structure \texttt{PIE}, this LPI may be solved for the associated PIE by calling
\begin{matlab}
\begin{verbatim}
 >> [prog_sol, Kop, gamma, Pop, Zop] = PIETOOLS_Hinf_control(PIE,settings);
\end{verbatim}
\end{matlab}
or
\begin{matlab}
\begin{verbatim}
 >> [prog_sol, Kop, gamma, Pop, Zop] = lpiscript(PIE, 'hinf-controller', settings);
\end{verbatim}
\end{matlab}
Here \texttt{prog} will be an LPI program structure describing the solved problem, \texttt{gam} will be the found optimal value for $\gamma$, and \texttt{Kop} will be an \texttt{opvar} object describing the optimal feedback $\mcl{K}$. Outputs \texttt{Pop} and \texttt{Zop} will be \texttt{opvar} objects describing the solved operators $\mcl{P}$ and $\mcl{Z}$. See Chapter~\ref{ch:LPIs} for more information on how LPIs are solved and on the \texttt{settings} input.

\subsection{$H_2$ Controller}\label{subsec:h2control}
For a PIE of the form
\begin{align}\label{eq:Sigma_c}
	\partial_t(\mcl T\mbf{x}_{\text{f}})(t)&=\mcl A\mbf{x}_{\text{f}}(t)+\mcl{B}_1w(t)+\mcl{B}_2u(t), \quad \mbf{x}_{\text{f}}(0)=\mbf{0}\notag\\
	z(t) &= \mcl{C}_1\mbf{x}_{\text{f}}(t)+ \mcl{D}_{12}u(t), 
\end{align}
with $w,z \in L_2[0,\infty)$. The problem is to determine $\mcl{K}$ such that, using the full-state feedback law $u(t) = \mathcal{K}\mbf{x}_{\text{f}}(t)$, the regulated output ${z}$ of the closed loop system, denoted $\Sigma_c$ has $H_2$ norm $\norm{\Sigma_c}_{H_2} \leq \gamma$ for a particular $\gamma >0$.

The controller can be found by solving the following LPI.
\begin{align}\label{H2_control_LPI}
&\min\limits_{\gamma,\mcl{Z,P,W}} ~~\gamma&\notag\\
	&\mcl{P}\succ0, \gamma>0&\notag\\
        &\trace(\mcl W) \leq \gamma&\notag\\
	&\bmat{-\gamma I & \mcl B_1^*  \\ \mcl B_1 & \mcl{AP T^*}+ \mcl{T P \mcl A^* +B}_2\mcl Z+\mcl Z^* \mcl B_2^*}	\preccurlyeq 0&\notag\\
    &\bmat{\mcl W&\mcl C_1\mcl P+\mcl D_{12}\mcl Z\\\mcl P^*\mcl C_1^* +\mcl Z^*\mcl D_{12}^* &\mcl P }\succcurlyeq 0&
\end{align}

 If the LPI is feasible, let $\mcl K = \mcl Z \mcl P^{-1}$. Given a structure \texttt{PIE}, this LPI may be solved for the associated PIE by calling
\begin{matlab}
\begin{verbatim}
 >>[prog, Kop, gam, Pop, Zop, Wop] = PIETOOLS_H2_control(PIE, settings);
\end{verbatim}
\end{matlab}
or
\begin{matlab}
\begin{verbatim}
 >> [prog, Kop, gam, Pop, Zop, Wop] = lpiscript(PIE, 'h2control', settings);
\end{verbatim}
\end{matlab}
Here \texttt{prog} will be an LPI program structure describing the solved problem, and \texttt{gam} will be the smallest value of $\gamma$ for which the LPI was found to be feasible, offering a bound on the $H_2$-norm of $\Sigma_c$. The outputs \texttt{Pop, Zop}, and \texttt{Wop} will be \texttt{opvar} objects describing the solved decision operators $\mcl{P, Z}$, and $\mcl{W}$. See Chapter~\ref{ch:LPIs} for more information on how LPIs are solved and on the \texttt{settings} input.

\part{Appendices}

\appendix
%

\chapter{PI Operators and their Properties}\label{appx:PI_theory}

In this appendix, we discuss in a bit more detail the crucial properties of PI operators that PIETOOLS relies on for implementation and analysis of PIEs. For this, in Section~\ref{sec:PI_theory:PI_defs}, we first recap the definitions of PI operators as presented in Chapter~\ref{ch:PIE}, also introducing some notation that we will continue to use throughout the appendix. In Section~\ref{sec:PI_theory:Addition},~\ref{sec:PI_theory:Composition}, and~\ref{sec:PI_theory:Adjoint}, we show that respectively the sum, composition, and adjoint of PI operators can be expressed as PI operators. In Section~\ref{sec:PI_theory:Inverse}, and~\ref{sec:PI_theory:Derivative}, we then show how the respectively the inverse of a PI operator, and the composition of a PI operator with a differential operator can be computed. Finally, in Section~\ref{sec:PI_theory:Positive_PI}, we show how a cone of positive PI operators can be parameterized by positive matrices, allowing an LPI constraint $\mcl{P}\succcurlyeq 0$ to be posed as an LMI $P\succcurlyeq 0$.

For more information on PI operators, and full proofs of each of the results, we refer to e.g.~\cite{peet_2020Aut}~\cite{shivakumar_2019CDC} (1D) and~\cite{jagt_2021PIEACC} (2D).

\section{PI Operators on Different Function Spaces}\label{sec:PI_theory:PI_defs}

Recall that we denote the space of square integrable functions on a domain $\Omega$ as $L_2[\Omega]$, with inner product
\begin{align*}
    \ip{\mbf{x}}{\mbf{y}}_{L_2}=\int_{\Omega}\bl[\mbf{x}(s)\br]^T\mbf{y}(s)ds.
\end{align*}
In defining the different PI operators, we will restrict ourselves to domains of 1D or 2D hypercubes. In 1D, such a hypercube is simply an interval $[a,b]$, for which we define the following PI operator:
    
\begin{defn}[3-PI Operator]\label{appx_def:3-PI}
 For given parameters 
 \begin{align*}
    R:=\{R_0,R_1,R_2\}\in\left\{L_2^{m\times n}[a,b],L_2^{m\times n}\bl[[a,b]^2\br],L_2^{m\times n}\bl[[a,b]^2\br]\right\}=:\mcl{N}_{1D}^{m\times n}[a,b],
 \end{align*}
 we define the associated 3-PI operator $\mcl{P}[R]:=\mcl{P}_{\{R_0,R_1,R_2\}}:L_2^{n}[a,b]\rightarrow L_2^{m}[a,b]$ as
 \begin{align}
  \bl(\mcl{P}[R] \mbf{x}\br)(s) := R_0(s) \mbf{x}(s) +\int_{a}^{s} R_1(s,\theta)\mbf{x}(\theta)d \theta +\int_s^b R_2(s,\theta)\mbf{x}(\theta)d \theta,
 \end{align} 
 for any $\mbf{x}\in L_2^{n}[a,b]$.
\end{defn}

We note that 3-PI operators can be seen as (one possible) generalization of matrices to infinite dimensional vector spaces. In particular, suppose we have a matrix $P\in\R^{m\times n}$, which we decompose as $P=D+L+U$, where $D$ is diagonal, $L$ is strictly lower triangular, and $U$ is strictly upper-triangular. Then, for any $x\in\R^n$, the $i$th element of the product $Px$ is given by:
\begin{align*}
 \bl(Px\br)_{i} &= D_{ii} x_i + \sum_{j=1}^{i-1} L_{ij}x_j + \sum_{j=i+1}^{n} U_{ij} x_j.
 \intertext{
 Compare this to the value of $\mcl{P}[R]\mbf{x}$ at a position $s\in[a,b]$ for some $\mbf{x}\in L_2[a,b]$ and 3-PI parameters $R$:}
 \bl(\mcl{P}[R]\mbf{x}\br)(s) &= R_0(s)\mbf{x}(s) + \int_{a}^{s} R_1(s,\theta)\mbf{x}(\theta)d\theta + \int_{s}^{b} R_2(s,\theta)\mbf{x}(\theta)d\theta.
 \end{align*}
 Replacing row and column indices $(i,j)$ by primary and dummy variables $(s,\theta)$, and performing integration instead of summation, 3-PI operators have a structure very similar to that of matrices, wherein we can recognize a diagonal, lower-triangular, and upper-triangular part. Accordingly, we will occasionally refer to a PI operator of the form $\mcl{P}_{\{R_0,0,0\}}$ as a diagonal 3-PI operator, and to PI operators of the forms $\mcl{P}_{\{0,R_1,0\}}$ and $\mcl{P}_{\{0,0,R_2\}}$ as lower- and upper-triangular PI operators respectively. The similar structure between matrices and PI operator also ensures that matrix operations such as addition and multiplication are valid for PI operators as well, as we will discuss in more detail in the next sections.

 To map functions on a domain $[a,b]\times[c,d]\subset\R^2$, we also define the $9$-PI operator:

\begin{defn}[9-PI Operator]\label{appx_def:9-PI}
 For given parameters
{\small
\begin{align*}
 R&:={\left[\!\!\begin{array}{lll}
 R_{00} & R_{01} & R_{02}\\ R_{10} & R_{11} & R_{12}\\ R_{20} & R_{21} & R_{22}
 \end{array}\!\!\right]}    \\
 &\hspace*{0.5cm}\in
 {\left[\begin{array}{lll}
    L_{2}^{m\times n}\bl[[a,b]\times[c,d]\br] & L_{2}^{m\times n}\bl[[a,b]\times[c,d]^2\br] & L_{2}^{m\times n}\bl[[a,b]\times[c,d]^2\br] \\ 
    L_{2}^{m\times n}\bl[[a,b]^2\times[c,d]\br] & L_{2}^{m\times n}\bl[[a,b]^2\times[c,d]^2\br] & L_{2}^{m\times n}\bl[[a,b]^2\times[c,d]^2\br]\\
    L_{2}^{m\times n}\bl[[a,b]^2\times[c,d]\br] & L_{2}^{m\times n}\bl[[a,b]^2\times[c,d]^2\br] & L_{2}^{m\times n}\bl[[a,b]^2\times[c,d]^2\br]
 \end{array}\right]}
 =:\mcl{N}_{2D}^{m\times n}\bl[[a,b]\!\times\![c,d]\br]
 \end{align*}
 }
 we define the associated 9-PI operator $\mcl{P}[R]:=\mcl{P}\sbmat{R_{00} & R_{01} & R_{02}\\ R_{10} & R_{11} & R_{12}\\ R_{20} & R_{21} & R_{22}}:L_2^{n}\bl[[a_1,b_1]\times[a_2,b_2]\br]\rightarrow L_2^{m}\bl[[a_1,b_1]\times[a_2,b_2]\br]$ as
{\small
\begin{align}
    \left(\mcl{P}[R]\mbf{x}\right)(s,r)= R_{00}(s,r)\mbf{x}(s,r) &+\hspace*{0.0cm} \int_{c}^{r}\! R_{01}(s,r,\nu)\mbf{x}(s,\nu)d\nu + \int_{r}^{d}\! R_{02}(s,r,\nu)\mbf{x}(s,\nu)d\nu \nonumber\\
    +\int_{a}^{s}\! R_{10}(s,r,\theta)\mbf{x}(\theta,r)d\theta &+ \int_{a}^{s}\!\int_{c}^{r}\! R_{11}(s,r,\theta,\nu)\mbf{x}(\theta,\nu)d\nu d\theta + \int_{a}^{s}\!\int_{r}^{d}\! R_{12}(s,r,\theta,\nu)\mbf{x}(\theta,\nu)d\nu d\theta  \nonumber\\
    +\int_{s}^{b} R_{20}(s,r,\theta)\mbf{x}(\theta,r)d\theta &+ \int_{s}^{b}\!\int_{c}^{r}\! R_{21}(s,r,\theta,\nu)\mbf{x}(\theta,\nu)d\nu d\theta + \int_{s}^{b}\!\int_{r}^{d}\! R_{22}(s,r,\theta,\nu)\mbf{x}(\theta,\nu)d\nu d\theta
 \end{align}
}
for any $\mbf{x}\in L_2^{n}\bl[[a_1,b_2]\times[a_2,b_2]\br]$.
\end{defn}

Note that, similar to how 3-PI operators can be seen as a generalization of matrices, operating on infinite-dimensional states $\mbf{x}(s)$ instead of a finite-dimensional vectors $x_i$, 9-PI operators are a generalization of (a particular class of) tensors, operating on infinite-dimensional states $\mbf{x}(s,r)$ instead of matrix-valued states $x_{ij}$. However, this comparison is not quite as easy to visualize as that between 3-PI operators and matrices, so we will mostly use 3-PI operators to illustrate the different properties of PI operators in the remaining sections.

Finally we define a general class of PI operators, encapsulating 3-PI operators and 9-PI operators, as well as matrices and ``cross-operators''. In particular, we consider operators defined on the set $Z^{\textnormal{n}}\br[[a,b],[c,d]\bl]:=${\scriptsize$\left[\begin{array}{l}
\R^{n_0}\\ L_2^{n_s}[a,b]\\ L_2^{n_r}[c,d]\\ L_2^{n_2}\bl[[a,b]\times[c,d]\br]
\end{array}\right]$}, where $\text{n}:=\{n_0,n_s,n_r,n_2\}$, with each element being a coupled state of finite-dimensional variables $x_0\in\R^{n_0}$, 1D functions $\mbf{x}_s\in L_2^{n_s}[a,b]$ and $\mbf{x}_r\in L_2^{n_r}[a,b]$, and 2D functions $\mbf{x}_2\in L_2^{n_2}\bl[[a,b]\times[c,d]\br]$.

\begin{defn}[PI Operator]

For any operator $\mcl{R}:Z^{\textnormal{n}}\br[[a,b],[c,d]\bl]\rightarrow Z^{\textnormal{m}}\br[[a,b],[c,d]\bl]$ with $\textnormal{m}:=\{m_0,m_s,m_r,m_2\}$ and $\textnormal{n}:=\{n_0,n_s,n_r,n_2\}$, we say that $\mcl{R}$ is a PI operator, denoted by $\mcl{R}\in\Pi^{\textnormal{m}\times\textnormal{m}}$ if there exist parameters
{\scriptsize
\begin{align*}
 &R:=\left[\!\!\begin{array}{llll}
    R_{00} & R_{0s} & R_{0r} & R_{02}\\
    R_{s0} & R_{ss} & R_{sr} & R_{s2}\\
    R_{r0} & R_{rs} & R_{rr} & R_{s2}\\
    R_{20} & R_{2s} & R_{2r} & R_{22}
 \end{array}\!\!\right] \\
 &\qquad \in
 \left[\!\!\begin{array}{llll}
    \R^{m_0\times n_0} & L_2^{m_0\times n_s}[a,b] & L_2^{m_0\times n_r}[c,d] & L_2^{m_0\times n_2}\bl[[a,b]\times[c,d]\br]\\
    L_2^{m_s\times n_0}[a,b] & \mcl{N}_{1D}^{m_s\times n_s}[a,b] & L_2^{m_s\times n_r}\bl[[a,b]\times[c,d]\br] & \mcl{N}_{1D\leftarrow 2D}^{m_s\times n_2}\bl[[a,b],[c,d]\br]\\
    L_2^{m_r\times n_0}[c,d] & L_2^{m_r\times n_s}\bl[[a,b]\times[c,d]\br] & \mcl{N}_{1D}^{m_r\times n_r}[c,d] & \mcl{N}_{1D\leftarrow 2D}^{m_r\times n_2}\bl[[c,d],[a,b]\br]\\
    L_2^{m_2\times n_0}\bl[[a,b]\times[c,d]\br] & \mcl{N}_{2D\leftarrow 1D}^{m_2\times n_s}\bl[[a,b],[c,d]\br] & \mcl{N}_{2D\leftarrow 1D}^{m_2\times n_r}\bl[[c,d],[a,b]\br] & \mcl{N}_{2D}^{m_2\times n_r}\bl[[a,b]\times[c,d]\br]
 \end{array}\!\!\right]
 =:\mcl{N}^{\textnormal{m}\times \textnormal{n}}\bl[[a,b]\times[c,d]\br]
\end{align*}
}
such that
{\small
\begin{align}
 \mcl{R}&=\left(\mcl{P}[R]\mbf{x}\right)(s,r)    \nonumber\\
 &=\mcl{P}\sbmat{R_{00} & R_{0s} & R_{0r} & R_{02}\\
    R_{s0} & R_{ss} & R_{sr} & R_{s2}\\
    R_{r0} & R_{rs} & R_{rr} & R_{s2}\\
    R_{20} & R_{2s} & R_{2r} & R_{22}}\sbmat{x_0\\ \mbf{x}_{s}\\ \mbf{x}_{r}\\ \mbf{x}_2}   \nonumber\\
    &:=
    \left[\!\begin{array}{llll}
     R_{00}x_0 &\!\! +\ \int_{a}^{b}\! R_{0s}(s)\mbf{x}_s(s)ds &\!\! +\ \int_{c}^{d}\! R_{0r}(r)\mbf{x}_r(r)dr &\!\! +\ \int_{a}^{b}\!\int_{c}^{d}R_{02}\mbf{x}_{2}(s,r)drds \\
     R_{s0}(s)x_0 &\!\! +\ \bl(\mcl{P}[R_{ss}]\mbf{x}_{s}\br)(s) &\!\! +\ \int_{c}^{d}\! R_{sr}(s,r)\mbf{x}_{r}(r)dr &\!\! +\ \bl(\mcl{P}[R_{s2}]\mbf{x}_{2}\br)(s) \\
     R_{r0}(r)x_0 &\!\! +\ \int_{a}^{b}\! R_{rs}(s,r)\mbf{x}_{s}(s)ds &\!\! +\ \bl(\mcl{P}[R_{rr}]\mbf{x}_{r}\br)(r)  &\!\! +\ \bl(\mcl{P}[R_{r2}]\mbf{x}_{2}\br)(r) \\
     R_{20}(s,r)x_0 &\!\! +\ \bl(\mcl{P}[R_{2s}]\mbf{x}_{s}\br)(s,r) &\!\! +\ \bl(\mcl{P}[R_{2r}]\mbf{x}_{r}\br)(s,r)  &\!\! +\ \bl(\mcl{P}[R_{22}]\mbf{x}_{2}\br)(s,r) 
    \end{array}\!\right],
\end{align}
}
for any $\mbf{x}=\sbmat{x_0\\ \mbf{x}_{s}\\ \mbf{x}_{r}\\ \mbf{x}_2}\in ${\scriptsize$\left[\begin{array}{l}
    \R^{n_0}\\ L_2^{n_s}[a,b]\\ L_2^{n_r}[c,d]\\ L_2^{n_2}\bl[[a,b]\times[c,d]\br]
    \end{array}\right]$}$=:Z^{\textnormal{n}}$, where for given parameters
{\small
\begin{align*}
P&:=\{P_0,P_1,P_2\}\\
&\qquad \in \left\{L_2^{m_s\times n_2}\bl[[a,b]\!\times\![c,d]\br],L_2^{m_s\times n_2}\bl[[a,b]^2\!\times\![c,d]\br],L_2^{m_s\times n_2}\bl[[a,b]^2\!\times\![c,d]\br]\right\}=:\mcl{N}^{m_s\times n_2}_{1D\leftarrow 2D}\bl[[a,b],[c,d]\br], \\
Q&:=\{Q_0,Q_1,Q_2\}\\
&\qquad \in \left\{L_2^{m_s\times n_2}\bl[[a,b]\!\times\![c,d]\br],L_2^{m_s\times n_2}\bl[[a,b]^2\!\times\![c,d]\br],L_2^{m_s\times n_2}\bl[[a,b]^2\!\times\![c,d]\br]\right\}=:\mcl{N}^{m_s\times n_2}_{2D\leftarrow 1D}\bl[[a,b],[c,d]\br],
\end{align*}
}
we define
{\small
\begin{align*}
 \bl(\mcl{P}[P] \mbf{x}_{2}\br)(s) &:= \int_{c}^{d}\bbbl[P_0(s,r) \mbf{x}_{2}(s,r) +\int_{a}^{s} P_1(s,r,\theta)\mbf{x}_{2}(\theta,r)d \theta +\int_s^b P_2(s,r,\theta)\mbf{x}(\theta,r)d \theta \bbbr]dr,  \nonumber\\
 \bl(\mcl{P}[Q] \mbf{x}_{s}\br)(s,r) &:= Q_0(s,r) \mbf{x}_{s}(s) +\int_{a}^{s} Q_1(s,r,\theta)\mbf{x}_{s}(\theta)d \theta +\int_s^b Q_2(s,r,\theta)\mbf{x}_{s}(\theta)d \theta,
\end{align*}
}
for any $\mbf{x}_2\in L_2^{n_2}\bl[[a,b]\times[c,d]\br]$ and $\mbf{x}_s\in L_2^{n_s}[a,b]$.

\end{defn}

\section{Addition of PI Operators}\label{sec:PI_theory:Addition}

An obvious but crucial property of PI operators is that the sum of two PI operators (of appropriate dimensions) is again a PI operator.

\begin{lem}
 For any PI parameters $Q,R\in\mcl{N}^{\text{m}\times\text{n}}\bl[[a,b],[c,d]\br]$, there exist unique parameters $P\in\mcl{N}^{\text{m}\times\text{n}}\bl[[a,b],[c,d]\br]$ such that
 \begin{align*}
    \mcl{P}[R]+\mcl{P}[Q]=\mcl{P}[P].
 \end{align*}
 That is, for any $\mbf{x}\in Z^{\textnormal{n}}\bl[[a,b],[c,d]\br]$,
 \begin{align*}
    \bl((\mcl{P}[Q]+\mcl{P}[R])\mbf{x}\br)(s)=\bl(\mcl{P}[P]\mbf{x})(s).
 \end{align*}
\end{lem}
\begin{proof}
    We outline the proof for 3-PI operators, for which it is easy to see that, by linearity of the integral,
    \begin{align*}
        &\bl(\mcl{P}_{\{R_0,R_1,R_2\}}\mbf{x}\br)(s)+\bl(\mcl{P}_{\{Q_0,Q_1,Q_2\}}\mbf{x}\br)(s)    \\
        &\qquad=[R_0(s)+Q_0(s)]\mbf{x}(s)+\int_{a}^{s}[R_1(s,\theta)+Q_1(s,\theta)]\mbf{x}(\theta)d\theta +\int_{s}^{b}[R_2(s,\theta)+Q_2(s,\theta)]\mbf{x}(\theta)d\theta  \\
        &\qquad\qquad= \bl(\mcl{P}_{\{R_0+Q_0,R_1+Q_1,R_2+Q_2\}}\mbf{x}\br)(s).
    \end{align*}
    Similar results can be easily derived for more general PI operators. For a full proof, we refer to~\cite{peet_2020Aut}.
\end{proof}
Comparing the addition operation for $3$-PI operators to that for matrices $A,B\in\R^{m\times n}$, we can draw direct parallels. In particular, where the sum $C=A+B$ of two matrices is simply computed by adding the elements $[C]_{ij}=[A]_{ij}+[B]_{ij}$ for each row $i$ and column $j$, the sum of two 3-PI operators is computed by simply adding the values of the parameters $P(s,\theta)=Q(s,\theta)+R(s,\theta)$ at each position $s$ and $\theta$ within the domain.

\section{Composition of PI Operators}\label{sec:PI_theory:Composition}

In addition to the sum of two PI operators being a PI operator, the composition of two PI operators can also be shown to be a PI operator, as stated in the following lemma:

\begin{lem}
 For any PI parameters $R_1\in\mcl{N}^{\text{m}\times\text{p}}\bl[[a,b],[c,d]\br]$ and $R_2\in\mcl{N}^{\text{p}\times\text{n}}\bl[[a,b],[c,d]\br]$, there exist unique parameters $R_3\in\mcl{N}^{\text{m}\times\text{n}}\bl[[a,b],[c,d]\br]$ such that
 \begin{align*}
    \mcl{P}[R_1]\circ\mcl{P}[R_2]=\mcl{P}[R_3].
 \end{align*}
 That is, for any $\mbf{x}\in Z^{\textnormal{n}}\bl[[a,b],[c,d]\br]$,
 \begin{align*}
    \bbl(\mcl{P}[R_1]\bl(\mcl{P}[R_2]\mbf{x}\br)\bbr)(s)=\bl(\mcl{P}[R_3]\mbf{x})(s).
 \end{align*}
\end{lem}
\begin{proof}
    We once again outline the proof only for 3-PI operators. For this, we define the indicator function
    \begin{align*}
    \mbf{I}(s-\theta)=\begin{cases}1,   &\text{if }s\geq \theta\\ 0.    &\text{else} \end{cases}
    \end{align*}
    allowing us to write, e.g.
    \begin{align*}
     \bl(\mcl{P}_{\{R_0,R_1,R_2\}}\mbf{x}\bl)(s)=R_0(s)\mbf{x}+\int_{a}^{b}\bbl[\mbf{I}(s-\theta)R_1(s,\theta)+\mbf{I}(\theta-s)R_2(s,\theta)\bbr]\mbf{x}(\theta)d\theta.
    \end{align*}
    Furthermore, we have the following relations for the indicator function
    \begin{align*}
        \mbf{I}(s-\eta)\mbf{I}(\eta-\theta)&=\begin{cases}
            \mbf{I}(s-\theta),    &\text{if } \eta\in[\theta,s],  \\
            0,                  &\text{else,}
        \end{cases} \\
        \mbf{I}(s-\eta)\mbf{I}(\theta-\eta)&=\mbf{I}(s-\theta)\mbf{I}(\theta-\eta) + \mbf{I}(\theta-s)\mbf{I}(s-\eta)
    \end{align*}
    Using the first relation, it follows that for any $R_1,Q_1\in L_2^{m_s\times n_s}\bl[[a,b]^2\br]$,
    \begin{align*}
     \bbl(\mcl{P}_{\{0,R_1,0\}}\bl(\mcl{P}_{\{0,Q_1,0\}}\mbf{x}\br)\bbr)(s)&=\int_{a}^{s}R_1(s,\eta)\int_{a}^{\eta}Q_1(\eta,\theta)\mbf{x}(\theta)d\theta d\eta \\
     &= \int_{a}^{b}\int_{a}^{b}\mbf{I}(s-\eta)\mbf{I}(\eta-\theta)R_1(s,\eta)Q_1(\eta,\theta)\mbf{x}(\theta)d\theta d\eta \\
     &= \int_{a}^{s}\left[\int_{\theta}^{s}R_1(s,\eta)Q_1(\eta,\theta) d\eta\right] \mbf{x}(\theta)d\theta  
     = \bl(\mcl{P}_{\{0,P_{11},0\}}\mbf{x}\br)(s),
    \end{align*}
    where $P_{11}(s,\theta):=\int_{\theta}^{s}R_1(s,\eta)Q_1(\theta,\eta)d\eta$. Similarly, we can show that
    \begin{align*}
     \bbl(\mcl{P}_{\{0,R_1,0\}}\bl(\mcl{P}_{\{0,0,Q_2\}}\mbf{x}\br)\bbr)(s)&=
     \int_{a}^{s}\left[\int_{a}^{\theta}R_1(s,\eta)Q_2(\eta,\theta)d\eta\right]\mbf{x}(\theta)d\theta + \int_{s}^{b}\left[\int_{a}^{s}R_1(s,\eta)Q_2(\eta,\theta)d\eta\right]\mbf{x}(\theta)d\theta \\
     \bbl(\mcl{P}_{\{0,0,R_2\}}\bl(\mcl{P}_{\{0,Q_1,0\}}\mbf{x}\br)\bbr)(s)&=
     \int_{a}^{s}\left[\int_{s}^{b}R_2(s,\eta)Q_1(\eta,\theta)d\eta\right]\mbf{x}(\theta)d\theta + \int_{s}^{b}\left[\int_{\theta}^{b}R_2(s,\eta)Q_1(\eta,\theta)d\eta\right]\mbf{x}(\theta)d\theta \\
     \bbl(\mcl{P}_{\{0,0,R_2\}}\bl(\mcl{P}_{\{0,0,Q_2\}}\mbf{x}\br)\bbr)(s)&=
     \int_{s}^{b}\left[\int_{s}^{\theta}R_2(s,\eta)Q_2(\theta,\eta)d\eta\right] \mbf{x}(\theta) d\theta = \bl(\mcl{P}_{\{0,0,P_{22}\}}\mbf{x}\br)(s),
    \end{align*}
    proving that the composition of lower-triangular and upper-triangular partial integrals can always be expressed as partial integrals as well. It is also easy easy to see that
    \begin{align*}
        \bbl(\mcl{P}_{\{R_0,0,0\}}\bl(\mcl{P}_{\{0,Q_1,Q_2\}}\mbf{x}\br)\bbr)(s)&=\int_{a}^{s}R_0(s)Q_1(s,\theta)\mbf{x}(\theta)d\theta + \int_{s}^{b}R_0(s)Q_2(s,\theta)\mbf{x}(\theta)d\theta = \bl(\mcl{P}_{\{0,P_{01},P_{02}\}}\mbf{x}\br)(s), \\
        \bbl(\mcl{P}_{\{0,R_1,R_2\}}\bl(\mcl{P}_{\{Q_0,0,0\}}\mbf{x}\br)\bbr)(s)&=\int_{a}^{s}R_1(s,\theta)Q_0(\theta)\mbf{x}(\theta)d\theta + \int_{s}^{b}R_2(s,\theta)Q_2(\theta)\mbf{x}(\theta)d\theta = \bl(\mcl{P}_{\{0,P_{10},P_{20}\}}\mbf{x}\br)(s),
    \end{align*}
    from which it follows that the composition of any 3-PI operators can be expressed as a 3-PI operator as well. Moreover, since we can repeat these steps along any spatial directions, this result extends to more general (2D) PI operators as well. For a full proof, we refer to~\cite{shivakumar2024extension}.
    
\end{proof}

We note again the similarity to matrices: just like the product of two lower triangular matrices $L_1,L_2$ is a lower triangular matrix $L_3$, the composition of two lower-triangular 3-PI operators $\mcl{P}_{\{0,R_1,0\}},\mcl{P}_{\{0,Q_1,0\}}$ is also a lower-triangular 3-PI operator $\mcl{P}_{\{0,P_{11},0\}}$. Similarly, the product of two upper-triangular 3-PI operators $\mcl{P}_{\{0,0,R_2\}},\mcl{P}_{\{0,0,Q_2\}}$ is also an upper-triangular 3-PI operator $\mcl{P}_{\{0,0,P_{22}\}}$, but the composition of lower- and upper-triangular PI operators need not be lower- or upper-triangular -- just as with matrices. Finally, the composition of a diagonal operator $\mcl{P}_{\{R_0,0,0\}}$ with a lower- or upper-diagonal PI operator is also respectively lower- or upper-diagonal.

\section{Adjoint of PI Operators}\label{sec:PI_theory:Adjoint}

To define the adjoint of a PI operator $\mcl{R}\in\Pi^{\textnormal{m}\times\text{normal{n}}}$, we first recall the definition of the function space that these operators map: $Z^{\textnormal{n}}\br[[a,b],[c,d]\bl]:=${\scriptsize$\left[\begin{array}{l}
\R^{n_0}\\ L_2^{n_s}[a,b]\\ L_2^{n_r}[c,d]\\ L_2^{n_2}\bl[[a,b]\times[c,d]\br]
\end{array}\right]$}, where $\text{n}:=\{n_0,n_s,n_r,n_2\}$, with each element being a coupled state of finite-dimensional variables $x_0\in\R^{n_0}$, 1D functions $\mbf{x}_s\in L_2^{n_s}[a,b]$ and $\mbf{x}_r\in L_2^{n_r}[a,b]$, and 2D functions $\mbf{x}_2\in L_2^{n_2}\bl[[a,b]\times[c,d]\br]$. We endow this space with the inner product
\begin{align*}
    \ip{\mbf{x}}{\mbf{y}}_{Z}&=\ip{x_0}{y_0}+\ip{\mbf{x}_s}{\mbf{y}_s}_{L_2}+\ip{\mbf{x}_r}{\mbf{y}_r}_{L_2}+\ip{\mbf{x}_2}{\mbf{y}_2}_{L_2}    \\
    &=x_0^T y_0 + \int_{a}^{b}[\mbf{x}_s(s)]^T\mbf{y}(s)ds + \int_{c}^{d}[\mbf{x}_r(r)]^T\mbf{y}(r)dr + \int_{a}^{b}\int_{c}^{d}[\mbf{x}_2(s,r)]^T\mbf{y}(s,r)dr ds
\end{align*}
Defining this inner product, we can also define the adjoint of PI operators.

\begin{lem}
 For any PI parameters $R\in\mcl{N}^{\text{m}\times\text{n}}\bl[[a,b],[c,d]\br]$, there exist unique parameters $Q\in\mcl{N}^{\text{n}\times\text{m}}\bl[[a,b],[c,d]\br]$ such that
 \begin{align*}
    \bl(\mcl{P}[R]\br)^*=\mcl{P}[Q],
 \end{align*}
 where $\mcl{P}^*$ denotes the adjoint of a PI operator $\mcl{P}$. 
 That is, for any $\mbf{x}\in Z^{\textnormal{n}}\bl[[a,b],[c,d]\br]$ and $\mbf{y}\in Z^{\textnormal{m}}\bl[[a,b],[c,d]\br]$,
 \begin{align*}
    \ip{\mcl{P}[R]\mbf{x}}{\mbf{y}}_{Z}=\ip{\mbf{x}}{\mcl{P}[Q]\mbf{y}}_{Z}.
 \end{align*}
\end{lem}
\begin{proof}
 We outline the proof only for 4-PI operators. In particular, let $\textnormal{n}=\{n_0,n_1,0,0\}$ and $\textnormal{m}=\{m_0,m_1,0,0\}$, and let $B=\sbmat{P&Q_1\\Q_2&\{R_0,R_1,R_2\}}$ define a 4-PI operator $\mcl{P}[B]\in\Pi^{\textnormal{n}\times\textnormal{m}}$. Define $\hat{B}=\sbmat{\hat{P}&\hat{Q}_1\\\hat{Q}_2&\{\hat{R}_0,\hat{R}_1,\hat{R}_2\}}$, where
 \begin{align*}
     \hat{P}&=P^T,           &   \hat{Q}_1(s)&=Q_2^T(s),    &   \hat{R}_1(s,\theta)&=R_2^T(\theta,s),   \\
     \hat{Q}_2(s)&=Q_1^T(s), &      \hat{R}_0(s)&=R_0^T(s),  &   \hat{R}_2(s,\theta)&=R_1^T(\theta,s), 
 \end{align*}
 Then, for arbitrary $\mbf{x}=\sbmat{x_0\\\mbf{x}_1}\in Z^{\textnormal{n}}$ and $\mbf{y}=\sbmat{y_0\\\mbf{y}_1}\in Z^{\textnormal{m}}$, we note that
 \begin{align*}
     \ip{\mcl{P}[B]\mbf{x}}{\mbf{y}}_{Z}&=[Px_0]^T y_0 + \left[\int_{a}^{b}Q_1(s)\mbf{x}_1(s)ds\right]^Ty_0 + \int_{a}^{s}[Q_2(s)x_0]^T\mbf{y}_1(s)ds +\int_{a}^{b}[R_0(s)\mbf{x}_1(s)]^T\mbf{y}_1(s)ds    \\
    &\qquad + \int_{a}^{b}\left[\left(\int_{a}^{b}\mbf{I}(s-\theta)R_1(s,\theta)+\mbf{I}(\theta-s)R_2(s,\theta)\right)\mbf{x}_1(\theta)d\theta\right]^T\mbf{y}_1(s)ds    \\
    &=x_0 [P^T y_0] + x_0^T\left[\int_{a}^{b}Q_2^T(s)\mbf{y}_1(s)ds\right] + \int_{a}^{b}\mbf{x}_1^T(s)\left[Q_1^T(s)y_0\right]ds +\int_{a}^{b}\mbf{x}_1^T(s)\left[R_0^T(s)\mbf{y}_1(s)\right]ds \\
    &\qquad +\int_{a}^{b}\mbf{x}_1^T(s)\left[\left(\int_{a}^{b}\mbf{I}(-\theta-s)R_1(\theta,s)+\mbf{I}(s-\theta)R_2(\theta,s)\right)\mbf{y}_1(\theta)d\theta\right]ds   \\
    &=\ip{\mbf{x}}{\mcl{P}[\hat{B}]\mbf{y}}_{Z}.
 \end{align*}
 
\end{proof}
We note again the similarities with matrices: Just like the adjoint of a matrix can be determined by switching the rows and columns, the adjoint of a 3-PI operator is determined by switching the primary and dual variables $(s,\theta)$, as well as switching the lower- and upper-triangular parts.

\section{Inversion of PI operators}\label{sec:PI_theory:Inverse}
In this section, we address the problem of invertibility, required to constructing the controllers and estimators from the feasible solutions to the LPIs described in the main text. For example, the controller gain $\mcl{K}$ is given by the relation $\mcl{K}=\mcl{Z}\mcl{P}^{-1}$. $\mcl{P}^{-1}$, although is a 4-PI operator, may not have polynomial parameters. Hence the inverse is approximated numerically. To find the inverse of a 4-PI operator, $\fourpi{P}{Q_1}{Q_2}{R_i}$, we first find the inverse of 3-PI operators of the form $\threepi{R_i}$.

\subsection{Inversion of 3-PI operators}
	First, we note that any matrix-valued polynomial $H(s,\theta)$ can be factored as $F(s)G(\theta)$. Then, for any given 3-PI operator $\threepi{I,H_1,H_2}$ with matrix-valued polynomial parameters $H_1$ and $H_2$, we have
	\begin{align*}
	\threepi{I,H_1,H_2} &= \threepi{I,-F_1G_1,-F_2G_2},~\text{where}\\
	H_i(s,\theta) &= -F_i(s)G_i(\theta),
	\end{align*}
	for some matrix-valued polynomials $F_i$ and $G_i$. We can now find an inverse for $\threepi{I,H_1,H_2}$ using the following result. 
	\begin{lem}\label{lem:3piINV}
	Suppose $F_1: [a,b]\to \R^{n\times p}$, $G_1:[a,b]\to \R^{p\times n}$, $F_2:[a,b]\to\R^{n\times q}$, $G_2:[a,b]\to \R^{q\times n}$ and $U$ is the unique function that satisfies the equation 
	\begin{align*}
	U(s) &= I_{(p+q)} +\int_a^s \bmat{G_1(t)F_1(t) & G_1(t)F_2(t)\\-G_2(t)F_1(t)&-G_2(t)F_2(t)} U(t) dt, 
	\end{align*} 
	where $U$ is partitioned as
	\[
	U = \bmat{U_{11}&U_{12}\\U_{21}&U_{22}}, \quad U_{11}(s)\in \R^{p\times p}, U_{22}(s)\in \R^{q\times q}.
	\]
	Then, the 3-PI operator $\threepi{I,-F_1G_1,-F_2G_2}$ is invertible if and only if $U_{22}(b)$ is invertible and 
	\begin{align*}
	(\threepi{I,-F_1G_1,-F_2G_2})^{-1} &= \threepi{I,L_1,L_2},
	\end{align*}
	where
	\begin{subequations}\label{eq:Li_inv}
	\begin{align}
	L_1(s,t) &= \bmat{F_1(s)&F_2(s)}U(s)V(t)\bmat{G_1(t)\\-G_2(t)}-L_2(s,t),\\
	L_2(s,t) &= -\bmat{F_1(s)&F_2(s)}U(s)PV(t)\bmat{G_1(t)\\-G_2(t)},
	\end{align}
	\end{subequations}
	\begin{align*}
	P &= \bmat{0_{p\times p}&0_{p\times q}\\U_{22}(b)^{-1}U_{21}(b)&I_{q}},	
	\end{align*}
	and $V$ is the unique function satisfying the equation 
	\[
	V(t) = I_{(p+q)} - \int_a^t V(s)\bmat{G_1(s)F_1(s) & G_1(s)F_2(s)\\-G_2(s)F_1(s)&-G_2(s)F_2(s)} ds.
	\]	
	\end{lem}
	\begin{proof}
	Proof can be found in \cite[Chapter IX.2]{gohberg2013classes}.
	\end{proof}
	Using Lemma 2.2. of \cite[Chapter IX.2]{gohberg2013classes}, we can use an iterative process and numerical integration to approximate $U$ and $V$ functions in the above result at discrete spatial points. Them, a polynomial that best fits the data, in a least-squares sense, can be used as an approximation. Thus, we find an approximated inverse for 3-PI operators of the form $\threepi{I,R_1,R_2}$ where $R_i$ are matrix-valued polynomials. By extension, given an invertible $R_0$, we can obtain the inverse of a general 3-PI operator as shown below.
	\begin{cor}\label{cor:3piINV}
	Suppose $R_0:[a,b]\to \R^{n\times n}$, $R_1,R_2: [a,b]^2\to \R^{n\times n}$, with $R_0$ invertible on $[a,b]$. Then, the inverse of the 3-PI operator, $\threepi{R_i}$, is given by $\threepi{\hat R_0,\hat R_1,\hat R_2}$ where
	\begin{align*}
	&\hat R_0(s) = R_0(s)^{-1}, \quad \hat{R}_i(s,\theta) = L_i(s,\theta)\hat R_0(\theta), \quad i\in \{1,2\},
	\end{align*}
	where $L_1$ and $L_2$ are as defined in \eqref{eq:Li_inv} for functions $F_i$ and $G_i$ such that $F_i(s)G_i(\theta)= R_0(s)^{-1}R_i(s,\theta)$.
	\end{cor}
	\begin{proof} 
	Let $R_i$ be as stated above.
	\begin{align*}
	&(\threepi{R_0(s),R_1(s,\theta),R_2(s,\theta)})^{-1} \\
	&= (\threepi{R_0(s),0,0} \threepi{I,R_0(s)^{-1}R_1(s,\theta),R_0(s)^{-1}R_2(s,\theta)})^{-1} \\
	&= (\threepi{I,R_0(s)^{-1}R_1(s,\theta),R_0(s)^{-1}R_2(s,\theta)})^{-1}(\threepi{R_0(s),0,0})^{-1} \\
	&= \threepi{I,L_1(s,\theta),L_2(s,\theta)}\threepi{R_0(s)^{-1},0,0} \\
	&= \threepi{R_0(s)^{-1},L_1(s,\theta)R_0(\theta)^{-1},L_2(s,\theta)R_0(\theta)^{-1}}.
	\end{align*}
	where $L_i$ are obtained from the Lemma \ref{lem:3piINV} and the composition of PI operators is performed using the formulae in \cite{shivakumar_2019CDC}.

	\end{proof}
	Note the above expressions for the inverse are exact, however, in practice, $R_0^{-1}$ may not have an analytical expression (or very hard to determine). Thus, finding $F_i$ and $G_i$ such that $F_i(s)G_i(\theta)= R_0(s)^{-1}R_i(s,\theta)$ may not be possible. To overcome this problem, we approximate $R_0^{-1}$ by a polynomial which guarantees that $R_0^{-1}R_i$ are polynomials and can be factorized into $F_i$ and $G_i$. Using this approach, we can find an approximate inverse for $\threepi{R_i}$ using Lemma \ref{lem:3piINV}. 

	\subsection{Inversion of 4-PI operators}
	Given, $R_0, R_1, R_2$ with $R_0$ invertible, we proposed a way to find the inverse of the operator $\threepi{R_i}$. Now, we use this method to find the inverse of a 4-PI operator $\fourpi{P}{Q_1}{Q_2}{R_i}$. Given $P$, $Q_1$, $Q_2$ and $R_i$ with invertible $P$ and $R_0$. For this aim, we assume invertibility of the 3-PI operator $\threepi{R_i}$ in the following result.
	\begin{cor}\label{lem:inverse}
	Suppose $P\in \R^{m\times m}$, $Q_1:[a,b]\to \R^{m\times n}$, $Q_2:[a,b]\to \R^{n\times m}$ $R_0:[a,b]\to\R^{n\times n}$ and $R_1,R_2:[a,b]^2\to \R^{n\times n}$ are matrices and matrix-valued polynomials such that $\threepi{R_i}^{-1}$ is invertible according to Cor.~\ref{cor:3piINV} and call $\threepi{\hat{R}_i}:=\threepi{R_i}^{-1}$. Then, $\mcl P := \fourpi{P}{Q_1}{Q_2}{R_i} \in \Pi_4$ is invertible if and only if the matrix 
    \begin{equation*}
        T=P-\fourpi{\emptyset}{Q_1}{\emptyset}{\emptyset} \threepi{\hat{R}_i} \fourpi{\emptyset}{\emptyset}{Q_2}{\emptyset}
    \end{equation*}
    is invertible. Furthermore
    \begin{align*}
        \mcl P^{-1} = \mcl U \fourpi{T^{-1}}{0}{0}{\hat{R}_i} \mcl V,
    \end{align*}
where 
\begin{align*}
    \mcl U &= \fourpi{I}{0}{0}{\hat{R_i}}\fourpi{I}{0}{-Q_2}{R_i},\\
    \mcl V = &= \fourpi{I}{-Q_1}{0}{\hat{R_i}}\fourpi{I}{0}{0}{\hat{R}_i}.
\end{align*}
	\end{cor}
	\begin{proof}
A proof may be found in Sec. VII of~\cite{shivakumar2022dual}.
	\end{proof}



\section{Composition of Differential and PI operator}\label{sec:PI_theory:Derivative}

Given the well-known relationship between integrals and derivatives (think e.g. the fundamental theorem of calculus), it is natural to assume that the composition of a differential operator and a PI operator may be expressed as a PI operator as well. Unfortunately, this is not true in general, as e.g. the operator $\mcl{P}$ defined as $\bl(\mcl{P}\mbf{v}\br)(s)=P(s)\mbf{v}(s)$ is a PI operator, but there clearly does not exist a PI operator $\mcl{Q}$ such that $\partial_{s}\bl(\mcl{P}\mbf{v}\br)(s)=\bl(\mcl{Q}\mbf{v}\br)(s)$. Nevertheless, if the function $\mbf{v}$ is differentiable, i.e. $\mbf{v}\in H_1$, then we can always express the derivative of $\bl(\mcl{P}\mbf{v}\br)(s)$ for a PI operator $\mcl{P}$ in terms of $\mbf{v}(s)$ and $\partial_{s}\mbf{v}(s)$ as $\partial_{s}\bl(\mcl{P}\mbf{v}\br)(s)=\bl(\mcl{Q}\sbmat{\mbf{v}\\\partial_{s}\mbf{v}}\br)(s)$, as we show in the next lemma.


\begin{lem}\label{lem:diff_PI}
	Suppose $\fourpi{P}{Q_1}{Q_2}{R_i}:\R^m\times H_1^n\to\R^p\times L_2^q$, and define $\partial_s \fourpi{P}{Q_1}{Q_2}{R_i}: \R^m\times H_1^n\times L_2^n\to \R^p \times L_2^q$ as
	\begin{align}\label{eq:diff_PI}
		\partial_s\fourpi{P}{Q_1}{Q_2}{R_i}= \fourpi{0}{0}{\bar{Q}_2}{\bar{R}_i}
	\end{align}
	where $\bar{Q}_2(s) = \partial_s Q_2(s)$, $\bar{R}_0(s) = \bmat{\partial_s R_0(s)+R_1(s,s)-R_2(s,s) & R_0(s)}$ and $\bar{R}_i(s,\theta)=\bmat{\partial_s R_i(s,\theta)&0}$ for $i\in\{1,2\}$. Then, for any $x\in\R^m$, $\mbf{x}\in H_1^n$,
	\begin{align*}
		\partial_s\left(\fourpi{P}{Q_1}{Q_2}{R_i}\bmat{x\\\mbf{x}}\right)
		=\left(\partial_s\fourpi{P}{Q_1}{Q_2}{R_i}\right)
		\bmat{x\\\mbf{x}\\\partial_s \mbf{x}}
	\end{align*}
\end{lem}

\begin{proof}
    The result can be easily derived using the Leibniz integral rule, stating that for any $F\in H_1[a,b]^2$,
    \begin{align*}
        \frac{d}{ds}\left(\int_{L(s)}^{U(s)}F(s,\theta) d\theta\right) = F(s,U(s))\frac{d}{ds}U(s) - F(s,L(s))\frac{d}{ds}L(s) + \int_{L(s)}^{U(s)}\frac{d}{ds}F(s,\theta)d\theta.
    \end{align*}
    For a similar result for PI operators in 2D, we refer to~\cite{jagt_2021PIEACC}.
\end{proof}


\section{Matrix Parametrization of Positive Definite PI Operators}\label{sec:PI_theory:Positive_PI}

In order to be able to solve optimization programs involving PI operator $\mcl{P}$, we need to be able to enforce positivity constraints $\mcl{P}\succcurlyeq 0$. For this, we parameterize PI operators by positive matrices, expanding them as $\mcl{P}=\mcl{Z}^*P\mcl{Z}$ for a fixed operator $\mcl{Z}$, and positive semidefinite matrix $P\succcurlyeq 0$. The following theorem provides a sufficient condition for positivity of a 4-PI operator defined as
\begin{align}\label{eq:4pi}
    \bl(\mcl{P}\sbmat{P,&Q_1\\Q_2,&\{R_i\}}\mbf{x}\br)(s)=
    \left[\begin{array}{ll}
        Px_0        \hspace*{-0.1cm}~& +\ \int_{a}^{b}Q_1(s)\mbf{x}_1(s)ds  \\
        Q_2(s)x_0   \hspace*{-0.1cm}& +\ R_{0}(s)\mbf{x}_1(s) + \int_{a}^{s}R_{1}(s,\theta)\mbf{x}_1(\theta)d\theta + \int_{s}^{b}R_{2}(s,\theta)\mbf{x}_1(\theta)d\theta
    \end{array}\right]
\end{align}
for $\mbf{x}=\bmat{x_0\\\mbf{x}_1}\in \bmat{\R^{n_0}\\L_2^{n_1}[a,b]}$. This result allows us to parameterize a cone of positive PI operators as positive matrices, implement LPI constraints as LMI constraints, allowing us to solve optimization problems with PI operators using semi-definite programming solvers. 
\begin{thm}\label{th:positivity}
	For any functions $Z_1:[a, b]\to\R^{d_1\times n}$, $Z_2:[a, b]\times [a, b]\rightarrow\mathbb{R}^{d_2\times n}$, if $g(s)\geq0$ for all $s\in[a,b]$ and 
	{\small
		\begin{align}
		P&= T_{11}\myint g(s) ds,\nonumber\\
		Q(\eta) &= g(\eta)T_{12}Z_1(\eta)+\myintb{\eta} g(s)T_{13}Z_2(s,\eta)\text{d}s+ \myinta{\eta}g(s) T_{14}Z_2(s,\eta)\text{d}s, \nonumber\\
		R_1(s,\eta) &=g(s)Z_1(s)^{\top}T_{23}Z_2(s,\eta)+g(\eta)Z_2(\eta,s)^{\top}T_{42}Z_1(\eta)+\myintb{s}g(\th)Z_2(\theta,s)^{\top}T_{33}Z_2(\theta,\eta)\text{d}\theta\nonumber\\
		&\hspace{1.3cm}~~+\int_{\eta}^{s}g(\th)Z_2(\theta,s)^{\top}T_{43}Z_2(\theta,\eta)\text{d}\theta+\myinta{\eta}g(\th)Z_2(\theta,s)^{\top}T_{44}Z_2(\theta,\eta)\text{d}\theta,\nonumber\\
		R_2(s,\eta) &=g(s)Z_1(s)^{\top}T_{32}Z_2(s,\eta)+g(\eta)Z_2(\eta,s)^{\top}T_{24}Z_1(\eta)+\myintb{\eta}g(\th)Z_2(\theta,s)^{\top}T_{33}Z_2(\theta,\eta)\text{d}\theta\notag\\
		&\hspace{1.3cm}~~+\int_{s}^{\eta}g(\th)Z_2(\theta,s)^{\top}T_{34}Z_2(\theta,\eta)\text{d}\theta+\myinta{s}g(\th)Z_2(\theta,s)^{\top}T_{44}Z_2(\theta,\eta)\text{d}\theta,\nonumber\\
		R_0(s) &= g(s)Z_1(s)^{\top} T_{22} Z_1(s).
		\label{eq:TH}
		\end{align}}
	where
	\begin{align*}
	T= \begin{bmatrix}
	T_{11} & T_{12} & T_{13} & T_{14}\\
	T_{21} & T_{22} & T_{23} & T_{24}\\
	T_{31} & T_{32} & T_{33} & T_{34}\\
	T_{41} & T_{42} & T_{43} & T_{44}
	\end{bmatrix}\succcurlyeq 0,
	\end{align*}\\
	then the operator $\fourpi{P}{Q_1}{Q_2}{R_i}$ ~as defined in~\eqref{eq:4pi} is positive semidefinite, i.e. $\ip{\mbf{x}}{\fourpi{P}{Q_1}{Q_2}{R_i}\mbf{x}}\ge 0$ for all $\mbf{x}\in \R^m\times L_2^n[a,b]$.
\end{thm}
To see the PIETOOLS implementation, check Section \ref{sec:poslpivar}. For an extension of this result to 2D PI operators, we refer to~\cite{jagt_2021PIEACC}.

\bibliographystyle{plain}
\bibliography{user_manual}

\appendix

\end{document}